\documentclass[12pt]{amsart}
\usepackage{mathrsfs}
\usepackage{color}
\usepackage[latin1]{inputenc}
\usepackage{fullpage}
\usepackage{amsmath}
\usepackage{amscd}
\usepackage{amstext}
\usepackage{amsbsy}
\usepackage{amsopn}
\usepackage{amsthm}
\usepackage{amssymb}
\usepackage{amsxtra}
\usepackage{verbatim}
\def\o{\otimes}
\title{Free transport for convex potentials}
\author[Yoann Dabrowski, Alice Guionnet, Dima Shlyakhtenko]{Yoann Dabrowski, Alice Guionnet, Dima Shlyakhtenko}
\thanks{YD: dabrowski@math.univ-lyon1.fr, Universit\'e de Lyon, Universit\'e Lyon 1, Institut Camille Jordan UMR 5208, 43 blvd. du 11 novembre 1918, F-69622 Villeurbanne cedex, France\\
AG: alice.guionnet@ens-lyon.fr,
 Universit\'e de Lyon, Ecole Normale Superieure, 46 all\'ee d'Italie, 69007, Lyon, France. Research supported by NSF Grant DMS-1307704 and Simons foundation.\\
DS: shlyakht@math.ucla.edu, Department of Mathematics, UCLA, Los Angeles, CA 90095, USA. Research supported by NSF Grant DMS-1500035\\
}

\newcounter{Step}
\newenvironment{step}[0]{\bigskip\addtocounter{Step}{1}\noindent\textbf{Step \theStep~:} }{\begin{flushright}\tiny \end{flushright}}

\def\N{\mbox{I\hspace{-.15em}N} }
\def\R{\mbox{I\hspace{-.15em}R} }

\def\C{\mathbb{C}}

\def\o{\otimes}

\def\H{\mathscr{H}}

\def\Tr{{\rm Tr}}

\def\H{\mathcal{H}}

\theoremstyle{plain}
      \newtheorem{Theorem}{Theorem}
      \newtheorem{Lemma}[Theorem]{Lemma}
      \newtheorem{Corollary}[Theorem]{Corollary}
      \newtheorem{Proposition}[Theorem]{Proposition}
     
      \theoremstyle{definition}
      \newtheorem{definition}[Theorem]{Definition}
      \newtheorem{Assumption}[Theorem]{Assumption}

     \theoremstyle{remark}

\def\oeh#1{\underset{#1}{\overset{\operatorname{eh}}{\otimes}}}
\def\oehc#1{\underset{#1, c}{\overset{\operatorname{eh}}{\otimes}}}
\def\oh#1{\underset{#1}{\overset{\operatorname{h}}{\otimes}}}

\begin{document}
\begin{abstract}
We construct non-commutative analogs of transport maps among free Gibbs state satisfying a certain convexity condition.  Unlike previous constructions, our approach is non-perturbative in nature and thus can be used to construct transport maps between free Gibbs states associated to potentials which are far from quadratic, i.e., states which are far from the semicircle law.  An essential technical ingredient in our approach is the extension of free stochastic analysis to non-commutative spaces of functions based on the Haagerup tensor product.
\end{abstract}
\maketitle
\tableofcontents
\sloppy

\section{Introduction} 
A transport map between two probability measures  is  a function  pushing the first measure onto the second. Finding transport maps which minimize a certain cost function  is the central question in transportation theory.
It was formalized by Monge in the eighteenth century, studied by Kantorovich during World War II and has known major advances in the last twenty years, starting with a work of Brenier \cite{brennier:polarFact}, see also the very inspiring book by Villani \cite{cedric}.
In fact, the mere  existence of a transport map is itself not completely trivial and was shown by von Neumann in 1930s, under very weak assumptions, as part of the program to classify measure spaces.  

A central question is to find appropriate generalizations of this result to the non-commutative setting, where measures are replaced by non-commutative distributions, that is, tracial states. In this case, there is no notion of density but in certain instances arising in Voiculescu's free probability theory, integration by parts makes sense.  It gives the adjoint in $L^2$  of Voiculescu's free difference quotient \cite{dvv:entropy5}, and is often a (cyclic) derivative of a non-commutative function that we call potential. 

Non commutative laws which are characterized by such an integration by parts formula are called free Gibbs laws.
In \cite{alice-shlyakhtenko-transport}, two of the authors of this article constructed transport maps between a class of free Gibbs laws. They used   ideas going back to Monge and Amp\`ere, based on the remark that transport maps must satisfy an equation given by the change of variables formula.  Solving this equation yields a transport map. Unfortunately, this equation was only solved in \cite{alice-shlyakhtenko-transport} in the case of potentials which are small perturbations of quadratic potentials, i.e., certain small perturbations of Voiculescu's free semicircular law.
However, already this result yielded isomorphisms between the associated $C^*$ and von Neumann algebras in such perturbative situations, solving a number of open questions \cite{voiculescu:conjectureAboutPotentials}. In particular, this approach was used to show that the  $C^*$ and von Neumann algebras of $q$-Gaussian laws \cite{Speicher:q-comm} are isomorphic for sufficiently small values of $q$.

The goal of the present article is to consider non-perturbative situations. We will see that we can tackle situations where the potential is ``strictly convex'' (in a sense we will make precise later in the paper). The idea is once again to use a non-commutative version of the Monge-Amp\`ere equation, but to solve it by interpolating the potential between the two given laws. This requires to solve a Poisson type equation. The latter, in strictly convex situations, can be solved by using the associated (free) semi-group. However, this program meets several difficulties in the non-commutative setting. First, smoothness properties of the semi-group were so far not studied. Furthermore, the appropriate 
notion of convexity has not yet been formulated. We detail our  framework in Section \ref{notation}, leaving to the appendix the elaboration of 
most of its properties. In Section \ref{sg}, we study the semi-group defined in this framework and derive its properties. Based on this, we finally construct the transport map in Section \ref{cons}. 

In the rest of this section, we detail the classical construction of transport maps from which we took our inspiration, and explain how it generalizes to the case of a single non-commutative variable. We then consider the general non-commutative multi-variable case and state our main theorem.

\subsection{Classical construction of transport maps}
For any suitable real-valued function $U$  from $\mathbb R^d$ to $\mathbb R$ we define  the probability measure
 $$\mu_U(dx)=\frac{1}{Z_{U}}e^{-U(x)}dx, \qquad Z_U=\int e^{-U(x)} dx\,.$$
We let $V$ and $V+W$ be two functions going fast enough to infinity so that $Z_V$ and $Z_{V+W}$ are finite. We would like to construct $F:\mathbb R^d\mapsto \mathbb R^d$ so that 
$\mu_{V+W}=F\#\mu_V$, i.e.,  so that for all test functions $h$
$$\int h(F(y))d\mu_V(y)=\int h(x)d\mu_W(x)=\int h(F(y)){\rm Jac}(F)(y) e^{-(V+W)(F(y))}dy/Z_W$$
where  ${\rm Jac}(F)$ denotes the Jacobian of $F$.  We have simply performed the change of variables $x=F(y)$ in the last line, assuming that  $F$ is $C^1$.
We therefore deduce that $F$ should satisfy the transport equation:
\begin{equation}\label{trans}
V(y)=(V+W)(F(y))-\ln {\rm Jac}(F)(y)+C
\end{equation}
for almost all $y$ where we set $C=\ln Z_{V+W}-\ln Z_V$. 

If $V-W$ is small we can seek a solution $F$ which is close to identity, so that its Jacobian stays 
away from the zero and therefore does not get close to the singularity of the logarithm. The resulting equation can in turn be 
 solved by the implicit function theorem. Such arguments were extended to the non-commutative setting in \cite{alice-shlyakhtenko-transport}.

To solve the transport equation in a non-perturbative situation, we shall in this article proceed by interpolating the potential. Namely, let us consider  potentials $V_\alpha=\alpha W+V$ and seek to  construct a  transport map $F_\alpha$ 
of $\mu_V$ onto $\mu_{V_\alpha}$. The advantage of smooth interpolation is that transporting $\mu_{V_\alpha}$ onto $\mu_{V_{\alpha+\varepsilon}}$ can a priori be solved for $\varepsilon$ small enough by the previous pertubative arguments, and the full transport $F_1=F$ of $\mu_V$ onto $\mu_W$ can then be recovered by integration along the interpolation. 

In fact, we shall solve the transport equation \eqref{trans} under the additional restriction that $F$ evolves according to a gradient 
flow: $\partial_\alpha F_\alpha=\nabla g_\alpha(F_\alpha)$. It turns out that $g$ must then be a  solution of the Poisson equation
\begin{equation}\label{poisson}L_{V_\alpha} g_\alpha=W+\partial_\alpha \ln Z_{V_\alpha}\,,\end{equation}
with $L_{V_\alpha}=\Delta-\nabla V_\alpha.\nabla$ the infinitesimal generator of the diffusion having $\mu_{V_\alpha}$ as its stationary measure. 
Solving the Poisson equation \eqref{poisson} amounts to inverting $L_{V_\alpha}$, that is, finding the Green function of the differential operator $L_{V_\alpha}$. This is a well known problem which can be solved under various boundary conditions or growth of $V$ at infinity. To simplify we shall assume that  $V_\alpha$ (that is $V$ and $V+W$)
 are uniformly convex. This insures that  the semi-group $P_s^\alpha=e^{sL_{V_\alpha}}$ converges uniformly towards the Gibbs measure $\mu_{V_\alpha}$ as $s$ goes to infinity. More precisely, there exists some $c>0$ such that for all Lipschitz functions $f$ with bounded 
 Lipschitz norm $\|f\|_L$ we have
 $$\|P_s^\alpha f-\mu_{V_\alpha}(f)\|_\infty\le 2 e^{-cs}\|f\|_L\,.$$
 As a consequence we can solve the Poisson equation \eqref{poisson} by setting
 \begin{equation}\label{solve}
 g_\alpha(x)=\int_0^\infty P^\alpha_s(W+\partial_\alpha \ln Z_{V_\alpha}) (x) ds\end{equation}
 where we noticed  that $\mu_{V_\alpha}\left(W+\partial_\alpha \ln Z_{V_\alpha}\right)=0$.
 Hence we see that in the classical setup \eqref{poisson} can be solved thanks to the associated semi-group. Moreover, by smoothness of $x\mapsto P^\alpha_s (W)(x)$, we see that $g_\alpha$ is smooth if $W$ is.
  To conclude, all that remains is to solve the transport equation 
 $\partial_\alpha F_\alpha=\nabla g_\alpha (F_\alpha)$.  In the rest of this article we generalize this strategy to the free probability framework.

  Let us first investigate the free set-up in the one variable case. Typically, one should think about the non-commutative law of one variable as
  the asymptotic spectral measure of a random matrix, confined by a potential $V$: the joint law of these eigenvalues is  given by
 $$dP_N^V(\lambda_1,\ldots,\lambda_N)=\frac{1}{Z_N^V}\prod_{1\le i\neq j\le N}|\lambda_i-\lambda_j| \exp\{-N\sum_{i=1}^N V(\lambda_i)\} \prod_{1\le i\le N} d\lambda_i\,.$$
 It is then well known (see e.g. \cite{guionnet-anderson-zeitouni}) that the spectral measure $L_N=\frac{1}{N}\sum_{i=1}^ N \delta_{\lambda_i}$ converges almost surely to the equilibrium measure $\mu_V$, which is  characterized by the fact that the function
 \begin{equation}\label{char}V(x)-2\int \ln |x-y|d\mu_V(y) \end{equation}
 is equal to a constant $c_V$ on the support of $\mu_V$ and is greater than this constant outside of the support. This equation implies the Schwinger-Dyson equation
 \begin{equation}
 \label{SDint}
 2\ P.V.\int \frac{1}{x-y} d\mu_V(y)= V'(x),\quad \mu_V\mbox{ a.s.}\end{equation}
 where $P.V.$ denotes the principal value.
We will call a free Gibbs law with potential $V$  a solution to \eqref{SDint}.  It may not be unique; in fact, there is a continuum of solutions as soon as solutions have disconnected support: a solution corresponds  to any choice  of masses of the connected pieces of the support. 
 This is not the case when  $V$ is uniformly convex. In this case, there is a unique solution and it has connected support. The interest in Schwinger-Dyson equation 
 is that it can be interpreted as an integration by parts identity for the non-commutative derivative $\partial f(x,y):=\frac{f(x)-f(y)}{x-y}$ since it implies that
 $$\int \int \frac{f(x)-f(y)}{x-y} d\mu_V(x)d\mu_V(y)=\int f(x)V'(x) d\mu_V(x)\,.$$
 As there is no notion of density in free probability, integration by parts can be seen as an important way to classify measures. 
 Moreover, as we shall soon describe, there is a natural generalization of free Gibbs laws to the multi-variable setting.

 Let now $V,W$  be two potentials. 
 We would like to construct a transport map from the Gibbs law $\mu_V$ with potential $V$ to the Gibbs law $\mu_{V+W}$  with potential $V+W$.
 We can follow the previous scheme and seek  $g_\alpha$ satisfying : $\partial_\alpha F_\alpha= g_\alpha' (F_\alpha)$ and  $F_\alpha\#\mu_V=\mu_{V_\alpha}$. By \eqref{char}, we find that $\mu_{V_\alpha}$ almost surely we must have
 \begin{equation}\label{tot}\Delta_{V_\alpha} g_\alpha(x):=2\int \frac{g'_\alpha(x)-g'_\alpha(y)}{x-y} d\mu_{V_\alpha}(y)-V_\alpha'(x) g'_\alpha(x)=W-\partial_\alpha c_{V_\alpha}\,.\end{equation}
 We recognize on the left  hand side the infinitesimal generator  $\Delta_{V_\alpha}$
 of the free diffusion driven by a free Brownian motion, \cite{biane-speicher0}.  More precisely, the infinitesimal generator of the free diffusion is given by
 $$\Delta_{V_\alpha}f(x)=2\mathbb E\left[ \frac{f'(x)-f(X)}{x-X} \right]-V'(x) f'(x)$$
 if $X$ has the same law as $x$. 
 
 The fact that this generator depends on the law of the variable complicates the resulting theory quite a lot. In particular,
 the operators $e^{s \Delta_{V_\alpha}}$ acting on the obvious space of functions do not form a semigroup. 
To restore the semi-group property, we have to enlarge the set of test functions to be functions of not just the real variable $x$, but also of expectations of this random variable. Our idea here is similar to the one introduced in \cite{Ceb13}.
 This in turn changes the generator of the diffusion to also involve differentiation under the expectation: we denote $\delta_V$ the derivative $\delta_V\mathbb E[f]=\mathbb E [\Delta_V f]$. 
 We can now check that $(e^{s (\Delta_{V_\alpha}+\delta_{V_\alpha})})_{s\ge 0}$ is a semi-group so that we can apply the previous analysis.

  Note here that when $x$ follows the invariant measure $\mu_{V_\alpha}$,
 $\delta_{V_\alpha}\mu_{V_\alpha}(f)=0$ and therefore the two generators coincide.  Thus invariant measures for the semi-group $(e^{s (\Delta_{V_\alpha}+\delta_{V_\alpha})})_{s\ge 0}$ will satisfy \eqref{tot}.

 As before, we shall solve \eqref{tot} under a gradient form. Again, the natural gradient that we shall use also differentiates under expectation. 
 Namely we let $\mathcal D$ to be given for any smooth functions $f,f_i, i\ge 0$ by 
 $$\mathcal D(f(x)\prod \mathbb E[f_j(x)])=f'(x)\prod \mathbb E[f_j(x)]+\mathbb E[f]\sum_i f_i'(x)\prod_{j\neq i}\mathbb E[f_j] \,.$$
  Then, we shall find a function $\mathcal D g_\alpha$ (of the variable $x$ and the expectation, see Lemma \ref{ODE}),  which satisfies a gradient form of \eqref{tot} (after adding $\delta_V$ to the generator and commuting $\mathcal D$ with $\Delta_{V_\alpha}+\delta_{V_\alpha}$) :
 \begin{equation}\label{cf}\mathcal D(W)=(\Delta_{V_\alpha} +\delta_{V_\alpha})(\mathcal D g_\alpha)+ V''_\alpha  \mathcal D g_\alpha\,.\end{equation}
Having obtained the solution $g_\alpha$, we finally solve
 \begin{equation}\label{transportflow}\partial_\alpha F_\alpha=\mathcal D g_\alpha(F_\alpha)\,.\end{equation}
 
 To make things clearer, let us  transport the measure $P_N^V$ onto $P_N^W$ and only afterwards take the large $N$-limit. Again, we consider the transport of $P_N^V$ onto $P_N^{V_\alpha}$. 
We may  expect, by symmetry, the  flow $F^\alpha=(F^\alpha_1,\ldots,F^\alpha_N)$  for the transport  map to be the gradient of a function of the empirical measure $L_N=\frac{1}{N}\sum \delta_{\lambda_i^N}$: 
$$F_i^\alpha(\lambda)=N\partial_{\lambda_i} G_\alpha(L_N)=\mathcal D G_\alpha(\lambda_i,L_N)\,.$$ The infinitesimal generator $L_V=\Delta-\nabla V.\nabla$ acting on functions 
of the form $F(L_N)=N\prod \frac{1}{N}\sum_{i=1}^N f_j(\lambda_i)$ reads
$$L_V F=\sum_{k}\prod_{j\neq k}  \frac{1}{N}\sum_{i=1}^N f_j(\lambda_i)  \frac{1}{N}\sum_{i=1}^N L_V f_k(\lambda_i) +O(\frac{1}{N})$$
where the last term comes from differentiation of two different functions and is at most of order $1/N$. Hence, when $N$ goes to infinity we see that functions of the distribution of the $\lambda_i$ should not be taken as constant but also differentiated under the expectation. 
Taking the gradient in the Poisson equation \eqref{poisson} shows that we seek  $G_\alpha$ such that for each $i$
$$ ( L_{V_\alpha}+\delta_{V_\alpha})\mathcal D  G_\alpha(\lambda_i) =\mathcal D W(\lambda_i)+V_\alpha''(\lambda_i) \mathcal D G_\alpha(\lambda_i)+O(\frac{1}{N})\,.$$
Hence, taking the large $N$ limit, we expect $G_\alpha$ to be given at first order by the solution $g_\alpha$ of \eqref{cf}.

The final step to finish our construction of the transport map is to introduce a notion of uniform convexity of $V$ such that the associated semi-group converges uniformly and sufficiently rapidly towards the invariant measure
as time goes to infinity (to make sense of the integral over time from $0$ to $\infty$), and such that  
if $f$ is smooth then also $x\mapsto e^{s (\Delta_{V_t}+\delta_{V_t})} f(x) $ is smooth, uniformly in $s$ (to be able to solve the transport equation). 
Our choice of the notion of uniform convexity of $V$ is designed to guarantee such properties. 

\subsection{Construction of transport maps in free probability}
We now want to explain our approach to the main goal of this article, which is to construct transport maps between  non-commutative distributions of several non-commutative variables. In free probability theory, laws of non-commutative variables are defined as linear forms $\tau$ on the space $\mathbb C\langle X_1,\ldots,X_n\rangle$ of polynomials in the self-adjoint  non-commutative letters $X_1,\ldots,X_n$ with coefficients in $\mathbb C$ which have mass one (so that $\tau(1)=1$), and which satisfy the  traciality property ($\tau(PQ)=\tau(QP)$) and the state property ($\tau(PP^*)\ge 0$). Here $*$ denotes the usual involution $(zX_{i_1}\cdots X_{i_k})^*= \bar z X_{i_k}\cdots X_{i_1}$.

An example 
one should  keep in mind is the asymptotic law  of several interacting random matrices with joint law given by
$$d\mathbb P_N^{V}(X_1^N,\ldots,X_n^N)=\frac{1}{Z_N^{V}} \exp\{-N\Tr(V(X_1^N,\ldots,X_n^N))\} d X_1^N\cdots dX_n^N$$
where $dX^N$ is the Lebesgue measure on the space of $N\times N$ Hermitian matrices and $V$ is a self-adjoint polynomial in $\mathbb C\langle X_1,\ldots,X_n\rangle$ so that $Z_N^{V}$ is finite.  In this case $$\tau_{X^N}(P)=\frac{1}{N}\Tr (P(X_1^N,\ldots,X_n^N))$$
is a non commutative law for any self-adjoint matrices $X_1^N,\ldots,X_n^N$.  So is  its expectation under $\mathbb P_N^{V}$ and the limit of these expected value as $N\to \infty$ (if the limit exists). 

Existence of such an (almost sure and $L^1(\mathbb P^V_N)$) limit was proven when $V$ is a small perturbation of a quadratic potential  \cite{guionnet-edouard:combRM} and when $V$ satisfies some property of convexity \cite{alice-shlyakhtenko-freeDiffusions}.

 In this paper we will introduce a more suitable notion of convexity yielding as well existence and uniqueness of a limit $\tau_V$. We shall see that it includes the case of quartic potentials.   By integration by parts, we see that the limit $\tau_V$ must satisfy 
that for any polynomial $P$
\begin{equation}\label{SD}\tau_V\otimes\tau_V(\partial_i P)=\tau_V(P \mathcal D_i V)\end{equation}
where $\partial_i$ is the free difference quotient with respect to the $i$th derivative from $\mathbb C\langle X_1,\ldots,X_n\rangle$ to $\mathbb C\langle X_1,\ldots,X_n\rangle\otimes \mathbb C\langle X_1,\ldots,X_n\rangle$ given by
$$\partial_i (PQ)=\partial_i(P)\times 1\otimes Q+P\otimes 1\times\partial_i Q,\qquad \partial_i X_j=1_{i=j}1\otimes 1\,,$$
and $\mathcal D_i=m\circ \partial_i $ the cyclic derivative, $m(a\times b)=ba$.  When $V=\sum_{i=1}^n X_i^2$, $\sigma_n:=\tau_{\sum_{i=1}^n X_i^2}$ is uniquely given recursively by \eqref{SD} and is the law of $n$ free semicircle variables.
In general, we say that a non-commutative law $\tau_V$ satisfying \eqref{SD} is a free Gibbs law with potential $V$. Alternatively we say that the conjugate variables $(\partial_i^*(1\otimes 1))_{1\le i\le n}$ are equal to the cyclic gradient $(\mathcal D_i V)_{1\le i\le n}$. 

The goal of this paper is to construct non-commutative transport maps between $\tau_V$ and $\sigma_n$, following the ideas developed in the previous section. In fact, constructing the transport map as the solution of the transport equation \eqref{transportflow}
where $g_\alpha$ is solution of a Poisson equation \eqref{cf} is a natural analogue thanks to existence of free diffusion and free semi-groups.
However, this program meets several issues that have to be addressed.

 \begin{itemize}
 \item One of the key  point to construct the solution to Poisson equation was the fast convergence of the semi-group towards the free Gibbs law. In the free context, it is well known that semi-groups with deep double well potentials do not always converge.  It is therefore natural to search for the appropriate notion of convexity in the non-commutative setting, which would imply convergence of the semi-group as time goes to infinity, uniformly on the initial condition. 
   In \cite{alice-shlyakhtenko-freeDiffusions}, the notion of convexity that was used turns out to be too strong to include many examples. It assumed that for all $n$-tuples of self-adjoint variables $(X,Y)$  bounded by some $R$, 
  $$\sum_{i=1}^n \left((\mathcal D_i V(X)-\mathcal D_iV(Y))(X_i-Y_i)+(X_i-Y_i)(\mathcal D_i V(X)-\mathcal D_iV(Y))\right)$$
  is non-negative. This is  not satisfied by $V(X)=X^4$ as can be checked by taking $(X,Y)$ to be two $2\times 2$ matrices given by $X_{11}=1,X_{12}=X_{21}=0,X_{22}=-6$, $Y_{11}=1,Y_{12}=Y_{21}=\sqrt{11}/4, Y_{22}=-5$.
  It would be more natural to assume that the Hessian of $\Tr V(X_1^N,\ldots,X_n^N)$ is bounded below for any $n$-tuple of Hermitian matrices $X_1^N,\ldots,X_n^N$. However, this Hessian lives in a tensor product space and saying that it is non-negative depends on the topology with which we equip the tensor product. We shall see that a good  topology is given by  the extended Haagerup tensor product and prove that our definition includes the case of quartic potentials.
  \item As in the one variable case, we have to consider functions not only of the variables but also of the expectation and the semi-group must also differentiate under expectation. Hence, we have to develop  free stochastic calculus applied to such functions.
  \item The solution of the Poisson equation is given in terms of the semi-group,
  and we need to show existence and smoothness of the transport maps which are the solution of the transport equation driven by this solution. This requires us to show that the semi-group acts  smoothly on appropriate spaces of non-commutative functions,
  and also understand its image under the cyclic gradient. 
   \end{itemize}
   
   We next state our result.  In Section \ref{diffopsection} we define several differential operators acting on functions of several non-commutative variables, some of them being well known in free probability, such as the difference quotient and the cyclic gradient. We extend their definition to functions which also depend on expectations, in order to define a proper semi-group on the appropriate function spaces. We then define the notion of $(c,R)$ $h$-convexity  of a function in
  Definition \ref{defhconvex}. It states that the Hessian of this function is bounded below by $cI$ in the extended Haagerup tensor product, uniformly when evaluated on non-commutative variables bounded by $R$. An important point is that this notion is stable under addition.
  We then show in Proposition \ref{ConvSDE} that the free SDE with strictly $h$-convex potential converges as time goes to infinity towards a free Gibbs law. To construct the transport map between $\tau_V$ and $\sigma_n$, we shall need an additional technical assumption. First, as we proceed by  interpolation of   the potential, we need to assume that a nice bounded free Gibbs law exists for all potentials $V_\alpha=\alpha V+(1-\alpha)\sum_{i=1}^n X_i^2,\alpha\in [0,1]$. This is the content of Assumption \ref{thehyphyphyp}. We are now in position to state one of our main theorems,
  see Corollary \ref{maincor} (with $B=D=\mathbb C$ and $W=c\sum_{i=1}^n X_i^2-V$).
  
  \begin{Theorem} Let $c,R>0$.  Assume that $V$ is a six times continuously differentiable  $(c,R)$ h-convex on the space of variables bounded by $2R$. Assume that
   $(V, c\sum_{i=1}^n X_i^2-V)$ satisfies the technical Assumption \ref{thehyphyphyp}.  Let $V_\alpha=V+\alpha(c\sum_{i=1}^n X_i^2-V)$.
   
   \begin{itemize}
   \item There exists $\alpha_0>0$ and functions $F_\alpha, \alpha\in [0,\alpha_0]$ and $G_\alpha,\alpha \in [0,\alpha_0]$, so that for all $\alpha\in [0,\alpha_0]$, $\tau_{V}$ (resp. $\tau_{V_\alpha}$) is the pushforward of $\tau_{V_\alpha}$ (resp. $\tau_V$) by $F_\alpha$ (resp. $G_\alpha$).  
   \item For any $\alpha\in [0,1]$, the von Neumann algebras associated to  the free Gibbs law with potential $V_\alpha$ are isomorphic; in particular, they are isomorphic to the von Neumann algebra generated by $n$ free semicircular variables.
     
     \end{itemize}
  \end{Theorem}   
  In the appendix, see Corollary \ref{corfullquartic}, we show that the following perturbation of quartic  potentials $\mathcal V$ satisfy all our hypotheses:
  
  $$\mathcal{V}(X)=V(X)+\varepsilon P(\frac{\sqrt{-1}+X_1}{\sqrt{-1}-X_1}, \cdots,\frac{\sqrt{-1}+X_n}{\sqrt{-1}-X_n})\,,$$
with
$$V(X)=\sum_{j=1}^k\mu_j\upsilon_j\left(\sum_{i=1}^n\lambda_{i,j}X_i\right)
+\sum_{i,j=1}^nA_{i,j}X_iX_j\,.$$ 
  Here
   $A=(A_{i,j})\in M_n(\R) $ is a positive matrix with  $A\geq cI_n$,  $(\lambda_{i,j})\in M_{n,k}(\R) , \mu\in [0,\infty[^k$, $\upsilon_j(x)=\nu_{j,2}\frac{X_1^2}{2}+\nu_{j,3} \frac{X_1^3}{3}+\nu_{j,4}\frac{X_1^4}{4}\in \C\langle X_1,...,X_n\rangle$ for $\nu_{j,4}>0,\nu_{j,3}^2\leq 8\nu_{j,2}\nu_{j,4}/3.$ Furthermore, $P$ is  a self-adjoint polynomial and $\varepsilon$ is small enough.

This is the first potential which is not a perturbation of a quadratic case for which isomorphism between the von Neumann algebras associated with its free Gibbs law and that of free semi-circle variables  is proven.

In the rest of the article we will consider a more general framework where the set of polynomials in $X_1,\ldots, X_n$ is replaced by the set of polynomials in $X_1,\ldots,X_n$ and elements in $B$, a von Neumann algebra. For $D$ a von Neumann  subalgebra of $B$, we shall consider variables $X_1,\ldots,X_n$ which commute with $D$. 
Our set of test functions will be converging series in such monomials, or closures of this space arising from certain non-commutative versions of $C^p$-norms. We shall consider the extended Haagerup tensor product of such spaces, and its cyclic variant which allows the action of cyclic permutations on these functions, space on which the cyclic gradient acts. Indeed, this gradient appears in the  right hand side of the Dyson-Schwinger equation \eqref{SD} and   the non-commutative version of the transport equation \eqref{transportflow}, and is therefore key to our analysis. Our main result in this general situation  is stated in Corollary \ref{maincor}.

    Our motivation for this generalization is two-fold. The first is to consider the crossed product $F_n\ltimes D$ of an action of the free group on $D$, as well as its $q$-deformation \cite{JLU14}. At this point we did not verify that these deformations correspond to potentials that satisfy our assumptions (for $q$ small enough). The motivation to also consider the algebra $B$ comes from the analysis of the free product $(\Gamma\ltimes D)*_{D} (W^*(S_s,s\le t)\otimes D)$: then $B=\Gamma\ltimes D$. Being able to construct transport maps in this setting would allow to construct solutions of free SDE's with initial conditions in $B$ as the image by transport maps of some process $S_{t_1},\ldots,S_{t_n}$.  For instance, one would want to obtain solutions of free SDE's similar to those considered in \cite{shlyakht:lowerEstimates} in the context of crossed product and for non-algebraic cocycles. Building such solutions in free products with amalgamation could enable the use of techniques similar to those in \cite{DabI,Ioana} and would lead to the study of algebras $B$ by  a free transport approach, for instance to answer questions such as uniqueness of Cartan decomposition up to unitary conjuguacy for non-trivial actions when $\Gamma$ is a group with positive first $\ell^2$ Betti number. Such interesting applications would thus require  to consider non smooth potentials $V$, something which is still far from our reach.   However, we feel that these potential applications outweigh the small additional difficulties involved in considering the more complex setting with non-trivial algebras $B$ and $D$.  Thus  our article lays the groundwork for future developments in this direction and our main example of relative algebra $B$ is exactly the kind of crossed-product that could be interesting for the above-mentioned potential applications.
    
\subsection*{Acknowledgements} The authors would like to acknowledge the hospitality of the Focus Program on Noncommutative Distributions in Free Probability Theory held at the Fields institute in July 2013 where an early part of this work has been completed.  We are also grateful to the  Oberwolfach Workshop on Free Probability Theory held in June 2015 during which we were able to make further progress.
  
 \section{Definitions and framework}
\label{notation}

\subsection{Spaces of analytic functions}
We denote by $M( X_1,\ldots,X_n)$  the set of monomials in $X_1,\ldots,X_n$.
Throughout this paper, $B$ will denote a finite von Neumann algebra, and $D$ a von Neumann subalgebra.

The extended Haagerup tensor product relative to $D$ is denoted by $\oeh{D}$.  We denote by $B^{\oehc{D} n}$ a version of the $n$-th extended Haagerup tensor power of $B$ that carries the action of the cyclic group of order $n$.

For $R>0$, we define formally 
 $$B\langle X_1,...,X_n:D,R\rangle:= B\oplus^{1}_D\ell^1_D\left(R^{|m|}B^{\oeh{D}}(|m|+1);m\in M(X_1,...,X_n), |m|\geq 1\right).$$
Here $R^{|m|}E$ means the space $E$ with standard  norm multiplied by $R^{|m|}$. This space can be regarded as the space of  power series in $X_1,\dots,X_n$ with coefficients in $B$ and radius of convergence at least $R$ by identifying a monomial $b_0 X_{i_1} b_1 \cdots  X_{i_p} b_p$ with the copy of the tensor $b_0\otimes \cdots \otimes b_p$ indexed by the monomial $m=X_{i_1}\cdots X_{i_p}$. The definition of the  Haagerup tensor product $\oeh{D}$ are discussed in section 1.2 and Lemma 5 of \cite{dabsetup} {(see also  \cite[chaper 5]{PisierBook}, \cite{M97,M05} for the general module case).}
 The above definition requires a direct sum of $D$-modules in order that $B\langle X_1,...,X_n:D,R\rangle\oeh{D} B\langle X_1,...,X_n:D,R\rangle$ is well defined. Modulo this (important) property, we could have more simply considered the (ordinary operator space) $\ell^1=\ell^1_{\C}$ direct sum (cf.  \cite[section 2.6]{PisierBook}): we denote $B\langle X_1,...,X_n:D,R,\mathbb C\rangle$ the corresponding smaller space. We will only use this sum in the cyclic case.

 Its cyclic variant  $B_c\langle X_1,...,X_n:D,R,\C\rangle$ is given
 by: 
 $$(D'\cap B)\oplus^1 \ell^1\left(R^{|m|}B^{\oehc {D}(|m|+1)};m\in M(X_1,...,X_n), |m|\geq 1\right)\,,$$
where $D'$ is the commutant of $D$ and $\oehc {D}$ stands for the cyclic version of Haagerup tensor product defined in subsection \ref{CyclicH}.
 This space can be regarded as the space of power series in $X_1,\dots,X_n$ with coefficients in $B$ and radius of convergence at least $R$, and such that variables $X_j$ commute with $D$.  As before, a monomial $b_0 X_{i_1} b_1 \cdots  X_{i_p} b_p$ is identified with the copy of the tensor $b_0\otimes \cdots \otimes b_p$ indexed by the monomial $m=X_{i_1}\cdots X_{i_p}$.  The use of the Haagerup tensor product $\oehc{D}$ ensures the possibility of cyclic permutation of various terms in the power series. 
 $\mathfrak{C}_{p+1}$ denotes the group of cyclic permutations acting on the cyclic tensor product, with generator $\rho(b_0\otimes\cdots\otimes b_p)=b_p\otimes b_0\cdots\otimes b_{p-1}$.
 We will define in subsection \ref{CyclicA} the cyclic gradient: it is roughly speaking a linear map on this space. We also define the analogue $B_c\langle X_1,...,X_n:D,R,\mathbb C\rangle$ of  $B\langle X_1,...,X_n:D,R,\mathbb C\rangle$.

 $B_c\langle X_1,...,X_n:D,R,\C\rangle$ and  $B\langle X_1,...,X_n:D,R\rangle$ are Banach algebras, see \cite[Theorem 39]{dabsetup} and subsection \ref{CyclicA}.

We let  for $n,m>0$, $i\in\{1,\ldots, n-1\}$, $\#_i: A^{\oeh D n}\times (D'\cap A^{\oeh D m}) \to A^{\oeh D n+m-2}$ the canonical extension of the map given on elementary tensors by 
$$(a_1\otimes\cdots a_n)\#_i( b_1\otimes\cdots b_m)=a_1\otimes\cdots \otimes a_{i-1}\otimes a_ib_1\otimes b_2\cdots \otimes b_ma_{i+1}\otimes\cdots\otimes a_n\,.$$
Having those operadic compositions, which will be crucial for non-commutative calculus, is another reason for using variants of Haagerup tensor products. The reader should note that by definition $A^{\oehc D m}\subset (D'\cap A^{\oeh D m})$. We will also use the restriction to cyclic variants as defined in subsection \ref{CyclicH}: 
$\#_i: A^{\oehc D n}\times A^{\oehc D m} \to A^{\oehc D n+m-2}$
We will denote in short $\#$ for $\#_1$.
{We may also write for instance $.\#(.,.): A^{\oehc D 3}\times A^{\oehc D n}\times A^{\oehc D m}\to A^{\oehc D m+n-1}$ for $U\#(V,W)=(U\#_1V)\#W=(U\#_2W)\#V$ and similarly  $U\#(V_1,\cdots,V_k)$.}

We endow $A^{\oehc D 2}$ with the adjunction $*$ so that $(a\otimes b)^*=a^*\otimes b^*$. Note that $(a\#b)^*=b^*\#a^*$, so that $(A^{\oehc D 2},*)$ is  a $*$-algebra.

\subsection{Spaces of analytic functions with expectations}
We will need a generalization of analytic functions  enabling functions of the conditional expectation $E_D$ on $D$.  For example, we would like to consider functions of the type
$$b_0X_{i_1}b_1\cdots  X_{i_p} b_p  E_D[ b_{p+1}X_{i_{p+2}}\cdots b_{p+k}E_D[b_{p+k+1}X_{i_{p+k+2}}\cdots b_{p+k+m}]]\qquad\quad$$
$$\qquad \times  
E_D[b_{p+k+m}X_{i_{p+k+m+1}}\cdots b_{p+k+m+\ell}] b_{p+k+m+\ell+1} X_{i_{p+k+m+\ell+2}}\cdots b_{p+k+m+\ell+r}$$
As the order in which conditional expectations are applied matters, we will label such a monomial by inserting an additional
letter $Y$ for each closing and opening parenthesis of the map. The matching  between the
closing and opening parenthesis then defines a non-crossing pair partitions of the set of positions of the letter $Y$. 
Conversely, given a non-commuting monomial in  letters $X_1,\ldots,X_n$ and $Y$ having even degree $2k$ in $Y$, and a non-crossing pair partition of the positions of the letter $Y$, we can define a unique expression of the type above. Thus, formally we set
$$B_k\{ X_1,...,X_n:E_D,R\}:=\ell^1_{D}\big(R^{|m|_X}B^{\oeh{D}(|m|+1)};\qquad$$
$$\qquad m\in M_{2k}(X_1,...,X_n,Y),\pi\in NC_2(2k), |m|\geq 1\big)\,,\quad k\ge 1$$
where $M_{2k}(X_1,...,X_n,Y)$ is the set of non-commuting monomial in  letters $X_1,\ldots,X_n$ and $Y$ having even degree $2k$ in $Y$, $|m|_X$ denotes the degree in the letter $X_1,\ldots,X_n$ of $m$ and $|m|=|m|_X+2k$. { We call $B_k\{ X_1,...,X_n:E_D,R,\C\}$ the corresponding space with (non-module) operator space $\ell^1$ sums (in the sense of   \cite[section 2.6]{PisierBook}).}
Similarly, we define 
$$B_{c,k}\{ X_1,...,X_n:E_D,R{,\C}\}:=\ell^1\big(R^{|m|_X}B^{\oehc{D}(|m|+1)};\qquad$$
$$\qquad m\in M_{2k}(X_1,...,X_n,Y),\pi\in NC_2(2k), |m|\geq 1\big)\,.$$
We set     $B_{c,0}\{ X_1,...,X_n:E_D,R{,\C}\}=B_c\langle X_1,...,X_n: D,R{,\C}\rangle$
   and $B_0\{ X_1,...,X_n:E_D,R\}=B\langle X_1,...,X_n: D,R\rangle$.
   
Finally we define:
\begin{gather*}
B_{c}\{ X_1,...,X_n:E_D,R{,\C}\}:=\ell^1\left(B_{c,k}\{ X_1,...,X_n:E_D,R{,\C}\}, k\in \N\right),\\
 B\{ X_1,...,X_n:E_D,R,\C\}:=\ell^1\left(B_{k}\{ X_1,...,X_n:E_D,R,\C\}, k\in \N\right),\\
   B\{ X_1,...,X_n:E_D,R\}:=\ell^1_D\left(B_{k}\{ X_1,...,X_n:E_D,R\}, k\in \N\right).\end{gather*}
   Above, $E_D$ should be considered as a variable taken in the space of $D$-bilinear completely bounded maps.

For $P\in B_{c}\{ X_1,...,X_n:E_D,R{,\C}\}$ and $E: B\langle X_1,...,X_n: D,R\rangle\to D$ unital $D$-bilinear completely-bounded map, we can define the map $P\mapsto P(E)$ taking $B_{c}\{ X_1,...,X_n:E_D,R{,\C}\}$ to $B\langle X_1,...,X_n: D,R\rangle$ by recursively replacing each sub-monomial $E_D(Q)$, $Q\in B\langle X_1,...,X_n: D,R\rangle$ inside $P$ by $E(Q)$. {A formal definition is explained in subsection \ref{CyclicE} where all the technical Lemmas we will need about those analytic functions are proved.}

\subsection{Spaces of differentiable  functions}
Let $A$ be a finite von Neumann algebra, $B\subset A$ {a von Neumann subalgebra}.
Set $$A^n_R:=\{  (X_1,\ldots,X_n)\in A^n: X_i=X_i^*\in A;\,  \|X_i\|< R, \quad [X_i,D]=0,\, 1\le i\le n \}.$$
Let $U\subset A^n_R$ be a closed subset of $A^n_R$. For convenience, we will first embed the algebra
$B_c\langle X_1,...,X_n:D,R{,\C}\rangle$ into a much larger algebra  $\cap_{S>R}C^0_b(U,B_c\langle X_1,...,X_n: D,S{,\C}\rangle)$, where
 $C^0_b(U,B)$ stands for the space of bounded continuous functions on $U$ with values in a Banach space $B$.  On this space we define the norm
 $$\Vert{}P\Vert{}_{A,U}=\sup\{\Vert{}P(X_1,...,X_n)\Vert{}_A\, ;\, (X_1,...,X_n)\in U\, \},$$
where by $P(X_1,\ldots,X_n)$ we mean the value of $P$ evaluated at $(X_1,\ldots,X_n)\in U$, itself evaluated as a power series 
in $(X_1,\ldots,X_n)$ {(see Proposition \ref{analytic} for some details on those evaluations)}. We call the corresponding completion $C^*_u(A,U:B,D)$ and $C^*_u(A,R:B,D)$ when $U=A_R^n$.

For $P\in \cap_{S>R}C^1_b(A_R^n,B_c\langle X_1,...,X_n: D,S{,\C}\rangle)\subset \cap_{S>R}C^0_b(A_R^n,B_c\langle X_1,...,X_n: D,S{,\C}\rangle)$ the set of continuously differentiable functions on $A_R^n$ with bounded first derivative, one can consider the differential $$dP\in \cap_{S>R}C^0_b(A_R^n,L(D'\cap A_{sa}^n,B_c\langle X_1,...,X_n: D,S{,\C}\rangle))\,,$$ 
where $L(G,G')$ is the set of bounded linear maps from $G$ into $G'$. Here, $(D'\cap(A_{sa}^n))$ should be thought of as a tangent space of $A_R^n$.
As usual, one writes for $X\in A^n_R$ and $H\in D'\cap A_{sa}^n$,  $$D_HP(X)=dP(X).H$$ and we see that $(D_HP:X\mapsto dP(X).H)\in \cap_{S>R}C^0_b(A_R^n,B_c\langle X_1,...,X_n: D,S{,\C}\rangle))$. Likewise for $$P\in \cap_{S>R}C^k_b(A_R^n,B_c\langle X_1,...,X_n: D,S,\C\rangle)\subset \cap_{S>R}C^0_b(A_R^n,B_c\langle X_1,...,X_n: D,S{,\C}\rangle)$$ an element of the set of $k$ times coefficientwise continuously differentiable functions on $A_R^n$ with bounded first $k$-th order differentials, one can consider the $k$-th order differential $$d^kP\in C^0_b(A_R^n,B((D'\cap(A_{sa}^n))^{\hat{\o} k},B_c\langle X_1,...,X_n: D,S,\C\rangle)).$$  Here $\hat{\o}$ denotes the projective tensor product.  

In this case $D_KD_H^{k-1}P(X)=d^kP(X).(K,H,...,H)$ and $$D_KD_H^{k-1}P:X\mapsto D_KD_H^{k-1}P(X)\in \cap_{S>R}C^0_b(A_R^n, B_c\langle X_1,...,X_n: D,S,\C\rangle).$$

We show in Proposition \ref{analytic2} that  on $B\langle X_1,...,X_n: D,R\rangle$, the $i$-th free difference quotient $$\partial_i:B\langle X_1,...,X_n: D,R\rangle\to B\langle X_1,...,X_n: D,R\rangle\oeh{ D} B\langle X_1,...,X_n: D,R\rangle $$ 
 is defined and is a canonical derivation satisfying  $\partial_i(X_j)=\delta_{i=j}1\o 1, \partial_i(b)=0.$ This can be extended to
 $ B\{ X_1,...,X_n: E_D,R\}$ by putting $\partial_i\circ E_D=0$.

We denote in short 
$$\partial^k_{(i_1,...,i_k)}:B\langle X_1,...,X_n: D,R\rangle\to B\langle X_1,...,X_n: D,R\rangle^{\oeh{D} (k+1)}$$ the map $$\partial^k_{(i_1,...,i_k)}=(\partial_{i_1}\o1^{\o k})\circ (\partial_{i_2}\o1^{\o (k-1)})\circ...\circ\partial_{i_k}.$$

Recall that $D_H$ stands for the directional derivative of a function in $C^1_u(A,U:B,E_D)$, viewed as a function from $U$ to the space of power series $B_c\langle X_1,...,X_n:D,R{,\C}\rangle$. However, this won't be the most convenient differential, since the non-commutative power series part will always be evaluated at the same $X\in U$ and we will rather need the full differential which uses also the free difference quotient on the powers series part. 

On the space of continuous differentiable functions $C^1(U,A)$  from $U$ to $A$, denote by $D_H^X$ the  derivative in the direction $H\in A^n$.  
Consider the map $\eta: C_u^{1,1}(A,U:B,E_D) \to C^1(U,A)$ given for $P\in \cap_{S>R}C^1_b(A_R^n,B_c\langle X_1,...,X_n: D,S{,\C}\rangle)$ by $\eta(P) = (P (X) )(X)$.  Then one has
\begin{equation}\label{eqdif}
D_H^X ( \eta ( P ) ) = \eta ( D_H (P) ) + \eta ( \sum_{j=1}^k (\partial_j (P))\#H_j )).
\end{equation}
We let $d_X$ be the differential associated with $D^X_H$.
We will also write:

\begin{eqnarray*}
d_X^pP(X).H&=&(D^X_H)^p{\eta(P)}=d^p[X\mapsto P(X)(X)](X).(H,...,H)\\
&=&\sum_{j=0}^p\sum_{i\in[1,n]^j}(d^{p-j}[\partial_i^j P(X)].(H,...,H))\#(H_{i_1},...,H_{i_j})
.\end{eqnarray*}

For $P\in \cap_{S>R}C^0_b(A_R^n,A\langle X_1,...,X_n: D,S,\C\rangle),X\in A_R^n$ , we set
$$
 \Vert{}P\Vert{}_{k,X}=\left(\Vert{}P(X)\Vert{}_{A}+\sum_{p=1}^k \sum_{i\in[1,n]^p}
\Vert{}\partial^p_{i}(P)(X)\Vert{}_{A^{\oeh{D}(p+1)}} \right).$$

We will consider the (separation) completion of $$ \bigcap_{S>R}C^{l+1}_b(A_R^n,B\langle X_1,...,X_n: D,S,\C\rangle)$$ with respect to the seminorms for $(k,l)\in \N^2$ given by
\begin{align*}&\Vert{}P\Vert{}_{k,l,U}=\sup_{X\in U} \Vert{}P\Vert{}_{k,X}+\sum_{p=1}^l\left\{\sup_{\textrm{\tiny$\begin{array}{c}X\in U\\H\in A_1^n\end{array}$}}\big(\Vert{}(D_H^X)^{p}\eta(P)(X)\Vert{}_{A}\qquad\right.\\
&\left.\qquad\qquad\qquad+
\sum_{\textrm{\tiny$\begin{array}{c}
i\in[1,n]^m\\ m\leq k\end{array}$}}\Vert{}(D_H^X)^{p}\partial^m_{i}(\eta(P))(X)\Vert{}_{A^{\oeh{ D}(m+1)}}
\big)\right\}\end{align*}

This seminorm controls $k$ free difference quotients and $l$ full differentials.


We will denote these (separation) completions by $C^{k,l}_{u}(A,U:B,D),$ and $C^{k,l}_{u}(A,R:B,D)$ when $U=A_R^n$. {Note that the above map $D_H^p$, for $p\leq l$ extends continuously to a map $C^{k,l}_{u}(A,R:B,D)\to C_u^{k-p,l-p}(A,R:B,D).$}
 
When in the definition of $\|.\|_{k,X}$ we replace $\|.\|_{A^{\oeh{D}(l+1)}}$ by $\|.\|_{A^{\oehc{D}(l+1)}}$, we distinguish  the corresponding seminorms by 
a subscript $c$, yielding the norm $\|.\|_{k,l,U,c}$ and the spaces $C^{k,l}_c(A,U:B,D)$, $C^k_c(A,U:B,D)$.

{Note that this require a supplementary assumption that $U\subset A_{R,UltraApp}^n$ where $A_{R,UltraApp}^n$ is defined before Proposition \ref{TensorCyclicEvaluation}: this assumption is necessary to define evaluation into cyclic tensor products. This is crucial to see that the image of cyclic analytic functions by the free difference quotient belongs to the cyclic Haagerup tensor product, see also Proposition \ref{analytic2}.
More precisely we define  $A_{R,UltraApp}^n$ the set of $X_1,...,X_n\in A, X_i=X_i^*, [X_i,D]=0, ||X_i||\leq R$ and such that $B,X_1,...,X_n$ is the limit in $E_D$-law (for the $*$-strong convergence of $D$) of variables in $B_c\langle X_1,...X_m:D,2,\C\rangle(S_1,...,S_m)$ with $S_i$  free semicircular variables over $D.$   We will thus always assume $U\subset A_{R,UltraApp}^n$ when we deal with spaces with index $c$. Note that consistently, we will write $C^k_c(A,R:B,D)$ when $U=A_{R,UltraApp}^n$.}

 For convenience later in writing estimates valid when there is at least one derivative, we also introduce a seminorm  
\begin{align*}&\Vert{}P\Vert{}_{k,l,U,\geq 1}=\sup_{X\in U}\left(\sum_{p=1}^k \sum_{i\in[1,n]^p}\Vert{}\partial^p_{i}(P)(X)\Vert{}_{A^{\oeh{D}(p+1)}} \right)+\sum_{p=1}^l\sup_{\textrm{\tiny$\begin{array}{c}X\in U\\H\in A_1^n\end{array}$}}\big(\Vert{}(D_H^X)^{p}\eta(P)(X)\Vert{}_{A}\qquad\\
&\qquad\qquad\qquad+
\sum_{\textrm{\tiny$\begin{array}{c}
i\in[1,n]^m\\ m\leq k\end{array}$}}\Vert{}(D_H^X)^{p}\partial^m_{i}(\eta(P))(X)\Vert{}_{A^{\oeh{ D}(m+1)}}
\big).\end{align*}

We next define differentiable functions 
depending on conditional expectations.

Using the conditional expectation $E_D: A\to D$, we can define a completely bounded map $E_{D,X} :  B\langle X_1,...,X_n: D,S\rangle
\to D$ by sending $P$ to $E_D ( P(X_1,\dots,X_n))$, for any $S>R$.

Consider the map $\omega$ taking $P\in B_c\{ X_1,...,X_n:E_D,R^+,\C\}:=\cap_{S>R}B_c\{ X_1,...,X_n:E_D,S,\C\}$ to the function $$\omega(P) :  X \mapsto P(E_{D,X})  \in 
B_c\langle X_1,...,X_n:E_D,R^+,\C\rangle:=\cap_{S>R}B_c\langle X_1,...,X_n:E_D,S,\C\rangle.$$  We denote by $C^0_{b,tr}(U, B\langle X_1,...,X_n: D,R\rangle)$ the image of this map.

The spaces $C_{tr}^{k,l}(A,U:B,E_D)$ (resp.  $C^*_{tr}(A,U:B,E_D)$, $C^*_{tr,{c}}(A,U:B,E_D)$ and $C_{tr,c}^{k,l}(A,U:B,E_D)$) are defined as the closures of the space $C^0_{b,tr}(U, B\langle X_1,...,X_n: D,S\rangle)$ inside $C^{k,l}(A,U:B,D)$ (respectively, $C^*_u(A,U:B,D)$, $C^*_c(A,U:B,D)$, $C^{k,l}_c(A,U:B,D)$
). When $U=A^n_R$, we replace in the notations $U$ by $R$.

We denote by  
$C^{k,l}(A,U:B,D)$ the closed subspace of $C^{k,l}_{tr}(A,U:B,D)$ generated by the  image under $\omega$ of 
$B\langle X_1,...,X_n:D,S\rangle, S>R$. We denote in short  $C^{k}(A,R:B,D)$ for $C^{k,k}(A,U:B,D)$.

Let $H\in A^n$.  Recall that $D_H$ stands for the directional derivative of a function in $C^1_u(A,U:B,E_D)$, viewed as a function from $U$ to the space of power series $B_c\langle X_1,...,X_n:D,R{,\C}\rangle$.  Given $P\in B_c\{ X_1,...,X_n: E_D,R{,\C}\}$ a monomial involving $E_D$, we note that  $D_H (\omega(P)$ amounts to replacing each sub-monomial of the form $E_D(Q)$ with $Q\in B_c\langle X_1,...,X_n:D,R{,\C}\rangle$ by $E_D(\sum_j \partial_j Q \# H_j)$.  For example if $H=(H_1,H_2)$, then 
\begin{eqnarray*}&&D_H(\omega ( X_1 X_2 E_D (X_1^2 (E_D (X_1) ) E_D(X_2))) (Y_1,Y_2) \qquad \qquad\qquad\\ 
\quad&&=X_1 X_2 E_D (H_1 Y_1 (E_D (Y_1) ) E_D(Y_2))) 
  +X_1 X_2 E_D (Y_1H_1 (E_D (Y_1) ) E_D(Y_2)))\\
 &&  + X_1 X_2 E_D (Y_1^2 (E_D (H_1) ) E_D(Y_2))) 
  + X_1 X_2 E_D (Y_1^2 (E_D (Y_1) ) E_D(H_2))).
  \end{eqnarray*}  In other words, $D_H$ corresponds to ``differentiation under $E_D$''.
\subsection{Differential operators}\label{diffopsection}
For $p,P\in B_c\{ X_1,...,X_n: E_D,S,\C\}$,
we define recursively the cyclic gradient   $(\mathscr{D}_{i,p}(P), 1\le i\le n)$ by
$\mathscr{D}_{i,p}(X_j)=1_{j=i} p$,
\begin{equation}\label{CyclicDerivationNotation}\mathscr{D}_{i,p}(PQ)=\mathscr{D}_{i,Qp}(P)+\mathscr{D}_{i,pP}(Q),\qquad \mathscr{D}_{i,p}(E_D(P))=\mathscr{D}_{i,E_D(p)}(P).\end{equation}
For instance, one computes $\mathscr{D}_{1,p}(X_2 E_D(X_1bX_2) X_1)=pX_2 E_D(X_1bX_2)+bX_2 E_D(X_1pX_2)$. Moreover, observe that  for polynomials $P$  in $\{X_1,\ldots,X_n\}$,
\begin{equation}\label{flip}\rho(\partial_i P)\#Q=\mathscr{D}_{i,Q}(P)\,.\end{equation}
We denote in short $\mathscr{D}_i=\mathscr{D}_{i,1}$. Its restriction to polynomials in $\{X_1,\ldots,X_n\}$ corresponds to the usual cyclic derivative.
We consider a flat Laplacian
defined for $P\in  B\{ X_1,...,X_n: E_D,R\}$ by

$$\Delta(P)=2\sum_{i}m\circ (1\o E_D\o 1)\partial_i\o 1\partial_i(P)\,.$$
We define $\delta_\Delta$ a derivation on $B\{ X_1,...,X_n: E_D,R\}$ by requiring that it vanishes on $B\langle X_1,...,X_n: D,R\rangle$
and satisfies
$$\delta_\Delta(P)=0\,,\quad \delta_\Delta(E_D(Q))=E_D((\Delta +\delta_\Delta)(Q))\,.$$

Likewise, for any $V\in B\langle X_1,...,X_n:D,R\rangle,$ the map \begin{equation}\label{deltaV}
\Delta_V=\Delta-\sum_{i}\partial_i(.)\#\mathscr{D}_iV\end{equation} produces a map $\delta_V$ such that $\delta_V(P)=0, $ for $P\in B\langle X_1,...,X_n:D,R\rangle$. Moreover,  $\delta_V$ is a derivation and for $Q$ monomial in $B\{ X_1,...,X_n:E_D,R\}$, 
$$\delta_V(E_D(Q))=E_D((\Delta_V +\delta_V)(Q)).$$
$\delta_V$ extends to $B\{ X_1,...,X_n:E_D,R\}$ {(see Proposition \ref{DeltaAnalytic}).} 
Moreover, we have for any $g\in B_c\{ X_1,...,X_n:E_D,R,\C\}$, 
\begin{equation}\label{derivcyclic}
\mathscr{D}_{i}(\Delta_V+\delta_V)(g)=(\Delta_V+\delta_V)\mathscr{D}_{i}(g)-\sum_{j=1}^n\mathscr{D}_{i,\mathscr{D}_{j}g}\mathscr{D}_{j}V.\end{equation}
We extend $\Delta_V$ and $\delta_V$ to $V\in C^3_c(A,2R:B,E_D)$  by adding the variables $Z_i$ to be evaluated at $
\mathscr{D}_iV(X)$, letting $V_0(Z)=\frac{1}{2}\sum Z_i^2$  and setting for $P\in B\{ X_1,...,X_n:E_D,R\}$
\begin{equation}
\label{defDeltaV}\Delta_V(P)(E_{D,X})(X):=\left(\Delta_{V_0(Z)}(P)\right)\left(E_{D,X,\mathscr{D} V(X)}\right)(X,\mathscr{D} V(X))\,.\end{equation}
$\Delta_V(P)$ belongs to $C^*_{tr}(A,U)$.  The extension of $\delta_V$ is similar. 
We   define, $C^{k,l}_{tr,V}(A,U:B,E_D), k\in \{*\}\cup\N^*,k\geq l$ as the separation-completion of $B_{ c}\{ X_1,...,X_n;E_D,R^+\}:= \cap_{S>R}B_{ c}\{ X_1,...,X_n;E_D,R\}$ for the semi-norm (with $\omega(P)=(X\mapsto P(E_{X,D}))$):

\begin{align*}||P||_{C^{k,l}_{tr,V}(A,U:B,E_D)}&=||\omega(P)||_{k,l,U}+
1_{k\geq 2}||(\Delta_V+\delta_V)(P)||_{C^*_{tr}(A,U)}\\&+\sum_{p=0}^{l-1}\sum_{i=1}^n\sup_{\textrm{\tiny$\begin{array}{c}Q\in (C^{k,p}_{tr}(A,U^{m-1}:B,E_D))_1\\ m\geq 2\end{array}$}} ||\mathscr{D}_{i,Q(X')}(P)||_{k,p,U^m},\end{align*}
{where $(X)_1$ denotes the unit ball around $0$ of the normed space $X$.}
We  also define a first order part seminorm $||P||_{C^{k,l}_{tr,V}(A,U:B,E_D),\geq 1}$ by replacing the first term in the sum with $\Vert{}\omega(P)\Vert{}_{k,l,U\geq 1}$. {We also define the space $C^{k,l}_{tr,V,c}(A,U:B,E_D)$ in the same way as before but considering everywhere {\em cyclic} extended-Haagerup tensor products.}

To sum up we have introduced the following spaces
$$\begin{array}{ccccccc}
C^{k+l}&&C^{k,l}_{tr,V}&\to&C^{k,l}_{tr}&\subset&C^{k,l}_u\cr
\cup&&\cup&&\cup&&\cup\cr
C^{k+l}_c&\subset^{
}&C^{k,l}_{tr,V,c}&\to&C^{k,l}_{tr,c}&\subset&C^{k,l}_{u,c}\cr
\uparrow&&\uparrow&&&&\cr
B_c\langle \cdots\rangle&\subset&B_c\{\cdots\}&&&&\cr
\end{array}$$
where $\subset$ means the existence of a canonical injective mapping, whereas $\to$ means the existence of a canonical map (with conditions written in index). We shall not discuss these mappings as we will not use them and leave the reader check them.
\subsection{Free brownian motion}\label{freebrownian}
$(S_t^i, t\ge 0, 1\le i\le n)$ will denote $n$ free Brownian motions. Let $U\subset A_{R}^n$. We denote by $*_D$ the free product  with amalgamation over $D$: see \cite{dnv} for a definition as well as for a definition of freeness with amalgamation over $D$. Let $\mathscr{A}=A*_D(D\otimes W^*(S_t^{(i)},i=1,...,n,t\geq 0))$ and assume that $A$ is big enough so that $\mathscr{A}$ is isomorphic to $A$. Set  $U_A=\{X\in \mathscr{A}_{R}^n , X\in U \}\subset A_{R}^n$ and $\mathscr{B}=B*_D(D\otimes W^*(S_t^{(i)},i=1,...,n,t\geq 0))$.

Define 
$$C^{k{},l}_{tr,V}(A,U:\mathscr{B},E_D:\{S_t^{(i)},i=1,...,n,t\geq 0\})\subset C^{k,l}_{tr,V}(\mathscr{A},U_A:\mathscr{B},E_D)$$ as the closure of $$\bigcup_{0\leq t_1\leq ...\leq t_m}\eta_S\big(B_c\{X_1,...,X_n,S_{t_1},...,S_{t_m}-S_{t_{m-1}}:\qquad\qquad$$
$$\qquad\qquad
\qquad E_D,\max[R,\max_{i=2,n} 2(t_i-t_{i-1})]\C\}\big)$$ 
where $\eta_S$ is the partial evaluation of the analytic functions in $X$'s and $S$'s at $S_{t_1},$ ${S_{t_2}-S_{t_1}},$ $ \ldots, S_{t_m}-S_{t_{m-1}}$, hence obtaining functions in $\mathscr{B}_c\{X_1,\ldots,X_n: E_D, R\}$. In other words, 
this is the  union of partial evaluation maps at the free brownian motions  of analytic functions with expectations. Write in short $\mathscr{S}=\{S_t^{(i)},i=1,...,n,t\geq 0))\},$ { and similarly for $u>0,\mathscr{S}_u=\{S_t^{(i)},i=1,...,n,u\geq t\geq 0))\},\mathscr{S}_{\geq u}=\{S_t^{(i)}-S_u^{(i)},i=1,...,n, t\geq u))\}$}

We call accordingly, {for $U\subset A_{R,UltraApp}^n$,} $C^{k}_{c}(A,U:\mathscr{B},D:\mathscr{S})\subset C^{k,k}_{tr,V}(A,U:\mathscr{B},E_D:\mathscr{S})\cap C^{k}_{c}(\mathscr{A},U_A:\mathscr{B},D)$ the space generated by analytic functions (without expectations) with norm $\Vert{}.\Vert{}_{k,l,U}$.{We also have analogously $C^{k}_{c}(A,U:\mathscr{B},D:\mathscr{S}_u)\subset C^{k,k}_{tr,V}(A,U:\mathscr{B},E_D:\mathscr{S}_u)$ (imposing above $t_m\leq u$).
Fix a trace preserving $*$-homomorphism $\theta_u:\mathscr{A}\to \mathscr{A}$ by $\theta_u(a)=a, a\in A$, $\theta_u(S_s)=S_{s+u}-S_u$
with  obvious induced maps $$\theta_u':C^{k,l}_{tr,V}(A,U:\mathscr{B},D:\mathscr{S})
\to C^{k,l}_{tr,V}(A,U:\mathscr{B},D:\mathscr{S}_{\geq u}),$$
and similarly $ \theta_u':C^{k,l}_{tr}(A,U:\mathscr{B},D:\mathscr{S})
\to C^{k,l}_{tr}(A,U:\mathscr{B},D:\mathscr{S}_{\geq u}).$}

For $u\ge 0$, we denote by $\mathscr{A}_u=A*_D(D\otimes W^*(S_t^{(i)},i=1,...,n,t\in [0,u]))$ and $E_u$ the associated conditional expectation.
We observe that when restricted to polynomial function, the conditional expectations take their values in polynomials. Under certain conditions on $U$,
see Proposition  \ref{ExpectationCkl},  we can extend $E_u$ as an application $
C^{k,l}_{tr,V}(A,U:\mathscr{B},D:\mathscr{S})\to C^{k,l}_{tr,V}(A,U:\mathscr{B},D:\mathscr{S}_{ u})\,.$

\section{Semi-groups and SDE's associated with a convex potential}\label{sg}
\subsection{Convex potentials}
With obvious notations, $ M_n(A^{\oehc {D}2})$ denotes the space of $n\times n$ matrices with entries in $A^{\oehc {D}2}$.
For $M\in  M_n(A^{\oehc {D}2})$, $(M^*)_{ij}:= (M_{ji})^*$ with for $b\in A^{\oehc {D}2}$, $b^*$ defined in {Theorem \ref{finite2} (1e)}.
 We don't equip this space with the norm induced by its natural  operator space structure as Haagerup tensor product. 
We rather see $M_n(A^{\oehc {D}2})$ as follows 
 $$M_n(A^{\oehc {D}2})\subset \bigcap_{m=1}^\infty B\left(\ell^2\Big([\![1,n]\!],(A^{\oehc {D}m})\Big),\ell^2\Big([\![1,n]\!],(A^{\oehc {D}m})\Big)\right)\,.$$
We equip it  with the matrix like $\#$ multiplication map defined for $M=[M_{ij}]\in M_n(A^{\oehc {D}2}),$ $ X\in \ell^2([\![1,n]\!],A^{\oehc {D}2})=(M^{\oehc {D}m})^n$ by 
 $$(A\#X)_i=\sum_{j=1}^n A_{ij}\#X_j\,, $$
and with the norm
$$||M||_{M_n(A^{\oehc {D}2})}:=\sup_{m\ge 0} \sup\{
||(M\#X)||_{(A^{\oehc {D}m})^n}, ||(M^*\#X)||_{(A^{\oehc {D}m})^n}: {||X||_{(A^{\oehc {D}m})^n}\le 1} \}\,.$$
By definition $||M||_{M_n(A^{\oehc {D}2})}=||M^*||_{M_n(A^{\oehc {D}2})}$, and 
$$||M \#N||_{M_n(A^{\oehc {D}2})}\le ||M||_{M_n(A^{\oehc {D}2})}||N||_{M_n(A^{\oehc {D}2})}\,.$$

We first recall a consequence of Hille-Yosida Theorem.
\begin{Proposition} The following are equivalent.

\begin{enumerate}
\item $Q=Q^*\in M_n(A^{\oehc{D}2})$ has a semigroup of contraction $e^{-t Q},$

\item $Q=Q^*\in  M_n(A^{\oehc {D}2})$ has a resolvent family for all $\alpha>0$, $\alpha+Q$ is invertible in $M_n(A^{\oehc {D}2})$ and $||\frac{\alpha}{\alpha+Q}||_{M_n(A^{\oehc{D}2})} \leq 1$.

\end{enumerate}
In this case we say $Q\geq 0$.
\end{Proposition}
\begin{proof}
We apply Hille-Yosida Theorem e.g. in the form of Theorem 1.12 in \cite{MaRockner}, to each  Banach space $\ell^2([\![1,n]\!],(A^{\oehc {D}m}))$ in the definition of the norm  of $M_n(A^{\oehc{D}2}).$
\end{proof}

Note that  the set of non-negative $Q=Q^*\in M_n(A^{\oehc{D}2})$ is a cone. Indeed, if $\alpha\ge 0$ and $Q\geq 0$, clearly $\alpha Q\ge 0$. Moreover,
 $Q\ge 0$
and $\tilde Q\ge 0$  implies that $Q+\tilde Q\ge 0$. Indeed, as $Q$ and $\tilde Q$ are
bounded, they are defined everywhere  as well as $Q+\tilde Q$, and one can use \cite{trotter} to see that 
$$e^{-t(Q+\tilde Q)}=\lim_{k\rightarrow\infty}(e^{-\frac{t}{k} Q}.e^{-\frac{t}{k} \tilde Q})^k$$
is a contraction as the right hand side is. Moreover, this set is closed as follows easily from  the characterization (2) (notice here that
the set $Q=Q^*\in  M_n(A^{\oehc {D}2})$ is closed).

Observe that if $V=V^*\in  C^{2}_c(A,R:B,D)$, $X\in A_R^n$,  $(\partial_i\mathscr{D}_jV(X))_{1\le i,j\le n}\in M_n(A^{\oehc {D}2})$
is self-adjoint.
\begin{definition}\label{defhconvex} Let $c,R> 0$.
$V=V^*\in  C^{2}_c(A,R:B,D)$ is said  $(c,R)$ h-convex if $(\partial_i\mathscr{D}_jV(X))_{1\le i,j\le n}-c\operatorname{Id}\geq 0$
for any {$X\in A_{R,UltraApp}^n$}.
\end{definition}


We show below that  $(c,R)$ h-convex potentials have well behaved solutions of linear ODE.
\begin{Lemma}\label{lemhyp} Assume $V$ is  $(c,R)$ h-convex. Consider  a continuous self-adjoint  process $(X_t)_{t\ge 0}$, $\|X_t\|\le R$, $X_t\in D'$.

(a)  Let $Y\in (A^{\oehc {D} m})^n$ be such that $Y_j^*=Y_j$ (with $(a_1\o ...\o a_m)^*=a_m^*\o...\o a_1^*$). Then, there exists a
  unique solution $\phi_{s,t}(Y,X)\in (A^{\oehc {D} m})^n$ of the following linear ODE for $t\geq s$:
\begin{equation}\label{eqphi}
\phi_{s,t}(Y,X)_j=Y_j-\frac{1}{2}\int_s^t du\sum_k (\partial_k \mathscr{D}_j V)( X_u) \# \phi_{s,u}(Y,X)_k\,.\end{equation} 
It satisfies $\phi_{s,t}(Y,X)_j^*=\phi_{s,t}(Y,X)_j$. Moreover,
 for any $\sigma \in \mathfrak{C}_{n}$,  the solution $\sigma.(\phi_{s,t}(Y,X)_j)$ of the equation transformed by $\sigma$ (that is 
the equation obtained by applying a cyclic permutation of the tensor 
indices)
satisfies  :
\begin{equation}\label{eq:estimateOnPhist}
||\sigma.(\phi_{s,t}(Y,X))||_{(A^{\oeh{ D} m})^n}\leq e^{-(t-s)c/2}\|Y\|_{(A^{\oehc{ D} m})^n}.
\end{equation}

(b)  Let $Y_s$  be a $C^1$ process 
with values in $ (A^{\oehc{D} m})^n$ such that $Y_s(t)_j^*=Y_s(t)_j$ (with $(a_1\o ...\o a_n)^*=a_n^*\o...\o a_1^*$). The (unique) solution $\Phi_{s,t}(Y,X)$ of the following linear ODE for $t\geq s$:
\begin{equation}\label{equationphi}\Phi_{s,t}(Y,X)_j=Y_s(t)_j-\frac{1}{2}\int_s^t du\sum_k (\partial_k \mathscr{D}_j V)( X_u) \# \Phi_{s,u}(Y,X)_k,\end{equation}
satisfies $\Phi_{s,t}(Y,X)_j^*=\Phi_{s,t}(Y,X)_j$  and
$$||\Phi_{s,t}(Y,X)||_{(A^{\oehc{D} m})^n}\leq e^{-(t-s)c/2}\|\|Y\|\|_{s,t}$$
with
$$\|\|Y\|\|_{s,t}=(\sum_j||(Y_s(s))_{j}||_{A^{\oehc{D} m}}^2)^{1/2}+  \int_s^t  e^{-c(s-u)/2}
(\sum_j||\partial_u Y_s(u)_{j}||_{A^{\oehc{D} m}}^2)^{1/2}du\,.$$

\end{Lemma}
\begin{proof}{\bf Proof of (a). } Let $X$ be a continuous self-adjoint process. The semigroup $\Theta^{X}$  associated to $Q=\frac{1}{2}(\partial_k\mathscr{D}_j V(X))_{kj}$
gives a solution  $(\Theta^X_{s,t}(Y))_{t\ge 0}$ to
$$Y_j(t)=Y_j-\frac{1}{2}\int_s^t \sum_{k=1}^n \partial_k\mathscr{D}_j V(X)\# Y_k(s) ds\,.$$
Therefore we can 
define the solution to 
$$\phi_{s,t}^p(Y,X)_j=Y_j-\frac{1}{2}\int_s^tdu\sum_k (\partial_k \mathscr{D}_j V)( X_{\frac{\lfloor up\rfloor}{p}}) \# \phi_{s,u}^p(Y,X)_k$$
in $(A^{\oehc {D} m})^n$ by putting 
\begin{equation}\label{defphip}\phi_{s,t+s}^p(Y,X)=\Theta^{X_{\frac{\lfloor tp \rfloor }{p} +s}}_{\frac{\lfloor pt\rfloor }{p} +s , t+s}\circ
\Theta^{X_{\frac{\lfloor tp-1 \rfloor }{p} +s}}_{\frac{\lfloor (pt-1)\rfloor }{p} +s , \frac{\lfloor pt\rfloor }{p} +s}\circ\cdots \circ \Theta^{X_{ s}}_{s , \frac{1}{p} +s}(Y)\,. \end{equation}
By assumption of $(c,R)$ h-convexity,  the semigroup $e^{-t(Q-\frac{c}{2}\operatorname{Id})}=e^{\frac{c}{2}t }e^{-tQ}$ is contractive,
which gives the bound
$$||\phi_{s,t}^p(Y,X)||_{(A^{\oehc{ D} m})^n}\leq e^{-(t-s)c/2}||Y||_{(A^{\oehc{ D} m})^n}\,.$$
In particular, this sequence is bounded uniformly.  By continuity of $X$, we can prove similarly that this sequence is Cauchy, and hence 
converges towards the solution of \eqref{eqphi}; the limit then clearly satisfies the bound \eqref{eq:estimateOnPhist}.  Uniqueness can be proved by Gronwall Lemma, as $(\partial_k \mathscr{D}_j V)( X_.)$ is uniformly bounded.

Selfadjointness  of $\phi_{s,t}(Y,X)_j$
 follows from the uniqueness of the solution to the linear ODE since $((a\o c)\#(b_1\o ...\o b_n))^\star=(c^\star\o a^\star)\# (b_1\o ...\o b_n)^\star$ and $((\partial_k \mathscr{D}_j V(X_s))^\star)_{kj}=(\partial_k (\mathscr{D}_j V^*)(X_s^*))_{kj}=(\partial_k (\mathscr{D}_j V)(X_s))_{kj}$ because $V=V^*$ and $X_s^*=X_s$.
 
 \noindent
{\bf Proof of (b). } Using the notation of (a), define : 
$$\Phi_{s,t}(Y,X)=\phi_{s,t}(Y_s(s),X)+\int_s^tdu\phi_{u,t}(\partial_uY_s(u),X).$$
Differentiating in $t$ shows that $\Phi_{s,t}$ is a solution of \eqref{equationphi}. The bounds follows readily from (a).
Again, uniqueness follows from Gronwall's Lemma. 

\end{proof}
\subsection{Free stochastic differential equation}
\begin{Proposition}\label{ConvSDE}
Assume $V\in C^{2}_c(A,R:B,D)$ is  $(c,R)$ h-convex. \\
(a) There exists $T>0$ so that for any $X_0\in  {A_{R,UltraApp}^n}$, there exists 
a unique solution to 
$$X_t(X_0)=X_0+ S_t -\frac{1}{2} \int_0^t \mathscr{D} V (X_u(X_0)) du$$ which is defined for all times $t<T$.
Moreover, for all $X_0,\tilde X_0\in {A_{R,UltraApp}^n}$ and $t\ge 0$
\begin{equation}\label{boundX}\|X_t(X_0)-X_t(\tilde X_0)\|\le e^{-ct/2}\| X_0-\tilde X_0\|.\end{equation} \\
(b) Assume that there exists $X^V = (X_1^V,\dots,X_n^V)\in  {A_{R/3,UltraApp}^n}$ for which the conjugate variables are equal to $\mathscr{D}_j V$.  Then part (a) holds with $T=\infty$ for any solution starting at $X_0\in A^n_{R/3,UltraApp}$.   As a consequence, there is at most one free Gibbs law with potential $V$ uniformly in ${A_{R/3,UltraApp}^n}$.
\end{Proposition}
\begin{proof}
Existence of $X_t(X_0)$ for all times $t<T$ for which $\sup_{s<T} \Vert X_s(X_0)\Vert <R$ 
 follows form the Picard iteration argument in \cite{biane-speicher}. The existence of $T>0$ (depending only on the Lipschitz constant of $\mathscr{D}V$) is also shown there.

 Applying the same argument as in the proof of  Lemma \ref{lemhyp} by writing $X^1_t=X_t(\tilde X_0),X^0_t=X_t(X_0)$,
$$X_t^1-X_t^0=\tilde X_0-X_0-\frac{1}{2} \int_0^t \int_0^1 \partial\mathscr{D} V(\theta X_u^0+(1-\theta) X_u^1) \# (X_u^1-X_u^0) d\theta du$$
and arguing that $\int_0^1 \partial\mathscr{D} V(\theta X^1_u+(1-\theta) X^0_u)d\theta -c\operatorname{Id} \ge 0$ as the set of non-negative elements of $M_n(A^{\oehc D 2})$ is a closed cone, the estimate  \eqref{boundX}  follows from \eqref{eq:estimateOnPhist}.

Assuming the Assumption of part (b), we see that the solution $X_t(X_V)$ is stationary; in particular, its norm is constant.  Part (a) and the estimate \eqref{boundX} then imply that any other solution starting at an element of $A^n_{R/3}$ stays in $A^n_{R}$, which means that $T$ can be chosen to be infinite.  Also, if there were two free Gibbs law with potential $V$, they would be stationary laws for the dynamics and \eqref{boundX} would imply that they are equal.
 \end{proof}
 Throughout this paper we assume that

\begin{Assumption}\label{thehyp} Let 
 $V,W\in C_c^3(A,2R:B, E_D)$ be two non-commutative functions such that
$V$ and $V+W$ are $(c,2R)$ h-convex for some $c>0$.  We assume   that for any $\alpha\in [0,1]$, there exists  a solution $(X_1^{V+\alpha W},\ldots,X_n^{V+\alpha W})\in  {A_{R/3,UltraApp}^n}$  with conjugate variables $(\mathscr{D}_i(V+\alpha W))_{1\le i\le n}$. 
\end{Assumption}

{In subsection  \ref{ConvexPotentialSection} we describe a class of quartic potentials satisfying this assumption. The existence of a solution to Schwinger-Dyson equations will be obtained from a random matrix model in the easiest case $B=\C$ and the convexity will be obtained by operator spaces techniques.}

This Assumption insures that 
 $$V_\alpha  = V +  \alpha W$$ 
 is $(c,2R)$ convex for all $\alpha\in [0,1]$.

 We consider the SDE
\begin{equation} \label{eqn:Xt}
X_t ^\alpha= X_0+ S_t - \frac{1}{2} \int_0^t \mathscr{D} V_\alpha (X_s^\alpha) ds
\end{equation}
where $S$ is the free Brownian motion relative to $D$ (with covariance map $id_D$).  
By Proposition \ref{ConvSDE}, we deduce that there exists  a unique solution $X_t$ satisfying $\Vert X_t \Vert < R$ for any $X_0\in
 A^n_{R/3}$.  We denote it by   $X_t^\alpha(X_0, \{S_s, s\in [0,t]\}), t\ge 0$, and $X_t^\alpha$ in short.
We  set  for $U\subset A_R^n$, $U_\alpha $ be the subset of its elements stable under the flow:
$$U_\alpha=\{X_0\in U : \forall t\,,\quad  X_t^\alpha\in U\}$$ 
\begin{Lemma}\label{EstimXt} Let {$U\subset A_{R,UltraApp}^n$}. Under  Assumption \ref{thehyp}, the map
 $${X_0\in U_{\alpha 
}\mapsto }X_t^\alpha(X_0, \{S_s, s\in [0,t]\})$$
 comes from an element in $C^{1,0}_{tr,V,c}(A,U_\alpha:\mathscr{B},E_D:\mathscr{S})$, {and we have for any $\tau<t$ the relation \begin{equation}\label{poll}
 X_t^\alpha(., \{S_s, s\in [0,t]\})=\theta_\tau'[X^\alpha_{t-\tau}(., \{S_{s}, s\in [0,t-\tau]\})]\circ_\tau X_\tau^\alpha(., \{S_s, s\in [0,\tau]\})\,\end{equation}}
where $\theta_u':C^{k,l}_{tr,V,c}(A,U_\alpha:\mathscr{B},E_D:\mathscr{S})
\to C^{k,l}_{tr,V,c}(A,U_\alpha:\mathscr{B},E_D:\mathscr{S}_{\geq u}),$ the map induced by the shift $\theta_u(S_s)=S_{s+u}-S_u$.
Moreover, if we also assume $V,W\in C^{k+l+2}_{c}(A,2R:\mathscr{B},D)$, then $X_0\mapsto X_t^\alpha(X_0, \{S_s, s\in [0,t]\})\in C^{k+l}_{c}(A,U_\alpha:\mathscr{B},D:\mathscr{S})\to C^{k,l}_{tr,V,c}(A,U_\alpha:\mathscr{B},E_D:\mathscr{S})$. Moreover, in each case $t\mapsto X_t^\alpha$ is continuous.

Finally there exists a finite constant $C_{k+l}$ such that, for $k+l\geq 1$ :
\begin{equation}\label{bornenorm}
||X_t^\alpha||_{k+l,0,U_\alpha
{,c,\geq 1}}\leq C_{k+l}e^{-ct/2}.\end{equation}

\end{Lemma}

{Note that $X_t^\alpha(X_0, \{S_s, s\in [0,t]\})$ above is a non-commutative function without expectation but can be thought of as an element of this larger space of functions, hence the reference to $l$. Note that most of the results only depends on $k+l$.}
\begin{proof} 
Let $k\geq 1,l\geq 0$ so that  $V,W\in C^{k+l+2}_{c}(A,2R:\mathscr{B},D)$.
We now prove that $X^\alpha$ can be seen as a smooth function of $X_0, S$, in the sense that it is an element of $C^{k+l}_{c}(A,U_\alpha:\mathscr{B},D:\mathscr{S})$. Fix $T$ small enough, 
 such that in particular  $2\sqrt{T}+T\sup_{X\in A_{2R}^n,i}||\mathscr{D}_i V_\alpha (X)||_A\leq R$ . We construct by Picard iteration the process on $[0,T]$. We let  $X^{[0,m]}$ be 
defined  recursively by $X^{[0,0]}_.=X_0$ and for $m\ge 1$, 
\begin{equation*} 
X_t^{[0,m]} = S_t-\frac{1}{2} \int_0^t \mathscr{D} V_\alpha (X_u^{[0,m-1]}) du+X_0, t\in [0,T]\,.
\end{equation*}
Because $\|X_0\|\le R$, one checks by induction on $m$ that $||X_t^{[0,m]}||\le 2R$, and the processes are indeed well defined for all $m$ {as a $C^{k,l}_{tr,V,c}$ function.}
 Since $X_t^{[0,m]}$ is obtained from $X_t^{[0,m-1]}$ by operations of integration over a subset of $[0,T]$ and composition with $\mathscr{D}V$, we may use Corollary \ref{compositionCkl} and $\mathscr{D} V\in C^{k+l+1}_c(A,2R:\mathscr{B},D)^n$ to prove that the Picard iteration procedure is first bounded (for $T$ small)  and then converges in the norm $\|.\|_{k+l,0,U_\alpha,c}$ (for $T$ even smaller so that the equation is locally lipschitz on the a priori bound obtained before in $\|.\|_{k+l,0,U_\alpha,c}$). We let $X_s, s\le T$ be the limit : it belongs to  $C^{k+l}_{c}(A,U:\mathscr{B},D:\mathscr{S})$ and is the unique solution of \eqref{eqn:Xt}. By the definition of $U_\alpha$, for $X_0\in U_\alpha$, $X_s\in U_\alpha$, {in particular $||X_s||\leq R$} . Hence, we can iterate the process by considering 
 for $s\in [0,T]$ the sequence
defined  recursively by $X^{[s,0]}_t=X_s, t\le T$ and for $m\ge 1$
\begin{equation*} 
X_t^{[s,m]} = S_t-S_s-\frac{1}{2} \int_s^t \mathscr{D} V_\alpha (X_u^{[s,m-1]}) du+X_s, t\in [s,s+T]\,.
\end{equation*}
 Again this sequence converges in the norm $\|.\|_{k+l,0,U_\alpha,c}$ to a limit $X^{[s,\infty]}$. As $V$ is $C^{k+l+2}_c(A,2R:\mathscr{B},D)$, such construction has a unique solution so that
 $X^{[s,\infty]}_t=X^{[s',\infty]}_t$ for all $s,s'\le t$. We denote this solution $X^\alpha$. It satisfies  \eqref{poll}. 
We continue by induction to construct $X^\alpha\in C^{k+l}_{c}(A,U_\alpha:\mathscr{B},D:\mathscr{S})$ for all times.  
 The continuity of $t\to X_t$ is clear, as a uniform  limit of continuous functions.
 
 We finally show \eqref{bornenorm}.
 Using the first formula in the proof of Lemma \ref{compositionCklFixed} on the equation on Picard iterates and then taking the limit $m\to\infty$, one gets for $k\geq 1$:

\begin{align} \label{HigherProcess}\begin{split}
&\partial^{k}_{(j_1,...,j_k)}X_t^{(i)} 
= -\frac{1}{2}\int_s^tdu \sum_j\partial_j\mathscr{D}_i V_\alpha (X_u) \# \partial^{k}_{(j_1,...,j_k)}X_u^{(j)}+ \text{ l.o.t}+ \partial^{k}_{(j_1,...,j_k)}X_s\end{split} 
\end{align}
where the lower order terms ($l.o.t$) are with respect to the degree $k$ of differentiation of $X_u$.
Evaluating the differentials and using Lemma \ref{lemhyp} (b), one gets the exponentially decreasing bound  on their norms by induction over $k$ (note that all the other lower order terms are non-linear in derivatives of $X_t$ and thus bring more than one exponential decreasing enabling to compensate the increase via time integrals).
\end{proof}
\begin{Lemma}[It\^o's formula] Under  Assumption \ref{thehyp},  for 
$
P\in B\{ X_1,...,X_n:E_D,R\}$ we have 
\begin{eqnarray}
P(E_{D,X_t^\alpha})(X_t^\alpha)&=&P(E_{D,X_0^\alpha})(X_0^\alpha)+\frac{1}{2}\int_0^t[(\Delta_{V_{\alpha}}+\delta_{V_{\alpha}})P](E_{D,X_s^\alpha})(X_s^\alpha)ds\label{eqito} \\ \nonumber
&&\qquad +\int_0^t\partial[P(E_{D,X_s^\alpha})(X_s^\alpha)]\#dS_s.\end{eqnarray}

\end{Lemma}
\begin{proof}For $P$ (later called polynomial) in the algebra generated by $B,X_1,...,X_n$  inside $ B\langle X_1,...,X_n:D,R\rangle$, this is the standard It\^o's formula, see \cite{biane-speicher0,biane-speicher}.
{ By the norm continuity of all operations appearing, the extensions to $\ell^1$ direct sums are obvious, so that it suffices to extend the formula to a monomial $P\in B\langle X_1,...,X_n:D,R\rangle$ having only one term in the direct sum. Finally, using the standard decomposition of elements in extended Haagerup tensor products \cite{M97} thanks to which $P\in B^{\oeh{D}n}$ can be written $P=x_1\otimes_D... \otimes_Dx_n$ with $x_1\in M_{1,I_1}(D), x_i\in M_{I_{i-1},I_{i}}(D)$ with $I_j$ infinite indexing sets but $I_n=1$.
We can truncate these infinite matrices  by finite matrices, giving a net of approximation $P_n$ of $P$. 
 All the terms in It\^o's formula, once evaluated at a given time, will then   converge  in $L^2(M)$ (while staying bounded in $M$).  
Unfortunately, to get convergence of the time integrals we have to be a bit more careful.  Considering evaluations into $L^\infty([0,T],A)$ it is only possible to get a   }
bounded net $P_n$ of polynomials such that   $P_n(X_t^\alpha),P_n(X_0^\alpha)$ converges weak-* to $P(X_t^\alpha),P(X_0^\alpha)$, in $A$, 
 $\partial[P_n(X_s^\alpha)]$ converges weak-* to $\partial[P(X_s^\alpha)]$ in $A\oeh{D}A.$ For every $s\in [0,t]$,  $s\mapsto [\Delta_{V_{\alpha}}P_n](X_s^\alpha)$ 
converges weak-* to $s\mapsto [\Delta_{V_{\alpha}}P](X_s^\alpha)$  in $L^\infty([0,t],A).$ Then considering constant functions with value in $L^1(A)$, it is easy to deduce the first line in the right hand side of It\^o formula for $P_n$ weak-* converges to the one for $P$  in $A$. To check the same result for the stochastic integral term, note that by Clarck-Ocone's formula and a priori boundedness of all the stochastic integrals, it suffices to check that for an adapted bounded $U_s$, we have convergence to $0$ of the pairing $$\langle \int_0^t\partial[(P_n-P)(X_s^\alpha)]\#dS_s,\int_0^tU_s\#dS_s\rangle =\int_0^t\langle\partial[(P_n-P)(X_s^\alpha)], U_s\rangle ds.$$
Since  $(P_n-P)$ is a bounded net in $ B\langle X_1,...,X_n:D,R\rangle$,  $r=\sup_{s\in [0,t]} ||X_s^\alpha||<R$ and 
$X_s^\alpha$ is continuous, for $p$ large enough  $\sup_{s\in [0,t]} ||X_s^\alpha-X_{\lfloor ps\rfloor/p}^\alpha||$ is so small that $||\partial[(P_n-P)(X_s^\alpha))]-\partial[(P_n-P)(X_{\lfloor ps\rfloor/p}^\alpha))]||\leq \epsilon$ uniformly in $n$ for an arbitrary $\epsilon>0$.

Finally $U\in L^2([0,t],L^2(A)\o_DL^2(A))$ so that approximating it by a process with finitely many values and using weak-* convergence of the finitely many values of $\partial[(P_n-P)(X_{\lfloor ps\rfloor/p}^\alpha)]$, one gets  $\int_0^t\langle\partial[(P_n-P)(X_{\lfloor ps\rfloor/p}^\alpha)], U_s\rangle ds\to 0$. This completes the proof of the formula for $P\in B\langle X_1,...,X_n:D,R\rangle$.

For $P$  in the algebra generated by $B,X_1,...,X_n$,  
notice that the previous computations show that
$$E_D[P(X_t^\alpha)]=E_D[P(X_0)]+\frac{1}{2}\int_0^t E_D[\Delta_{V_\alpha} P(X^\alpha_s)] ds$$
so that by induction over the number of conditional expectations, if $P$ belongs to the algebra generated by $B,X_1,...,X_n,E_D$,  
$$E_D[P(X_t^\alpha)]=E_D[P(X_0)]+\frac{1}{2}\int_0^t \delta_{V_\alpha}(E_D(P))(X^\alpha_s) ds$$
Formula \eqref{eqito} follows for $P$ polynomial in  the algebra generated by $B,X_1,...,X_n,E_D$.

The reduction from $
P\in B\{ X_1,...,X_n:E_D,R\}$ to an element of the algebra generated by $B,X_1,...,X_n,E_D$ is similar.  Indeed, we can canonically embed $ \iota: B\{ X_1,...,X_n:E_D,R\}\to B\langle X_1,...,X_n, S_j,j\in\mathbb N:D,R\rangle$ where the $S_i$ are free semi-circle, free with amalgamation over $D$. Each term in $E_D$ corresponds to a different set of $S_i$  and 
$$P(E_{D,X})(X)=E_{W^*(X_1,\ldots,X_n,B)}[\iota(P)(X_1,\ldots,X_n,S_i,i\in\mathbb N)]\,.$$
We can conclude 
by the previous considerations and the weak-* continuity of $E_{W^*(X_1,\ldots,X_n,B)}$.

\end{proof}

\subsection{Semigroup}\label{semigroup}
Hereafter, we will often need a second technical assumption on $D\subset B$ to apply  Theorem \ref{Finite3}.(3)  and Proposition \ref{CyclicPermutations}.(2) in the appendix. The appropriate definitions are given in the appendix in subsection \ref{CyclicH}. \begin{Assumption}\label{thehyp2}
Assume  \begin{itemize}\item[$\bullet$] either that there exists a $D$-basis of  $L^2(B)$ as a right $D$ module $(f_i)_{i\in I}$  which is also a $D$-basis of  $L^2(B)$ as a left $D$ module 
\item[$\bullet$] or that $D$ is a $II_1$ factor and that $L^2(B)$ is an extremal $D-D$ bimodule.\end{itemize}
\end{Assumption}

 As discussed in the appendix, the easiest non-trivial example of a pair $(B,D)$ satisfying this assumption is $B=\Gamma \ltimes D$ a crossed-product by a countable (or finite) discreate group $\Gamma$. In particular, when $B=D$ this assumption is obviously satisfied.

We write $A_{R,App}^n\subset A_{R,UltraApp}^n$ the set $ A_{R,UltraApp}^n$ if $D=\C$ and otherwise the set requiring additionally $M=W^*(B,X_1,...,X_n)\subset W^*(B,S_1,...S_m)=B*_{D}(D\otimes W^*(S_1,...,S_m)$ included into the algebra generated by $m$ semicircular variables over $D$. Here, $m$ can be infinite. This will be crucial  when we will assume  $D\subset B$ satisfying the assumption of Theorem \ref{Finite3}.(3) so that the conclusion of this Theorem and Proposition \ref{CyclicPermutations}.(2) will then be available for $M$ in the sense that $\langle e_D,.\#e_D\rangle$ will be a trace on $D'\cap M\oeh{D} M$. 
 
We define: $$A_{R,\alpha}^n=( {A_{R,App}^n})_{\alpha}.$$
Proposition \ref{ConvSDE} implies $A_{R/3,App}^n\subset A_{R,\alpha}^n.$
Let $$A_{R,\alpha,conj}^n=\{X\in A_{R,\alpha}^n, \partial_i^*(1\o 1)\in W^*(X, B),i=1,...,n\}\,.$$
 Using \cite[Theorem 27]{dabrowski:SPDE}
 (first for $V$ polynomial and then for all $V$  by density), one gets that for any $X\in A_{R,\alpha}^n$, 
$X_t^\alpha \in  A_{R,\alpha,conj}^n$ for any $t>0.$
Hereafter we thus assume that $X_0\in A_{R,\alpha,conj}^n$.

Denote 
$A_{R,\alpha, conj1}^n=
A_{R,\alpha, conj}^n$, $A_{R,\alpha, conj0}^n=
A_{R,\alpha}^n$. Hereafter, we will consider only functions of $X$ and $E_{D,X}$, we therefore drop the dependency in $E_{D,X}$ in the notations.
Because we will need later to apply the cyclic gradient to the image of the semi-group, we will need the following  ad'hoc space 
$C_{tr}^{k,l{;}-1}(A,A_{R,\alpha,conj}^n)$ which is the completion of  $B_c\{X_1,\ldots,X_n:E_D,R,\C\}$ for 
 \begin{align*}&||P||_{C^{k,l{;}-1}_{tr}(A,U:B,E_D)}=||\iota(P)||_{k,l,U}+1_{k\geq 1}{\sum_{p=1}^l\sum_{i=1}^n}||\mathscr{D}_{i}(P)||_{k,p,U}
 \end{align*}
 Generalizations of this norm are discussed in the appendix \eqref{normehorrible}.

\begin{Proposition}\label{semig} Suppose Assumptions \ref{thehyp} and \ref{thehyp2} hold.  Let $k\in\{ 2,3\},l\geq 0$ be given and assume
$V,W\in C_c^{k+l+2}(A,2R:B,D)$.  The process $X_t^\alpha$ of Lemma \ref{EstimXt} defines a 
strongly continuous semigroup $
\varphi_t^{\alpha}$  
on $C_{tr}^{k,l}(A,A_{R,\alpha,conj}^n:B,E_D)$  and, {on $C_{tr}^{k,l{;}-1}(A,A_{R,\alpha,conj}^n:B,E_D),$ if moreover $V,W\in C_c^{k+l+3}(A,2R:B,D)$.}
They are given by  the formula
$$\varphi_t^{\alpha}(P)=E_0(P(X_t^\alpha))\,.$$
It satisfies the exponential bounds  :
$$||\varphi^\alpha_t(P)||_{k,l,A_{R,\alpha,conj }^n\geq 1}\leq C_{k,l}||P||_{k,l,A_{R,\alpha,conj}^n\geq 1}e^{-ct/2},$$

Moreover, when restricted to $C_c^{k+l}(A,A_{R,\alpha,conj}^n:B,E_D)$, one gets  strongly continuous one parameter families of maps $$\varphi_t^{\alpha\prime}:C_c^{k+l}(A,A_{R,\alpha,conj}^n:B,E_D)\to C_{tr,V_\alpha}^{k,l}(A,A_{R,\alpha,conj}^n:B,E_D),$$
 with $\varphi_t^{\alpha}=\iota\varphi_t^{\alpha\prime}$ for 
the canonical map

\noindent $\iota:C_{tr,V_\alpha}^{k,l}(A,A_{R,\alpha,conj }^n:B,E_D)\to C_{tr}^{k,l}(A,A_{R,\alpha,conj}^n:B,E_D)$. It satisfies
$$||\varphi_t^{\alpha\prime}(P)||_{C_{tr,V_\alpha}^{k,l}(A,A_{R,\alpha,conj}^n),\geq 1}
\leq C_k||P||_{k,l,A_{R,\alpha,conj }^n\geq 1}e^{-ct/2}.$$

\end{Proposition}
\begin{proof}
$\varphi_t^\alpha$ is well defined  in all cases by composing the maps $X_t^\alpha$ from Lemma \ref{EstimXt}, the composition $(P,X_t)\to P(X_t)$,
see Lemma \ref{compositionCklFixed},   and expectations $E_B$ from Proposition \ref{ExpectationCkl}. 
To get a semigroup we apply composition for $\tilde U=U= A_{R,\alpha,conj}^n$, we have to check the consistency condition for composition, i.e. for any $X_0\in U,$ we have to check that $X_t^\alpha(X_0,\{S_s, s\in [0,t]\})$ has one  conjugate variable. {This is proved in Proposition \ref{boundConj} in the  appendix. Note this is where we need Assumption \ref{thehyp2} and the condition $A_{R,\alpha}^n\subset A_{R,App}^n$ in order to apply Proposition \ref{CyclicPermutations} to get $M=W^*(B,X_0)*_D(D\otimes W^*(S_t,t>0)$,  $\tau=\langle e_D, .\#e_D\rangle$ is a trace on $D'\cap M\oeh{D} M$.} The construction of   $\varphi_t^{\alpha\prime}$ and the consistency follow similarly.
 
 Let us check the semigroup property. It follows from the following formal computation :
\begin{align*}\varphi_u^\alpha(\varphi_{t-u}^\alpha(P))&=E_0(\varphi_{t-u}^\alpha(P)\circ X_u(., \{S_s, s\in [0,u]\}))
\\&=E_0([E_0(P\circ X_{t-u}(., \{S_s, s\in [0,t-u]\}))]\circ X_u(., \{S_s, s\in [0,u]\}))
\\&=E_0(E_u(\theta_u'[(P\circ X_{t-u}(., \{S_s, s\in [0,t-u]\})]\circ_u X_u(., \{S_s, s\in [0,u]\})))
\\&=E_0(E_u(P\circ[\theta_u'[( X_{t-u}(., \{S_s, s\in [0,t-u]\})]\circ_u X_u(., \{S_s, s\in [0,u]\})]))
\\&=E_0(E_u(P\circ X_t(., \{S_s, s\in [0,t]\})))
\\&=E_0(P\circ X_t(., \{S_s, s\in [0,t]\}))
\\&=\varphi_t^\alpha(P)
\end{align*}
where $\circ_u$ is the composition defined in Proposition \ref{ExpectationCkl}.
To justify this computation, the two first and last equations are the definitions of the ``semigroup",  third, fourth and next-to -last lines come from Proposition \ref{ExpectationCkl} and the fifth line from Lemma \ref{EstimXt}.
We thus have the semigroup relation. The  strong continuity  on both spaces comes from continuity in $u$ of $X_u^\alpha$  
(see Lemma \ref{EstimXt}) and continuity of various compositions and $E_0$  (see  Proposition \ref{ExpectationCkl}).

Using the variant of \eqref{composition1stOrder} { with $U=V=A^n_{R,\alpha,conj}$,}
 in the context of Proposition \ref{ExpectationCkl},  that is adding Brownian motions filtration, we get
\begin{align*}&||P(X_t)||_{k,l,U\geq 1}\\&\leq C(k,l,n)||P||_{k,l,V\geq 1}
\left(1+\max_{i=1,...,n}||X_t^i||_{k,l,U\geq 1}\right)^{k+l-1}\max_{i=1,...,n}||X_t^i||_{k,l,U\geq 1}\,.\end{align*}

Using  contractivity of expectations in Proposition \ref{ExpectationCkl} and bounds in Lemmas \ref{Ckc} and \ref{EstimXt}, one gets the  exponential bounds  as claimed. The bounds for the seminorm $||P(X_t)||_{k,l,U\geq 1}$ follow similarly. 

Moreover, we get similar results for $\varphi^{\alpha\prime}$ by noticing that if $P\in C_c^{k+l}(A,A_{R,\alpha,conj }^n,B,E_D)$, 
$P(X_t^\alpha)\in C_c^{k+l}(A,A_{R,\alpha,conj }^n,B,E_D:\mathscr S)$. The seminorm on this space is equivalent to 
$||P(X_t)||_{k,l,U\geq 1}$ which we already estimated. Continuity of conditional expectation, see Proposition \ref{ExpectationCkl},
allows to get the exponential bounds for $\varphi^{\alpha\prime}$.
\end{proof}
We next find the generator for the semi-group $\varphi_t^{\alpha \prime}$:  it is given by $L_\alpha=\frac{1}{2}(\Delta_{V_{\alpha}}+\delta_{V_{\alpha}})$ and we precise in the next Lemma some dense domains of this generator (without looking for the maximal one).   
\begin{Proposition}\label{infgen}Assume Assumption \ref{thehyp} and \ref{thehyp2} and let $k\in\{2,3\},l\ge 2,$ be given with
$V,W\in C_c^{k+l+2}(A,2R:B,D)$ as before.
We let $\iota'$ be the canonical map
$$\iota':C_{tr,V_\alpha}^{k,l}(A,A_{R,\alpha,conj }^n:B,E_D)\to C_{tr}^{k-2,0{;}-1}(A,A_{R,\alpha,conj}^n:B,E_D),$$
then for any $P\in C_c^{k+l}(A,A_{R,\alpha,conj }^n:B,E_D), k\geq 2$,  $t\mapsto \iota'(\varphi_t^{\alpha\prime}(P))$ is $C^1$ and
$$\frac{\partial}{\partial t}\iota'(\varphi_t^{\alpha\prime}(P))=L_\alpha(\varphi_t^{\alpha\prime}(P)),$$
where  $L_\alpha : C_{tr,V_\alpha}^{k,l}(A,A_{R,\alpha,conj }^n:B,E_D)\to C_{tr}^{k-2,0{;}-1}(A,A_{R,\alpha,conj}^n:B,E_D)$ is given by   $L_\alpha=\frac{1}{2}(\Delta_{V_{\alpha}}+\delta_{V_{\alpha}})$. 
 \end{Proposition}  
\begin{proof}
To compute the generator we start with It\^o formula \eqref{eqito}.
Taking a conditional expectation, we deduce for $P\in B_c\{X_1,\ldots,X_n:E_D,R,\C\}$,
\begin{equation}\label{ItoSemig}\varphi_t^\alpha(P)(X_0)-P(X_0)-\frac{t}{2}(\Delta_{V_{\alpha}}+\delta_{V_{\alpha}})P(X_0) =\frac{1}{2}\int_0^t(\varphi^\alpha_s-\varphi^\alpha_0)[(\Delta_{V_{\alpha}}+\delta_{V_{\alpha}})P](X_0)ds\,.\end{equation}
 We now want to check the same relation under a full cyclic gradient $\mathscr{D}$.  We need to check that all the terms above are in $C^{k,1{;}-1}_{tr}(A,R:B,E_D)$ for our chosen $P$. But  we won't check that  the relation \eqref{ItoSemig} is valid in this space, we will only show this relation holds after application of the cyclic gradient in   each representation. Indeed, we do not know if the full cyclic gradient $\mathscr D$ is closable, on the contrary to the free difference quotient.  
 
From the definition of  $[(\Delta_{V_{\alpha}}+\delta_{V_{\alpha}})P]$  (see Def.~\eqref{defDeltaV})  as an evaluation of $$[\Delta_{V_0(Z)}+\delta_{V_0(Z)}](P)\in B_c\{X_1,\ldots,X_n,Z_1,\ldots,Z_n:E_D,R,\max_i(||\mathscr{D}_i V_\alpha||_{0,0,A_R^n})\C\}$$ at $(X,Z)=(X,\mathscr{D} V_\alpha(X))$
\noindent
$ \in (C^{k+2}_c(A,2R,B,E_D))^{2n}$, the fourth composition result in Corollary \ref{compositionCkl} gives the expected $[(\Delta_{V_{\alpha}}+\delta_{V_{\alpha}})P]\in C^{k,1{;}-1}_{tr}(A,R:B,E_D)$. The fact that the terms below semigroups are in the expected space then follows from Proposition \ref{semig} since $V,W\in C^{k+4}_c(A,2R,B,E_D)$.

Note that all our terms are known to be in our expected space, we can apply \eqref{CyclicDifferential} so that the equation \eqref{ItoSemig} under $\mathscr{D}$  is  true in any representation $X_0\in A_{R,\alpha}^n$ if it is true under the differential $d_{X_0}$. 
 Integrals are dealt with thanks to continuity of the semigroup with value in $C^{k,l{;}-1}_{tr}(A,R:B,E_D)$ from the previous Lemma. Seeing both sides of the equation \eqref{ItoSemig} as a function of $X_0$, one can differentiate both sides of \eqref{ItoSemig} under $d_{X_0}$ and obtain equality of both sides  in each representation.  We deduce the equality under 
  the abstract $d_{X_0}$-differential  in $C^*_{tr}$ by injectivity of the map from  $C^{0,l}_{tr}$ to $C^0(A_{R,\alpha}^n,A)$  (contrary to the space $C^{k,1{;}-1}_{tr}$ before where this is unknown). We have thus deduced the equality in each representation :
\begin{eqnarray*}
&&\mathscr{D}_{X_0,i}\varphi_t^{\alpha}(P)(X_0)-\mathscr{D}_{X_0,i}P(X_0)-\frac{t}{2}\mathscr{D}_{X_0,i}(\Delta_{V_{\alpha}}+\delta_{V_{\alpha}})P(X_0)\qquad \qquad\\
&&\qquad\qquad\qquad\qquad  =\frac{1}{2}\int_0^t\mathscr{D}_{X_0,i}(\varphi_s^{\alpha}-\varphi_0^{\alpha})[(\Delta_{V_{\alpha}}+\delta_{V_{\alpha}})P](X_0)ds.\end{eqnarray*}
Applying Lemma \ref{cyclic} and seeing $P$ as an element of $  C_{tr,V_\alpha}^{k,l}(A,A_{R,\alpha,conj}^n)$, one knows that all the terms of the equality are in the domain of  order $k-2$  free difference quotient 
and without having applied cyclic derivative, also in the domain of  order $k-2$ free difference quotient (since $k,l\geq 2$). By closability, if $X_0\in A_{R,\alpha,conj}^n$ we can apply the $k-2$ order free difference quotient to the relation above and deduce corresponding relations.
Therefore,  the following bound extends for $k\geq 2$ to $P\in C_{tr,V_\alpha}^{k,l}(A,A_{R,\alpha,conj}^n)$:
$$||\frac{1}{t}(\varphi_t^{\alpha}(P)-P)-\frac{1}{2}(\Delta_{V_{\alpha}}+\delta_{V_{\alpha}})P||_{k-2,0{;}-1,A_{R,\alpha,conj}^n}\qquad$$
$$\qquad\leq 
\frac{1}{2t}\int_0^t||(\varphi_s^{\alpha}-\varphi_0^{\alpha})[(\Delta_{V_{\alpha}}+\delta_{V_{\alpha}})P]||_{k-2,0{;}-1,A_{R,\alpha,conj}^n}\to 0$$
goes to zero when $t\to0^+$, by the  strong continuity of $\varphi_s^{\alpha}$ on 
 $C_{tr}^{k-2,0{;}-1}(A,A_{R,\alpha,conj}^n)$. This  gives the right derivative of $\varphi_t^\alpha$ at zero.


Now for $Q\in C_c^{k+l}(A,A_{R,\alpha,conj }^n,B,E_D),$ by the semigroup property $\varphi_{s+t}^{\alpha}(Q)=\varphi_{s}^{\alpha}(\iota'\varphi_t^{\alpha\prime}(Q))$ and applying the reasoning above to $P=\varphi_t^{\alpha\prime}(Q)$, one gets the right derivative at any time.

To compute the left derivative, we start similarly from the result of It\^o Formula to $P=\varphi_{t-s}^{\alpha\prime}Q$ starting at time $t-s$ and using also the semigroup property
\begin{align*}\varphi_t&(Q)(X_0)-\varphi_{t-s}(Q)(X_0)-\frac{s}{2}(\Delta_{V_{\alpha}}+\delta_{V_{\alpha}})\varphi_t^{\alpha\prime}(Q)(X_0) \\&=\frac{1}{2}\int_{t-s}^t(\varphi_{u-t+s}^{\alpha}[(\Delta_{V_{\alpha}}+\delta_{V_{\alpha}})\varphi_{t-s}^{\alpha\prime}Q]-(\Delta_{V_{\alpha}}+\delta_{V_{\alpha}})\varphi_t^{\alpha\prime}(Q))(X_0)du\\&=\frac{1}{2}\int_{t-s}^t\varphi_{u-t+s}^{\alpha}[(\Delta_{V_{\alpha}}+\delta_{V_{\alpha}})(\varphi_{t-s}^{\alpha\prime}-\varphi_t^{\alpha\prime})(Q)](X_0)du\\&+\frac{1}{2}\int_{t-s}^t(\varphi_{u-t+s}^{\alpha}-\varphi_{0}^{\alpha})[(\Delta_{V_{\alpha}}+\delta_{V_{\alpha}})\varphi_{t}^{\alpha\prime}(Q)](X_0)du.\end{align*}

Thus, using  strong continuities  of $\varphi^\alpha$ and $\varphi^{\alpha\prime}$, and reasoning as before in the more general spaces with some free difference quotient and cyclic derivative, we conclude that the left derivative is  in $C_{tr}^{k-2,0{;}-1}(A,A_{R,\alpha,conj}^n).$ 
\end{proof}
\section{Construction of the transport map}\label{cons}

Let $F\in C_{tr}^{k,l}(A,U)^n$, $k,l\ge 1$. Let $X=(X_1,\ldots,X_n) \in U$. Then we define $\partial_F=(\partial_{F^1},\ldots,\partial_{F^n})$ on $B\langle F^1(X),\ldots,F^n(X)\rangle$ as the free difference quotient 
of the variables $F^1(X),\ldots, F^n(X)$. Assume $ W^*(B,X_{1},...,X_{n})=M\subset (A,\tau)$ and let $S$ be a semicircle variable, free from $M$  with amalgamation over $D$.  Let $q\in D'\cap M\oeh{D} M$.
The adjoint  $\partial_F^*$ of $\partial_F$, when it exists, is given by
$$\tau((q\# S)^*\partial_{F^i} P\#S)=\tau( (\partial_{F^i}^*(q))^* P),\quad 1\le i\le n\,.$$
The Jacobian matrix is given by $\mathscr J(F)=(\partial_j F_i)_{ij}$. We define for $G\in  C_{tr}^{1,1}(A,U)^n$,  $\mathscr J_F(G)=(\partial_{F^j}G^i)_{1\le i,j\le n}$.
Its adjoint is given for $q\in M_{n}(D'\cap M\oeh{D} M)$ by
$$\mathscr J^*_F (q)=\left(\sum_{i} \partial_{F^i}^* (q_{ji})\right)_{j=1}^n\,.$$

We will need the following preparatory Lemma regarding conjugate variables.  We will need a temporary technical assumption, satisfied under Assumption \ref{thehyp2} if $X_0\in A_{R,App}^n$ as shown in the proof of Proposition \ref{semig}. This will thus be the case for semicircular variables and then via our transport map for other models with h-convex potential.

\begin{Assumption}\label{thehyp3}
Assume  $ W^*(B,X_{0})=M\subset (A,\tau)$  is such that $X\mapsto \tau(S X\#S)$ is a trace on $D'\cap M\oeh{D} M$ if $S$ is a semicircle variable, free from $M$  with amalgamation over $D$.
\end{Assumption}

\begin{Lemma}\label{adjoint}Assume {Assumption \ref{thehyp3}.} Fix such an $X\in U$ with $U\subset A_{R,conj}^n$. Take $l\ge 0$.
Consider  a $C^1$ map $\alpha\mapsto F_\alpha\in C_{tr}^{k,l}(A,U)^n$,  on $[0,\alpha_0]$ for $k\geq 2$,  so that $F_0=X_0$, $\Vert{}1-\mathscr{J}({ F_\alpha})\Vert{}_{M_n(M\oehc{D} M)}<1$.
Let $1\otimes 1$ be the diagonal matrix with  entries $1\otimes_D 1$ on the diagonal. 
Then 
 $\mathscr{J}_{F_{\alpha}}^* (1\otimes 1)\in M^n$ exists for any  $\alpha\in [0,\alpha_0]$,  $\alpha\mapsto\mathscr{J}_{F_{\alpha}}^* (1\otimes 1)$ is in $C^{1}([0,\alpha_{0}],M^n)$ and \begin{equation}\label{poi}
{ \frac{d}{d\alpha} }\mathscr{J}_{F_{\alpha}}^* (1\otimes 1)=-\mathscr{J}_{F_{\alpha}}^*\left([\mathscr{J}_{F_{\alpha}}(\partial_\alpha F_\alpha)]^*\right)\,. \end{equation}
\end{Lemma}
\begin{proof}
The existence of the conjugate variable is a technical variant of \cite{alice-shlyakhtenko-transport} explained in the appendix, see Lemma \ref{adjointAppendix}.  
It is also shown there that
$$\mathscr{J}^*_{F_\alpha}(1\o 1)=\mathscr{J}^*([ \mathscr{J} F_\alpha]^{-1,*})\,.$$
where we  denoted in short
$A^{-1,*}= (A^{-1})^*$.  Let us compute the time derivative of the above right hand side. From the elementary equation $A^{-1}-B^{-1}=A^{-1}(B-A)B^{-1}$,  one deduces  an equation on $(\mathscr{J} F_{\alpha+h})^{-1}$ which after taking the adjoint
reads 
\begin{align*}[&\mathscr{J} F_{\alpha+h}]^{-1,*}=[ \mathscr{J} F_\alpha]^{-1,*}-h [\mathscr{J} F_\alpha]^{-1,*} [\mathscr{J} \partial_\alpha F_\alpha]^*
 [\mathscr{J} F_\alpha]^{-1,*}
 \\&- [\mathscr{J} F_{\alpha+h}]^{-1,*} [\mathscr{J} F_{\alpha+h}-\mathscr{J} F_{\alpha}-h\mathscr{J} \partial_\alpha F_\alpha]^*
 [\mathscr{J} F_\alpha]^{-1,*}
\\& {+}h([\mathscr{J} F_{\alpha+h}]^{-1,*} [\mathscr{J} F_{\alpha+h}-\mathscr{J} F_{\alpha}]^*
 [\mathscr{J} F_{\alpha}]^{-1,*}[\mathscr{J}\partial_\alpha F_\alpha]^*
 [\mathscr{J} F_\alpha]^{-1,*})
\end{align*}

Since all the terms are in a matrix variant of $D(\overline{\partial_i\o_D1}^{eh}\oplus\overline{1\o_D\partial_i}^{eh})$ which is an algebra by Lemma \ref{conjvar}.(1) and $\mathscr{J} F_{\alpha+h}$ is differentiable in this space (using $k\geq 2$), we can conclude from Lemma \ref{conjvar}.(4) to the differentiability under $\mathscr{J}^*$ so that we conclude after letting $h$ going to zero that

$${ \frac{d}{d\alpha} } ( \mathscr{J}^*_{F_\alpha}1\o 1  )= -\mathscr{J}^* [ [\mathscr{J} F_\alpha]^{-1,*}] [\mathscr{J}\partial_\alpha F_\alpha]^*
 [\mathscr{J}_{F_0} F_\alpha]^{-1,*}.$$
 We have the chain rule for any $g\in C^{1,0}_{tr}(A,U)^n$
 \begin{equation}\label{chainrule}
 \mathscr{J} g(F_\alpha)=\mathscr{J}_{F_\alpha}g\#\mathscr{J} F_\alpha\,,
 \end{equation}
 we have, by taking $g=\partial_\alpha F_\alpha$,
 $$[\mathscr{J} \partial_\alpha F_\alpha]^*=[\mathscr{J}_{F_\alpha} \partial_\alpha F_\alpha\#\mathscr{J} F_\alpha]^*=[\mathscr{J} F_\alpha]^*\#[\mathscr{J}_{F_\alpha} \partial_\alpha F_\alpha]^*\,.$$
which completes the proof.
\end{proof}

We will now proceed with the construction of the transport map $F_\alpha$.

 \begin{Lemma}\label{ODE}
Assume that $V,W,B,D,X=X_0$ satisfy  Assumptions \ref{thehyp}, \ref{thehyp2} and \ref{thehyp3}  and that $V,W\in C^6_c(A,2R:B,D)$.  {Fix such an $X\in A_{R/4,conj}^n$.}
Let
 $$
 \mathscr{D}g_\alpha   := -\frac{1}{2}\int_0^\infty  \mathscr{D}(\varphi_t^{\alpha\prime}( W )) dt\in C_{tr,V_{\alpha}}^{2,1}(A,A_{R,\alpha,conj}^n: B,E_D)\,. 
 $$  Then 
 $\mathscr{D}g_\alpha$ satisfies the equation in $C^{0,0}_{tr}(A,A_{R,\alpha,conj}^n: B,E_D)$ :
 \begin{equation}\label{Laplaceg}\mathscr{D}( W )=(\Delta_{V_\alpha}+\delta_{V_\alpha})( \mathscr{D}g_\alpha)- \sum_{j=1}^n\mathscr{D}_{.,\mathscr{D}_jg_\alpha}\mathscr{D}_{j}V_\alpha.\end{equation}
 
 Moreover the differential equation 
 $$ 
 { \frac{d}{d\alpha} } {F}_\alpha = \mathscr{D} g_\alpha (F_\alpha)=(\mathscr{D}_1 g_\alpha (F_\alpha),...,\mathscr{D}_n g_\alpha (F_\alpha))
 $$
 has a unique solution in the space $C^{2,2}_{tr}(A,A^n_{R/4,conj}:B,E_D)$ with the initial condition $F_0= X$ on a small time $[0,\alpha_0]$  for some $\alpha_0\in (0,1]$ which only depends on $c,R, \sup_{\beta\in [0,1]} \| \mathscr{D} g_\beta\|_{C^{2,1}_{tr}(A,A^n_{R/3,conj}:B,E_D)}$, non-increasing in the last variable.
  \end{Lemma}
\begin{proof}
{The integral defining 
$\mathscr{D}g_\alpha$ exists  in the  space $C^{2,1}_{tr,c}(A,A^n_{R,\alpha, conj}:B,E_D)$ 
because of the exponential bound in Proposition \ref{semig} (with $k=2,l=2$).
From the computation of the derivative in $C_{tr,V_\alpha}^{0,0{;}-1}(A,A_{R,\alpha,conj}^n:B,E_D)$ 
in Proposition \ref{infgen},
one gets the derivative in $C^{0,0}_{tr}(A,A_{R,\alpha,conj}^n:B,E_D)$,
 $$\frac{\partial}{\partial t}\mathscr{D}(\varphi_t^{\alpha\prime}( W ))=\mathscr{D}(L_\alpha(\varphi_t^{\alpha\prime}( W )))=L_\alpha\mathscr{D}(\varphi_t^{\alpha\prime}( W ))-\frac{1}{2}\sum_{j=1}^n\mathscr{D}_{.,\mathscr{D}_{j}(\varphi_t^{\alpha\prime}( W ))}\mathscr{D}_{j}V_\alpha$$
where the last identity comes from Lemma \ref{cyclic} with $g=\varphi_t^{\alpha\prime}( W )$ and $k=0$.
Integrating in $t$ and since $\mathscr{D}(\varphi_t^{\alpha\prime}( W ))$ tends to $0$ when $t\to \infty,$ one gets the identity in $C^{0,0}_{tr}(A,A_{R,\alpha,conj}^n: B,E_D)$ :$$\mathscr{D}( W )=(\Delta_{V_\alpha}+\delta_{V_\alpha})( \mathscr{D}g_\alpha)- \sum_{j=1}^n\mathscr{D}_{.,\mathscr{D}_jg_\alpha}\mathscr{D}_{j}V_\alpha.$$ }
Fix $\alpha>0$. We next define an appropriate space on which  the following map $$\chi :F\mapsto \left(\gamma\in [0,\alpha]\mapsto \chi_\gamma=F_0+\int_0^\gamma \mathscr{D} g_\beta(F_\beta) d\beta\right)$$
will be a contraction   for $\alpha$ small enough.  We take $F_\beta\in A_{R/3,conj}^n\subset A_{R,\beta,conj}^n$ to stay in a space independent of $\beta$.
We set, for $\alpha$ to be chosen small enough and for any fixed $K>||F_{0}^i||_{2,0,A_{R/4,conj}^nc}$,
\begin{align*}&\mathcal E_{\alpha,K}=\{ F\in C^0([0,\alpha], (C^{2,0}_{tr,c}(A,A_{R/4,conj}^n:B,E_D))^n):F_0(X)=X,\forall \beta\in [0,\alpha]\\
&\qquad  \Vert 1 - \mathscr{J} F_\beta \Vert_{M_n(M\oehc D M )} \le \frac{1}{2}, \sup_{X\in A_{R/4,conj}^n} ||F_\beta^i(X)||\leq R/3,||F_{\beta}^i||_{2,0,A_{R/4,conj}^n,c}\leq K\}\,.\end{align*}
First, note that $\mathcal E_{\alpha,K}$ is a closed convex set of $C^0([0,\alpha], (C^{2,0}_{tr,c}(A,A_{R/4,conj}^n:B,E_D))^n)$, thus it is complete metric space.

By the previous Lemma (note that we don't need at this point $\alpha\mapsto F_\alpha$  $C^1$), for $F\in \mathcal E_{\alpha,K}$
, $\mathscr{J}^*_{F_\beta}(1\otimes 1)$ exists for $\beta\le \alpha$. Thus for any $X\in A_{R/4,conj}^n$, $F_\beta(X)\in A_{R/3,conj}^n$ and we are in position to apply Lemma \ref{compositionCklFixed} to get $\mathscr{D} g_\beta(F_\beta)\in C^{2,0}_{tr}(A,A_{R/3,conj}^n:B,E_D)^n$. Moreover, applying Lemma \ref{ContSemigAlpha} and the same exponential decay as before  to deal with the tail of the integral,  $\beta\in [0,1]\mapsto \mathscr{D} g_\beta\in C^{2,1}_{tr,c}(A,A^n_{R/3, conj}:B,E_D)$ is continuous. Using our Lemma \ref{compositionCklFixed} for composition, $\alpha\in [0,1]\mapsto \mathscr{D} g_\alpha(F_\alpha)\in C^{2,0}_{tr,c}(A,A^n_{R/4, conj}:B,E_D)$ is also continuous so that the integral defining $\chi$ makes sense. Hence $\chi$ is well defined on $\mathcal E_{\alpha,K}$ with value in $C^0([0,\alpha], (C^{2,0}_{tr}(A,A_{R/4,conj}^n:B,E_D))^n)$. For  $\alpha$ such that
$$\frac{R}{4}+\alpha \sup_{\beta\in [0,1]} \| \mathscr{D} g_\beta\|_{C^0_{tr}(A,A^n_{R/3,conj}:B,E_D)}\le \frac{R}{3}$$
the image of $\chi$ belongs to $A^n_{R/3}$.  Similarly, $\|\chi^i_\beta\|_{2,0,A_{R/4,conj}^n,c}\leq K$ if $\alpha$ is small enough.

Finally, by the chain rule \eqref{chainrule}, we have $\mathscr{J}\mathscr{D} g_\beta (F_\beta)=  \mathscr{J}_{F_\beta}\mathscr{D} g_\beta\#\mathscr{J}F_\beta$ so that 
\begin{eqnarray*}
\Vert \mathscr{J}\mathscr{D} g_\beta (F_\beta)  \Vert_{M_n(A\oehc D A )}&\leq &\Vert \mathscr{J} F_\beta \Vert_{M_n(A\oehc D A )}\Vert \mathscr{J}\mathscr{D} g_\beta  \Vert_{M_n(A\oehc D A )}\\
&\leq & 3/2 \Vert \mathscr{J}\mathscr{D} g_\beta  \Vert_{M_n(A\oehc D A )}.\end{eqnarray*}

Recalling $\mathscr{J}F_0=1$ and using the continuity of $\mathscr{J}\mathscr{D} g_\beta$ one can choose $\alpha=\alpha(K)$ small enough such that $\chi$ is valued in $\mathcal E_{\alpha,K}$. It remains to obtain a contraction, up to choose  $\alpha$ even smaller.

Since $\mathscr{D} g_\beta$ lies in a bounded set in $C^{2,1}_{tr,c}(A,A^n_{R/3, conj}:B,E_D)$ and  $\mathcal{E}_{\alpha,K}$ is bounded, $\mathscr{D} g_\beta$ is uniformly Lipschitz   by Lemma 
\ref{compositionCklFixed}, with a Lipschitz norm which does not depend on $\beta\in (0,1)$.
$\chi$ is thus a contraction on $\mathcal E_{\alpha,K}$. It has therefore a unique fixed point which is our solution which is necessarily  in $C^1([0,\alpha_0], (C^{2,0}_{tr}(A,A_{R/4,conj}^n:B,E_D))^n)$.

\end{proof}

\begin{Lemma}\label{Lemma34}
Assume the Assumption of Lemma \ref{ODE}. Let $\Upsilon_\alpha = \mathscr{J}_{F_{\alpha}(X)}^* (1\otimes 1) - \mathscr{D} V_\alpha (F_\alpha(X))$, where $F_\alpha$ is constructed in Lemma \ref{ODE}.

Then $\Upsilon_\alpha$ satisfies the differential equation in $L^\infty([0,\alpha_0),W^*(X))$:
$$
{ \frac{d}{d\alpha} } \Upsilon_\alpha = -d_{F_\alpha}[\mathscr{D}g_\alpha({F_\alpha}).(\Upsilon_\alpha)].
$$
As a consequence, if $\Upsilon_0=0$, then  $\Upsilon_\alpha=0, \forall \alpha\in [0,\alpha_0]$. 
\end{Lemma}
In other words, for $\alpha\in [0,\alpha_0]$, $F_\alpha(X)$ has conjuguate variables $\mathcal{D} V_\alpha$.
\begin{proof}
Using  our previous computation of derivative of conjugate variables in Lemma \ref{adjoint}, we compute
\begin{equation}\label{der1}
{ \frac{d}{d\alpha} } \Upsilon_\alpha =-\mathscr{J}^*_{F_\alpha} [\mathscr{J}_{F_\alpha}\mathscr{D} g_\alpha(F_\alpha)]- \mathscr{J}\mathscr{D}V_\alpha (F_\alpha)\# \mathscr{D}g_\alpha(F_\alpha)- \mathscr{D}W(F_\alpha)\end{equation}
We next  rewrite the right hand side. To this end, notice  that \eqref{Dan} yields
$$(\mathscr{J}_{F_\alpha}^*\mathscr{J}_{F_\alpha}\mathscr{D}g_\alpha(F_\alpha)) = \mathscr{J}\mathscr{D}g_{\alpha}(F_\alpha) \#(\mathscr{J}_{F_\alpha}^*(1\o 1))
-\Delta(  \mathscr{D}g_\alpha(F_\alpha))\,.$$
Moreover, \eqref{deltaV} gives
$$\Delta_{V_\alpha}(  \mathscr{D}g_\alpha(F_\alpha))=\Delta(  \mathscr{D}g_\alpha(F_\alpha))-
\mathscr{J}\mathscr{D}g_{\alpha}(F_\alpha) \#(\mathscr{D}V_\alpha(F_\alpha))\,.$$
Hence, we have
\begin{align*}&-(\mathscr{J}_{F_\alpha}^*\mathscr{J}_{F_\alpha}\mathscr{D}g_\alpha(F_\alpha)) + \mathscr{J}\mathscr{D}g_{\alpha}(F_\alpha) \#(\mathscr{J}_{F_\alpha}^*(1\o 1)
-\mathscr{D}V_\alpha(F_\alpha)) \\&=(\Delta_{V_{\alpha}}+\delta_{V_{\alpha}})\mathscr{D}g_{\alpha}(F_\alpha)-\delta_{V_{\alpha}}\mathscr{D}g_{\alpha}(F_\alpha)\end{align*} 
Moreover \eqref{flip} and $\partial_i\mathscr{D}_jV_\alpha=\rho(\partial_j\mathscr{D}_iV_\alpha)$  result with
$$ \sum_j[\mathscr{D}_{X,\mathscr{D}_jg_\alpha}\mathscr{D}_jV_\alpha](F_\alpha)= \mathscr{J}\mathscr{D}V_\alpha(F_\alpha) \# \mathscr{D}g_\alpha(F_\alpha)\,.$$
Putting these equalities together give:
\begin{align*}&-(\mathscr{J}_{F_\alpha}^*\mathscr{J}_{F_\alpha}\mathscr{D}g_\alpha(F_\alpha)) + \mathscr{J}\mathscr{D}g_{\alpha}(F_\alpha) \#(\mathscr{J}_{F_\alpha}^*(1\o 1)
\\&\qquad-\mathscr{D}V_\alpha(F_\alpha))- \mathscr{J}\mathscr{D}V_\alpha(F_\alpha) \# \mathscr{D}g_\alpha(F_\alpha) \\&=(\Delta_{V_{\alpha}}+\delta_{V_{\alpha}})\mathscr{D}g_{\alpha}(F_\alpha)-\delta_{V_{\alpha}}\mathscr{D}g_{\alpha}(F_\alpha)- \sum_j[\mathscr{D}_{X,\mathscr{D}_jg_\alpha}\mathscr{D}_jV_\alpha](F_\alpha)
\\&=\mathscr{D}W(F_\alpha)-[d\mathscr{D}g_\alpha(E_{F_\alpha,D}).( \mathscr{J}_{F_{\alpha}}^* (1\otimes 1)-\mathscr{D}V_{\alpha}(F_\alpha))](F_{\alpha})\,.
\end{align*}
where we have finally used  equation \eqref{Laplaceg}  and  Lemma \ref{cyclic}(1) applied to  $\mathscr{D} g_\alpha$.

Hence, \eqref{der1} yields
\begin{equation}\label{der2}
{ \frac{d}{d\alpha} } \Upsilon_\alpha
= -\mathscr{J} \mathscr{D}g_\alpha(F_\alpha) \# \Upsilon_\alpha-[d\mathscr{D}g_\alpha(E_{F_\alpha,D}).(\Upsilon_\alpha)](F_{\alpha})\end{equation}
We thus obtain the expected equation from which we deduce the bound :
\begin{align*}&\|\Upsilon_\alpha\|_\infty:=\max_i\Vert{}\Upsilon_{\alpha}^i\Vert{}_{A}\leq \Vert{}\Upsilon_{0}\Vert{}_{\infty}+\int_0^\alpha  ||g_\beta||_{C^{0,2}_{tr,c}(A,A_{R/3,conj}^n)} \Vert{}\Upsilon_{\beta}\Vert{}_{\infty}d\beta\end{align*}
so that Gronwall's Lemma yields the claim. 
\end{proof}

{Recall that $V_0=\frac{1}{2}\sum_{i=1}^nX_i^2.$} We have to reinforce slightly Assumption \ref{thehyp}, a reinforcement which is still satisfied by our examples of quartic potentials.
\begin{Assumption}\label{thehyphyphyp} Let 
 $V,W\in C_c^3(A,2R:B, E_D)$ be two non-commutative $(c,2R)$ h-convex  functions satisfying Assumption \ref{thehyp} and moreover   for any $\alpha\in [0,1]$, there exists  a solution $(X_1^{V+\alpha W},\ldots,X_n^{V+\alpha W})\in  {A_{R/4,UltraApp}^n}$.\end{Assumption}
\begin{Corollary}\label{maincor}
Let $V,W,B,D$ satisfy Assumption \ref{thehyphyphyp} and \ref{thehyp2}  and $V,W\in C^6_c(A,2R:B,D)$.
Assume also the pair $(cV_0,V-cV_0)$ satisfy Assumption \ref{thehyphyphyp}.{ Fix  an $X\in A_{R/4,conj}^n$ and suppose it follows the free Gibbs law with potential $V$.}

 Let $F_\alpha$, $0\leq \alpha\leq \alpha_0$ be the solution constructed in Lemmas~\ref{ODE} and ~\ref{Lemma34}.  Then:
 \begin{itemize}
 \item[(i)] The law of $F_\alpha(X)$ is the free Gibbs law with potential $V_\alpha = V+\alpha W$;
\item[(ii)] The $W^*$-algebras $W^*(F_\alpha(X),B)$ are equal for all $\alpha\in [0,\alpha_0]$.
\end{itemize}

For any $\alpha\in [0,1]$, the von Neumann algebras generated by $B$ and generators of the free Gibbs law with potential $V_\alpha = V+\alpha W$ are isomorphic.
\end{Corollary}
\begin{proof}
 We first check that Assumption \ref{thehyp3} is satisfied under our assumptions.  First start with the case $V=cV_0$, in which case Assumption \ref{thehyp3} is satisfied thanks to Assumption \ref{thehyp2} and Proposition \ref{CyclicPermutations} (2). Then building the transport map for the pair $(cV_0,V-cV_0)$ the same Assumption \ref{thehyp3} is satisfied for $X\in A_{R/4,conj}^n$.

By the previous Lemma \ref{Lemma34}, we find that $\Upsilon_\alpha=0$, which means that $\mathscr{J}_{F_\alpha}(1\otimes 1) =\mathscr{D}V_\alpha$ for $\alpha\in[0,\alpha_0]$.  Since $V_\alpha$ is by assumption
$(c,2R)$-convex and $\Vert F_\alpha\Vert < R/3$ it follows that the law of $F_\alpha$ is the free Gibbs law with potential $V_\alpha$.  This proves (i).

To see part (ii) fix $\alpha_1\in [0,\alpha_0]$.   Let $\hat{V}_\alpha = V_{\alpha_1} - \alpha W$, with $\alpha\in [0,\alpha_1]$, and consider the same ODE as in Lemma~\ref{ODE}, and call $\hat{F}_\alpha$ the solution.  $V_\alpha$ replaced by $\hat{V}_\alpha$. Note that $F_{\alpha_1}(X)\in A_{R/4,UltraApp}^n$ by Assumption \ref{thehyphyphyp}.
It is not hard to see that $\hat{F}_\alpha(F_{\alpha_1}(X)) , F_{\alpha_1-\alpha}(X)$ are  solutions to the same ODE (only $W$ is changed into $-W$ as it should be since the time is reversed), and is thus the unique solution.  Thus by what we proved, $W^*(\hat{F}_\alpha(F_{\alpha_1}(X)),B)\subset W^*(F_{\alpha_1}(X),B)$, which proves the reverse inclusion and thus $W^*(F_\alpha(X),B)= W^*(X,B),$ for $\alpha\in [0,\alpha_0]$.

Let us  prove the last point of the Corollary. We have just checked the case $\alpha\in [0,\alpha_0]$.
 Moreover, $(V_{\alpha_0},(1-\alpha_0)W)$ satisfies the same assumption as $(V,W)$ with the same constants $(c,R)$. We can therefore perform the previous construction of a function $F_\alpha$ with $(V,W)$ replaced by $(V_{\alpha_0},(1-\alpha_0)W)$.  This can be done until a parameter $\alpha_0'$ which can be chosen to be equal to $\alpha_0$ as the constants $(c,R)$ are the same and the semi-groups under consideration are the same. 
Note also that Assumption \ref{thehyphyphyp} enables to check that $|| F_{\alpha_0}(X)||\leq R/4$ and thus $ F_{\alpha_0}(X)$ satisfies the same assumption as $X$. Applying (i),(ii) in that case concludes to the isomorphism of $W^*(X^{V+\alpha W},B)$  for $\alpha\in [\alpha_0,\alpha_0+\alpha_{0}(1-\alpha_0)]$ if $X^{V+\alpha W}$ are the unique variables with conjugate variables $\mathscr{D}_i(V+\alpha W).$
 
Inductively, one concludes to the isomorphism for any $\alpha\in [0,1[$.
 To complete the proof, it suffices to note that for $\epsilon$ small enough, $V,(1+\epsilon)W$ satisfy the same assumptions (a priori with a different convexity constant and replacing $R/4<R/3$ by any larger value).
\end{proof}

\section{Appendix 1: Cyclic Haagerup Tensor Products}

Let $M$ be a finite von Neumann algebra and $D\subset M$ be  a von Neumann subalgebra.  
Our goal is to define a {notion of $n$-fold} cyclic tensor product $M^{\oehc{D}n}$ which will be a certain subspace of the Haagerup tensor product $M^{\oeh{D}n}$.  We start by considering the case $n=2$, and then use amalgamated free products to build the more general cyclic tensor powers.  

The inspiration for the construction comes from subfactor theory.  Indeed, if $M_0\subset M_1$ is a finite-index inclusion of II$_1$ factors and if $M_k$ denotes the $k$-th step in the iterated Jones basic construction, then (see e.g. \cite[Prop 4.4.1(ii)]{JonesSunders}) $L^2(M_k)$ are precisely the tensor powers of $L^2(M_1)$ regarded as an $M_0$ Hilbert bimodule: $L^2(M_k) = L^2(M_1)^{\otimes_{M_0} k}
$.  Moreover, the higher relative commutants $M_0' \cap M_k$ are precisely the {\em cyclic} tensor powers of $M_1$.  These ideas have been extended to the infinite-index case  \cite{BurnsPhD,PenneysPhD,Penneys}.  In particular, the notion of Burns rotation will be useful for us to get a certain traciality property.  

\subsection{Preliminaries}

\subsubsection{Background and basic results on tensor powers of Hilbert bimodules}
Let $D$ be a $II_1$-factor and let  ${}_D H_D$ be a $D$-Hilbert bimodule, i.e., a Hilbert space carrying a pair of commuting normal actions of $D$.  Recall that a vector $\xi\in H$ is called left (resp. right) bounded if the left (resp. right) action of $M$ on $\xi$ extends to an action of $L^2(M)$ on $\xi$.  There is always a $D$-basis $\{\alpha\}$ of vectors for $H$ which are both right and left bounded \cite{Popa86}. We write $H_{L^2(D)}$ the set of right bounded vectors and ${}_{L^2(D)} H$ the set of left bounded vectors. We call $B_H={}_{L^2(D)} H\cap H_{L^2(D)}$ the set of vectors which are  both left and right bounded. 

Let us denote by  $H^{\o_D n}$ the $n$-fold Hilbert module relative tensor product (for convenience, we set $H^{\o_D 0}=L^2(D)$). Denote by $P_H^n=D'\cap H^{\o_D n}$ the set of central vectors. Following \cite{Penneys}, we denote by  $\{\alpha^n\}$ the basis {for} $H^{\o_D n}$ of tensors of elements of $\{\alpha\}$. 
Similarly, fix $D$-bases  $\{\beta\}, \{\beta^n\}$  for $H_D$ and $(H^{\o_D n})_D$, respectively. 

Let $$C_{n,H}=D^{op\prime}\cap B(H^{\o_D n})$$ and endow it with the canonical trace $$Tr_n=\sum_{\beta^n}\langle \cdot \beta^n,\beta^n\rangle.$$ An example of this is the Jones basic construction, that we denote $\langle M,e_D\rangle$ for
 $D\subset M$.  Then $\langle M,e_D\rangle=D^{op\prime}\cap B(L^2(M))=C_{1,L^2(M)}.$ 
Similarly, 
let $$C_{n,H}^{op}=D'\cap B(H^{\o_D n})$$ with canonical trace $$Tr_n^{op}=\sum_{\alpha^n}\langle \cdot \alpha^n,\alpha^n\rangle.$$ Finally, define the centralizer algebras $$Q_{n,H}=C_{n,H}\cap C_{n,H}^{op}.$$ 

We recall the following definitions from \cite{Penneys}:
\begin{definition}
(i) A Hilbert bimodule  $H$ on a factor $D$ is said to be \textbf{extremal} if $Tr_1=Tr_1^{op}$ on the positive cone $Q_{1,H}^+.$

(ii) A \textbf{Burns rotation} is a map $\rho:P_H^n\to P_H^n$ such that for all $\zeta\in P_H^n, b_1,...,b_n\in B_H$, we have:
$$\langle \rho(\zeta),b_1\o ...\o b_n\rangle=\langle \zeta,b_2\o ...\o b_n\o b_1\rangle.$$
\end{definition}
Examples are given in \cite[section 5.2]{Penneys}. The easiest example is when $H$ has a two-sided basis \cite[Rmk 4.5]{Penneys}.

\begin{Theorem}\cite[Theorems 4.7, 4.20]{Penneys}\label{BurnsR}If $H$ is extremal, $H^{\o_D n}$ is also extremal and for all $n$, there exists a Burns rotation $\rho$ on $P_H^n$ which is a unitary map.
\end{Theorem}
There is also a partial converse \cite[Th 1.4]{Penneys}, although it is not needed for our purposes.

\subsubsection{Haagerup tensor products and the basic construction}
With these preliminaries recalled, we now turn to the definition of cyclic Haagerup tensor product.
We start by a well-known technical result concerning the Jones basic construction. 

If $A$ is an operator space, we write $A^*$ for its dual as an  operator space \cite{PisierBook}.  When $A$ is a  $D-D$ bimodule, we write $A^{\natural }$ for the dual operator $D'-D'$ bimodule in the sense of Magajna \cite{M05}. We will also denote by  $A^{\natural D norm}$ the normal dual defined {when $A$ is itself a tensor product over $D$} in  \cite[Th 3.2]{M05}.  While we will not recall the general definition of the normal dual here, we will mention that in the case that $A$ is itself a tensor product over $D$ (and therefore its dual can be viewed as the space of certain linear maps), the normal dual corresponds to maps that satisfy a normality condition on basic tensors. In the case that $D=\C$, the bimodule dual is the same  as the operator space dual $A^*$.

Let $D\subset M$ be finite von Neumann algebras, let $e_D$ be the Jones projection onto $D$, and denote by $\langle M, e_D\rangle$ the basic construction for $D\subset M$. Let $$A(M,e_D)=\operatorname{Span}\{xe_Dy:x\in L^2(M)_{L^2(D)},y\in {}_{L^2(D)}L^2(M) \}.$$  Denote  by $\mathcal{I}_0(\langle M,e_D\rangle)$ the compact ideal space (cf. \cite[section 1.3.3]{Po02}).  Let $E_{D'}:\mathcal{I}_0(\langle M,e_D\rangle)\to D'\cap \mathcal{I}_0(\langle M,e_D\rangle)$ be the conditional expectation constructed in \cite[Prop 1.3.2]{Po02}.

\begin{Lemma}
With the above notations, $A(M,e_D)$ is weak-* dense in $\langle M, e_D\rangle$, dense in $L^2(\langle M,e_D\rangle)$, $\mathcal{I}_0(\langle M,e_D\rangle)$  as well as $L^1(\langle M,e_D\rangle)$.

The following hold isometrically: $$L^1(\langle M,e_D\rangle)\simeq L^2(M)^*\o_{h D^{op}}L^2(M)= \mathcal{I}_0(\langle M,e_D\rangle)^{* D norm}\subset \mathcal{I}_0(\langle M,e_D\rangle)^*.$$ 

The restriction of  $E_{D'}$  to a normal projection on  $\mathcal{I}_0(\langle M,e_D\rangle)\cap L^2(M)\o_{D}L^2(M)$  induces a cross-section to the quotient map $\mathcal{I}_0(\langle M,e_D\rangle)\to \mathcal{I}_0(\langle M,e_D\rangle)/\overline{[D,\mathcal{I}_0(\langle M,e_D\rangle)]}$. {The Dixmier conditional expectation $E_{D'}:\langle M,e_D\rangle \to D'\cap \langle M,e_D\rangle$ is an extension of $E_{D'}$.}

 The map $E_{D'}$ is pointwise normal in $D$ and thus its adjoint $E_{D'}^*$ induces a projection $E_{D'}^* : L^1(\langle M,e_D\rangle)\to D'\cap L^1(\langle M,e_D\rangle)$  agreeing with the usual projection on $L^1(\langle M,e_D\rangle)\cap L^2(M)\o_{D}L^2(M),$ and giving an isomorphism $$ D'\cap L^1(\langle M,e_D\rangle)\simeq L^1(\langle M,e_D\rangle)/\overline{[D,L^1(\langle M,e_D\rangle)]}.$$ 
\end{Lemma}

\begin{proof}
The  identification $$L^1(\langle M,e_D\rangle)\simeq L^2(M)^*\o_{h D^{op}}L^2(M)= \mathcal{I}_0(\langle M,e_D\rangle)^{* D norm}$$ comes from the fact that both spaces are  preduals of the same von Neumann algebra as follows from the computation of their duals in \cite[Corollary 3.3]{M05}, {the computation of $\mathcal{I}_0(\langle M,e_D\rangle)$ as Haagerup tensor product below} and the identification with extended Haagerup products \cite[Rmk 2.18]{M05}: $$L^1(\langle M,e_D\rangle)\simeq L^2(M)^*\o_{h D^{op}}L^2(M)\simeq L^2(M)^*\oeh{D^{op}}L^2(M).$$

From \cite[Th 3.2, 
 Ex 3.15]{M05} we have the isomorphism $$[{}_{\C}({}_ML^2(M)_{L^2(D)})_D\o_{h D}{}_D({}_{L^2(D)}L^2(M)_M)_{\C}]^{* D norm}\simeq L^1(\langle M,e_D\rangle).$$ (Note that here the operator space structure ${}_D({}_{L^2(D)}L^2(M)_M)_{\C}$ is the one of the indicated Hilbert module structure, not the one as a module over $D^{op}$). It remains to check \begin{align*}\mathcal{I}_0(\langle M,e_D\rangle)&\simeq [{}_{\C}({}_ML^2(M)_{L^2(D)})_D\o_{h D}{}_D({}_{L^2(D)}L^2(M)_M)_{\C}]\\&\subset [{}_{\C}({}_ML^2(M)_{L^2(D)})_D\oeh{D}{}_D({}_{L^2(D)}L^2(M)_M)_{\C}]\simeq \langle M,e_D\rangle\end{align*} but the last inclusion comes again from \cite[Th 3.2, 
 Ex 3.15]{M05}. In this way we identify the compact ideal space with the norm closure of basic tensors in the extended Haagerup tensor product.  This norm closure is exactly the Haagerup tensor product  and thus we deduce the first isomorphism.
 
On the dense space  $\mathcal{I}_0(\langle M,e_D\rangle)\cap L^2(M)\o_{D}L^2(M)$, $E_{D'}$ vanishes on $[D,U]$ for any $U$.  Since $E_{D'}(U)$ is a limit of convex combinations of $u^*Uu= U+ [u^*U,u]$, $u\in D$, $E_{D'}(U)$ has the same image as $U$ in the quotient  $\mathcal{I}_0(\langle M,e_D\rangle)/\overline{[D,\mathcal{I}_0(\langle M,e_D\rangle)]}$.  This gives the claimed isomorphism between the image of $E_D$,  $D'\cap \mathcal{I}_0(\langle M,e_D\rangle)$, and the quotient, as well as the identification with the Dixmier conditional expectation.

The key part of our Lemma is to check $D$-normality of $d\mapsto Tr(E_{D'}^*(V)\xi de_D\eta),$ $V\in L^1(\langle M,e_D\rangle),\xi\in L^2(M)_{L^2(D)},\eta\in {}_{L^2(D)}L^2(M).$ Since $E_{D'}$ is bounded, one may assume $V\in L^2(M)\o_{D}L^2(M)$ in which case obviously
$E_{D'}^*(V)=E_{D'}(V).$ This one is again close to $\sum \lambda_u uVu^*$ so that since $d\mapsto Tr(\sum \lambda_u uVu^*\xi de_D\eta),$ is normal, one gets our result. The second quotient statement is analogous. 
 \end{proof}
{The reader should note that the  identification  $L^1(\langle M,e_D\rangle)\simeq L^2(M)^*\o_{h D^{op}}L^2(M)$ is given on basic tensors by:
\begin{equation}\label{L1basic}xe_Dy\mapsto y\o_{D^{op}}x.\end{equation}
 This will be the key to various flips appearing naturally later.}
 
 \subsection{The cyclic Haagerup tensor product, case $n=2$.}
Recall that the spaces $L^p(\langle M,e_D\rangle)$ are made in compatible couples in the sense of interpolation theory \cite{PisierBook}. {We can see them as} the inductive limit of $L^p(q\langle M,e_D\rangle q)$ for $q$ finite projection.  Thus these spaces are realized as an interpolation pair as a subspace of the topological direct sum  $\oplus_{q\in P_{f}(\langle M,e_D\rangle)}L^1(q\langle M,e_D\rangle q)$. 

We refer to \cite[Th 2]{dabsetup} for a literature overview of the main algebraic operations available on module Haagerup tensor products. We will use them extensively. We single out several operations.  The first is the map $^\star$ (see  Section \ref{notation}) which is given on basic tensors by $(a\o b)^\star=b^*\o a^*.$ Next, for a basic tensor $X=a\otimes b\in M\oeh{D}M$ and a basic product $U= x e_D y \in \langle M , e_D\rangle$ we write:
{$$ U \# X = E_{D'}(bx  e_D  ya),\qquad \textrm{(inner action)}.$$  } and {if $U\in D'\cap   \langle M , e_D\rangle$:}
$$ X\# U = a x e_D y b,\qquad \textrm{(outer action)}$$

With these notations, we have the following statements, which we group into { three} Theorems for convenience of presentation.

\begin{Theorem}\label{Finite1}
Let $D\subset (M,\tau)$ finite von Neumann algebras.
\begin{enumerate}
\item[(1a)]  The outer action  $(X,U)\mapsto X\#U$ extends to all $X\in M\oeh{D}M$ and $U\in D'\cap \langle M,e_D\rangle\subset D'\cap B(L^2(M))$, taking values in $\langle M,e_D\rangle$.  The inner action $(X,V)\mapsto V\#X$ extends to all $X\in M\oeh{D}M$ and $V\in L^1(\langle M,e_D\rangle)$ with values in $D'\cap L^1(\langle M,e_D\rangle)$

\item[(1b)] If in addition $X\in D'\cap M\oeh{D}M$, $U\in D'\cap \langle M,e_D\rangle$, then $X\#U \in D'\cap \langle M, e_D \rangle$.

\item[(1c)] The inner and outer multiplication actions give rise to inclusions $\sigma_1$, $\sigma_2$, 
\begin{align*}
\sigma_i:  D'\cap M\oeh{D}M &\to \\ & B\left(D'\cap \langle M,e_D\rangle \cap L^1(\langle M,e_D\rangle), D'\cap(\langle M,e_D\rangle +  L^1(\langle M,e_D\rangle)\right).
\end{align*}
\end{enumerate}
\end{Theorem}

\begin{proof}
The {$M^{op}$-modularity} of the action on $D'\cap B(L^2(M))$ whose definition is recalled in \cite{dabsetup} Theorem 2.(4) insures stability of $\langle M,e_D\rangle=(D^{op})'\cap B(L^2(M))$ under the outer action.
 
Let us give an explicit description of the predual map giving the inner action on $D'\cap L^1(\langle M,e_D\rangle).$
From the canonical map $M\otimes_h L^2(M)=M\otimes_{eh} L^2(M)\to L^2(M)$ and its row analogue $L^2(M)^*\otimes_{eh}M\to L^2(M)^*$, (see \cite[Prop 3.1.7]{BLM}), one gets a map from
$$L^2(M)^*\otimes_{h}(M\oeh{D}M)\otimes_{h}L^2(M)\simeq (L^2(M)^*\otimes_{eh}M)\otimes_{eh D}(M\otimes_{eh} L^2(M))$$
into $ L^2(M)^*\otimes_{eh D}L^2(M)$,
inducing in particular a map $$m:L^2(M)^*\otimes_{h}L^2(M)\times M\oeh{D}M\to L^2(M)^*\otimes_{eh D}L^2(M)=L^2(M)^*\otimes_{h D}L^2(M)$$  which is our inner multiplication. 
{Composing with $E_{D'}$ one induces a map $$E_{D'}\circ m:L^2(M)^*\otimes_{h D^{op}}L^2(M)\times M\oeh{D}M\to D'\cap L^2(M)^*\otimes_{eh D}L^2(M)\,.$$
The latter is isomorphic to
 $D'\cap L^2(M)^*\otimes_{eh D^{op}}L^2(M)$, the last inclusion following for instance from the identification of this commutant with a quotient or because $E_{D'}(dU -Ud)=0$.
Note that the last isomorphism sends $a\o_Db\in D'\cap L^2(M)^*\otimes_{eh D}L^2(M)$ to $a\o_{D^{op}}b$  and thus on basic tensors $$E_{D'}\circ m(y\o_{D^{op}} x,a\o_D b)= E_{D'}(ya\o_{D^{op}} bx)$$
which is identified with $E_{D'}(bxe_Dya)$ in $D'\cap L^1(\langle M,e_D\rangle)$ via \eqref{L1basic} and coincides with our inner action.
}

For $X\in D'\cap (M\oeh{D}M),$ $U\in D'\cap \langle M,e_D\rangle$, 
 $X\#U\in D'\cap \langle M,e_D\rangle$. This proves (1b).

We first claim that for $V\in L^1(\langle M,e_D\rangle), X\in (M\oeh{D}M),$ $ U\in D'\cap\langle M,e_D\rangle$ :
\begin{equation}\label{adjointActions}Tr(U [V\#X])=Tr([X\#U] V).\end{equation}
To show this, it
suffices to take 
$V\in A$ by density. 
We can also assume $X$ is a finite sum. Indeed, if  $X=x\o_Dy$ a standard decomposition for $X$ \cite[(2.4),(2.5)]{M05} the ultrastrong convergence of finite families $x_F^*\to x^*$, $y_F\to y$ implies if $X_F=x_F\o_D y_F$ $X_F\#U\to  X\#U$ ultraweakly. Likewise if $V=\xi\o_{D^{op}}\eta$
we have the convergence $$\|V\# (X_F-X)\|_{L^2(M)^*\otimes_{eh D^{op}}L^2(M)}^2 \leq 2\langle \xi\sum_{i\not\in F}x_ix_i^* ,\xi\rangle\|y_F\eta\|_2^2+ 2\|\xi x\|_2^2\langle \sum_{i\not\in F}y_i^*y_i\eta ,\eta\rangle \to 0.$$
Now for the remaining case 
$V=\xi\o_{D^{op}}\eta$, $X=x\o_D y$ (without matrix tensor products), we note that the image of $V$ in the identification with $L^1(\langle M,e_D\rangle)$ is $\eta e_D\xi,$ {as explained in \eqref{L1basic}} so that  $[V\#X]=E_{D'}([y\eta e_D\xi x])$ and 
$$Tr(U [V\#X])=Tr(U[y\eta e_D\xi x])=Tr([xUy] V)=Tr([X\#U] V).$$
 We have also shown the existence of  {an extension for the definition of our inner action}, namely that for $V\in D'\cap\langle M,e_D\rangle\cap L^1(\langle M,e_D\rangle)$, and $x,y\in M$, \begin{equation}\label{flipflap}[V\#(x\o_D y)]=E_{D'}([y\o_D x]\#V).\end{equation}

We now prove (1c); all we need to show is that $\sigma_1(X):U\to X\#U,$ $\sigma_2(X):V\to V\#X$ give inclusions.  Note that  $\sigma_1(\cdot)(e_D)$ is the canonical inclusion $M\oeh{D}M\to \langle M,e_D\rangle=L^2(M)_{L^2(D)}\oeh{D}{}_{L^2(D)}L^2(M)$ given by the theory of extended Haagerup product (see e.g. \cite[Prop 14]{dabsetup}), so that $\sigma_1$ is injective.

By the definition of $\sigma_2$, $\sigma_2(X)(e_D)=E_{D'}(i(X))$ with $i:M\oeh{D}M \to L^2(M)^*\oeh{D} L^2(M)$ since it equals $i(X)$ for $X\in D'\cap M\oeh{D}M$; this gives injectivity of $\sigma_2$. 
\end{proof}

\begin{definition} \label{def:HaagerupCyclic}
Denote by $M^{\oehc{D}2}$ the intersection space of the images $\sigma_i(D'\cap M\oeh{D}M)$, $i=1,2$, in the sense of interpolation theory. This space is called the cyclic extended Haagerup tensor square of $M$.
\end{definition}

\begin{Theorem}  \label{finite2} We keep the notations and assumptions of Theorem~\ref{Finite1} and Definition~\ref{def:HaagerupCyclic}.
\begin{enumerate}

\item[(1d)]
The restriction of the map $^\star$ defined in \cite{dabsetup} Theorem 2.(4) to $M^{\oehc{D}2}$  and the map $\sigma = \sigma_2 \circ \sigma_1^{-1}$ define two commuting isometric involutions on $M^{\oehc{D}2}$. 
\item[(1e)] The involution $U\mapsto U^* := (\sigma(U))^\star$ and the product induced on $M^{\oehc{D}2}$ via $\sigma_1$ give rise to an involutive Banach algebra structure on $M^{\oehc{D}2}$.

\item[(1f)]For each $X\in  M^{\oehc{ D}2}$, $\sigma_1^{-1}(X)\#\cdot :D'\cap \langle M,e_D\rangle\to D'\cap \langle M,e_D\rangle$ and $\cdot \#\sigma_2^{-1}(X):D'\cap L^1(\langle M,e_D\rangle)\to D'\cap L^1(\langle M,e_D\rangle)$ interpolate to give an action of  $X\in  M^{\oehc{ D}2}$ on $D'\cap L^2(M)\o_{D}L^2(M)$.

\item[(1g)] There is also an outer action denoted $X\#_{L^1}.$ of $M^{\oehc{D}2}$ on  $ L^1(\langle M,e_D\rangle)$ leaving $D'\cap L^1(\langle M,e_D\rangle)$ globally invariant and commuting with the inner action.
\item[(2a)]  The map  $Y\in  (M\oeh{D}M)\mapsto Y\#e_D\in  \langle M,e_D\rangle\cap L^1(\langle M,e_D\rangle)$ gives the canonical weak-* continuous inclusion of $M\oeh{D}M$  into $L^2(\langle M,e_D\rangle)\simeq L^2(M)\o_D L^2(M)$  (cf. \cite[Proposition 14]{dabsetup}).

\item[(2b)] For any $Y,Z\in D'\cap 
M\oeh{D}M$ the map $X\mapsto \langle Z\#e_D,X\#Y\#e_D\rangle$ is weak-* continuous on bounded sets of $ 
M\oeh{D}M.$ 
\item[(2c)]  $M^{\oehc{D}2}\#e_D $ is dense in $D'\cap L^2(M)\o_{D}L^2(M)$ and $M^{\oehc{D}2}$ weak-* dense in $D'\cap M^{\oeh{D}2}$.  
\item[(2d)]{The multiplication map $(U,V)\mapsto U\#V$  is separately weak-* continuous on bounded sets in the second variable as a map $$(M\oeh{D}M)\times (D'\cap 
M\oeh{D}M)\to (M\oeh{D}M),$$
 and on each variable when restricted to: $$(D'\cap(M\oeh{D}M))\times (D'\cap 
M\oeh{D}M)\to (D'\cap(M\oeh{D}M)).$$}

\end{enumerate}
\end{Theorem}
\begin{proof}

Note that the intersection space $M^{\oehc{D}2}$ is thus well-defined because of (1c).

We start by proving (1d).  If $X\in M^{\oehc{D}2}$ , let $X'=\sigma(X)\in \sigma_2(D'\cap M^{\oeh{D}2})$ so if we show $X'=\sigma'(X):=\sigma_1(\sigma_2^{-1}(X))$ we will have shown $\sigma$ leaves  $M^{\oehc{D}2}$ globally invariant.  The adjoint relation \eqref{adjointActions} gives for $U,V\in  D'\cap \langle M,e_D\rangle \cap L^1(\langle M,e_D\rangle)$
$$Tr(X(U) V)=Tr([\sigma_1^{-1}(X)\# U]V)=Tr(U[V\#\sigma_1^{-1}(X)])=Tr(U[\sigma(X)(V)]),$$
$$Tr(X(U) V)=Tr([U\#\sigma_2^{-1}(X)]V)=Tr(U[\sigma_2^{-1}(X)\#(V)])=Tr(U[\sigma'(X)(V)]).$$
Since $U$ and $V$ are arbitrary in dense spaces this shows the desired relation and as a consequence that $\sigma$ is involutive.

With the same notation and using the definitions, $\sigma=\sigma'$  and the adjoint relation \eqref{adjointActions} several times, we have: \begin{align*}Tr([X^\star(U)]V)&:=Tr([\sigma_1^{-1}(X)^\star\# U]V)\\&=Tr([\sigma_1^{-1}(X)\# U^*]^*V)
=\overline{Tr([\sigma_1^{-1}(X)\# U^*]V^*)}\\&=\overline{Tr( U^*[V^*\#(\sigma_1^{-1}(X))])}=Tr( U[V^*\#(\sigma_1^{-1}(X))]^*)
\\&=Tr( U[\sigma(X)(V^*)]^*)=Tr( U[\sigma_2^{-1}(X)\#V^*)]^*)
\\&=Tr( U[\sigma_2^{-1}(X)^\star\#V)])=Tr( [U\# \sigma_2^{-1}(X)^\star]V)]).
\end{align*}

This shows both the two possible inductions of ${}^{\star}$ coincide and stability of $M^{\oehc{D}2}$ by ${}^\star$.
The commutation with $\sigma$ also follows since we showed $\sigma_2^{-1}(X)^\star=\sigma_2^{-1}(X^\star),$ $\sigma_1^{-1}(X)^\star=\sigma_1^{-1}(X^\star),$ thus $\sigma(X^\star)=\sigma_1(\sigma_2^{-1}(X)^\star)=\sigma_1(\sigma_1^{-1}(\sigma(X))^\star)=\sigma(X)^\star.$  

To prove (1e), it remains to check the composition and the adjunction ${}*$ give the expected Banach algebra structure.

We can reason similarly using  our formula \eqref{adjointActions} and $\sigma=\sigma'$ to check closure under the product:
\begin{align*}Tr([(XY)(U)]V)&:=Tr([\sigma_1^{-1}(X)\sigma_1^{-1}(Y)\# U]V)\\&=Tr([\sigma_1^{-1}(Y)\# U][V\#\sigma_1^{-1}(X)])
=Tr( U[\sigma(X)(V)\#\sigma_1^{-1}(Y)])\\&=Tr( U[\sigma(Y)(\sigma(X)(V))])
\\&=Tr( U[\sigma_2^{-1}(Y)\#(\sigma_2^{-1}(X)\#V)]
=Tr( [U\#\sigma_2^{-1}(Y)\sigma_2^{-1}(X)]V).
\end{align*}
The middle relation then also shows $\sigma(XY)=\sigma(Y)\sigma(X)$.
Similarly, $(UV)^\star=(U)^\star(V)^\star$ which gives the only missing relation between ${}^{*}$ and product to get an involutive Banach algebra.

We next prove (1f).  Since commutants have conditional expectations on them $D'\cap(L^2(\langle M,e_D\rangle))$ is indeed an interpolation of commutants (see e.g. \cite[Prop 2.7.6]{PisierBook}). For $X\in  M^{\oehc{D}2}$, the very definition of $M^{\oehc{D}2}$ give the compatibility for interpolation of the pair of maps $\sigma_1^{-1}(X)\#\cdot :D'\cap \langle M,e_D\rangle\to D'\cap \langle M,e_D\rangle$ and $\cdot \#\sigma_2^{-1}(X):D'\cap L^1(\langle M,e_D\rangle)\to D'\cap L^1(\langle M,e_D\rangle)$. This gives the action on $D'\cap L^2(M)\o_{D}L^2(M)$.

We now turn to (1g).   
Because $L^2(M)^*\o_{h}L^2(M)=L^2(M)^*\o_{eh}L^2(M)\supset M\o_{eh} M$ (obviously weak-* continuous injection), one can extend the projection $E_{D'}$
 from  $M\o_{eh} M\to D'\cap M\o_{eh} M$
to  a map $L^2(M)^*\o_{h}L^2(M)\to D'\cap L^2(M)^*\o_{h}L^2(M)$.

Indeed, by construction, the projection $E_{D'}(U)$ is built as a weak-* limit of convex combinations $\sum \lambda_u uUu^*$ converging thanks to the embedding $M\o_{eh} M\subset L^2(M)\o L^2(M)$. Moreover, we have $\|\sum \lambda_u uUu^*\|_{L^2(M)^*\o_{h}L^2(M)}\leq \|U\|_{L^2(M)^*\o_{h}L^2(M)}.$ Because the injection is weak-* continuous, one also gets weak-* convergence of the convex combination in $L^2(M)^*\o_{h}L^2(M)$ and thus, for any $U\in M\o_{eh} M$, $$\|E_{D'}(U)\|_{L^2(M)^*\o_{h}L^2(M)}\leq \|U\|_{L^2(M)^*\o_{h}L^2(M)}.$$
By density, $E_{D'}$ extends to a bounded map on $L^2(M)^*\o_{h}L^2(M)$ which obviously induces a map $L^2(M)^*\o_{h D^{op}}L^2(M)\to D'\cap L^2(M)^*\o_{h}L^2(M)$, a cross-section to  the quotient map (as seen first for $U\in M\o_{eh} M$ by the weak-* limit above).

 Now take $U\in L^1(\langle M,e_D\rangle)\simeq L^2(M)^*\o_{h D^{op}}L^2(M), X\in M^{\oehc{D}2}$  write $\sigma_2^{-1}(X)=y\o_D x$, a canonical decomposition with $y\in M_{1,I}(M), x\in M_{I,1}(M)$ and take $U'=E_{D'}(U)=\sum_j u_j\o v_j\in D'\cap L^2(M)^*\o_{h}L^2(M)$ sent to $U$ by the quotient map $\pi: L^2(M)^*\o_{h}L^2(M)\to L^2(M)^*\o_{h D^{op}}L^2(M)$.

Then $X\#_{L^1}U:=\sum_{i,j} \pi(x_iu_j\o v_jy_i)$ is well defined in $L^2(M)^*\o_{h D^{op}}L^2(M)$. Indeed, if $\sigma_2^{-1}(X)=0\in M\oeh{D}M$, by 
 \cite{M05} (2.5) there exists $P\in M_{I}(D)$ with { $Px=x$}, $yP=0$ so that 
 $\sum \pi(xE_{D'}(U)y)=\sum \pi(PxE_{D'}(U)y)=\sum \pi(xE_{D'}(U)yP)=0$.
Moreover, we have a bound $\sum_{i,j}\|x_iu_j\|_2^2\leq \|\sum x_i^*x_i\|\sum_{j}\|u_j\|_2^2$ so that $(x_iu_j)$ is indeed a row vector in $L^2(M)^*,$ and similarly $(v_jy_i)$ is a column vector in $L^2(M).$ Thus we have indeed $\sum_{i,j} x_iu_j\o v_jy_i\in 
 L^2(M)^*\o_{eh }L^2(M)\simeq L^2(M)^*\o_{h }L^2(M)$ as claimed.
 
Moreover, by the definition of the norm, it is now easy to see $$\|X\#_{L^1}U\|_{L^2(M)^*\o_{eh}L^2(M)}\leq \|\sigma(X)\|_{M\oeh{D} M}\|E_{D'}(U)\|\leq \|X\|_{M^{\oehc{D}2}}\|U\|_{L^2(M)^*\o_{h D^{op}}L^2(M)}.$$ 
This gives the outer action on $L^1(\langle M,e_D\rangle)$ as is easily seen using the identity $\sigma_2^{-1}(XY)=\sigma_2^{-1}(Y)\#\sigma_2^{-1}(X)$. The stability and commutation are easy.

We now turn to (2a)--(2d). 
 First note that $\sigma_{TC}: L^2(M^{op})^*\o_{hD}L^2(M^{op})\to  L^2(M)^*\o_{hD^{op}}L^2(M)$, given by $\sigma_{TC}(a\o_D b)=b\o_{D^{op}} a$, is isometric. This uses that a row vector of $ L^2(M^{op})^*$ is the same as a column vector of  $L^2(M)$.
 
To prove (2a) note that 
the canonical map  $j: M\oeh{D} M\to L^2(M)^*\oeh{D}L^2(M)=L^2(M)^*\o_{h D}L^2(M)$ composed with $\sigma_{TC}$ above  gives the map $\sigma_{TC} j$ valued in $L^2(M)^*\o_{hD^{op}}L^2(M)=L^1(\langle M,e_D\rangle)$ such that $Y\#e_D$ coincides in the canonical identification with $\sigma_{TC} j(Y)$, proving $Y\#e_D\in   \langle M,e_D\rangle\cap L^1(\langle M,e_D\rangle)$. 
The statement about the agreement with canonical inclusion is then obvious.
 
{Let us prove (2b).} Since  $D'\cap L^1(\langle M,e_D\rangle)\cap \langle M,e_D\rangle$ is dense in $D'\cap L^2(\langle M,e_D\rangle)$, by approximating $Z\#e_D$ by $Z'\in D'\cap L^1(\langle M,e_D\rangle)\cap \langle M,e_D\rangle$ and even $Z'=\sum_i z_i'e_Dz_i\in Me_DM$ in $L^2$ norm, we see that 
it suffices to prove that
$X\to \langle Z',X\#Y\#e_D\rangle$ is weak-* continuous on bounded sets. 

For $Y\in D'\cap M\oeh{D}M$, note that $Y\#e_D\in D'\cap \langle M,e_D\rangle\subset \langle M^{op},e_D\rangle$.  Since $Y\#e_D\in  D'\cap L^1(\langle M,e_D\rangle)=(D'\cap \langle M,e_D\rangle)_*\subset L^1(\langle M',e_D\rangle)$, we see that $Y\#e_D\in L^1(\langle M',e_D\rangle)\cap \langle M',e_D\rangle$.  Since $L^1(\langle M',e_D\rangle)\simeq L^2(M)^*\o_{h D} L^2(M)$ we have a canonical form $Y\#e_D=\sum (y_k')^{op}e_D{y_k}^{op}$ with $(y_k')$ column vector in $L^2(M)$ and $(y_{k})$ row vector in $L^2(M)^*$. Note that for $\xi\in M$, one can compute the evaluation with the formula above $[(x'\o x)\#(Y\#e_D)](\xi)=\sum_k x'E_D(x\xi y_k)y_k'\in L^1(M)$ (one can first approximate $y_k,y_k'$ by elements of $M$ to establish the formula).

If we take $(g_j)_{j\in J}$ a basis of $L^2(M)$ as a right D-module (of elements of $M$ if we want), then one can use the well-known formula 
 
  $$ Tr( Z'[X\#(Y\#e_D)])=\sum_j \langle g_j, Z'[X\#(Y\#e_D)(g_j)]\rangle.$$
We compute a term in the last formula.    We continue our computation by applying $Z'$ which also gives a map on $L^1(M)$ : 
$$Z'[(x'\o x)\#(Y\#e_D)(g_j)]=\sum_i z_i'E_D(z_i\sum_k x'E_D(xg_jy_k)y_k')\in L^1(M).$$ Then since $g_j\in M$, one can compute the trace :
$$\tau(g_j^*Z'[(x'\o x)\#(Y\#e_D)(g_j)])=\sum_i \tau(E_D(g_j^*z_i')z_i\sum_k x'E_D(xg_jy_k)y_k')$$ which could be expressed as a duality formula for $Y\#e_D\in D'\cap L^1(\langle M',e_D\rangle)$ since the sum in $i$ is finite, thus one can use its commutativity with $D$ :
\begin{eqnarray*}
\tau(g_j^*Z'[(x'\o x)\#(Y\#e_D)(g_j)]&=&\sum_i \tau(z_i\sum_k x'E_D(xg_jE_D(g_j^*z_i')y_k)y_k')\\
&=&\sum_i\tau(z_i (x'\o x)\#(Y\#e_D)(g_jE_D(g_j^*z_i')),\end{eqnarray*}
where we finally used one of our previous formulas with $\xi=g_jE_D(g_j^*z_i')$ instead of $g_j$ before. But looking again at  $(x'\o x)\#(Y\#e_D)$ as the bounded operator on $L^2$ and using the relation for a right basis $\sum_jg_jE_D(g_j^*z_i')=z_i'$ with convergence in $L^2$,  one may use operator weak-* convergence to replace $(x'\o x)$ by $X=\sum(x_l'\o x_l) $:

\begin{eqnarray*} Tr( Z'[X\#(Y\#e_D)])&=&\sum_i\tau(z_i X\#(Y\#e_D)(z_i')))\\
&=&\sum_{i,k,l} \tau(z_i x_l'E_D(x_lz_i'y_k)y_k')=\langle X, \sum_{k,i}y_k'z_i\o_{D^{op}}z_i'y_k \rangle.\end{eqnarray*}
Since $\sum_{k,i}y_k'z_i\o_{D^{op}}z_i'y_k\in L^{2}(M^*)\o_{ehD^{op}}L^2(M)\subset L^1(M)^{\o_{hD'}2},$ the predual of the weak-* Haagerup tensor product, one gets the claimed weak-* continuity and thus the proof of (2b) is complete.

To prove  the density part in (2c), it is enough to show that for a finite sum, $E_{D'}(\sum_i x_i\o_D y_i)\in M^{\oehc{D}2}$. 

More precisely, we will show that \begin{equation}\label{finitesum}\sigma_1(E_{D'}(\sum_i x_i\o_D y_i))=\sigma_2(E_{D'}(\sum_i y_i\o_D x_i)).\end{equation}
We thus want to prove, for any $U,V\in D'\cap \langle M,e_D\rangle \cap L^1(\langle M,e_D\rangle)$ : 
\begin{eqnarray*}
Tr([(E_{D'}(\sum_i x_i\o_D y_i))\#U]V)&=&Tr([U\#(E_{D'}(\sum_i y_i\o_D x_i))]V)\\
&=&Tr(U[(E_{D'}(\sum_i y_i\o_D x_i))\#V]).\end{eqnarray*}
By density (simultaneous weak-* and $L^1$ using the agreeing conditional expectations) it suffices to take $U=X\#e_D,V=Y\#e_D, X,Y\in D'\cap M^{\oeh{D}2}.$

But now we can use the weak-* continuity we just proved to replace the conditional expectations by the limit of a net of convex combinations of conjugates by unitaries of $D$, and thus by commutativity with $D$, the conditional expectations can be removed, and the relation then becomes obvious.

Finally, for (2d), taking bounded nets $U_n\to U, V_\nu \to V$ we note that $U_n\#V, U\#V_\nu$ are still bounded, thus weak-* precompact and it thus suffices to show that $U\#V$ is the unique cluster point, 
 for instance by showing the nets converge weakly in $L^2(M)\o_DL^2(M)$ or $D'\cap L^2(M)\o_DL^2(M)$.
 {For $Z\in D'\cap 
M\oeh{D}M$, by (2b) we have $\langle Z\#e_D, U_n\#V\#e_D\rangle \to \langle Z\#e_D, U\#V\#e_D\rangle$, and since the elements $Z\#e_D$ are dense in $D'\cap L^2(M)\o_DL^2(M)$, one deduces the wanted weak convergence in $D'\cap L^2(M)\o_DL^2(M)$. Applying formula \eqref{adjointActions} to $(Z\#e_D)^*\in  L^1(\langle M, e_D\rangle)$, one gets for $Z\in M\oeh{D}M$, \begin{align*}\langle Z\#e_D, U\#V_\nu\#e_D\rangle&= Tr(U\#(V_\nu\#e_D)(Z\#e_D)^*)=Tr((V_\nu\#e_D)[(Z\#e_D)^*\#U])\\&\to Tr((V\#e_D)[(Z\#e_D)^*\#U])=\langle Z\#e_D, U\#V\#e_D\rangle.\end{align*}}
{The convergence is due to the weak-* continuity of the map $.\#e_D$ $M\oeh{D} M\to \langle M,e_D\rangle$ (following from the corresponding one with value $L^2(\langle M,e_D\rangle)$).}
{Again, by density we deduce the weak convergence in $L^2(M)\o_DL^2(M)$, and since $.\#e_D$ is the canonical weak-* continuous map to $L^2(M)\o_DL^2(M)$, this concludes.}

\end{proof}

\begin{Theorem} We keep the assumptions and notation of Theorem~\ref{Finite1} and Definition~\ref{def:HaagerupCyclic}. \label{Finite3}
\begin{enumerate}
\item[(3)] Assume  either that there exists a $D$-basis of  $L^2(M)$ as a right $D$ module $(f_i)_{i\in I}$  which is also a $D$-basis of  $L^2(M)$ as a left $D$ module or that $D$ is a $II_1$ factor and that $L^2(M)$ is an extremal $D-D$ bimodule. Then (writing $\sigma_1^{-1}(X)\#e_D=X\#e_D$) $\tau(X)=\langle e_D,X\#e_D\rangle$ is a trace on $D'\cap M^{\oeh{D}2}$ such that $L^2(M^{\oehc{D}2},\tau)=D'\cap L^2(M)\o_{D}L^2(M).$
 Moreover the involution on $M^{\oehc{D}2}$ coincides with the adjoint in its action on $D'\cap L^2(M)\o_{D}L^2(M).$

\item[(4)]  Assuming the conclusion of (3), the inner action of  $M^{\oehc{D}2}$ on $L^2(M^{\oehc{D}2},\tau)=D'\cap L^2(M)\o_{D}L^2(M)$ extends to an action on $L^2(M)\o_{D}L^2(M).$

\end{enumerate}
\end{Theorem}

We may later identify $M^{\oehc{D}2}$ as a subset of $D'\cap M^{\oeh{D}2}$ via $\sigma_1^{-1}.$
\begin{proof}
(3) Our proof relies on the existence of a unitary Burns rotation, which exists in the extremal case.  The case with a two-sided basis is an easy variant of that case and is left to the reader.

First, note that, without any assumption on $M$ related to traciality, for $X\in M^{\oehc{D}2},Y\in D'\cap M^{\oeh{D}2}$, one can apply the relation established  during the proof of (1):
$$Tr([\sigma_1^{-1}(X)^\star\# U]V)=Tr( U[\sigma_2^{-1}(X)\#V^*)]^*)$$ to $U=e_D,V=(Y\#e_D)$ to get  \begin{equation}\label{ScalarProductTrace}\begin{split}\tau(X^*Y)&=Tr( e_D[(\sigma_1^{-1}(\sigma(X)))\#(Y\#e_D)^*]^*)\\&=Tr( Y\#e_D[\sigma_2^{-1}(\sigma(X))\#e_D)]^*)\\&=Tr([\sigma_1^{-1}(X)\#e_D)]^* (Y\#e_D))
\\&
=\langle X\#e_D,Y\#e_D \rangle. \end{split}\end{equation}
In particular, this realizes canonically isometrically $L^2(M^{\oehc{D}2},\tau)$ as a subspace of $D'\cap L^2(M)\o_{D}L^2(M)$ and as a consequence shows the agreement of the previously defined adjoint with the Hilbert space one.
 The density in our part (2c)
  give the identification $L^2(M^{\oehc{D}2},\tau)=D'\cap L^2(M)\o_{D}L^2(M)$.

  It remains to prove traciality $\tau(XY)=\tau(YX)$;  it is enough to prove it  for $X,Y\in M\oehc{D} M$. Indeed, using the proof of the density and weak-* continuity in our part (2), we only need to consider $X=E_{D'}(x_1\o x_2),  Y=E_{D'}(y_1\o y_2)$ for $x_i,y_i\in M.$ But from our previous computation, this reduces to: \begin{align*}\langle X^*\#e_D,Y\#e_D \rangle&=\langle E_{D'}(x_1^*e_Dx_2^*),E_{D'}(y_1e_Dy_2) \rangle\\&=\langle E_{D'}(y_1^*e_Dy_2^*),E_{D'}(x_1e_Dx_2) \rangle\\&=\langle Y^*\#e_D,X\#e_D \rangle\end{align*}
  
  Now the key equality in the middle line comes {from} the extremality of $L^2(M)$ that gives from Theorem \ref{BurnsR} a unitary Burns rotation. From unitarity it is easy to see that $\rho(E_{D'}(y_1\o_Dy_2))=E_{D'}(y_2\o_Dy_1)$ so that the equality in the middle line comes from 
\begin{align*}\langle E_{D'}(x_1^*e_Dx_2^*),E_{D'}(y_1e_Dy_2) \rangle&
=\langle E_{D'}(x_1^*\o_Dx_2^*),  E_{D'}(y_1\o_Dy_2)\rangle\\&=\langle \rho(E_{D'}(x_1^*\o_Dx_2^*)), \rho( E_{D'}(y_1\o_Dy_2))\rangle\\&=\langle (E_{D'}(x_2^*\o_Dx_1^*)), ( E_{D'}(y_2\o_Dy_1))\rangle\\&=\langle E_{D'}(x_2^*e_Dx_1^*),E_{D'}(y_2e_Dy_1)\rangle\\& =Tr(x_1e_Dx_2{E_{D'}(y_2e_Dy_1)}) \\&=\langle E_{D'}(y_1^*e_Dy_2^*),E_{D'}(x_1e_Dx_2) \rangle.\end{align*}

(4) The extension of the inner action of  $M^{\oehc{D}2}$ to an action on $L^2(M)\o_{D}L^2(M)$ will require more work.
The action of $X\in M^{\oehc{D}2}$ will extend for $U\in L^1(\langle M,e_D\rangle) \cap L^2(M)\o_{D}L^2(M)$, $$U\#X:=\sigma(X)\#_{L^1}U,$$
with the outer action on $L^1(\langle M,e_D\rangle)$ built at the end of (1).

We aim to construct  the action of  $M^{\oehc{D}2}$ by interpolation of the previous action with a dual action on $\langle M,e_D\rangle$, defined by duality for $V\in \langle M,e_D\rangle$:
$$Tr((V\#_{L^\infty}X)U)=Tr(V(X\#_{L^1}U)).$$ It thus remains to see these two actions agree on a common dense subspace. 

Take $U=Y\#e_D\in L^1(\langle M,e_D\rangle)\cap \langle M,e_D\rangle$, for $Y\in \pi(M\o_{alg} M)\subset M\o_{h D}M\subset M\oeh{D}M$. We already noticed they form a dense subspace in both $L^1(\langle M,e_D\rangle)$ and (for the weak-* topology) in  $\langle M,e_D\rangle$. Note that  this indeed gives (even for $Y\in M\oeh{D}M$)  the expected inner action $$\sigma(X)\#_{L^1}U=\sigma(X)\#_{L^1}(\sigma_{TC}j(Y))=\sigma_{TC}j(Y\#X)$$

For the last key equality, take a canonical representation of $Y=\sum y_j\o_{D} y_j',X=\sum x_i\o_{D} x_i'$ then we note that
$$\sigma(X)\#_{L^1}(\sigma_{TC}j(Y))=\sigma(X)\#_{L^1}(\sum y_j'\o_{D^{op}} y_j)=\sum_{ij} x_i'y_j'\o_{D^{op}} y_jx_i=\sigma_{TC}j(Y\#X)$$

Now, take also $V=Z\#e_D\in L^1(\langle M,e_D\rangle)\cap \langle M,e_D\rangle$, for $Z=\sum z_i'\o_Dz_i\in  \pi(M\o_{alg} M)$ to compute $V\#_{L^\infty}X$:
\begin{align*}Tr((V\#_{L^\infty}X)U)&=Tr((Z\#e_D)[(Y\#\sigma(X))\#e_D])
\\&=Tr((Z\#e_D)[Y\#(\sigma(X)\#e_D)])
\\&=Tr([(Z\#e_D)\#Y](\sigma(X)\#e_D))
\\&=Tr([\sum_{ij}y_j'z_i'e_Dz_iy_j](\sigma(X)\#e_D))
\\&=Tr(e_D[(E_{D'}(\sum_{ij} z_iy_j\o_Dy_j'z_i')\#\sigma(X)\#e_D)),\end{align*}
where we started by using the relations we just established, the adjoint relation \eqref{adjointActions} in the third line, an explicit computation in the fourth valid for finite sums and the weak-* continuity on bounded sets of our part (2b) to introduce a conditional expectation.

Now having elements in $D'\cap M\oeh{D}M$ we can use the traciality we just proved,  the adjoint relation \eqref{adjointActions}, then in the third line the definition of $\sigma$ and a removal of conditional expectation since $X\#e_D\in D'\cap \langle M,e_D\rangle$ and finally again explicit computations for finite sums to get:
\begin{align*}Tr((V\#_{L^\infty}X)U)&=Tr(e_D[\sigma(X)\#((E_{D'}(\sum_{ij} z_iy_j\o_Dy_j'z_i')\#e_D))])
\\&=Tr((e_D\#\sigma(X))[((E_{D'}(\sum_{ij} z_iy_j\o_Dy_j'z_i')\#e_D))]) 
\\&=Tr((X\#e_D)[((\sum_{ij} z_iy_j\o_Dy_j'z_i')\#e_D)]) 
\\&=Tr([\sum_iz_i'(X\#e_D)z_i](Y\#e_D)) 
\\&=Tr([(Z\#X)\#e_D](Y\#e_D)) 
.\end{align*}
Thus $(V\#_{L^\infty}X)=(Z\#X)\#e_D=\sigma(X)\#_{L^1}V$ and we can thus interpolate both maps to get the desired action. 
Finally, the agreement with the inner action on the commutant comes from the equality $\sigma(X)\#_{L^1}(Y\#e_D)=(Y\#X)\#e_D$ we proved for $Y\in M\oeh{D}M$. 
\end{proof}


\subsection{$k$-fold cyclic module extended Haagerup tensor products }
\label{CyclicH}
We now turn to the construction of $k$-fold cyclic tensor powers $M^{\oehc{D}k}$ extending the case $k=2$ we have just dealt with. The desired properties of these tensor powers include the action of cyclic permutations, commutation with left-right actions of $D$ as well as compatibility with various multiplication and evaluation operations.  Elements in these modules will serve as  coefficients for our generalized analytic functions, on which free difference quotient and cyclic gradients will be well-defined.

We will use  free products with amalgamation as a convenient trick to reduce to the case of  $2$-fold cyclic modules we have already considered. 

We thus now fix the appropriate notation. Let $D\subset M$ finite von Neumann algebras and consider $D\subset N_\kappa=M*_D(D\otimes W^*(S_1,...,S_\kappa))$ the free product with amalgamation with a free semicircular element $S_1,...,S_\kappa$ for $\kappa$ an ordinal. This of course gives an isomorphic result for each ordinal of same cardinality. 
Note that as $D$-bimodules, $L^2(N_\kappa)\simeq \bigoplus_{n=0}^\infty (L^2(M)^{\o_D n})^{\kappa^{n-1}}$, with $\bigoplus_{n=0}^k (L^2(M)^{\o_D n})^{\kappa^{n-1}}$ being the usual orthonormalisation of $\overline{Span\{ (MS_{i_1})...(MS_{i_{n-1}})M,1\leq n \leq k, i_j\in [1,\kappa]\}}$ (``Wick words'').
 
 In particular, for any word $n=n_1...n_{|n|}$ in $\kappa$ letters there is an embedding $$\iota_n: M^{\oeh{D}(|n|+1)}\to L^2(N_\kappa)$$ valued in $L^2(M)^{\o_D |n |+1}\cap N_\kappa$ obtained by first sending the tensor $x_0\otimes\cdots \otimes x_{|n|}$ to $x_0 S_{n_1} x_1 \cdots S_{n_{|n|}} x_{|n|}$ and then projecting onto the orthogonal complement of  the space 
 $\overline{Span\{ (MS_{i_1})...(MS_{i_{k-1}})M,1\leq k \leq |n|, i_j\in [1,\kappa]\}}$.   We will write $$L^2(M)^{\o_D n}\simeq L^2(M)^{\o_D (|n|+1)}$$ for the closure of the image of $\iota_n$.
 
 \subsubsection{Construction of intersection spaces}
 To handle the action of a basic cyclic permutation, we need an intersection space similar to the intersection  $L^1(\langle M,e_D\rangle)\cap \langle M,e_D\rangle$ in the previous section (which corresponds to the case $|n|+1=2$). For this, we will use $L^1(\langle N_\kappa,e_D\rangle)\cap \langle N_\kappa,e_D\rangle$ (for any fixed $\kappa\geq k$, e.g. $\kappa=\omega$)

Let  $\mathcal{K}_{m,m}=L^2(M)^{\o_D |m|+1}$, $\mathcal{K}_{m,n}=L^2(M)^{\o_D |m|+1}\oplus L^2(M)^{\o_D |n|+1},$ if $m\neq n$, considered with the right normal action of $D$, and consider the corresponding basic construction  $B(M:D,(m, n))=B(\mathcal{K}_{m,n},\mathcal{K}_{m,n})_D$ 
 with a canonical semifinite trace $Tr$ (see e.g. \cite[section  2.3]{PV11} or the beginning of section 6.1). In our operator space terminology, we have, by \cite[Corol 3.3]{M05} (and the preceding Theorem to change the reference Hilbert space structure to compute duality), $B(M:D,(m, n))\simeq (\mathcal{K}_{m,n})_{L^2(D)}\oeh{D}{}_{L^2(D)}(\mathcal{K}_{m,n}^*).$ Via this isomorphism $\xi d\o_D\overline{\eta}=\xi \o_D\overline{\eta d^*}$ is send to $L_{\xi d} L_\eta^*=L_{\xi }d L_\eta^*=L_{\xi} L_{\eta d^*}^*,$ where $L_\xi$ denotes left multiplication by $\xi$, see \cite[Section 2.3]{PV11}.
Its predual is $\mathscr{TC}(M:D,(m,n)):=L^1(B(M:D,(m, n)),Tr)\simeq \mathcal{K}_{m,n}^*\o_{h D^{op}}\mathcal{K}_{m,n}.$
The spaces  $B(M:D,(m, n))$ and $\mathscr{TC}(M:D,(m,n))$ are considered as an interpolation pair as before.

We will be mostly interested in off-diagonal block matrices in these constructions, namely (for $k\neq l$), $$\mathscr{TC}(M:D,k,l):=L^2(M)^{\o_D |k|+1*}\o_{h D^{op}}L^2(M)^{\o_D |l|+1},$$ $$B(M:D,k,l):=B(L^2(M)^{\o_D |k|+1},L^2(M)^{\o_D |l|+1})_D$$ so that $B(M:D,k,l)=\mathscr{TC}(M:D,l,k)^*$ and the duality can be seen as induced by $Tr$ above when they are seen as block matrices in the space above. 

Let us start with a Lemma making explicit this relation.
Consider, for $n$ a word in $\kappa$ letters, $P_n\in \langle N_\kappa ,e_D\rangle\cap B(L^2(N_\kappa),L^2(M)^{\o_D n})$ the orthogonal projection on the $n$-th component in the decomposition $L^2(N_\kappa)\simeq \bigoplus_{k=0}^\infty \bigoplus_{|n|=k}L^2(M)^{\o_D n}$. Note that we make the difference between the adjoint $P_n^*\in B(L^2(M)^{\o_D n},L^2(N_\kappa))$ and the map $\overline{P_n}\in B(L^2(N_\kappa)^*,L^2(M)^{\o_D n*})$: $\overline{P_n}(\overline{\xi})=\overline{\xi}\circ P_n^*=\overline{P_n\xi},$ even though they may be conjugated by some isomorphisms above. 
\begin{Lemma}\label{basic2}
\begin{enumerate}
\item  Let $X\in \langle N_\kappa ,e_D\rangle$ and $Y\in L^1(\langle N_\kappa ,e_D\rangle)\simeq L^2(N)^*\o_{h D^{op}}L^2(N).$ $X$ and $Y$ agree in the classical intersection space, if and only if for all $k,l$ words in $\kappa$ letters, $P_lXP_k^*\in B(M:D,k,l)$ and $(\overline{P_k}\o_{D^{op}}P_l)(Y)$ agree in the intersection space coming from the inclusions $B(M:D,k,l)\subset B(M:D,(k,l))$, $\mathscr{TC}(M:D,k,l)\subset\mathscr{TC}(M:D,(k,l)).$

\item We have the inclusions: $$\mathscr{TC}(M:D,k,l)\subset B(L^2(M)^{\o_D k},L^1(D)\o_{h D^{op}}L^2(M)^{\o_D l})_D\supset B(M:D,k,l)$$ (the right module structure on $L^1(D)\o_{h D^{op}}L^2(M)^{\o_D l}$ given by right multiplication on $L^1(D)$). Moreover the intersection space of interpolation theory 
$\mathscr{TC}(M:D,k,l)\cap B(M:D,k,l)$ coincides with the one coming from the inclusions $B(M:D,k,l)\subset B(M:D,(k,l))$, $\mathscr{TC}(M:D,k,l)\subset\mathscr{TC}(M:D,(k,l)),$ those spaces being realized as classical compatible couple for interpolation  of $L^p$ spaces of a semifinite von Neumann algebra.


\end{enumerate}
\end{Lemma}

\begin{proof} 

(1) This point readily comes from the agreement of the trace induced by projections from $\langle N_\kappa,e_D\rangle$ with the one defined on $B(M:D,(k,l)).$ Thus if $p$ finite projection in $B(M:D,(k,l))$, $P_k^*pP_k$ is finite in $\langle N_\kappa,e_D\rangle$.  Hence agreement of $X$ and $Y$ which boils down to the agreement for any finite projection of their compressions, gives 
$P_l^*pP_lXP_k^*pP_k=(\overline{P_k^*pP_k}\o_{D^{op}}P_l^*pP_l)(Y)$ and thus the agreement after removing one application of $P_i^*,$ i.e. as we said since this is for all finite projection $p$, 
$P_lXP_k=(\overline{P_k}\o_{D^{op}}P_l)(Y).$
Conversely, since $P_{\leq n}=P_0+...+\sum_{|m|=n}P_m$ increases to identity, it suffices to consider finite projection $q\in \langle N_\kappa,e_D\rangle$ with $q\leq P_{\leq n}$ which readily reduces to compression by $q\wedge P_k=P_k^*(q\wedge P_k)P_k$ 
(on the right and $q\wedge P_l$ on the left) for a projection $p$ on $B(M:D,(k,l))$. And we can then apply the converse reasoning.

(2) Note that \begin{align*}(L^2(M)^{\o_D k})&=(L^2(M)^{\o_D k}_{L^2(D)})\oeh{D} L^2(D)\\&\simeq CB(L^2(M)^{\o_D k*},L^2(D^{op})^*)_{D^{op}}\oeh{D} L^2(D)\\&\simeq CB(L^2(M)^{\o_D k*},L^2(D^{op})^*\oeh{D} L^2(D))_{D^{op}}=CB(L^2(M)^{\o_D k*}, L^1(D))_{D^{op}}.\end{align*}
For any $\phi\in (L^2(M)^{\o_D k})$ we have a map $\phi\o_{h D^{op}}1:L^2(M)^{\o_D k*}\o_{h D^{op}}L^2(M)^{\o_D l}\to L^1(D)\o_{h D^{op}}L^2(M)^{\o_D l}.$ Moreover, take $Z=x\o_{D^{op}} y$ a typical element in $L^2(M)^{\o_D k*}\o_{h D^{op}}L^2(M)^{\o_D l}$, if its image vanishes, this means for all $\phi\in L^2(M)^{\o_D k}$, $\phi(x)\o_{D^{op}} y=0$, thus by \cite{M05} formula (2.5) there is $P_{\phi}\in M_J(D^{op})$ such that $\phi(x)P_{\phi}=0$, $P_{\phi}y=y$. Take $P=\bigwedge_{\phi \in L^2(M)^{\o_D k}}P_\phi$ then $Py=y$ and $\phi(xP)=\phi(x)P=\phi(x)P_{\phi}P=0$ thus since $\phi$ is arbitrary in a space containing the dual of the space of $x$, $xP=0$ and thus $x\o_{D^{op}} y=0$; thus we get the first claimed injectivity. 

The agreement of intersections spaces comes from the fact that the intersection space of $L^1$ and $L^\infty$ can be reduced to equality when compressed by rank 1 projections coming from 
elements in a fixed right-module basis. Then the agreement corresponds in the second picture to agreement when evaluating at this fixed basis (and evaluating by duality at this basis too).
\end{proof}

\subsubsection{Wick formula}
We will also need a straightforward tensor variant of Wick formula.            
For $k=k_1...k_{|k|},m=m_1...m_{|m|}$ words in $\kappa$ letters, we write $k\circ m=k_1...k_{|k|}m_1...m_{|m|}$ for the concatenation, and also $k\circ_i m=k_1...k_{|k|-i}m_{1+i}...m_{|m|}$,$|k|\wedge |m|\geq i\geq 0$ (defined only if  the last $i$ letters of $k$ and the first $i$ letters of $m$ form identical words).  Note that  $|k\circ_im|=|k|+|m|-2i.$ We also write $\overline{k}=k_{|k|}...k_1$. Sometimes, we will need to emphasize the following isomorphism: \begin{align*}\iota_{m_1,m_2,l_1,l_2}:&B(M:D,m_1m_2,l_1l_2)\simeq (L^2(M)^{\o_D l_1l_2})_{L^2(D)}\oeh{D}{}_{L^2(D)}(L^2(M)^{\o_D m_1m_2*})\\&\simeq (L^2(M)^{\o_D |l_1|})_{L^2(D)}\oeh{D}B(M:D,m_2,l_2)\oeh{D}{}_{L^2(D)}(L^2(M)^*)^{\o_D |\overline{m_1}|}\end{align*} given by 
$\iota_{m_1,m_2,l_1,l_2}(\xi_{1}\o_D...\o_D\xi_{|l_1|+|l_2|+1})\o_D\overline{(\eta_{1}\o_D...\o_D\eta_{|m_1|+|m_2|+1})}=\xi_{1}\o_D...\o_D\xi_{|l_1|}\o_D(\xi_{|l_1|+1}\o_D...\o_D\xi_{|l_1|})\o_D\overline{(\eta_{|m_1|+1}\o_D...\o_D\eta_{|m_1|+|m_2|+1})}\o_D\overline{\eta_{|m_1|}}\o_D...\o_D\overline{\eta_{1}}.$

Likewise, we have : \begin{align*}&\hat{\iota}_{m_1,m_2,l_1,l_2}:\mathscr{TC}(M:D,m_1m_2,l_1l_2)\simeq (L^2(M)^{\o_D m_1m_2*})\oeh{D^{op}}(L^2(M)^{\o_D l_1l_2})\\&\simeq (L^2(M)^{\o_D m_2*})\oeh{D}{}_{L^2(D)}(L^2(M)^*)^{\o_D |\overline{m_1}|}\oeh{D^{op}}(L^2(M)^{\o_D |l_1|}){}_{L^2(D)}\oeh{D}(L^2(M)^{\o_Dl_2})\end{align*} given by 
\begin{align*}&\hat{\iota}_{m_1,m_2,l_1,l_2}\overline{(\eta_{1}\o_D...\o_D\eta_{|m_1|+|m_2|+1})}\o_{D^{op}}(\xi_{1}\o_D...\o_D\xi_{|l_1|+|l_2|+1})\\&=\overline{(\eta_{|m_1|+1}\o_D...\o_D\eta_{|m_1|+|m_2|+1})}\o_D\\&(\overline{\eta_{|m_1|}}\o_D...\o_D\overline{\eta_{1}})\o_{D^{op}}(\xi_{1}\o_D...\o_D\xi_{|l_1|})\o_D(\xi_{|l_1|+1}\o_D...\o_D\xi_{|l_1|+|l_2|+1}).\end{align*}

\begin{Lemma}\label{Wick}
Let $X\in \langle N_\kappa ,e_D\rangle, Y\in L^1(\langle N_\kappa ,e_D\rangle),$ $k,l,m,n,p,q$ words in $\kappa$ letters, $U\in D'\cap M^{\oeh{D}(|k|+|l|+2)}$ and $V=(\iota_k\oeh{D}\iota_l)(U)\in D'\cap N^{\o_{eh D }2}$.

If we consider
$P_mXP_{n}^*\in 
B(M:D, n,m)\cap \mathscr{TC}(M:D,n,m)$, we have 
$$\iota_{\overline{l},n,k,m}(P_{k\circ m}[V\#(P_mXP_{n}^*)]P_{\overline{l}\circ n}^*)\in M^{\oeh{D} |k|}\oeh{D}B(M:D,n,m)\oeh{D}\overline{M}^{\oeh{D} |l|},$$ and $P_{p}[V\#(P_mXP_{n}^*)]P_{q}^*=0$ for either $|q|> |n|+|l|$ or $|q|<|n|+|l|-2(|l|\wedge |n|)$ or $|p|<|m|+|k|-2(|k|\wedge |m|)$ or $|p|>|m|+|k|.$ 

Moreover, if we consider the canonical map$$m_{\infty}^{(|k|,k\circ k',|l|,l\circ l')}:M^{\oeh{D} |k|}\oeh{D}B(M:D,l\circ l',k\circ k')\oeh{D}\overline{M}^{\oeh{D} |l|}\to B(M:D,l',k')$$ extending: 
\begin{align*}m_{\infty}^{(|k|,k\circ k',|l|,l\circ l')}(&m_1\o \cdots m_{|k|}\o \xi_{|k|+1}\o \xi_{|k|}\o \cdots \o \xi_{1}\o \xi_{|k|+2}\o \cdots \o \xi_{|k|+|k'|+1}\\&\o\overline{\eta_{|l|+1}\o \eta_{|l|}\o \cdots \o \eta_{1}\o \eta_{|l|+2}\o \cdots \o \eta_{|l|+|l'|+1}}
\o \overline{n_{|l|}}\o \cdots \o \overline{n_1})\\&= m_1E_D(m_2\cdots E_D(m_{|k|}E_D(\xi_{|k|+1})\xi_{|k|})\cdots \xi_2)\xi_1\o  \xi_{|k|+2}\o \cdots \o \xi_{|k|+|k'|+1}\\&\o\overline{n_1E_D(n_2\cdots E_D(n_{|l|}E_D(\eta_{|l|+1})\eta_{|l|})\cdots \eta_2)\eta_1\o  \eta_{|l|+2}\o \cdots \o \eta_{|l|+|l'|+1}}
\end{align*}for $|k|,|l|\geq 0$, (by convention $m_{\infty}^{(0, k',0, l')}=Id$) then we have the relation for $P\in [\![0,|k|\wedge |m|]\!],Q\in [\![0,|l|\wedge |n|]\!]$:
 \begin{equation}\label{TensorWick}\begin{split}&\iota_{\overline{l_{[Q+1,|l|]}},n_{[Q+1,|n|]},k|_{[1,|k|-P]},m_{[P+1,|m|]}}(P_{k\circ_Pm}[V\#(P_mXP_{n}^*)]P_{\overline{l}\circ_Qn}^{*})=\prod_{i=1}^P1_{k_{|k|-(i-1)}=m_i}\times\\&\times\prod_{i=1}^Q1_{l_{i}=n_i}[1^{\o |k|-P}\o m_{\infty}^{(P,m,Q,n)}\o 1^{\o |l|-Q}]\iota_{\overline{l},n,k,m}\left(P_{k\circ m}[V\#(P_mXP_{n}^*)]P_{\overline{l}\circ n}^*\right).\end{split}\end{equation}
 
Likewise we have:
$$\hat{\iota}_{\overline{k},n,l,m}(\overline{P_{\overline{k}\circ n}}\o_{D^{op}} P_{l\circ m})[ (\overline{P_n}\o_{D^{op}}P_m)](Y)\#V]\qquad$$
$$\qquad\qquad \in D'\cap [(L^2(M)^{\o_D n*}\oeh{D}\overline{M}^{\oeh{D} |k|}]\oeh{D^{op}}[M^{\oeh{D} |l|}\oeh{D}L^2(M)^{\o_D m})]$$ and $(\overline{P_{q}}\o_{D^{op}} P_{p})[ (\overline{P_n}\o_{D^{op}}P_m)](Y)\#V]=0$ for $|q|> |n|+|k|$ or $|q|<|n|+|k|-2(|k|\wedge |n|)$ or $|p|<|m|+|l|-2(|l|\wedge |m|)$ or $|p|>|m|+|l|.$ 
Moreover there is a canonical map  $$m_{1}^{(|l|,P,m,|k|,Q,n)}:((L^2(M)^{\o_D n*}\oeh{D}\overline{M}^{\oeh{D} |k|})\oeh{D^{op}}({M}^{\oeh{D}|l|}\oeh{D}L^2(M)^{\o_D m}))\quad\qquad$$
$$\qquad\qquad \to \mathscr{TC}(M:D,k\circ_Qn,l\circ_Pm ),$$ given on elementary tensors by: \begin{align*}&m_{1}^{(|l|,P,m,|k|,Q,n)}(\overline{\eta_{|k|+1}\o \eta_{|k|}\o \cdots \o \eta_{1}\o \eta_{|k|+2}\o \cdots \o \eta_{|n|+1}}
\\&\o \overline{n_{|k|}}\o \cdots \o \overline{n_1})\o (m_1\o \cdots m_{|l|}\o \xi_{|l|+1}\o \xi_{|l|}\o \cdots \o \xi_{1}\o \xi_{|l|+2}\o \cdots \o \xi_{|m|+1}))=\\& [\overline{n_1\o \cdots 
 \o n_{|l|-Q+1}E_D(n_{|l|-Q+1}\cdots E_D(n_{|l|}E_D(\eta_{|l|+1})\eta_{|l|})\cdots \eta_{|l|-Q+2})\eta_{|l|-Q+1}  }
 \\ & \qquad \overline{ 
\o \cdots \o \eta_{1}\o  \eta_{|l|+2}\o \cdots }
\\&\o [(m_1\o \cdots m_{|k|-P}\o m_{|k|-P+1}E_D(m_{|k|-P+2}\cdots E_D(m_{|k|}E_D(\xi_{|k|+1})\xi_{|k|})\cdots \xi_{|k|-P+2})\xi_{|k|-P+1}\\&\o \xi_{|k|-P}\o \cdots \o \xi_{1}\o \xi_{|k|+2}\o \cdots \o \xi_{|k|+|k'|+1}].
\end{align*}
These maps satisfy: $$\iota^{-1}_{\overline{l_{[Q+1,|l|]}},n_{[Q+1,|n|]},k|_{[1,|k|-P]},m_{[P+1,|m|]}}\circ(1^{\o |k|-P}\o m_{\infty}^{(P,m,Q,n)}\o 1^{\o |l|-Q})\circ \iota_{\overline{l},n,k,m}=m_{1}^{(|l|,Q,n,|k|,P,m)}\hat{\iota}_{\overline{l},n,k,m}$$ when restricted to the intersection of their domain viewed as a subset of $B(M:D,\overline{l}\circ n,k\circ m)+\mathscr{TC}(M:D,\overline{l}\circ n,k\circ m)$.   For $P\in [\![0,|l|\wedge |m|]\!],Q\in [\![0,|k|\wedge |n|]\!]$:
 \begin{equation}\label{TensorWickL1}\begin{split}
 (\overline{P_{\overline{k}\circ_Qn}}&\o_{D^{op}} P_{l\circ_Pm})[ (\overline{P_n}\o_{D^{op}}P_m)](Y)\#V]=\prod_{i=1}^Q1_{k_{i}=n_i} \prod_{i=1}^P1_{l_{|l|-(i-1)}=m_i}\times \\\times&m_{1}^{(|l|,P,m,|k|,Q,n)}\hat{\iota}_{\overline{k},n,l,m}\left((\overline{P_{\overline{k}\circ n}}\o_{D^{op}} P_{l\circ m})[ (\overline{P_n}\o_{D^{op}}P_m)](Y)\#V]\right).\end{split}\end{equation}
\end{Lemma}
\begin{proof}
The definition of the map $m_{\infty}^{(|k|,k\circ k',|l|,l\circ l')}$ and its weak-* continuity in the  variable $B(M:D,l\circ l',k\circ k')$ are easy. 
Thus one can assume $X\in [Alg(S,M)]e_D[Alg(S,M)].$
Then using canonical forms for the extended Haagerup tensor product and strong convergence of corresponding finite sums, we are reduced to the case where $V$ is a finite sum.
For finite tensors, the relation reduces to the usual Wick formula. The second part of the statement is similar using norm- instead of weak-* density.
\end{proof}
What really matters for us in the previous result is that the highest component of the product is a tensor product, while the remaining terms are then determined by applying multiplication and conditional expectations to its various components. For convenience for words $m,n$ and $k\leq |m|,$ we write : $$m\#_Kn=m_1...m_{K-1}n_1...n_{|n|}m_{K+1}...m_{|m|},$$ $$m\hat{\#}_Kn=m_1...m_Kn_1...n_{|n|}m_{K+1}...m_{|m|}.$$ 

\subsubsection{Flips and cyclic permutations.} 
We start by interpreting a cyclic permutation $\sigma=(l+2,l+3,\ldots l+k+2,1,2,\ldots,l+1)$  in $\mathfrak{C}_n$, $n=l+k+2$, as the flip (i.e., period two permutation) of the blocks $[\![ l+2,\dots,l+k+2]\!]$ and 
$[\![1,\dots,l+1]\!]$.  We mimic this point of view in terms of injections in our free product von Neumann algebra $N_\kappa$.  We thus make use of our results on the two-fold cyclic Haagerup tensor product in this context to construct a suitable intersection space, using which we then construct the $n$-fold cyclic Haagerup tensor product.

\begin{Proposition}\label{OnePermutation}
Let $D\subset M$ finite von Neumann algebras and $N=N_\kappa=M*_D(D\otimes W^*(S_1,...,S_\kappa))$. We assume $\kappa$ infinite $k,l$ words in $\kappa$ letters. 

  Let $\sigma\in \mathfrak{C}_n$  be a cyclic permutation as above, $n=|k|+|l|+2$, $\sigma(1)=|l|+2$. Using $\sigma_i$ of  Theorem \ref{Finite1} for $D\subset N_\kappa$ we have two inclusions $I_1(\sigma)=\sigma_1\circ(\iota_k\oeh{D}\iota_l),I_2(\sigma)=\sigma_2\circ(\iota_l\oeh{D}\iota_k)$ \begin{align*}I_i(\sigma):  &D'\cap M^{\oeh{D}n}\to       B(D'\cap \langle N,e_D\rangle \cap L^1(\langle N,e_D\rangle), D'\cap(\langle N,e_D\rangle +  L^1(\langle N,e_D\rangle)).\end{align*}  The intersection space in the sense of interpolation of these inclusions, written $M^{\oeh{D}(k\circ l,\sigma)}=M^{\oeh{D}(k,l)}$, has a change  of inclusion $I(\sigma)=I_2(\sigma^{-1})\circ I_1(\sigma)^{-1}:M^{\oeh{D}(k\circ l,\sigma)}\to M^{\oeh{D}(\sigma.(k\circ l),\sigma^{-1})}$ which satisfies $I(\sigma)=I(\sigma^{-1})^{-1}$ (with $\sigma\cdot k\circ l= l\circ k$).
 
Moreover,  the isometric involution ${}^\star$ induced on $M^{\oeh{D}n}\subset NCB((D')^{n-1},B(L^2(M))$ 
 given by $U^\star(X_1,...,X_{n-1})=U(X_{n-1}^*,...,X_{1}^*)^*$ extending  $(x_1\o_{D}\cdots\o_{D} x_n)^\star=(x_n^*\o_{D}\cdots\o_{D}x_1^*)$  sends $M^{\oeh{D}(k\circ l,\sigma)}$ to  $M^{\oeh{D}(l\circ k,\sigma^{-1})}.$ 

The product $.\#_{K}.:M^{\oeh{D}|n|+2}\times (D'\cap M^{\oeh{D}|m|+1})\to M^{\oeh{D}(|n|+|m|+1)}$ for $K\in [\![1,|n|+1]\!],$ induced by the composition in the $K$-th entry of $NCB((D')^{|n|+1}$ 
 $$NCB((D')^{|n|+1},B(L^2(M))\times NCB((D')^{|m|},D'\cap B(L^2(M)))\to NCB((D')^{|n|+|m|},B(L^2(M))$$ corresponds on tensors to the map $(x_1\o_{D}\cdots\o_{D} x_{|n|+2})\#_{K}(y_1\o_{D}\cdots\o_{D} y_{|m|+1})=x_1\o_{D}\cdots\o_{D}x_{K}y_1\o_D y_2 \o_{D}\cdots\o_{D} y_{|m|+1}x_{K+1}\o_{D}\cdots\o_{D} x_{|n|+2}$.  The product   is separately weak-* continuous on bounded sets in each variable and has the following stability properties:
\begin{itemize}
\item If $\sigma\in \mathfrak{C}_{|n|+2}$, $\tau\in \mathfrak{C}_{|n|+|n'|+1}$ $\sigma(1)=|n|-k'+2$, $\tau(1)=|n|+|n'|+1-k'$, $k'<K-1$,  then for any $U\in M^{\oeh{D}(n,\sigma)}, V\in D'\cap  M^{\oeh{D}(|n'|+1)} $ we have $U\#_KV\in  M^{\oeh{D}(n\#_Kn',\tau)},$
\item  If $\sigma\in \mathfrak{C}_{|n|+2}$, $\tau\in \mathfrak{C}_{|n|+|n'|+1}$ $\sigma(1)=|n|-k'+2=\tau(1)$, $k'\geq K$, then for any $U\in M^{\oeh{D}(n,\sigma)}, V\in D'\cap  M^{\oeh{D}(|n'|+1)} $ we have $U\#_KV\in  M^{\oeh{D}(n\#_Kn',\tau)},$
\item  If $\sigma\in \mathfrak{C}_{|n|+2},\rho\in \mathfrak{C}_{|n'|+2}$, $\tau\in \mathfrak{C}_{|n|+|n'|+2}$ $\sigma(1)=|n|-K+3,\rho(1)=|n'|-k'+2$ $\tau(1)=|n|+|n'|-K-k'+3$, $k'\in [\![1 , |n'|-1 ]\!]$, then for any $U\in M^{\oeh{D}(n,\sigma)}, V\in  M^{\oeh{D}(n',\rho)} $ we have $U\#_KV\in  M^{\oeh{D}(n\hat{\#}_Kn',\tau)}.$
\end{itemize}
 
 Similarly, the map $M^{\oeh{D}|n|+1}\times ( M^{\oeh{D}|m|+1})\to M^{\oeh{D}(|n|+|m|+1)}$  induced by the product in $B(L^2(M))$ 
 $$NCB((D')^{|n|},B(L^2(M))\times NCB((D')^{|m|},B(L^2(M)))\to NCB((D')^{|n|+|m|},B(L^2(M))$$ and corresponds on tensors to the map $(x_1\o_{D}\cdots\o_{D} x_{|n|+1})(y_1\o_{D}\cdots\o_{D} y_{|m|+1})=x_1\o_{D}\cdots\o_{D}x_{|n|+1}y_1\o_D y_2 \o_{D}\cdots\o_{D} y_{|m|+1}$.  It has the following stability properties:
 \begin{itemize}
\item If $\sigma\in \mathfrak{C}_{|n|+2}$, $\tau\in \mathfrak{C}_{|n|+|m|+2}$ $\sigma(1)=|n|-k'+2$, $\tau(1)=|n|+|m|+2-k'$, then for any $U\in M^{\oeh{D}(n,\sigma)}, V\in D'\cap  M^{\oeh{D}|m|+1} $ we have $UV\in  M^{\oeh{D}(nm,\tau)},$
\item If $\sigma\in \mathfrak{C}_{|m|+2}$, $\tau\in \mathfrak{C}_{|n|+|m|+2}$ $\sigma(1)=|m|-k'+2=\tau(1)$,  then for any $V\in M^{\oeh{D}(m,\sigma)}, U\in D'\cap  M^{\oeh{D}|n|+1} $ we have $UV\in  M^{\oeh{D}(nm,\tau)}.$
\end{itemize}
 
\end{Proposition}

\begin{proof}
\noindent (i) For the first statement, we only have to prove that $I_2(\sigma^{-1})\circ I_1(\sigma)^{-1}=I_1(\sigma^{-1})\circ I_2(\sigma)^{-1}$. In this it becomes clear that the image of  $I(\sigma)$ is indeed  $M^{\oeh{D}(l\circ k,\sigma^{-1})}$ and that $I(\sigma)^{-1}=I(\sigma^{-1})$. Take $X\in M^{\oehc{D}(k,l)}$
We know there is $U\in D'\cap M^{\o_{eh D(|k|+|l|+2)}}$ such that $U'=\iota_k\oeh{D}\iota_l(U)\in  D'\cap N^{\oeh{D} 2}$ and $U'=\sigma_{1}^{-1}(X)$ there is also $V\in D'\cap M^{\o_{eh D(|k|+|l|+2)}}$ such that $V'=\iota_l\oeh{D}\iota_k(V)\in D'\cap N^{\oeh{D} 2}$ is $V'=\sigma_2^{-1}(X)$.
Then by definition $$I_2(\sigma^{-1})\circ I_1(\sigma)^{-1}(X)=\sigma_2(U')=\sigma(X)=\sigma_1(\sigma_2^{-1}(X))=I_1(\sigma^{-1})\circ I_2(\sigma)^{-1}(X),$$ using in the middle the key relation proved in Theorem \ref{Finite1}.(1) and then the definition of our maps $I_i.$

\noindent (ii) For the statement about the adjoint, one  uses $$[(\iota_k\oeh{D}\iota_l)(U)]^\star=[(\iota_l\oeh{D}\iota_k)(U^\star)],$$
and our previous results in Theorems \ref{Finite1} and Theorem \ref{finite2}.(1) to deduce:
$$I(\sigma)(U^\star)=[I(\sigma^{-1})(U)]^\star.$$

\noindent (iii) For the weak-* continuity of composition products, take bounded nets $U_n\to U, V_\nu \to V$.   
By weak-* precompactness of $U_n\#_KV, U\#_KV_\nu$, it suffices to show that they converge weakly in $L^2(M)^{\o_D(|n|+|m|+1)}$ to $U\#_KV$.  By density it is enough to check convergence dually against any $Z\in M^{\oeh{D}(|n|+|m|+1)}$. Take any word $o$ of length $|o|=|m|-1$. We claim that $(\iota_{\epsilon}\o_D\iota_{o})(V_\nu-V)\to 0$ weak-* in $N\oeh{D}N$. This is obvious again by the isometric embedding at $L^2$ level and since it suffices to check weak convergence in $L^2(N)\o_DL^2(N)$. Take similarly $n=kl$, $|k| =K-1$, then $(\iota_{k}\o_D\iota_{l})(U_n-U)\to 0$. From the result in Theorem \ref{finite2}.(2), $$[(\iota_{k}\o_D\iota_{l})(U_n)]\#[(\iota_{\epsilon}\o_D\iota_{o})(V)]\to [(\iota_{k}\o_D\iota_{l})(U)]\#[(\iota_{\epsilon}\o_D\iota_{o})(V)],$$
$$ [(\iota_{k}\o_D\iota_{l})(U)]\#[(\iota_{\epsilon}\o_D\iota_{o})(V_\nu)]\to [(\iota_{k}\o_D\iota_{l})(U)]\#[(\iota_{\epsilon}\o_D\iota_{o})(V)],$$
in $N^{\oeh{D}2}$. 
Since from the computation below coming from Lemma \ref{TensorWick} $$\langle Z, U\#_KV\rangle=\langle(\iota_{k\epsilon}\o_D\iota_{ol})(Z)\#e_D, [(\iota_{k}\o_D\iota_{l})(U)]\#[(\iota_{\epsilon}\o_D\iota_{o})(V)]\rangle,$$ 
we get the weak convergence by duality against $(\iota_{k\epsilon}\o_D\iota_{ol})(Z)\#e_D$. 
 
\noindent (iv) For the stability of composition products, consider first the situation of the third point, $U\in M^{\oeh{D}(n,\sigma)}=M^{\oeh{D}(k_1,l_1)}, V\in  M^{\oeh{D}(n',\rho)}=M^{\oeh{D}(k_2,l_2)} $  $n=k_1l_1,n'=k_2l_2$, $|k_1|=K-1,|k_2|=k' $ and consider $U'=I_2(\sigma)^{-1}I_1(\sigma)(U),V'=I_2(\rho)^{-1}I_1(\rho)(V).$
But from the definitions, one easily gets for $X\in \langle N, e_D\rangle\cap L^1(\langle N, e_D\rangle)$:
\begin{align*}P_{k_1\circ k_2\circ m}&\left([\iota_{k_1k_2}\o_{D}\iota_{l_2l_1}(U\#_KV)]\#(P_mXP_l^*)\right)P_{\overline{l_2\circ l_1}\circ l}^*\\&=P_{k_1\circ k_2\circ m}\left([\iota_{k_1}\o_{D}\iota_{l_1}(U)]\#[\iota_{k_2}\o_{D}\iota_{l_2}(V)]\#(P_mXP_l^*)\right)P_{\overline{l_2\circ l_1}\circ l}^*.\end{align*}
and similarly :
\begin{align*}&(\overline{P_{\overline{l_2\circ l_1}\circ l}}\o_{D^{op}} P_{k_1\circ k_2\circ m})\left([ (\overline{P_l}\o_{D^{op}}P_m)](X)\#[\iota_{l_2l_1}\o_{D}\iota_{k_1k_2}(V'\#_{|n'|+1-k'}U')]\right)\\&=(\overline{P_{\overline{l_2\circ l_1}\circ l}}\o_{D^{op}} P_{k_1\circ k_2\circ m})\left([ (\overline{P_l}\o_{D^{op}}P_m)](X)\#[\iota_{l_2}\o_{D}\iota_{k_2}(V')]\#[\iota_{l_1}\o_{D}\iota_{k_1}(U')]\right)
\end{align*}

From the assumptions on $U$ and $V$ the two second lines are equal, and then, from \eqref{TensorWick} and \eqref{TensorWickL1}, one deduces the conclusion we wanted, for all $p,q$ :
\begin{align*}P_{p}&\left([\iota_{k_1\circ k_2}\o_{D}\iota_{l_2\circ l_1}(U\#_KV)]\#(P_mXP_l^*)\right)P_{q}^*\\&=(\overline{P_q}\o_{D^{op}} P_{p})\left([ (\overline{P_l}\o_{D^{op}}P_m)](X)\#[\iota_{l_2l_1}\o_{D}\iota_{k_1k_2}(V'\#_{|n'|+1-k'}U')]\right),
\end{align*}
which, using Lemma \ref{basic2}.(1), implies our statement  and :
$$I_2(\tau)^{-1}I_1(\tau)(U\#_KV)=V'\#_{|n'|+1-k'}U'.$$

The other statements about composition product and product are similar, the first statement in each case always following from the second using the stability by adjoint proved before. We  give a few details concerning the second point for the composition product.

Take $U\in M^{\oeh{D}(n,\sigma)}=M^{\oeh{D}(k_1,l_1)}, V\in D'\cap  M^{\oeh{D}n'}, k'\geq K, n=k_1l_1$, $|k_1|=k'$ and let $U'=I_2(\sigma)^{-1}I_1(\sigma)(U)$. Note that $n\#_K(n')=[k_1\#_K(n')]\circ l_1.$ As before it suffices to prove :
\begin{align*}&P_{[k_1\#_K(n')]\circ  m}\left([\iota_{k_1\#_K(n')}\o_{D}\iota_{l_1}(U\#_KV)]\#(P_mXP_l^*)\right)P_{\overline{l_1}\circ  l}^*\\&=(\overline{P_{\overline{l_1}\circ  l}}\o_{D^{op}} P_{[k_1\#_K(n')]\circ  m})\left([ (P_l\o_{D^{op}}P_m)](X)\#[\iota_{l_1}\o_{D}\iota_{k_1\#_K(n')}(U'\#_{|n|+2-K}V)]\right),
\end{align*}

But now, by assumption, we know :
\begin{align*}&P_{k_1\circ m}\left([\iota_{l_1}\o_{D}\iota_{k_1}U]\#(P_mXP_l^*)\right)P_{\overline{l_1}\circ  l}^*\\&=(\overline{P_{\overline{l_1}\circ  l}}\o_{D^{op}} P_{k_1\circ m})\left([ (P_l\o_{D^{op}}P_m)](X)\#[\iota_{l_1}\o_{D}\iota_{k_1}U']\right),
\end{align*}
Moreover, by Lemma \ref{Wick} they are valued respectively in 
$\iota_{\overline{l_1},l,k_1,m}^{-1} [D'\cap M^{\oeh{D} |k_1|}\oeh{D}B(M:D,l,m)\oeh{D}\overline{M}^{\oeh{D} |l_1|}]$ and $\hat{\iota}_{\overline{l_1},l,k_1,m}^{-1}[D'\cap ((L^2(M)^{\o_D l*}\oeh{D}\overline{M}^{\oeh{D} |l_1|})\oeh{D^{op}}(M^{\oeh{D} |k_1|}\oeh{D}L^2(M)^{\o_D m}))]$. By Lemma \ref{basic2}.(2), it suffices to see that the two elements we wish to prove equal in $B(L^2(M)^{\o_D \overline{l_1}\circ  l},L^1(D)\o_{h D^{op}}L^2(M)^{\o_D [k_1\#_K(n')]\circ  m})_D$ have the same value on any $\xi\in L^2(M)^{\o_D \overline{l_1}\circ  l}$. 

Since $P_{k_1\circ m}\left([\iota_{l_1}\o_{D}\iota_{k_1}U]\#(P_mXP_l^*)\right)P_{\overline{l_1}\circ  l}^*
(\xi)\in L^2(D)\o_{h D^{op}}(M^{\oeh{D} |k_1|}\oeh{D}L^2(M)^{\o_D m})$ the equality we want can be obtained from the one we know by applying the multiplication $.\#V$ which is well defined on the appropriate extended Haagerup tensor powers of $M$ in the range of our maps.

The reader should note that in this case, we thus actually proved $$I_2(\tau)^{-1}I_1(\tau)(U\#_KV)=U'\#_{|n|+2-K}V.$$
\end{proof}

\subsubsection{Cyclic Haagerup tensor products: the general case}
 We are now ready to introduce our cyclic extended Haagerup tensor product as an intersection space with enough compatibility condition to have a cyclic group action on it. Once those cyclic group actions are obtained, our various products and actions leave stable our intersection space as expected. We also obtain a density result saying that our spaces are non-trivial as soon as $D'\cap L^2(M)^{\o_D n}$ are. We also obtain traciality and functoriality results crucial to build later evaluations maps.

\begin{Proposition}\label{CyclicPermutations}
Let $D\subset M$ finite von Neumann algebras and $N_\kappa=M*_D(D\otimes W^*(S_1,...,S_\kappa))$, $\kappa$ infinite. We write $n$ a generic word in $\kappa$ letters of length $N$. 
 Let $M^{\o_{ehsc D}(N+2)}$ the intersection space of $$I_1(\sigma,n)^{-1}(M^{\oeh{D}(n,\sigma)})=I_2(\sigma^{-1},\sigma.n)^{-1}(M^{\oeh{D}(\sigma.n,\sigma^{-1})})\subset (D'\cap  M^{\oeh{D}(N+2)})$$ for $\sigma\in \mathfrak{C}_{N+2}$, completely isometrically included via $I=\bigoplus_{n, |n|=N}(Id\oplus (J(\sigma,n))_{ \sigma\in (\mathfrak{C}_{n}-\{Id\})})$
 into $(D'\cap  M^{\oeh{D}(N+2)})^{\oplus(\mathfrak{C}_{N+2}\times \kappa^N)},$ (with operator space direct sum norm) and write $J(Id)=Id,$ $J(\sigma,n)=I_1(\sigma^{-1},\sigma.n)^{-1}\circ I_2(\sigma^{-1},\sigma.n),$ with $I_i$ associated to $n$. This intersection space is independent of $\kappa$ infinite.

Consider $M^{\o_{ehSc D}N+2}=\left(\bigcap_{n\neq m, |n|=|m|=N}Ker(J(\sigma,n)-J(\sigma,m)) \right)\subset  M^{\o_{ehsc D}N+2}$ and on $(M^{\o_{ehSc D}n})^{\mathfrak{C}_{N+2}}$, $P_{\sigma}$ the projection on the $\sigma$ component and the maps $J(\sigma_1,\sigma_2)=J(\sigma_1)P_{\sigma_2,n}$, with $J(\sigma)=J(\sigma,n)$ for any $n$, and a corresponding $I$ without repetition over $n$ and then define   \begin{align*}M^{\oehc{D}N+2}&:=I^{-1}\left(\bigcap_{(\sigma_1,\sigma_2)\in \mathfrak{C}_{N+2}^2}Ker(J(\sigma_1,\sigma_2)-J(Id,\sigma_1\sigma_2)) \right)\\&\subset  I^{-1}((M^{\o_{ehSc D}N+2})^{\mathfrak{C}_{N+2}})\subset M^{\o_{ehsc D}N+2},\end{align*}
with the induced norm, for which we have equality with the previous definition when $N=0$.

\begin{enumerate} 
\item For any $U\in M^{\o_{ehsc D}N+2}, V\in M^{\o_{ehsc D}M+2},$ we deduce $U^\star\in M^{\o_{ehsc D}N+2}, U\#_iV\in M^{\o_{ehsc D}N+M+2}$ for all $i\in [\![1,N+1]\!],$  $UV\in M^{\o_{ehsc D}N+M+3},$ and similarly for $s$ replaced by $S$.  

Moreover the maps $ J(\sigma)$  
induce a continuous action of $\mathfrak{C}_{N+2}$ on 
$M^{\oehc{D}N+2}$. 
 For any $U\in M^{\oehc{D}N+2}, V\in M^{\oehc{D}M+2},$ we have:  $U^\star\in M^{\oehc{D}N+2}, U\#_iV\in M^{\oehc{D}N+M+2}$ for all $i\in [\![1,N+1]\!],$  $UV\in M^{\oehc{D}N+M+3}.$ Moreover, $M^{\oehc{D}N+2}$ is weak-* dense in $D'\cap M^{\oeh{D}N+2}$ and dense in $D'\cap L^2(M)^{\o_D N+2}.$
\item  Assume  either that there exists a $D$-basis of  $L^2(M)$ as a right $D$ module $(f_i)_{i\in I}$  which is also a $D$-basis of  $L^2(M)$ as a left $D$ module or that $D$ is a $II_1$ factor and that $L^2(M)$ is an extremal $D-D$ bimodule. 
Then, the linear map $J(\sigma)$  extends to an isometry on the subspace generated  $D'\cap L^2(M)^{\o_D n}.$ 
As a consequence, $\tau(X)=\langle e_D,X\#e_D\rangle$ is a trace on $D'\cap N_{ \kappa}^{\oeh{D}2}.$

 \item {[Partial fonctoriality] If $\phi_1:M^{\oeh{D}n_1}\to N_\kappa$,...,$\phi_p:M^{\oeh{D}n_p}\to N_\kappa$  are multiplication maps to canonical semicircular variables in $N_\kappa$, then $$\phi_1\o_D...\o_D\phi_p: M^{\oehc{D}n}\to {N_\kappa}^{\oehc{D}p}$$  is a completely bounded map with $n=\sum n_i$. In particular, in the degenerate case $\forall i, n_i=1$, we have a complete isometry $M^{\oehc{D}n}\subset {N_\kappa}^{\oehc{D}n}.$
 
 Moreover if $E:N_\kappa \to M$ is the canonical conditional expectation, $E^{\o_D p}:{N_\kappa}^{\oehc{D}p}\to M^{\oehc{D}p}$ is a completely contractive map.}
\end{enumerate}
\end{Proposition}
\begin{proof}
The independence of the intersection space of $\kappa$ infinite is obvious since any equation to check can be reduced to a countably generated algebra, and thus to countably many $S_i$ as variables. The agreement with the previous definition in the case $N=0$ is easy from Lemma \ref{basic2}.(2) and left to the reader.

(1) The stability of $M^{\o_{ehsc D}n}$ by adjoint, composition product and product are obvious from Proposition \ref{OnePermutation}. The stability of $M^{\oehc{D}n}$ comes from the equations we  (could have) got on $J(\tau)(U\#_kV),J(\sigma(U^\star))$ in the proof in each case.
We fix $n$ and first compute the inverse of $J(\sigma)=J(\sigma,n)=I_1(\sigma^{-1})^{-1}\circ I_2(\sigma^{-1})$ on $M^{\o_{ehsc D}N+2}$. Note first that $J(\sigma^{-1})J(\sigma)=I_1(\sigma)^{-1}\circ I_2(\sigma)I_1(\sigma^{-1})^{-1}\circ I_2(\sigma^{-1})=I_1(\sigma)^{-1}I(\sigma^{-1})I_2(\sigma^{-1})$ so that $J(\sigma^{-1})J(\sigma)I_1(\sigma)^{-1}=I_1(\sigma)^{-1}I(\sigma^{-1})I(\sigma)=I_1(\sigma)^{-1}$
by (1) and since $I_1(\sigma)^{-1}$ is surjective, one gets $J(\sigma^{-1})J(\sigma)=Id$ and likewise the converse to that $J(\sigma,n)^{-1}=J(\sigma^{-1},\sigma.n)$. 

 By definition as an intersection, $ (J(\sigma))$ defines an action on $M^{\oehc{D}n}$ since on the intersection of kernels we exactly have $J(\sigma_1)J(\sigma_2)=J(\sigma_1\sigma_2).$

It mostly remains to show the density results.
For, we prove that for any $x_1,...,x_{N+2}\in M$, then $E_{D'}(x_1\o_D...\o_D x_{N+2})\in M^{\oehc{D}N+2}$. From the weak-* continuity on bounded sets of $E_{D'}$, this implies the weak-* density. The $L^2$ density is even easier.
More precisely, we show that $E_{D'}(x_1\o_D...\o_D x_{N+2})\in M^{\o_{ehsc D}N+2}$ and  $$
J(\sigma,n)(E_{D'}(x_1\o_D...\o_D x_{N+2}))=E_{D'}(x_{\sigma^{-1}(1)}\o_D...\o_Dx_{\sigma^{-1}(N+2)})$$
and on this formula one reads it is also in the intersection of kernels defining $M^{\oehc{D}n}.$

Thus we can fix $\sigma\in \mathfrak{C}_{N+2}$ and $n=kl$, $\sigma(1)=|l|+2$, $|l|=N+1-|k|$. We have to show for any $X\in \langle {N_\kappa},e_D\rangle\cap L^1(\langle {N_\kappa},e_D\rangle)$:
$$(\iota_{k}\o_D\iota_l(E_{D'}(x_1\o_D...\o_D x_{N+2})))\#X=X\#(\iota_{l}\o_D\iota_k(E_{D'}(x_{k+1}\o_D...\o_D x_k))).$$

This reduces to \eqref{finitesum} if we show that $$\iota_{k}\o_D\iota_l(E_{D'}(x_1\o_D...\o_D x_{N+2}))=E_{D'}(\iota_{k}\o_D\iota_l(x_1\o_D...\o_D x_{N+2})).$$
But we saw both sides can be further included in $L^2(M)^{\o_D {N+2}}$ as a subspace with both $E_{D'}$ agreeing with the projection there. This concludes.

(2) From the action property in (1) on the dense set where $J(\sigma)$ is defined, it suffices to consider $\sigma$ a generator of the cyclic group. We thus extend $J(\sigma)$ isometrically in the case $\sigma$ is  such that $\sigma(1)=N+2$.

Moreover, by the density of (linear combinations of) vectors of the form $E_{D'}(x_1\o_D...\o_D x_{N+2})$ obtained in the proof of (1), it suffices to show that the restriction of $J(\sigma)$ to those vectors is an isometry. 

But note that with our fixed $\sigma$, we have obtained the relation :
$$J(\sigma)[E_{D'}(x_1\o_D...\o_D x_{N+2})]=E_{D'}(x_2\o_D...\o_D x_{N+2}\o_D x_1).$$
Moreover, assuming extremality, there is by Theorem \ref{BurnsR} a unitary Burns rotation, and by its defining relation, it coincides with $J(\sigma)^{-1}$ so that $J(\sigma)$ is an isometry as stated. The case with a basis is left to the reader.

For the last statement about traciality of $\tau(X)=\langle e_D,X\#e_D\rangle$ 
on $D'\cap N_\kappa^{\oeh{D}2},$ 
we start from the result we obtained using the action for a general $\sigma$. Let $U,U'\in D'\cap M^{\oehc{D}(N+2)}$, $V=J(\sigma)(U),V'=J(\sigma)(U')\in D'\cap M^{\oeh{D}(N+2)}.$ One easily gets from the isometry relation :
\begin{align*}Tr( e_D[(\iota_{\overline{k}}\o \iota_{\overline{l}}(V^\star))\#[(\iota_k\o \iota_{l}(U'))\#e_D]]&=\langle U,U'\rangle  =\langle V,V'\rangle \\& = Tr( [(\iota_k\o \iota_{l}(U'))\#[(\iota_{\overline{k}}\o \iota_{\overline{l}}(V^\star))\#e_D]]e_D)
\end{align*} and one easily gets zero for various other injections.

Finally, we know that linear combinations of $E_{D'}(n\o_D n')$, $n,n'\in N_\kappa$ are weak-* dense in $D'\cap {N_\kappa}^{\oeh{D}2},$  and then using the strong density of $Span( \iota_k(M^{\oeh{D}k}),k\in \kappa^N,N\geq 0)$ in ${N_\kappa}$, we get the same result, with $n,n'$ in this span. But now, we already noticed that $E_{D'}(\iota_k\o \iota_{l}(U))=(\iota_k\o \iota_{l}(E_{D'}(U))$ thus proving the weak-* density of $Span\{ E_{D'}(\iota_k\o \iota_{l}(D'\cap M^{\oehc{D}N+2}), |k|+|l|=N\geq 0\}$ in $D'\cap {N_\kappa}^{\oeh{D}2}$ (and even of the intersection of the unit ball in the intersection of the unit ball using moreover Kaplansky density Theorem in the reasoning above). Now, the weak-* continuity proved in  Theorem \ref{finite2}.(2) of $X\mapsto \langle e_D,X\#(Y\#e_D)\rangle$ (and the obvious one of $Y\mapsto \langle e_D,X\#(Y\#e_D)\rangle$ using \eqref{adjointActions}
) concludes.

(3) The complete boudedness statements follow from replacing $M$ by $M_n(M)$ and checking the bounds don't depend on $n$.  For the first statement, using Wick expansion, it suffices to prove boundedness of $\iota_{k_1}\o_D...\o_D\iota_{k_p}: M^{\oehc{D}n}\to {N_\kappa}^{\oehc{D}p}$ for $|k_p|=n_p-1.$ Since the map is defined $D'\cap M^{\oeh{D}n}\to D'\cap {N_\kappa}^{\oeh{D}p}$ by the universal property, it suffices to check the stability of corresponding subspaces. Since ${N_\kappa}$ is involved, we consider $N'=W^*({N_\kappa},S_1',...,S_\kappa')$ and $\iota'_k$ the corresponding evaluation for a word in $\kappa$ letters (with primes), $\iota''_k$ the evaluation for $M$ with a word $k$ in $2\kappa$ letters. If $|l|=p-1$ is a word in $\kappa$ letters with primes , and $k_i$'s are word in $\kappa$ letters without prime as before, we write $l\circ(k_1,...,k_p)=k_1l_1...l_{p-1}k_p$ and one then notices (using some orthogonality in free products) that $\iota_l'\circ (\iota_{k_1}\otimes_D...\o_D\iota_{k_p})=\iota''_{l\circ(k_1,...,k_p)}.$ One easily deduces from this the stated stability, the boundedness following from the very definitions of norms involving more specific evaluation and from $(N',E_M)\simeq ({N_\kappa},E_M)$ since $\kappa$ infinite.

For the statement on conditional expectations, it suffices to prove the boundedness on $E^{\o_D p}:{N_\kappa}^{\o_{ehscD}p}\to M^{\o_{ehscD}p}$ by the symmetry of this map which induces easily the stability of kernels involving the action of the cyclic group.
It suffices to check that  $I_1(\sigma)\circ E^{\o_D p}(X)=E_{W^*(M,S_1',...,S_\kappa')}[I_1(\sigma)(X)]E_{W^*(M,S_1',...,S_\kappa')} $ and $I_2(\sigma)\circ E^{\o_D p}(X)=(E_{W^*(M,S_1',...,S_\kappa')}\o_{D^{op}}E_{W^*(M,S_1',...,S_\kappa')})I_2(\sigma)(X),$ which are easily checked on elementary tensors by freeness with amalgamation over $M$ of ${N_\kappa}$ and $W^*(M,S_1',...,S_\kappa')$.
\end{proof}

\section{Appendix 2: Function spaces}

In this section, we study several function spaces crucial to our constructions.  We start by considering spaces of analytic functions as well as cyclic analytic functions (these can be regarded as enlargements of spaces of non-commutative polynomials and cyclically {symmetrizable} non-commutative polynomials). We then consider analytic functions that depend on expectations, i.e., enlargements of functions of the form $X_{i_1} E(X_{i_2}X_{i_3} E(X_{i_4} X_{i_5}) X_{i_6}) X_{i_7}$, where $E$ is a (formal) conditional expectation.  Finally, we consider analogues of spaces of $C^k$-functions, defined as completions in certain $C^k$ norms. 

 \subsection{Generalized Cyclic non-commutative analytic functions}
\label{CyclicA}

In this section we study the properties of cyclic $B_c\langle X_1,...,X_n:D,R,\C\rangle$ and ordinary $B\langle X_1,...,X_n:D,R,\C\rangle$  generalized analytic functions in $n$ variables with radius of convergence at least $R$, defined in subsection \ref{notation}. Here, as before, $D\subset B$ are finite von Neumann algebras.
We will also consider a variant with several radius of convergence $R,S$, $B\langle X_1,...,X_n:D,R;Y_1,...,Y_m: D,S\rangle,B_c\langle X_1,...,X_n:D,R;Y_1,...,Y_m: D,S\C\rangle.$ We will use it freely later.
If $X=(X_1,\cdots,X_n)$, we also write $B\langle  X:D,S\rangle$ for $B\langle X_1,...,X_n:D,S\rangle$, etc.



We have the following basic result:


\begin{Proposition}\label{analytic} Let $X=(X_1,\dots,X_n)$, $Y=(Y_1,\dots, Y_m)$.  Then
(a) The linear spaces $B_c\langle X:D,R,\C\rangle,B\langle X:D,R,\C\rangle$ (resp. $B\langle X:D,R\rangle$) are Banach $*$-algebras 
as well as operator spaces  (resp. Banach algebra and strong operator $D$ module).  {Moreover, $B\langle X:D,R,\C\rangle$, $B\langle X:D,R\rangle$ are dual operator spaces when seen as (module) duals of (module) $c_0$ direct sums of the fixed preduals of each term of the $\ell^1$ direct sum. We always equip them with this weak-* topology. Finally the algebra generated by $B,X$ is weak-* dense in those spaces.}

(b) For $P\in B\langle X:D,R\rangle, Q_1,...,Q_n\in D'\cap B\langle X:D,S,\C\rangle$, such that $\|Q_i\|\leq R,$  there is a  well defined composition obtained by evaluation at $Q_j$:  $P(Q_1,...,Q_n)\in B\langle  X:D,S\rangle$.  The composition also makes sense on the  cyclic variants and is compatible with  canonical inclusion maps on these function spaces. 

(c) If $B_{\otimes k}\langle X:D,R,C\rangle$  (with $C=\C$ or $C=D$) is the subspace of $B\langle X ,Y:D,R,C\rangle$ consisting of functions linear in each $Y_1,\dots,Y_m$ and so that in each monomial each letter $Y_j$ only appears to the right of all letters $Y_i$ with $i<j$,then there are canonical maps 
 $$.\#(.,...,.):B_{\otimes k}\langle X:D,R\rangle\times \prod_{i=1}^kB_{\otimes l_i}\langle X:D,R,\C\rangle\to B_{\otimes(\sum_il_i)}\langle X:D,R\rangle,$$ $l_i\geq 0$ induced from composition in the $Y$ variables. (Note that by definition $B_{\otimes 0}\langle X:D,R\rangle = B\langle X:D,R\rangle$.)

(d) For any  $N\supset B$ a finite von Neumann algebra, $P\in B\langle  X:D,R\rangle$ defines a map $(D'\cap N)_R^n\to  N,$  by evaluation, with $P(X_1,...,X_n)\in W^*(B,X_1,...,X_n).$ 
\end{Proposition}

\begin{proof}
The fact that $B\langle X_1,...,X_n:D,R\rangle$ is a Banach algebra is obtained in \cite[Th 39]{dabsetup}. {The dual operator space structure and weak-* density also come from this result.}  The stability by adjoint only works for direct sums over $\C$ (since adjoint is not a module map and would require the conjugate module structure). The stability by multiplication obtained in Proposition \ref{CyclicPermutations} gives the same result for 
$B_c\langle X_1,...,X_n:D,R,\C\rangle.$
For the stability by composition, the well-known composition map in \cite{dabsetup} Theorem 2  is completely bounded in each of the middle variables and it is easy to see that the compositions built in Proposition \ref{CyclicPermutations} also are (since the intersection norm is obtained from Haagerup norms dealt with in the non-cyclic case). Thus, $\ell^1$ direct sums are dealt with  using universal property, the only key point is that we use operator space (and not module) $\ell^1$ direct sum for composition in $Q_i$ variables since the multilinear map $(P,Q_1,...,Q_n)\to P(Q_1,...,Q_n)$ is a $D-D$ module map only in the variable $P$. In this way, the previous complete contractivity can be used in each variable with the right universal property for each type of $\ell^1$ direct sum. In order to use the universal property in $P$, one also needs to know the source and target modules are strong operator modules over $D$ in the non-cyclic case, and they are since those extended Haagerup products are even normal dual operator modules. 
 The statements for $B_{\otimes k}$ are obvious consequences.
  The evaluation map comes from the  standard inclusion $B^{\oeh{D}n}\subset  N^{\oeh{D}n}$ (see e.g. \cite{dabsetup} Theorem 2.(2)), and from the multiplication maps explained e.g. in \cite{dabsetup} Theorem 2.(4). The reader should note that they can be applied on a larger space than the one in \cite[Th 39]{dabsetup} since in general $D'\cap N\supset E_D'\cap N$. Note the evaluation maps used here may not a have any kind of weak-* continuity, contrary to those of \cite[Th 39]{dabsetup}.
 \end{proof}

\subsubsection{Difference quotient derivations and cyclic derivatives} 
\begin{Proposition} \label{analytic2}Let $S<R$. 
(a) The   iterated free difference quotients  $\partial_{(i_1,...,i_k)}^{k}=(\partial_{X_{i_1}}\o 1^{\o k-1})\circ \ldots \circ \partial_{X_{i_k}}$  define 
completely bounded maps from  $B\langle X:D,R,C\rangle$  to 
$B_{\otimes k}\langle X:D,S,C\rangle$  {(with $C=\C$ or $C=D$, and thus in both cases to $B\langle X:D,S\rangle^{\oeh{D} k+1}$)}.

(b) The space  $B_c\langle X:D,R,\C\rangle$  is mapped by  $\partial_{(i_1,...,i_k)}^{k}$ to $B_{\otimes k c}\langle X:D,R,\C\rangle$. 

(c) For  $d\in B_c\langle X:D,S\C\rangle$, the cyclic gradient  $\mathscr{D}_{X_i,d}$ defines a bounded map from   $B_c\langle X:D,R,\C\rangle$ to $B_c\langle X_1,...,X_n:D,S\C\rangle$ 

(d) The following cyclic derivation relation holds: \begin{equation}\label{CyclicDerivation}\mathscr{D}_{X_i,d}(PQ)=\mathscr{D}_{X_i,Qd}(P)+\mathscr{D}_{X_i,dP}(Q).\end{equation}

(e) The following relations between derivatives and composition hold, denoting $Q=(Q_1,\dots,Q_n)$: 
 \begin{align}\label{CompositionAnalytic}\begin{split}&\partial^{k}_{(j_1,...,j_k)}(P(Q))=\sum_{l=2}^k\sum_{n_1,...,n_l}\sum_{1\leq i_1<...<i_l=k} (\partial_{(n_1,...,n_l)}^{l}(P))(Q)  \# \\ & \qquad\qquad\qquad(\partial_{(j_1,...,j_{i_1})}^{i_1}Q_{n_1},\partial_{(j_{i_1+1},...,j_{i_2})}^{i_2-i_1} Q_{n_2},...,\partial_{(j_{i_{l-1}+1},...,j_{k})}^{k-i_{l-1}} Q_{n_l}).\end{split}\end{align}  and 
\begin{align}\label{CompositionCyclic}\begin{split}&\mathscr{D}_{X_i,d}(P(Q))=\sum_{j=1}^n \mathscr{D}_{X_i,\mathscr{D}_{Q_j,d}(P)(Q)}(Q_j),\end{split}\end{align}
where we wrote $\mathscr{D}_{Q_j,d}(P)(Q)=[\mathscr{D}_{X_j,[d(X_1',...,X_n')}(P)](Q,X)$ considering $P\in B_c\langle X:D,R,\C\rangle\subset B_c\langle X,X':D,R,\C\rangle, d(X')\in B_c\langle X':D,R,\C\rangle\subset B_c\langle X,X':D,R,\C\rangle,$ so that $\mathscr{D}_{X_j,[d(X')}(P)]\in B_c\langle X,X':D,R,\C\rangle$ is well defined and can be evaluated at $X_i=Q_i$, $X'_i=X_i.$ 
\end{Proposition}

\begin{proof}
Let us write $n_{X_i}(m)$ for the $X_i$ degree of a monomial $m$, i.e., the number of times the variable $X_i$ occurs in $m$.
To define the free difference quotient and cyclic gradient, we start from the formal differentiation on monomial, add appropriate change of radius of convergences $S<R$ to allow boundedness of the map and then gather the monomials at the $\ell^1$ direct sum level by the universal property:
$$\partial_{X_i} :B\langle X_1,...,X_n:D,R,C\rangle\to\ell^1_C\left(S^{|m|}(B^{\oeh{D}(|m|+1)})^{\oplus^1_C n_{X_i}(m)};m\in M(X_1,...,X_n), |m|\geq 1\right),$$
and similarly in the cyclic cases. 

In order to for the value to belong to the claimed space, we also need to specify a canonical map $I$ with values in  $B_{\otimes 2}\langle X:D,S,C\rangle$. Of course, we want it to send the $j$-th component in the ${\oplus^1_C n_{X_i}(m)}$ direct sum to the component of the monomial $m_{X_i,j}$ which is identical to $m$ but with the $j$-th $X_i$ replaced by $Y_1$. Since there is a bijection between the disjoint union over monomials of  $\{m\}\times [\![1,n_{X_i}(m)]\!]$ and the set of monomials in $X$ and $Y_1$ linear in $Y_1$, it is easy to see that $I$ extend to a complete isomorphism of $\ell^1_C$ direct sums. We still write $\partial_{X_i}$ for $I\circ \partial_{X_i}$.

{For the cyclic gradient, one can then apply a different cyclic permutation on each term of the direct sum and we gather them in a map $\sigma: B_{\otimes 2 c}\langle X:D,S,\C\rangle\to B_{\otimes 2 c}\langle X:D,S,\C\rangle$  and a multiplication map $m_d:B_{\otimes 2 c}\langle X:D,S,\C\rangle\to B_{c}\langle X:D,S,\C\rangle$ (based on composition $\#$ at $d$ on the appropriate term of the tensor product and extending $m_d(P\otimes Q)=PdQ=(P\otimes Q)\#d$) to get the expected cyclic gradient: $\mathscr{D}_{X_i,d}=m_d\sigma\partial_{X_i}.$}

For the free difference quotient, to see there is a canonical map to the range space $B\langle X_1,...,X_n:D,R\rangle\oeh{D}B\langle X_1,...,X_n:D,R\rangle,$ one applies the following Lemma to each term of the direct sum inductively, and then the universal property of $\ell^1$ direct sums to combine them. (We of course apply after mapping $\ell^1_C$ to $\ell^1_D$ direct sums).
The various relations then follow by construction from the various associativity properties of the compositions and multiplication defined in Proposition \ref{CyclicPermutations}. We explain those associated to cyclic gradients. First, we obtain the derivation property of $\partial_{X_i}$ and $\partial_{X_i}(PQ)=\partial_{X_i}(P)Q+P\partial_{X_i}(Q)$ so that :
$$\sigma\partial_{X_i}(PQ)=[\sigma\partial_{X_i}(P)]\#(Q\otimes 1)+ [\sigma\partial_{X_i}(Q)]\#(1\otimes P)$$
and applying $m_d$ one gets \eqref{CyclicDerivation}.
Similarly, one obtains first the relation $$\partial_{X_i}(P(Q))=\sum_{j=1}^n \partial_{X_j}P(Q)\#(\partial_{X_i}Q_j)$$ and then $$\sigma\partial_{X_i}(P(Q))=\sum_{j=1}^n [\sigma(\partial_{X_i}Q_j)]\#[\sigma\partial_{X_j}P(Q)]$$ and applying $m_d$ gives \eqref{CompositionCyclic}.
\end{proof}

The following result is a module extended Haagerup variant of \cite[Lemma 7]{OP97}, the proof is the same using universal property of $\ell^1$ direct sums and \cite[Th 3.9]{M97}. We thus leave the details to the reader.
\begin{Lemma}
Let $E_1,E_2\in {}_DSOM_D,F_1,F_2\in {}_DSOM_D$, let $X=(E_1\oplus^1_D E_2)\oeh{D}(F_1\oplus^1_D F_2)$. Let $S$ be the closure of the subspace obtained by injectivity of Haagerup tensor product $(E_1\oeh{D} F_1)+(E_2\oeh{D} F_2).$ Then we have:
$$S\simeq (E_1\oeh{D} F_1)\oplus^1_D(E_2\oeh{D} F_2),$$
completely isometrically.
\end{Lemma}

{We will also need a more subtle evaluation result for $B_{\otimes k c}\langle X_1,...,X_n:D,R,\C\rangle$. which require  that our variables are nice functions of semi-circular variables.

We write $A_{R,UltraApp}^n$  for the set of $X_1,...,X_n\in A, X_i=X_i^*, [X_i,D]=0, \|X_i\|\leq R$ and such that $B,X_1,...,X_n$ is the limit in $E_D$-law (for the $*$-strong convergence of $D$) of variables in $B_c\langle X_1,...X_m:D,2,\C\rangle(S_1,...,S_m)$ with $S_i$ a family of semicircular variables over $D$, that is of elements in the set of analytic functions evaluated in $S_{1},\ldots,S_{m}$. Here $m$ is some large enough fixed integer number.

\begin{Proposition}\label{TensorCyclicEvaluation}
For any $(X_1,...,X_n)\in A_{R,UltraApp}^n,$ 
if $\phi_j:B^{\oeh{D}n_j}\to M$, $j=1,\dots,p$  are multiplication maps { $\phi_j(Z) = Z\# (X_{1+\sum_{l=1}^{j-1}n_l},\dots,X_{\sum_{l=1}^j n_l})$}, 
 $M=W^*(B,X_1,...,X_n)$, then $\phi_1\o_D...\o_D\phi_p: B^{\oehc{D}n}\to M^{\oehc{D}p}$  is a completely bounded map of norm less than $R^{n-p}$ with $n=\sum n_i$. 
As a consequence, any $(X_1,...,X_n)\in A_{R,UltraApp}^n,$  induces an evaluation map $B_{\otimes k c}\langle X_1,...,X_n:D,R,\C\rangle\to M^{\oehc{D}(k+1)}.$
\end{Proposition}
\begin{proof}
Assuming first $X_i\in B_c\langle X_1,...X_m:D,2,\C\rangle(S_1,...,S_m)$ the result is obvious in a similar way as for composition of corresponding analytic functions and from the evaluation map to $(S_1,...,S_m)$ in Proposition \ref{CyclicPermutations}.(3). 
At first, the result is valued in $N_1^{\oehc{D}p}$ with $N_1=W^*(B,S_1,...,S_m)$ but one easily deduces the more restricted space of value.

We now consider the more general case {with $$X_i\in C^{*,+}(B,S_1,...,S_m):= C^*(ev_{S_1,...,S_m}(B\langle X_1,...,X_m:D,2,\C\rangle)),$$ in the $C^*$ algebra generated in  $W^*(B,S_1,...,S_m)$ by evaluations of our analytic functions at semicircular variables.}  There is a map $\phi_1\o_D...\o_D\phi_p$ on the extended Haagerup tensor product by functoriality and nothing is required to get a map on the intersection space $\phi_1\o_D...\o_D\phi_p: B^{\o_{ehscD}n}\to M^{\o_{ehscD}p}.$
To get the stated map and even first a map $\phi_1\o_D...\o_D\phi_p: B^{\o_{ehScD}n}\to M^{\o_{ehScD}p}$, we have to check various stability properties of kernels appearing in their definition as an intersection space. From the formula below describing the commutation of the cyclic action and various tensor products of the maps $\phi$, this stability of kernels will become obvious. More precisely, let $U\in B^{\o_{ehscD}n}$ for $\sigma\in \mathfrak{C}_{p}$, we write $\hat{\sigma}$ the induced permutations on blocks  and $V=J(\hat{\sigma}) (U)$ and $n=kl$, $\sigma(1)=|l|+2$, $|l|=p-2-|k|$. We want to show for any $X\in \langle N,e_D\rangle\cap L^1(\langle N,e_D\rangle)$ with $N=W^*(M,S_1,...,S_\kappa):$
$$(\iota_{k}\o_D\iota_l)(\phi_{1}\o_D...\o_D\phi_p(U))\#X=X\#(\iota_{l}\o_D\iota_k)((\phi_{\sigma^{-1}(1)}\o_D...\o_D\phi_{\sigma^{-1}(p)}(V))))).$$
For, it suffices to evaluate them to $Y,Z\in [B\langle X_1,...,X_n:D,R,\C\rangle(X_1,...,X_n)] \langle S_1,...,S_\kappa\rangle$ $=:C \subset L^2(N)$ as in Lemma \ref{basic2}.(2) and to take $X\in Ce_DC$, and see equality in $L^1(D)$. The statement for $X_1,...,X_n$ analytic as above gives exactly this in this case. In the evaluated form, the convergence in $E_D$-law is clearly enough to get the general case from this one. The evaluation map is then obtained by the universal property of $\ell^1$ direct sums. It crucially uses the bound on the norm of the completely bounded map above $R^{n-p}$ that easily follows from the bounds on canonical evaluations and the sup norm on $M^{\o_{ehScD}p},M^{\oehc{D}p}.$
\end{proof}

}

\subsection{Analytic functions with expectations}
\label{CyclicE}

For $X=(X_1,\dots,X_n)$, the spaces $B_{c}\{ X:E_D,R,\C\}$, $B\{ X:E_D,R\}$ have been defined in section \ref{notation}. To prove various results on them, we need  some formal notation to explain several computations combinatorially.
First, since those spaces are defined as $\ell^1$ direct sums over pairs of monomials $m$ and non-crossing partitions $\sigma\in NC_2(2k)$ (indexing the parenthesizing where conditional expectations are inserted), we can write $\pi_{m,\sigma}$ for the projection on the corresponding component of the $\ell^1$ direct sum, and $\epsilon_{m,\sigma}$ for the corresponding injection.

We write $E_D$ for the formal conditional expectation characterized for $P\in$ $ B_{c,k}\{ X_1,...,X_n:E_D,R\}$ by $E_D(P)\in B_{c,k+1}\{ X_1,...,X_n:E_D,R\}$ and such that the only-nonzero projections $\pi$ are of the form $$\pi_{YmY,\hat{\sigma}}(E_D(P))=\pi_{m,\sigma}(P)$$ for $\hat{\sigma}=\{\{1,2i+2\}\}\cup (\sigma+1)$ where the blocks of $\sigma+1$ are $\{a+1,b+1\}$ if $\{a,b\}$ are the blocks of $\sigma$. All other components of $\pi_{m',\sigma'}(E_D(P))$ are $0$.  
 $E_D$ is obviously $D-D$ bimodular and completely bounded.

The scalar case $D=\mathbb{C}$  was considered in \cite{Ceb13}; in this case we note the density of 
$\C\{ X_1,...,X_n\}\supset Span\{P_0tr(P_1)...tr(P_k), P_i\in \C\langle X_1,...,X_n\rangle\}.$ 

For $P\in \C\{ X_1,...,X_n\}$ and a linear form $\tau \in (\C\langle X_1,...,X_n\rangle)^*$ there is a canonical element $P(\tau)\in \C\langle X_1,...,X_n\rangle$ defined by extending linearly $[P_0tr(P_1)...tr(P_k)](\tau)= P_0\tau(P_1)...\tau(P_k)$. In this way, one embeds $$\C\{ X_1,...,X_n\}\hookrightarrow C^0( (\C\langle X_1,...,X_n\rangle)^*,\C\langle X_1,...,X_n\rangle)$$ (where the continuity is coefficientwise on the range and for the weak-* topology induced by $\C\langle X_1,...,X_n\rangle$ on the source).

Similarly, for $P\in B\{ X_1,...,X_n;E_D,R\}$ and a unital $D$ bimodular 
 completely bounded linear map $E\in UCB_{D-D}(B\langle X_1,...,X_n:D,R\rangle,D)$, there is a canonical element $P(E)\in  B\langle X_1,...,X_n:D,R\rangle$.
Since $P\mapsto P(E)$ will be completely bounded $D-D$ bimodular on monomials, by the universal property of $\ell^1$ direct sums, it suffices to define it for monomials 
$P=\pi_{m,\sigma}(P),$  $\sigma\in NC_2(2k)$. It is defined by induction on $k$. Write $\sigma_-\in NC_2(2(k-1))$ the unique pair partition obtained by removing from $\sigma$ the  pair $\{i,i+1\}$ of smallest index $i$ and re-indexing by the unique increasing bijection $[\![1,2k]\!]-\{i,i+1\}\to [\![1,2(k-1)]\!].$ Let  also $j(i)$ the index in the word $m$ of the i-th $Y$ (this being $1$ if $i=1$ and $m$ starts by $Y$).
Then $P=\pi_{m,\sigma}(P)\in B^{\oehc{D}(|m|+1)}$, then  $$P(E)=[\epsilon_{m',\sigma_{-}}[1^{\o j(i)}\o E\circ \epsilon_{m''}\o 1^{\o |m|-j(i+1)+1}](P)](E),$$ with $m'=m_1...m_{j(i)-1}m_{j(i+1)+1}...m_{|m|}$, $m''=m_{j(i)+1}...m_{j(i+1)-1}$. Indeed the letters between the index $j(i)$ and $j(i+1)$ in $m''$ are only $X$'s and we can thus apply $E$ identifying $B^{\oeh{D}j(i+1)-j(i)}$ via $
\epsilon_{m''}$ with the corresponding subspace of $B\langle X_1,...,X_n:D,R\rangle$. Since $E$ is $D-D$ bimodular $[\epsilon_{m',\sigma_{-}}[1^{\o j(i)}\o E\circ \epsilon_{m''}\o 1^{\o |m|-j(i+1)+1}]]$ is well defined and we can apply $E$ inductively.

In this way, we have a canonical map  $$B\{ X_1,...,X_n:E_D,R\}\to
C^0( UCB_{D-D}(B\langle X_1,...,X_n:D,R\rangle,D),B\langle X_1,...,X_n:D,R\rangle).$$
where the topology on $UCB_{D-D}(B\langle X_1,...,X_n:D,R\rangle,D)$ is the topology of pointwise normwise convergence of $id_{M_I}\o E$ on all $M_I(B\langle X_1,...,X_n:D,R\rangle)$ {(for $I$ a cardinal smaller than the cardinal of $B$).}

To state the algebraic and differential properties we will use, we also need the following variant (for $C=\C$ or $C=D$):
\begin{align*}&B_{op (l)}\{ X_1,...,X_n:E_D,R,C\}\\&:=\ell^1_{ C}\left(R^{|m|_X}B^{\oeh{D}(|m|+1)};m\in M_{2k}'(X_1,...,X_n;Z_1,...,Z_l;Y),\pi\in NC_{2}(2k),k\geq 0\right),\end{align*}
where $M_{2k}'(X_1,...,X_n;Z_1,...,Z_l;Y)$ is the set of monomials linear in each $Z_i$, without constraint on the order of appearance of $Z_1,...,Z_n$ 
 and  of order $2k$ in $Y$ 
The blocks in $Z_i$ are made to evaluate a variable in $D'\cap N$. {We call $B_{\otimes (l)}\{ X_1,...,X_n:E_D,R,C\}$ the subspace involving monomials with $Z_k$ ordered in increasing order of $k$ and with all variables $Z_i$ having an even number of $Y$ before them  and with their pair partitions unions of those restricted to the intervals between them (thus $Z_i$'s are interpreted as not being inside conditional expectations.) We write $B_{\otimes (l)c}\{ X_1,...,X_n:E_D,R,C\}$ the cyclic variant generalizing $B_{\otimes (l)c}\langle X_1,...,X_n:E_D,R,C\rangle.$ }

The following result is obvious:

\begin{Proposition}\label{ExpectationAnalytic} Let $X=(X_1,\dots,X_n)$. 
(a) The spaces 
$B_{c}\{ X:E_D,R,{\C}\}$, $B\{ X:E_D,R,\C\}$ are Banach *-algebras for usual adjoint and multiplication, extending the ones of $B\langle X:D,R,\C\rangle$.   $B\{ X:E_D,R,\C\}$ is a dual Banach space and the smallest algebra generated by $B, X$ and stable by $E_D$ is weak-* dense in it.

(b) $B\{ X:E_D,R\}$ is a Banach algebra. $B\{ X:E_D,R\}$ is a dual Banach space and the smallest algebra generated by $B, X$ and stable by $E_D$ is weak-* dense in it.

(c) There is a composition rule, for $P\in B\{ X:E_D,R\}, Q_1,...,Q_n\in D'\cap B\{ X:E_D,S,\C\}$, such that $\|Q_i\|\leq R,$  then there is a composition $P(Q_1,...,Q_n) \in B\{ X:E_D,S\}$ extending the composition on $B\langle  X:D,S\rangle$.  There are similar cyclic variants compatible with canonical maps and with the evaluation map below.

(d) For finite von Neumann algebras $N\supset B$, $P\in B\{ X:E_D,R\}$ defines a map $(D'\cap N)_R^n\to  N$  by evaluation, with $P(X):=P(E_{X,D})(X)\in W^*(B,X),$ thus extending the value on $B\langle X:D,R\rangle$ and where $E_{X,D}\in UCB_{D-D}(B\langle X:D,R\rangle,D)$ comes from the conditional expectation.

(e) 
Similarly there is a canonical evaluation $ev_{op}(P,E_{X,D},X)\in CB((D'\cap N)^{\otimes_h l},N)$, $P\in  B_{op l}\{ X:E_D,R\}$, where $N$ are evaluated in the  $Z_i$'s and then each pair of $Y$'s is replaced by a conditional expectation.

(f)
There are also canonical continuous compositions (in the $Z_i$ variables) with commuting with evaluation {(with variants for $B_{\otimes (l)c}\{ X:E_D,R,C\},B_{\otimes (l)}\{ X:E_D,R,C\}$)}:  $$.\circ(.,...,.):  B_{op (k)}\{ X:E_D,R\}\times \prod_{i=1}^k B_{op (l_i)}\{ X:E_D,R,\C\}\to B_{op(\sum_il_i)}\{ X:E_D,R\}.$$ 

(g)
{ Finally for $(X_1,...,X_n)\in A_{R,UltraApp}^n$ we in particular have an evaluation map $B_{\otimes (l)c}\{ X:E_D,R,C\}\to M^{\oehc{D}(l+1)}$ with $M=W^*(B,X_1,...,X_n)$ as in Proposition \ref{TensorCyclicEvaluation}.}
\end{Proposition}

\subsubsection{Various derivatives of analytic functions with expectations} 
\begin{Proposition}\label{CalcDiffAnalytic} For  $C=\C$ or $C=D$ and any $S<R$,
(a) The free difference quotient (FDQ) derivations give rise to bounded  maps \begin{align*}\partial_i:B\{ X:E_D,R,C\}&{\to B_{\otimes (1)}\{ X:E_D,S,C\}}\\&\to B\{ X:E_D,S\}\oeh{D}B\{ X:E_D,S\}\end{align*} extending the free difference quotient from $B\langle X:D,R,C\rangle$ and {determined by weak-* continuity of the first line and} by the requirement that the composition with the formal $E_D$ is zero: $\partial_i E_D=0$. 

(b) The iterated FDQ $\partial_{(i_1,...,i_k)}^{k}:B_{c}\{X:D,R,\C\}\to B_{\otimes (k) c}\{ X:D,S,\C\}$ and $\partial_{(i_1,...,i_k)}^{k}:B\{ X:D,R,\C\}\to B_{\otimes k }\{ X:D,S,\C\}$ are also bounded maps.

(c) Let $d:B\{ X:E_D,R,C\} \to B_{op}\{ X:E_D,S,C\}^n$ and the operator variant $d:B_{op (l)}\{ X:E_D,R,C\} \to B_{op (l+1)}\{ X:E_D,R,C\}^n$  
be the {\em formal differentiation}, i.e.  a derivation {uniquely determined among weak-* continuous maps} by  $$d(B\langle X,Z_1,...Z_l:D,R\rangle)=0$$ and for any monomial $P\in B_{op (l)}\{ X:E_D,R\}$ (possibly $l=0$): $$dE_D(P)=E_D(d_XP),\ \ \  d_XP:=dP+(\partial_i(P)\#Z_{l+1}))_{i}$$
and $d_{X(i_1,...,i_l)}^l=d_{Xi_l}...d_{Xi_1}:B\{ X:E_D,R,C\} \to B_{op (l)}\{ X:E_D,S,C\}.$ Then $d$ and $d^l$ 
 are bounded maps.

(d) We define the cyclic gradients on $B_c\{ X:E_D,R,\C\}\to B_c\{ X:E_D,S,\C\}$ for $d\in B_c\{ X:E_D,S,\C\},S<R$ 
   as a natural continuous extension of the cyclic gradient  on $B_c\langle X:D,R,\C\rangle$,   satisfying $\mathscr{D}_{X_i,d}(X_j)=d1_{i=j}$, \eqref{CyclicDerivation} 	and for $P, Q$ monomials and for $d,P$ monomials $$\mathscr{D}_{i,d}(E_D(P))=\mathscr{D}_{i,E_D(d)}(P).$$  

(e) The following relation with compositions \eqref{CompositionAnalytic}, \eqref{CompositionCyclic} holds:
\begin{align}\label{CompositionAnalExpectation}\begin{split}&d^{k}_{X(j_1,...,j_k)}(P(Q))=\sum_{l={1}}^k\sum_{n_1,...,n_l}\sum_{1\leq i_1<...<i_l=k}\sum_{l_{1,1}=1,l_{m-1,1}<l_{m,1}<...<l_{m,i_m-i_{m-1}}\leq k} \\ & (d_{X(n_1,...,n_l)}^{l}(P))(Q)\circ
(d_{X(j_{l_{1,1}},...,j_{l_{1,i_1})})}^{i_1}Q_{n_1},d^{i_2-i_1}_{X(j_{l_{2,1}},...,j_{l_{2,i_2-i_1})}} Q_{n_2},...,d_{X(j_{l_{l,1}},...,j_{l_{l,k-i_{l-1}})}}^{k-i_{l-1}} Q_{n_l}),\end{split}\end{align} and in particular:
$$d_X(P(Q_1,...Q_n))=\sum_{i=1}^n((d_X(P)
)(Q_1,...,Q_n))_i\circ d_X(Q_i).$$\end{Proposition}

[Note the sum of $l_{i,j}$ in formula \eqref{CompositionAnalExpectation} is only a sum over partitions, the first term of the first set being written $l_{1,1}$, the first term of the second set in the partition $l_{2,1}$, the ordering between sets in the partition being by the ordering of the smallest element]
\begin{proof}
For the most part, we only have to give a combinatorial formula for the derivations acting on monomials. Then by the bimodularity of the formula and explicit uniform bounds, the universal property of the $\ell^1$ sum will extend them to module $\ell^1$ direct sums. They will be moreover weak-* continuous as soon as they are weak-* continuous when restricted to monomial components since the $c_0$ sum of predual maps will then give a predual map. The derivation properties then determine $d,\partial$ on the $E_D$-algebra  generated by $B,X_1,...,X_n$ which is weak-* dense in the $\ell^1$ direct sum (actually in each monomial space by properties of the extended Haagerup product and then, the finite sum of monomial spaces are normwise dense), thus weak-* continuity determine those maps everywhere.

 For $\sigma\in NC_{2}(2k)$, $m\in M_{2k}(X_1,...,X_n,Y)$, let us say a submonomial $m'\subset m$ (with a fixed starting indexed, $m'$ is thus formally a pair of the monomial and the starting index) is compatible with $\sigma$ and write $m'\in C(\sigma,m)$ if $m'\in M_{2l}(X_1,...,X_n,Y)$, $l\leq k$ and $l'$ the index in $m$ of the first $Y$ in $m'$, then $\sigma|_{m'}:=\sigma|_{[\![l',l'+2l-1]\!]}\subset \sigma$ (which means there is no pairing in $m$ broken in $m'$ by our extraction of $m'$). We then write $Sub(\sigma,m')\in NC_{2}(2l)$ the partition $\sigma|_{[\![l',l'+2l-1]\!]}$ reindexed. 
 
 Then we define:
$$\partial_i(\epsilon_{m,\sigma}(P))=\sum_{\substack{m=m'X_im'',\\ m',m''\in C(\sigma,m)}}(\epsilon_{m',sub(\sigma,m')}\o_D\epsilon_{m'',sub(\sigma,m'')})(P).$$

Of course the sum is $0$ if its indexing set is empty, this in particular explains $\partial_iE_D=0$ and the remaining properties are easy.

The definition of $d$ is complementary. When $m'$ or $m''$ are not both in $C(\sigma,m)$ and $m=m'X_im''$ (some $i$), we write $(m',m'')\in IC(\sigma,m)$ (and this corresponds to a differentiation of $X_i$ below a conditional expectation). 

 Then we define 
$$d(\epsilon_{m,\sigma}(P))=\left(\sum_{\substack{m=m'X_im'',\\ (m',m'')\in IC(\sigma,m)}}(\epsilon_{m'Z_1m'',\sigma})(P)\right)_i.$$

For $\sigma_1,\sigma_2\in NC_2(2k_i)$ we define for $i\in [\![0,2k_1]\!]$ the obvious insertion  $\sigma_1\#_i\sigma_2=\sigma$  such that $\sigma|_{[\![i+1,i+k_2]\!]}=\sigma_2$, $\sigma|_{[\![i+1,i+k_2]\!]^c}=\sigma_1$ the equalities being understood after increasing reindexing. Likewise  $\rho_i(\sigma_1)=\{\{i_j+i,i_k+i\}: \{i_j,i_k\}\in \sigma_1\}$ addition being understood modulo $2k_1$ so that $\rho_{2k_1}=\rho_0=id,$ and write also $\rho_i$ the corresponding permutation $\rho_i(k)=k+i$ modulo $2k_1$.

We now define the cyclic gradient as follows:
$$\mathscr{D}_{i,\epsilon_{M,\Sigma}(d)}(\epsilon_{m,\sigma}(P))=\sum_{m=m'X_im''}\epsilon_{m''Mm',\rho_{|m''|_Y}(\sigma)\#_{|m''|_Y}\Sigma}((\rho_{|m''|}.P)\#_{|m''|}d)$$
and the relations are then easy. We give details for two of them involving cyclic gradients.

Let us explain \eqref{CyclicDerivation} on spaces of monomials. We have to compute $\mathscr{D}_{i,\epsilon_{M,\Sigma}(d)}(\epsilon_{m,\sigma}(P)\epsilon_{\mu,\pi}(Q))$. 
First note that $\epsilon_{m,\sigma}(P)\epsilon_{\mu,\pi}(Q)=\epsilon_{m\mu,\sigma\cup \pi}(PQ)$. Here $\sigma\cup \pi$ is merely the concatenation of non-crossing partitions and $PQ$ the product of tensors defined in Proposition \ref{CyclicPermutations} (1).
Note that the sum over $m\mu=m'X_im''$ splits into two sums depending on whether $X_i$ comes from $m$ or $\mu$. This gives the following computation (using relations on rotation and product such as $\rho_{|m''\mu|}.(PQ)=\rho_{|m''|}.(P))\#_{|m''|}(Q\o 1)$) :
\begin{align*}\mathscr{D}_{i,\epsilon_{M,\Sigma}(d)}&(\epsilon_{m,\sigma}(P)\epsilon_{\mu,\pi}(Q))\\&=\sum_{m=m'X_im''}\epsilon_{m''\mu Mm',\rho_{|m''\mu|_Y}(\sigma\pi)\#_{|m''\mu|_Y}\Sigma}((\rho_{|m''\mu|}.PQ)\#_{|m''\mu|}d)
\\&+\sum_{\mu=m'X_im''}\epsilon_{m'' Mmm',\rho_{|m''|_Y}(\sigma\pi)\#_{|m''|_Y}\Sigma}((\rho_{|m''|}.PQ)\#_{|m''|}d)
\\&=\sum_{m=m'X_im''}\epsilon_{m''\mu Mm',\rho_{|m''|_Y}(\sigma)\#_{|m''|_Y}(\pi\Sigma)}((\rho_{|m''|}.(P))\#_{|m''|}(Qd))
\\&+\sum_{\mu=m'X_im''}\epsilon_{m'' Mmm',\rho_{|m''|_Y}(\pi)\#_{|m''|_Y}(\Sigma\sigma)}((\rho_{|m''|}.Q)\#_{|m''|}(dP))
\\&=\mathscr{D}_{i,\epsilon_{\mu,\pi}(Q)\epsilon_{M,\Sigma}(d)}(\epsilon_{m,\sigma}(P))+ \mathscr{D}_{i,\epsilon_{M,\Sigma}(d)\epsilon_{m,\sigma}(P)}(\epsilon_{\mu,\pi}(Q)).
\end{align*}

Let us finally explain \eqref{CompositionCyclic}. By linearity (in $P$) and continuity (in $P$ and $Q$), it suffices to consider the case of finite sums 
$$Q_k=\sum_{i}\epsilon_{M_{k,i},\sigma_{k,i}}(Q_{k,i}),\qquad k=1,...,n$$ and where $P$ is replaced by a monomial $\epsilon_{m,\sigma}(P)$. Then write $Q_{X_k,i}=Q_{k,i}$ and $Q_{Y,i}=1\o 1,$ $M_{Y,k}=Y,M_{X_l,k}=M_{l,k}$ 	and note that $$[\epsilon_{m,\sigma}(P)](Q)=\sum_{i_1,....,i_{|m|}}\epsilon_{M_{m_1,i_1}...M_{m_{|m|},i_{|m|}},\sigma\#^m(\sigma_{m_1,i_1},...,\sigma_{m_{|m|},i_{|m|})}}\left(P\#(Q_{m_1,i_1},...,Q_{m_{|m|},i_{|m|}})\right)$$
where if $m_{j_1},..., m_{j_{2l}}$ is the set of $Y$'s in $m$, $\sigma\in NC(2l)$, $\sigma_{X_k,l}=\sigma_{k,l}$ and \begin{align*}\sigma\#^m(\sigma_{m_1,i_1},...,\sigma_{m_{|m|},i_{|m|}})=
&(\dots((\sigma\#_{2l}(\sigma_{m_{j_{2l}+1},i_{j_{2l}+1}}\cdots\sigma_{m_{|m|},i_{|m|}}))\\&\#_{2l-1}(\sigma_{m_{j_{2l-1}},i_{j_{2l-1}}}\cdots\sigma_{m_{j_{2l}-1},i_{j_{2l}-1}}))\dots \#_{0}(\sigma_{m_{1},i_{1}}\cdots\sigma_{m_{j_1-1},i_{j_1-1}}))\end{align*}

Thus one gets in writing for short $M_{m,i,L,+}=M_{m_{L+1},i_{L+1}}...M_{m_{|m|},i_{|m|}}$:
\begin{align*}&\mathscr{D}_{j,\epsilon_{M,\Sigma}(d)}[\epsilon_{m,\sigma}(P)](Q)=\sum_{L=1...|m|, m_L\neq Y}\sum_{M_{m_L,i_L}=m'X_jm''}\sum_{i_1,....,i_{|m|}}\\&\epsilon_{m''M_{m,i,L,+}MM_{m_1,i_1}...M_{m_{L-1},i_{L-1}}m',\rho_{|m''M_{m,i,L,+}|_Y}(\sigma\#^m(\sigma_{m_1,i_1},...,\sigma_{m_{|m|},i_{|m|})})\#_{|m''M_{m,i,L,+}|_Y}\Sigma}
\\&\Big((\rho_{|m''M_{m,i,L,+}|}.\left(P\#(Q_{m_1,i_1},...,Q_{m_{|m|},i_{|m|}})\right))\#_{|m''M_{m,i,L,+}|}d\Big)
.
\end{align*}

Then note the following combinatorial identities. We fix $m=M'X_jM''$ with $|M'|=L-1$, $\overline{m}=M''X_jM'$ 
\begin{align*}&\rho_{|m''M_{m,i,L,+}|_Y}(\sigma\#^m(\sigma_{m_1,i_1},...,\sigma_{m_{|m|},i_{|m|})})\#_{|m''M_{m,i,L,+}|_Y}\Sigma\\&=
\rho_{|m''|_Y}(\sigma_{{m}_{L},i_{L}})\#_{|m''|}\big([\rho_{|M''|_Y}(\sigma)]\#^{\overline{m}}(\sigma_{{m}_{L+1},i_{L+1}},...,\sigma_{m_{|m|},i_{|m|}},\Sigma, \sigma_{m_{1},i_{1}},...,\sigma_{m_{L-1},i_{L-1}})\big),
\end{align*}
and similarly:
\begin{align*}&\Big((\rho_{|m''M_{m,i,L,+}|}.\left(P\#(Q_{m_1,i_1},...,Q_{m_{|m|},i_{|m|}})\right))\#_{|m''M_{m,i,L,+}|}d\Big)=\\&
(\rho_{|m''|}.Q_{m_{L},i_{L})})\#_{|m''|}\left((\rho_{|M''|}(P))\#(Q_{{m}_{L+1},i_{L+1}},...,Q_{m_{|m|},i_{|m|}},d,Q_{m_{1},i_{1}},...,Q_{m_{L-1},i_{L-1}})\right)
.
\end{align*}
An inspection shows that gathering these terms leads to the definition of the right hand side in \eqref{CompositionCyclic} as expected.

\end{proof}
Finally, we will need  a second order operator and its commutation with cyclic gradients.

\begin{Proposition}\label{DeltaAnalytic}
There are continuous maps $\Delta,\delta_\Delta$ on  $B\{ X:E_D,R\}{\to B\{ X:E_D,S\}}$ for $S<R$ {uniquely defined as weak-* continuous map} by the following properties (a) and (b):

(a) For $P\in B\{ X:E_D,R\}$ monomial 
$$\Delta(P)=\sum_{i}m\circ (1\o E_D\o 1)\partial_i\o 1\partial_i(P)$$
and $\Delta E_D=0$

(b)  $\delta_\Delta$ is a derivation, $\delta_\Delta(P)=0$ for any $P\in  B\langle X:D,R\rangle$, and for $Q$ monomial in $B\{ X:E_D,R\}$, $\delta_\Delta(E_D(Q))=E_D((\Delta +\delta_\Delta)(Q)).$

(c)
Moreover, \begin{equation}\label{commutationDDelta}\mathscr{D}_{i}(\Delta+\delta_\Delta)=(\Delta+\delta_\Delta)\mathscr{D}_{i}.\end{equation}

(d) Likewise, for any $V\in B\langle X:D,R\rangle,$ the map $\Delta_V=\Delta+\sum_{i}\partial_i(.)\#\mathscr{D}_iV$ produces a derivation $\delta_V$ such that $\delta_V(P)=0$ for $P\in B\langle X:D,R\rangle$ and for $Q$ monomial in $B\{ X:E_D,R\}$, $\delta_V(E_D(Q))=E_D((\Delta_V +\delta_V)(Q)).$
Moreover, for any $g\in B\langle X:D,R\rangle$: $$\mathscr{D}_{i}(\Delta_V+\delta_V)(g)=(\Delta_V+\delta_V)\mathscr{D}_{i}(g)+\sum_{j=1}^n\mathscr{D}_{i,\mathscr{D}_{j}g}\mathscr{D}_{j}V.$$
\end{Proposition}
\begin{proof}
{Again it suffices to define those $D-D$ bimodular maps on monomials spaces, i.e., at the level of extended Haagerup tensor products. Then the universal property of the direct sum will extend them as weak-* continuous maps as soon as each component map is weak-* continuous. The algebraic relation then determines the maps on the $E_D$ algebra generated by $B,X_1,...,X_n$ and weak-* density of this algebra implies the uniqueness of the weak-* continuous extension.} For $\Delta$ we use the formula  above.  
Let $\sigma\in NC_{2}(2k)$, $m\in M_{2k}(X_1,...,X_n,Y)$.

For $m=nX_in'X_in'', n'\in C(\sigma,m)$ with the notation of the previous proof, we define \begin{align*}Add(\sigma, n,n',n'')&=\left\{\{|n|_Y+1,|n|_Y+|n'|_Y+2\}\}\right.\\&\cup \{\{i+1,j+1\}:\{i,j\}\in \sigma, |n|_Y<i<j\leq |n|_Y+|n'|_Y\}\\&\cup \{\{i,j+2\}:\{i,j\}\in \sigma, i\leq |n|_Y< |n|_Y+|n'|_Y<j\}\\&\cup \{\{i,j\}:\{i,j\}\in \sigma, i<j\leq |n|_Y\}\\&\left.\cup \{\{i+2,j+2\}:\{i,j\}\in \sigma, |n|_Y+|n'|_Y<i<j\right\}\in NC_{2}(2k+2).\end{align*} 
Then we define for a monomial $\epsilon_{m,\sigma}(P)$ :
$$(\Delta+\delta_\Delta)(\epsilon_{m,\sigma}(P))=\sum_{j=1}^n\sum_{m=nX_jn'X_jn'', n'\in C(\sigma,m) }\epsilon_{nYn'Yn'',Add(\sigma, n,n',n'')}(P).$$
All properties but the last equation \eqref{commutationDDelta} are easy. By definition, we have:
\begin{align*}&\mathscr{D}_{i}((\Delta+\delta_\Delta)(\epsilon_{m,\sigma}(P))\\&=\sum_{j=1}^n\sum_{m=nX_jn'X_jn'', n'\in C(\sigma,m) }\sum_{nYn'Yn''=m'X_im''}\epsilon_{m''m',\rho_{|m''|_Y}(Add(\sigma, n,n',n''))}((\rho_{|m''|}.(P)))\end{align*}
The sums can be divided into 3 cases depending on whether $X_i\in n,n',n''$.
Similarly, we have \begin{align*}(\Delta+\delta_\Delta)\mathscr{D}_{i}(\epsilon_{m,\sigma}(P))=\sum_{m=M'X_iM''}&\sum_{j=1}^n\sum_{M''M'=NX_jN'X_jN'', N'\in C(\rho_{|M''|_Y}(\sigma),M''M') }\\&\epsilon_{NYN'YN'',Add(\rho_{|M''|_Y}(\sigma), N,N',N''))}((\rho_{|M''|}.(P)))\end{align*}
and there are also 3 cases depending $X_j$'s are both in $M''$, in $M'$ or one in each. The proof of the equality is combinatorial, we check we have a bijection of the indexing sets of the sum, with equality of the terms summed in each case.

If $X_i\in n$, then $n=oX_io'$ and $m=oX_io'X_jn'X_jn''$ this suggests $M'=o$, $M''=o'X_jn'X_jn''$ corresponding bijectively to a term where both $X_j$'s are in $M''$, $N=o',N'=n',N''=n''M'$, $m'=o$, $m''=o'Yn'Yn''$ so that $m''m'=NYN'YN''$ as expected, $|M''|=|m''|$ implying the same rotation of $P$ and $Add(\rho_{|M''|_Y}(\sigma), N,N',N''))=\rho_{|m''|_Y}(Add(\sigma, n,n',n''))$, as is easily checked with the same condition on $n'=N'$, implying the final equality. The case $X_i\in n''$ is similar corresponding bijectively to the case where both $X_j$'s are in $M'$.

If $X_i\in n',$ $n'=oX_io'$ and $m=nX_joX_io'X_jn'', m'=nYo,m''=o'Yn''$. This suggests, $M'=nX_jo,M''=o'X_jn''$ corresponding bijectively to a term where one $X_j$ is in $M''$ the other in $M'$ with $N=o',N'=n''n,N''=o.$ Since $N''$ is related to a complement  of $n'$, the relations imposed on $n'$, $N'$ are equivalent after rotation. We also have $m''m'=NYN'YN''$ as expected, $|M''|=|m''|$ implying the same rotation of $P$ and $Add(\rho_{|M''|_Y}(\sigma), N,N',N''))=\rho_{|m''|_Y}(Add(\sigma, n,n',n''))$, as is easily checked, implying the final equality.
\end{proof}

\subsection{Non-commutative $C^{k,l}$-functions and their stability properties} 
 
 \subsubsection{$C^{k,l}$ norms.}
As in the main text, we consider {several} variants  $C^{k,l;\epsilon_1,\epsilon_2}_{tr,V}(A,U:B,E_D)$, $\epsilon_1\in\{0,1\},\epsilon_2\in\{-1,0,1,2\}$:
 
 \begin{align}\label{normehorrible}\begin{split}&\|P\|_{C^{k,l;\epsilon_1,\epsilon_2}_{tr,V}(A,U:B,E_D)}=\|\iota(P)\|_{k,l,U}+{\epsilon_1}\|{(\Delta_V+}\delta_V)(P)\|_{C^*_{tr}(A,U)}\\&\qquad + 1_{k\geq \max(\epsilon_2-1,- \epsilon_2)}\sum_{p=0}^{l-1+1_{odd}(|\epsilon_2|)}{\sum_{i=1}^n}\max\Big[\|\mathscr{D}_{i,1}(P)\|_{k,p,U},\\ &\qquad\qquad (0\vee\frac{\epsilon_2}{2})\sup_{\textrm{\tiny$\begin{array}{c}Q\in (C^{k,p}_{tr}(A,U^{m-1}:B,E_D))_1\\ m\geq 2\end{array}$}} \|\mathscr{D}_{i,Q(X')}(P)\|_{k,p,U^m}\Big].\end{split}\end{align}

{  We of course also define a first order part seminorm $\|P\|_{C^{k,l;\epsilon_1,\epsilon_2}_{tr,V}(A,U:B,E_D),\geq 1}$ only replacing the first term in the sum by $\Vert{}\iota(P)\Vert{}_{k,l,U\geq 1}$.
Note that $\|P\|_{C^{k,l;1_{k\geq 2},2}_{tr,V}(A,U:B,E_D)}=\|P\|_{C^{k,l}_{tr,V}(A,U:B,E_D)}$ enables to include our previous case in an ad-hoc way.
{ We may write $C^{k,l{;0,}\epsilon_2}_{tr,V}(A,U:B,E_D)=C^{k,l;0,\epsilon_2}_{tr}(A,U:B,E_D)$ since there is no more dependence in $V$ in this case. \textbf{ Note that we wrote $C^{k,l{;}\epsilon_2}_{tr}(A,U:B,E_D)=C^{k,l{;0,}\epsilon_2}_{tr,V}(A,U:B,E_D)$ for short in the text before the appendices since we only used this case $\epsilon_1=0$.}}

In the last seminorm we considered $P$ in variable $X=(X_1,...,X_n)$ and $Q$ in variable $X'=(X_{(1)}',...,X_{(m-1)}')\in U^{m-1}$ and $U^m\subset A_R^{mn}=(A_R^{n})^m.$
{ In order to get a consistent definition, we still have to check the last term is finite for $P\in B_c\{ X_1,...,X_n;E_D,R,\C\}$. We gather this and a complementary estimate in the following Lemma. A variant explains the inclusion $C^{k,l}_c\subset C^{k,l}_{tr,V,c}$ at the end of subsection \ref{diffopsection} with norm equivalent to the restricted norm (explaining why the completions are included in one another)  
}

\begin{Lemma}\label{Ckc}
{Assume $U\subset A_{R, appB-E_D}^n$.} For any $P\in B_{c}\{ X_1,...,X_n;E_D,R,\C\}$, we have $$\sup_{\textrm{\tiny$\begin{array}{c}Q\in (C^{k,p}_{tr}(A,U^{m-1}:B,E_D))_1\\ m\geq 2\end{array}$}} \|\mathscr{D}_{i,Q(X')}(P)\|_{k,p,U^m}<\infty$$ and moreover if $P\in B_c\langle X_1,...,X_n;D,R,\C\rangle$, for any $p\geq0$ we have:

$$\sup_{\textrm{\tiny$\begin{array}{c}Q\in (C^{k,p}_{tr}(A,U^{m-1}:B,E_D))_1\\ m\geq 2\end{array}$}} \|\mathscr{D}_{i,Q(X')}(P)\|_{k,p,U^m}\leq C\Vert{}P\Vert{}_{k+1,p,U,c{,\geq 1}}$$ so that we have {extensions} of the identity which give injective bounded linear maps: \begin{gather*}C^{k+l}_{c}(A,U:B,D)\to C^{k,l}_{tr,V}(A,U:B,E_D),\\ C^{k+l+1}_{c}(A,U:B,D)\to C^{k,l;0,1}_{tr,V}(A,U:B,E_D),\end{gather*} {and we have for some $C>0$:
$$\|P\|_{C^{k,l}_{tr,V}(A,U:B,E_D),\geq 1}\leq C\Vert{}P\Vert{}_{k+1,l-1,U,c{,\geq 1}}.$$}
\end{Lemma}
\begin{proof}
We can assume $Q\in B_c\{ X'_{(1)1},...,X'_{(1)n},...,X'_{(m)1},...,X'_{(m)n};E_D,R^+,\C\},m\geq 1$ $X'=X'_{(1)1},...,X'_{(1)n},...,X'_{(m)1},...,X'_{(m)n}$. We detail only the second estimate, since the first one mainly needs $P$ monomial and is an easy extension.

To compute differentials we introduce partial differentials $d^{s}_{(X,X')(r_1,...,r_s)}$ so that a full differential is $$\sum_{r\in [1,(m+1)n]^s}d^{s}_{(X,X')(r_1,...,r_s)}\mathscr{D}_{i,Q(X')}(P)(X,X').(H_1^{r_1},...,H_s^{r_s}).$$ 

Recall this $d^{s}_{(X,X')}$ is the full differential so that $d^{s}_{X}$ applied to  $P\in B_c\langle X_1,...,X_n;D,R,\C\rangle$ is a certain expression involving free difference quotients but is not necessarily $0$ (unlike $d^{s}$ by its definition).

We have to compute as easily checked on monomials, for $s,l\leq k-1$ \begin{align}\label{CompoCkltrFormula}\begin{split}d^{s}_{(X,X')(r_1,...,r_s)}&(\partial^{l}_{(j_1,...,j_k)}\mathscr{D}_{i,Q(X')}(P))
\\&=\sum_{0\leq o\leq p\leq l}[d^{\# R}_{X R}(\rho^{-(l-p+1)}.\partial^{(o+l-p+1)}_{(j_{p+1},...,j_l,i,j_1,...,j_o)}(P))]\#d^{\# R'}_{X' R'}\partial^{(p-o)}_{(j_{o+1},...,j_p)}(Q)
\end{split}\end{align}
 where $R=(r_{i_1},...,r_{i_{\# R}})$ with the underlying set $uR=\{r_{i_1},...,r_{i_{\# R}}\}=\{r_i, r_i \in [1,n]\}$, $i_1<...<i_{\# R}$ and $R'=(r_{j_1}-n,...,r_{j_{\# R'}}-n)$ with $\{r_{j_1},...,r_{j_{\# R'}}\}=\{r_1,...,r_s\}-uR$ $j_1<...<j_{\# R'}$ so that $d_{(X,X') (R'+n)}=d_{X' R'}$, and note there is no real sum to split the derivatives between $P,Q$ (the sum can contain only one non-zero term) since the variables of $Q$ and $P$ are not the same.
 
Using this remark and the natural bound on products defined in Proposition \ref{CyclicPermutations}, one gets the term in the seminorm to estimate for a fixed order $s$ of differentials $d^s$:
\begin{align*}&\left(\Vert{}\sum_{r\in [1,(m+1)n]^s}d^{s}_{(X,X')(r_1,...,r_s)}\mathscr{D}_{i,Q(X')}(P)(X,X').(H_1^{r_1},...,H_s^{r_s})\Vert{}_{A}+\sum_{l=1}^{k} \right.
\\&\ \ \ \ \left.\sum_{j\in[1,n(m+1)]^l}\Vert{}\sum_{r\in [1,(m+1)n]^s}d^{s}_{(X,X')(r_1,...,r_s)}\partial^l_{j}\mathscr{D}_{i,Q(X')}(P)(X,X').(H_1^{r_1},...,H_s^{r_s})\Vert{}_{A^{\oehc{D}(l+1)}} \right)
\\&\leq k\sum_{\textrm{\tiny$ \begin{array}{c}V\subset[1,s]\\V=\{i_1,...,i_v\}\\V^c=\{j_1,...,j_{s-v}\}\end{array}$}}\left(\|d^{v}_{X}\partial_i P\|_{A^{\oehc{D}2}}+\sum_{l=1}^{k}\sum_{j\in[1,n]^l} \Vert{}d^{v}_{X}\partial^{l+1}_{(i,j)}(P)(X).(H_{i_1},...,H_{i_{v}})\Vert{}_{A^{\oehc{D}(l+2)}}\right) 
\\&\times\left(\Vert{}d^{s-v}_{X'}Q(X')(H_{j_1},...,H_{j_{s-v}})\Vert{}_{A}+\sum_{l=1}^{k} 
\sum_{j\in[1,nm]^l}\Vert{}d^{s-v}_{X'}\partial^l_{j}(Q)(X').(H_{j_1},...,H_{j_{s-v}})\Vert{}_{A^{\oehc{D}(l+1)}} \right).\end{align*}
The factor $k$ appears for a the same reason as the sum over $V$, because in the sum over $j$ (resp. over $r$) the position of differentials $X$, $X'$ need to be determined by a starting point for the block of  $X'$ variables (resp. a set of $X$ variables) and in the first case the number is less than $l\leq k.$

Thus taking suprema in the definition of seminorms, one gets 
 the concluding result for any $p$:
$$\|\mathscr{D}_{i,Q(X')}(P)\|_{k-1,p,U^{m+1},c}\leq (k-1)2^p \|P\|_{k,p,U,c}\|Q\|_{k-1,p,U^m,c},$$
and similarly $$\|\mathscr{D}_{i,Q(X')}(P)\|_{k-1,p,U^{m+1}}\leq (k-1)2^p \|P\|_{k,p,U,c}\|Q\|_{k-1,p,U^m}.$$

The definition of the two bounded linear maps are then straightforward and injectivity comes from the fact that the bounds enable us to get equivalent norms on the image so that the separation completion defining the first space can be computed in the second.\end{proof}
}

\subsubsection{Composition of functions}
To understand the relationship between the Laplacian and composition of functions we need the following basic remark.  Let $P,Q_1,...,Q_n\in  \cup_{R>0}B_c\{ X_1,...,X_n: E_D,R,\C\}$.  Then: 

\begin{align*}\Delta(P\circ Q)&=\sum_{i,j}m\circ (1\o E_D\o 1)\partial_i\o 1((\partial_jP)\circ Q\#\partial_i(Q_j))\\&=\sum_{i,j}((\partial_jP)\circ Q\#m\circ (1\o E_D\o 1)(\partial_i\o 1\partial_i(Q_j)))\\&+\sum_{i,j,k}m\circ (1\o E_D\o 1)((\partial_k\o 1\partial_jP)\circ Q\#(\partial_i(Q_k),\partial_i(Q_j))\,.\end{align*}

Thus we have a lack of stability of the form of the second order term so that it is natural to introduce  for $P\in  B\{ X_1,...,X_n: E_D,R\}$, ${\mathcal{R}}=({\mathcal{R}}^{kl})=\sum_K(R_{1,K}^{kl}\o R_{2,K}^{kl})_{kl}\in [(D'\cap A^{\oeh D 2})\widehat{\o}(D'\cap A^{\oeh D 2})]^{n^2}$ the following expression:

$$\Delta_{\mathcal{R}}(P)=\sum_{i,j,K}m\circ (1\o E_D\o 1)[\partial_i\o 1\partial_j(P)\#(R_{1,K}^{ij},R_{2,K}^{ij})]\in A\{ X_1,...,X_n: E_D,R\}\,,$$
and similarly $$(\partial(Q)\otimes \partial(Q))\#R)^{kj}=\sum_{K,l,i}[(\partial_i(Q_k)\#(R_{1,K}^{il})]\o[\partial_l(Q_j))\#(R_{2,K}^{il}).]$$
In this way one gets 
\begin{equation}\label{basicsecondorder}\Delta_{\mathcal{R}}(P\circ Q)=(\partial_{\Delta_{\mathcal{R}}(Q)}P)\circ Q+\Delta_{(\partial(Q)\o\partial(Q))\#{\mathcal{R}}}(P)\circ Q\,.\end{equation}

As before we can also define $\delta_{\mathcal{R}}$ as a derivation  $$\delta_{\mathcal{R}}:B\{ X_1,...,X_n: E_D,R\}\to A\{ X_1,...,X_n: E_D,R\}$$ by requiring that it vanishes on $B\langle X_1,...,X_n: D,R\rangle\ni P$
and satisfies
$$\delta_{\mathcal{R}}(P)=0\,,\quad \delta_{\mathcal{R}}(E_D(Q))=E_D((\Delta_{\mathcal{R}} +\delta_{\mathcal{R}})(Q))\,.$$

We consider the variants  $C^{k,l;\epsilon_1,\epsilon_2}_{tr,(2)}(A,U:B,E_D)$, $\epsilon_1\in\{-1,0,1\},\epsilon_2\in\{-1,0,1,2\},o\in[\![0,\max(0,l-2)]\!]$: 
 
 \begin{align*}&\|P\|_{C^{k,l;\epsilon_1,\epsilon_2}_{tr,(2,o)}(A,U:B,E_D)}=\|\iota(P)\|_{k,l,U}+1_{k\geq \max(\epsilon_2-1,- \epsilon_2)}\sum_{p=0}^{l-1+1_{odd}(|\epsilon_2|)}{\sum_{i=1}^n}\\&\quad \max\Big[\|\mathscr{D}_{i,1}(P)\|_{k,p,U} , \\
 &  \qquad\qquad (0\vee\frac{\epsilon_2}{2})\sup_{\textrm{\tiny$\begin{array}{c}Q\in (C^{k,p}_{tr}(A,U^{m-1}:B,E_D))_1\\ m\geq 2\end{array}$}} \|\mathscr{D}_{i,Q(X')}(P)\|_{k,p,U^m}\Big]
 \\&\quad+\max\Big[{(0\vee\epsilon_1)}\sup_{\|{\mathcal{R}^{kl}}\|_{[(D'\cap A^{\oeh D 2})\widehat{\o}(D'\cap A^{\oeh D 2})]}\leq 1}\|(\Delta_{\mathcal{R}} +\delta_{\mathcal{R}})(P)\|_{0,o,U}, \\ & \qquad\qquad {(0\vee(-\epsilon_1))}\|(\Delta +\delta_{\Delta})(P)\|_{0,o,U}\Big].\end{align*}

{
Finally to deal with our universal norms we need to consider in what space of variables our functions are valued to handle  composition properly. For this consider $U\subset A_R^n,V\subset A_S^n$ 
 sets, $S\geq R$ and $C$ a class of functions on $U$ as before or one defined later, $B_C$ the space of analytic function (either $B_{c}\{ X_1,...,X_n;E_D,R^+,\C\}$ for classes with index $tr$ or $B_c\langle X: D,R,\C\rangle$ or $\cap_{T>R}C^{l+1}_b(A_R^n,B_c\langle X_1,...,X_n: D,T,\C\rangle)$ for classes with index $u$ etc.) used to define it as a separation-completion with canonical map $\iota : B_C\to C$. We define two candidates of sets admissible for composition $$Comp(U,V,C)=\{Q=(Q_1,...,Q_n)\in C^n, 
 \forall X\in U, Q(X)\in V\},$$
 $$Comp^-(U,V,C)=Comp(U,V,C)\cap \overline{Comp(U,V,C)\cap (\iota(B_C))^n)}^{C^n},$$
 which are 
 subspaces of $Comp(U,A_S^n,C)$ 
 .}
 We first define composition on the dense subspace of $Q_i \in \cap_{T>R}C^{l+1}_b(U,B_c\langle X: D,T,\C\rangle)$, with $Q(X)\in V$ for all $X\in U$, 
   for $P\in \cap_{T>S}C^{l+1}_b(V,B_c\langle X: D,T,\C\rangle)$ 
    by $$P(Q_1,...,Q_n):X\in U\mapsto P[(Q_1(X),...,Q_n(X))](Q_1[X],....,Q_n[X])$$
     where $P[(Q_1(X),...,Q_n(X))]\in \cap_{T>S} B_c\langle X:E_D,S,\C\rangle$ is then composed with $Q_i[X]\in B_c\langle X:E_D,R\rangle$, since $\|Q_i[X]\|\leq T$ for some $T\geq S$ one can apply the definition of composition at analytic level from Propositions \ref{analytic}, \ref{ExpectationAnalytic}. 
     
    { If $P\in B_{ c}\{ X_1,...,X_n;E_D,S^+,\C\}$, $P$ defines $X\mapsto P(E_{D,X})$ on any $V\subset A_S^n$, so that we can define $P(Q_1,...,Q_n)$ assuming only $\|Q_i(X)\|<S$ (case $V=A_S^n$ above).}

We can now extend these maps. We first deal with the cases of stability by compositions and then deal with the variants we used in the main texts obtained via various compositions with canonical maps.

\begin{Lemma}\label{compositionCklFixed} Fix $V,U$ as above {with $U\subset V$} {(with $V\subset A_{R,UltraApp}^n$ as soon as a space with index $c$ is involved)}.
The above map $(P,Q_1,...,Q_n)\mapsto P(Q_1,...,Q_n)$ extends continuously to 
$Q_1,...,Q_n\in {Comp^-(U,V,}C^{k,l}_{u}(A,U:B,E_D))$ 
to give a map $$C^{k,l}_{u}(A,V:B,E_D)\times {Comp^-(U,V,}(C^{k,l}_{u}(A,U:B,E_D)))\to C^{k,l}_{u}(A,U:B,E_D),$$
 { for $k\geq l$}.
Moreover, { for any $(k,l)\in \N^2$},  it also extends { continuously consistently } to 
\begin{gather*}C^{k,l}_{tr}(A,V:B,E_D)\times {Comp(U,V,}(C^{k,l}_{tr}(A,U:B,E_D)))\to C^{k,l}_{tr}(A,U:B,E_D),\\ C^{k,l}_{tr,c}(A,V:B,E_D)\times {Comp(U,V,}(C^{k,l}_{tr,c}(A,U:B,E_D)))\to C^{k,l}_{tr,c}(A,U:B,E_D),\\ C^{k,l;0,\epsilon_2}_{tr}(A,V:B,E_D)\times {Comp(U,V,}(C^{k,l;0,1\vee \epsilon_2}_{tr}(A,U:B,E_D)))\to C^{k,l;0,\epsilon_2}_{tr}(A,U:B,E_D),\\ C^{k,l;1,\epsilon_2}_{tr,(2,o)}(A,V:B,E_D)\times {Comp(U,V,}(C^{k,l;\epsilon_1,1\vee \epsilon_2}_{tr,(2,o)}(A,U:B,E_D)))\to C^{k,l;\epsilon_1,\epsilon_2}_{tr,(2,o)}(A,U:B,E_D),\end{gather*} $\epsilon_1\in\{-1,1\},\epsilon_2\in\{-1,0,1,2\},o=0$ and with the constraint $k,l\geq 1$ in case $\epsilon_1=1$. 
{ Finally, for $P\in C^{ k,l+1}_{u}(A,V:B,E_D)$  $(Q_1,...,Q_n)\mapsto P(Q_1,...,Q_n)$ is Lipschitz on bounded sets of $Comp^-(U,V,C^{k,l}_{u}(A,U:B,E_D))$ with corresponding statements on all other spaces  in adding to the $P$ variable only $1$ more derivative to $l$ and to $o$. Moreover, the Lipschitz property  is uniform on bounded sets for $P$ in the space it can be taken.} 
\end{Lemma} 
Although the case $o\in [\![1,\max(0,l-1)]\!]$ is not needed in this paper, it can be treated similarly but this is left to the reader.
\begin{proof}

{ Note first that for composition on $Comp^-$
we can extend the first definition of composition since then we have approximate $Q\in \iota(B)^n$ with $Q(X)\in V$. For all extension to $Comp$ we use the second definition since we can start from $P\in B_{c}\{ X_1,...,X_n;E_D,S^+,\C\}$ by density in the corresponding spaces. 
As we will see, we will always extend first in $Q$, and for $P$ fixed as above this extension can be done with $V=A_S^n,$ using $Comp(U,A_S^n,C)=Comp^-(U,A_S^n,C)$ (since $A_S^n$ open and using  compatibility with the topology of considered $C$) and then restrict this first extension to our space $Comp(U,V,C)\subset Comp(U,A_S^n,C).$
}
We have to estimate various norms using  \eqref{CompositionAnalytic} and \eqref{CompositionAnalExpectation} (and its variant which is the elementary differentiation of composition of functions):
\begin{align*}&d^{s}_{X(r_1,...,r_s)}(\partial^{k}_{(j_1,...,j_k)}P(Q_1,...Q_n))=\sum_{l=1}^k\sum_{n_1,...,n_l}\sum_{1\leq i_1<...<i_l=k} \\ &d^{s}_{X(r_1,...,r_s)}[(\partial_{(n_1,...,n_l)}^{l}(P))(Q_1,...Q_n)\#(\partial_{(j_1,...,j_{i_1})}^{i_1}Q_{n_1},\partial_{(j_{i_1+1},...,j_{i_2})}^{i_2-i_1} Q_{n_2},...,\partial_{(j_{i_{l-1}+1},...,j_{k})}^{k-i_{l-1}} Q_{n_l})]
\\&=\sum_{l=1}^k\sum_{n_1,...,n_l}\sum_{1\leq i_1<...<i_l=k}\ \ \sum_{V=\{\{t_{0,1}<...<t_{0,u_0}\},...\{t_{l,1}<...<t_{1,u_l}\}\in Part([1,s])\}} \sum_{m=1}^{u_0}\sum_{o_1,...,o_m}\sum_{1\leq j_1<...<j_m=u_0}
\\&
\sum_{L\in Part([1,u_0]):L_{1,1}=1,L_{.-1,1}<L_{.,1}} \left[[d_{X(o_1,...,o_m)}^{m}(\partial_{(n_1,...,n_l)}^{l}(P)](Q_1,...Q_n)\right.\\&\left.\circ
(d_{X(r_{t_{0,L_{1,1}}},...,r_{t_{0,L_{1,j_1})})}}^{j_1}Q_{o_1},,...,d_{X(r_{t_{0,L_{m,1}}},...,r_{t_{0,L_{m,u_0-j_{m-1}})})}}^{u_0-j_{m-1}} Q_{o_m}))]\right]\\&\#(d^{u_1}_{X(r_{t_{1,1}},...,r_{t_{1,u_1}})}\partial_{(j_1,...,j_{i_1})}^{i_1}Q_{n_1},...,d^{u_l}_{X(r_{t_{l,1}},...,r_{t_{l,u_l}})}\partial_{(j_{i_{l-1}+1},...,j_{k})}^{k-i_{l-1}} Q_{n_l})]
\end{align*}
(the sum over $V$ runs over partitions of $[1,s]$ (not ordered) and the sum over $L=\{\{L_{1,1}<...<L_{1,j_1}\},\cdots,\{L_{m-1,1}<L_{m,1}<...<L_{m,j_m-j_{m-1}}\}\}),L_{.-1,1}<L_{.,1}$ over partitions $Part([1,u_0])$ of $[1,u_0]$ with the extra inequalities written ordering the blocks of the partitions by the index of the smallest element).
Now for $P\in \cap_{T>S}C^{l+1}_b(A_S^n,B_c\langle X: D,T,\C\rangle)$, one checks (using we started from one more derivative on $U$ than necessary, namely $l+1$ instead of $l$) that $(Q_1,...,Q_n)\mapsto P(Q_1,...,Q_n)$ is uniformly continuous (on balls) thus extends by uniform continuity to $Comp^-(U,V,C^{k,l}_{u}(A,U:B,E_D))$.

Obviously, if one does not care about constants, we have from the previous computation, a bound of the form 
$$\|P(Q_1,...,Q_n)\|_{k,l,U}\leq C(k,l,n)\|P\|_{k,l,V}
\left(1+\max_{i=1,...,n}\|Q_i\|_{k,l,U}\right)^{k+l}$$
thus $P\mapsto P(Q_1,...,Q_n)$ is Lipschitz with value in the space continuous functions with supremum norm on $Q_i$ and thus extend to all $P$ in the  space $C^{k,l}_{u}(A,V:B,E_D)$. This concludes to the extension part. 
{ Note that one deduces from the computations above the estimate of independent interest :
\begin{equation}\label{composition1stOrder}\|P(Q_1,...,Q_n)\|_{k,l,U,\geq 1}\leq C(k,l,n)\|P\|_{k,l,V,\geq 1}
\left(1+\max_{i=1,...,n}\|Q_i\|_{k,l,U}\right)^{k+l-1}\max_{i=1,...,n}\|Q_i\|_{k,l,U,\geq 1}\end{equation}

}

For the Lipschitz property, the only problematic term in the expression above is the composition $d^{s}_{X(r_1,...,r_s)}[(\partial_{(o_1,...,o_l)}^{l}(P))(Q_1,...,Q_n).$ We note  that under the supplementary assumption of differentiability for $P$, it is always differentiable with differential   $$\sum_{i}d^{s+1}_{X(r_1,...,r_s,i)}[(\partial_{(n_1,...,n_l)}^{l}(P))(Q_1,...,Q_n)(\cdot,...,\cdot,H_i).$$ The conclusion follows by the fundamental Theorem of calculus. 

 Now, the case of $C^{k,l}_{tr}$ spaces is obvious because $P(Q_1,...,Q_n)$ exactly comes from the composition in Proposition \ref{ExpectationAnalytic} 
  and the discussion at the beginning of the proof to deal with $Comp$. $C^{k,l}_{tr,c}$ is also a variant.

We now turn to the spaces $C^{k,l;0,\epsilon_2}_{tr}$ first with $\epsilon_2=1$. For $P$ fixed analytic, the extension in $Q_i$ is as easy as before {(using the estimate below)}, it remains to prove uniform Lipschitz property  in $P$. Recall the basic formula \eqref{CompositionCyclic} 
 and since in our  case $\mathscr{D}_{Q_i,R}(P)(Q_1,...,Q_n)\in C^{k,l}_{tr}(A,U^{m-1}:B,E_D)$ we have the following bound for $p\leq l$ :
\begin{align}
 \label{eq:cycliccompo}
\begin{split} \sup_{R\in (C^{k,p}_{tr}(A,U^{m-1}:B,E_D))_1}& \|\mathscr{D}_{i,R(X')}(P(Q_1,...,Q_n))\|_{k,p,U^m}\\&\leq \sum_{j=1}^n\sup_{S\in (C^{k,p}_{tr}(A,U^m:B,E_D))_1} \|\mathscr{D}_{i,S(X'')}(Q_j))\|_{k,p,U^{m+1}}\\&\sup_{R\in (C^{k,p}_{tr}(A,U^{m-1}:B,E_D))_1} \|\mathscr{D}_{Q_j,R(X')}(P)(Q_1,...Q_n)\|_{k,p, U^{m}}\end{split}\end{align}
where we of course took the variables $S=\mathscr{D}_{Q_j,R(X')}(P)(Q_1,...Q_n), X''=(X',X)\in U^m,$  and used  $\|\mathscr{D}_{i,\mathscr{D}_{Q_j,R(X')}(P)(Q_1,...Q_n)}(Q_j))\|_{k-1,p,U^{m}}\leq \|\mathscr{D}_{i,S(X'')}(Q_j))\|_{k-1,p,U^{m+1}}.$
And from a variant with parameter 
 of our previous estimates for the change of variable $(Q_1(X),...,Q_n(X),X')$ {(based on the fact that no additional sum related to composition is involved for the variables $X'$ so that the constant $C(k-1,p,n)$ below only involves the number of variables of $X$)}, the last term is bounded by \begin{align*}
 \begin{split}\|\mathscr{D}_{Q_j,R(X')}&(P)(Q_1,...Q_n)\|_{k,p,U^{m}}\\&\leq C(k,p,n)\|\mathscr{D}_{X_j,R(X')}(P)\|_{k,p,V\times U^{m-1}}\left(1+\max_{i=1,...,n}(\|Q_i\|_{k,p,U}) \right)^{p+k}.\end{split}\end{align*}
This gives the expected Lipschitz bound in $P$ (using $U\subset V$ in taking $Q_j(X)=X_j$) for the part with cyclic gradients.  The Lipschitz property  in $Q$ is dealt with as before.

We now consider the case $\epsilon_2=0$.  In this case the norm becomes
 \begin{align*}\begin{split}&\|P\|_{C^{k,l;\epsilon_1,0}_{tr,V}(A,U:B,E_D)}=\|\iota(P)\|_{k,l,U}+{\epsilon_1}\|{}(\Delta_V+\delta_V)(P)\|_{C^*_{tr}(A,U)}+\sum_{p=0}^{l-1}{\sum_{i=1}^n}\|\mathscr{D}_{i,1}(P)\|_{k,p,U}.\end{split}\end{align*}
 and thus we can use the estimate \eqref{eq:cycliccompo} with $R=1$ to conclude.
 
 We now turn to the case $\epsilon_2 = -1$.  In this case the norm becomes
 \begin{align*}\begin{split}&\|P\|_{C^{k,l;\epsilon_1,-1}_{tr,V}(A,U:B,E_D)}=\|\iota(P)\|_{k,l,U}+{\epsilon_1}\|{}(\Delta_V+\delta_V)(P)\|_{C^*_{tr}(A,U)}+ 1_{k\geq 1} \sum_{p=0}^{l}{\sum_{i=1}^n}\|\mathscr{D}_{i,1}(P)\|_{k,p,U}.\end{split}\end{align*}
The term 
 $\sum_{p=0}^{l}{\sum_{i=1}^n}\|\mathscr{D}_{i,1}(P)\|_{k,p,U}$ is controlled by the similar term (with summation up to $l$) in \eqref{normehorrible} which gives the norm of $Q$ (noting that $\epsilon_2 \vee 1 =1$).  The other terms are treated as before. 
 
 Finally, we consider the case $\epsilon_2 = 2$.  This time $1\vee \epsilon_2 =2$ and the summation over $p$ goes up to {$l-1$}; thus we can use essentially the same estimate as in \eqref{eq:cycliccompo} in this case.

It remains to deal with the case $\epsilon_1=1$ with $o=0$.  It is based on \eqref{basicsecondorder}
\begin{align*}(\Delta_{\mathcal{R}}+\delta_{\mathcal{R}})(P\circ Q)&=(d_{Q(X)}P(E_{D,Q(X)}).(\Delta_{\mathcal{R}}+\delta_{\mathcal{R}})(Q))\\&+[\Delta_{(\partial Q\otimes \partial Q)\#\mathcal{R})}+\delta_{(\partial Q\otimes \partial Q)\#\mathcal{R})}](P)(E_{D,Q(X)})(Q(X))\,.\end{align*}
so that one gets :
\begin{align*}&\sup_{\|{\mathcal{R}^{kl}}\|_{[(D'\cap A^{\oeh D 2})\widehat{\o}(D'\cap A^{\oeh D 2})]}\leq 1}\|(\Delta_{\mathcal{R}} +\delta_{\mathcal{R}})(P\circ Q)\|_{C^*_{tr}(A,U)}\\&\leq \sup_{\|{\mathcal{R}^{kl}}\|_{[(D'\cap A^{\oeh D 2})\widehat{\o}(D'\cap A^{\oeh D 2})]}\leq 1}\|(\Delta_{\mathcal{R}} +\delta_{\mathcal{R}})(P)\|_{C^*_{tr}(A,V)}\left(\max_{i=1,...,n}\|Q_i\|_{1,0,U}^2\right)\\&+\sup_{\|{\mathcal{R}^{kl}}\|_{[(D'\cap A^{\oeh D 2})\widehat{\o}(D'\cap A^{\oeh D 2})]}\leq 1}\|(\Delta_{\mathcal{R}} +\delta_{\mathcal{R}})(Q)\|_{C^*_{tr}(A,U)}\|P\|_{1,1,U}.\end{align*}
This enables the extension in $P$ after extension in $Q$ if $k,l\geq 1$ and gives the Lipschitz property  in $Q$ on bounded set as required (using $o$ became $o+1$ for dealing with the annoying new term). The case $\epsilon_1=-1$ is possible because taking $\mathcal{R}^{kl}=1_{k=l}(1\o 1)\o (1\o 1)$ recovers the Laplacian and using a general $\mathcal{R}$ on the $P$ variable enables to deal with the particular case (and remove the sup) for $Q,P\circ Q$ variables.\end{proof}


\begin{Corollary}\label{compositionCkl}In the setting of the previous Lemma (in particular for $U\subset V\subset A_{R,UltraApp}^n$),
for  any $l\geq 1$ (and $k\geq 2$ in any case with $W$) the map 
 $(P,Q_1,...,Q_n)\mapsto P(Q_1,...,Q_n)$ also extends continuously  consistently  
 to $$C^{k+l}_c(A,V:B,D)\times {Comp(U,V,}(C^{k,l}_{tr,W}(A,U:B,E_D)))\to C^{k,l}_{tr,W}(A,U:B,E_D),$$ $$C^{k+l}_c(A,V:B,D)\times {Comp(U,V,}(C^{k,l}_{tr,W,c}(A,U:B,E_D)))\to C^{k,l}_{tr,W,c}(A,U:B,E_D),$$  $$C^{l}_c(A,V:B,D)\times {Comp(U,V,}(C^{l}_{c}(A,U:B,D)))\to C^{l}_{c}(A,U:B,D)$$ 
$$C^{k,l;0,-1}_{tr}(A,V:B,E_D)\times {Comp(U,V,}(C^{k+l+1}_{c}(A,U:B,D)))\to C^{k,l;0,-1}_{tr}(A,U:B,E_D),$$

Similarly as before if we require one more derivative in $P$ in the $l$ variable, one gets  the Lipschitz property  on bounded sets in the space for $Q$. 
\end{Corollary} 

\begin{proof}
This is a consequence of the previous result using the canonical maps :$C^{k+l}_c(A,V:B,D)\to C^{k,l;1,2}_{tr,(2,0)}(A,V:B,E_D),$  $C^{k,l}_{tr,W}(A,U:B,E_D)\to C^{k,l;-1,2}_{tr,(2,0)}(A,V:B,E_D),$ for $k\geq 2,l\geq 1,$
and $C^{k+l+1}_c(A,U:B,D)\to C^{k,l;0,1}_{tr}(A,U:B,E_D),$ from Lemma \ref{Ckc}.

 The last variant for $C^{k}_{c}(A,U:B,D)$ is easy since it is defined as a subspace with equivalent norm with respect to the previous space (with $l=0$) and thus a consequence of stability of analytic functions (without expectation) by composition.\end{proof}

An easy computation shows that \begin{equation}\label{CyclicDifferential}\sum_{i=1}^n\tau(\mathscr{D}_{i,e}(P)(E_{D,X})(X)H_i)=\tau(e[d_XP(E_{D,X}).(H_1,...,H_n)]).\end{equation}

 Note that \eqref{CyclicDifferential} extends for $e=1$ to $C^{k,l;\epsilon_1,\epsilon_2}_{tr,V}(A,U:B,E_D)$, $X\in U$ as soon as $l\geq 1.$

We will need later the following consequence of Proposition \ref{DeltaAnalytic}. 

Let us define the first space introducing a conjugate variable assumption that will be frequently used in the next subsection: $$A_{R,conj}^n=\{X\in A_{R,UltraApp}^n, \partial_i^*(1\o 1)\in W^*(X),i=1,...,n\}.$$ 
\begin{Lemma}\label{cyclic}Let  $U\subset A_{R,conj}^n.$
\begin{enumerate}\item {Let $V\in C^{1}_c(A,A_{R,conj}^n:B,D).$
For $g\in B\{ X_1,...,X_n:E_D,S,\C\}$, $X=(X_1,...,X_n)\in A_{R,conj}^n$, $\xi_i=\partial_i^*1\o 1,\xi=(\xi_1,...,\xi_n)$ the conjugate variables of $X$ relative to $E_D$ in presence of $B$, then $$(\delta_V(g))(E_{X,D})=dg(E_{X,D}).(\xi-\mathscr{D}V(X_1,...,X_n)),$$
and this extends to $g\in C^{k,l}_{tr,V}(A,U), k\geq 2,l\geq 1.$}

\item Let $k\in\{ 0,1,2,3\},V\in C^{{k+1}}_c(A,A_{R,conj}^n:B,D)$. For any $g\in C^{k+2,2}_{tr,V}(A,U:B,E_D)$, we have $h=(\Delta_V+\delta_V)(g)\in C^{k,{0 ;0,-1}}_{tr,V}(A,U{\cap A_{R,conj}^n}:B,E_D)$,{$(\mathscr{D}_{i}g)\in C^{{k+2,1 }}_{tr}(A,U{\cap A_{R,conj}}:B,E_D)$} and we have equality in $C^{{k,0}}_{tr}(A,U{\cap A_{R,conj}}:B,E_D)$:
$$\mathscr{D}_{i}h=(\Delta_V+\delta_V)(\mathscr{D}_{i}g)-\sum_{j=1}^n\mathscr{D}_{i,\mathscr{D}_{j}g}\mathscr{D}_{j}V.$$
\end{enumerate}
\end{Lemma}

\begin{proof}(1) Because of the norm continuity of the various maps, by density, the first assertion needs only to be checked for $V=0$ and $g=P$ a monomial. By the standard form of tensor products in extended Haagerup tensor products \cite[(2.4), (2.5)]{M05}, one can even reduce terms in those tensor products to finite linear combinations of products. Thus it suffices to check this on the algebra generated by $B,X_1,...,X_n$ where this is then an easy consequence of  the definition of conjugate variables. The extension to 
$C^{k,l}_{tr,V}(A,U), k\geq 2,l\geq 1$ is then obvious by norm continuity of the various maps.

(2) We first need to extend  \eqref{derivcyclic} to $V\in C^{k+1}_c(A,A_{R,conj}^n:B,D)$, still for $g=P\in B_c\{ X_1,...,X_n:E_D,R,\C\}$. If  one uses the notation after this formula extending the definition of $\Delta_V+\delta_V$ to these values of $V$ and notes from the formula \eqref{CompositionCyclic} for cyclic gradient of compositions above (extended beyond analytic functions since $[\Delta_{V_0(Z)}+\delta_{V_0(Z)}](P)$ is a non-commutative analytic function with expectation and we can use the composition Lemma as in the proof of Proposition \ref{infgen}), one gets the expected relation:
 \begin{align*}
\mathscr{D}_{X_i}&\left([\Delta_{V_0(Z)}+\delta_{V_0(Z)}](P)(X,\mathscr{D} V(X))\right)=\left(\mathscr{D}_{X_i}[\Delta_{V_0(Z)}+\delta_{V_0(Z)}](P)\right)(X,\mathscr{D} V(X))\\&\qquad\qquad +\sum_{j=1}^n\left(\mathscr{D}_{X_i,\mathscr{D}_{Z_j}([\Delta_{V_0(Z)}+\delta_{V_0(Z)}](P))}\mathscr{D}_{X_j} V(X)\right)\\&=\left([\Delta_{V_0(Z)}+\delta_{V_0(Z)}](\mathscr{D}_{X_i}P)\right)(X,\mathscr{D} V(X))+\sum_{j=1}^n\left(\mathscr{D}_{X_i,\mathscr{D}_{X_j}(P))}\mathscr{D}_{X_j} V(X)\right)
  \end{align*}
where we used \eqref{derivcyclic} for the extra variables $Z$ to get 
\begin{eqnarray*}
\mathscr{D}_{Z_j}([\Delta_{V_0(Z)}+\delta_{V_0(Z)}](P))&=&[\Delta_{V_0(Z)}+\delta_{V_0(Z)}](\mathscr{D}_{Z_j}P)+\sum_{k=1}^n\mathscr{D}_{Z_j,\mathscr{D}_{j}P}\mathscr{D}_{Z_k}V_0(Z)\\
&=&\mathscr{D}_{j}P\end{eqnarray*}
 since $\mathscr{D}_{Z_j}P=0$ and 
similarly $$\mathscr{D}_{X_i}[\Delta_{V_0(Z)}+\delta_{V_0(Z)}](P)=[\Delta_{V_0(Z)}+\delta_{V_0(Z)}]\mathscr{D}_{X_i}(P)$$ since $\mathscr{D}_{X_i}V_0=0$.

It now remains to extend the relation in $P$ to apply it to our $g$.

For the second statement we check that the map $g\mapsto (\Delta_V+\delta_V)(g)$ is bounded for $g$ analytic function with expectation between the  spaces \[\Delta_V+\delta_V:C^{k+2,2}_{tr,V}(A,U:B,E_D)\to C^{k,{0 ;0,-1}}_{tr,V}(A,U{\cap A_{R,conj1}^n}:B,E_D),\] where the identity has just been checked. 
 We need to bound the $k$-th order free difference quotient of $h$ and  $\mathscr{D}h$. We of course use $\mathscr{D}g$ is controlled in $C^{k+2,1}(A,U:B,E_D)$ thus by closability we can apply a $k$-th order free difference quotients to the relation for $\mathscr{D}h$ (using Lemma \ref{Ckc} for the term with second order derivative on $V$). We can also apply a $k$-th order free difference quotient to the formula for $h$, each time using the relation for $\delta_V(g)$ in terms of differential.
The bounds are now easy using for the term $\partial\delta_V$ the identity checked before in (1) in any representation for $\delta_V$ and commutation of $\partial$ and $d$.\end{proof}

\subsection{Free Difference Quotient with value in extended Haagerup tensor products}

We now consider closability properties of the free difference quotient with value in the extended Haagerup tensor product. 



 For later uses we consider variants of the spaces considered in subsection \ref{semigroup}:
$A_{R,conj0}^n= A_{R,UltraApp}^n,A_{R,conj}^n=A_{M,conj1}^n$ with all conjugate variables  relative to $B,E_D$:
$$A_{R,conj(1/2)}^n=\{X\in A_{R,conj0}^n, \partial_i^*(1\o 1)\in L^2(W^*(X)),i=1,...,n\},$$
 $$A_{M,conj2}^n=\{X\in A_{R,conj}^n, \partial_i^*(\partial_i^*(1\o 1)\o 1)\in W^*(X),i=1,...,n\}$$
They are motivated by the various cases in the next Lemma:
\begin{Lemma}\label{conjvar}Let $M=W^*(X_1,...,X_n,B)$ for $(X_1,...,X_n)\in (A,\tau).$
\begin{enumerate}
\item If $(X_1,...,X_n)\in (A,\tau)$ have conjugate variables $(\partial_1^*1\o 1,...,\partial_n^*1\o 1)\in L^2(M,\tau)$ relative to $B,E_D$ 
 then the unbounded densely defined operator $$\partial_i:M\to M\oeh{D} M$$ is weak-* closable with closure $\overline{\partial_i}^{eh}$. 
{ Moreover,  $\partial_i\o_D 1, 1\o_D \partial_i$ are weak-* closable $M\oeh{D} M\to M\oeh{D} M\oeh{D}M,$ and the closures are derivations for the natural multiplication: for  $U\in M\oeh{D} M, V\in D'\cap M\oeh{D} M$, with $U,V\in D(\overline{\partial_i\o_D 1}^{eh})$ (resp. $U,V\in D(\overline{1\o_D\partial_i}^{eh})$) so is $U\#V$ and $$\overline{\partial_i\o_D 1}^{eh}(U\#V)=\overline{\partial_i\o_D 1}^{eh}(U)\#_2V+U\#\overline{\partial_i\o_D 1}^{eh}(V)$$ (resp. $\overline{1\o_D\partial_i}^{eh}(U\#V)=\overline{1\o_D\partial_i}^{eh}(U)\#_1V+U\#\overline{1\o_D\partial_i}^{eh}(V)$).}
\item If $(X_1,...,X_n)\in (A,\tau)$ have conjugate variables $(\partial_1^*1\o 1,...,\partial_n^*1\o 1)\in L^2(M,\tau)$ then $(1\o E_D)\partial_i$ extends to a bounded operator from $M$ to $L^2(M),\tau)$ or from $L^2(M,\tau)$ to $L^1(M,\tau)$. If moreover $(X_1,...,X_n)\in (A,\tau)$ have conjugates variables $(\partial_1^*1\o 1,...,\partial_n^*1\o 1)\in M$ and second order conjugate variables  $(\partial_1^*(1\o \partial_1^*1\o 1),...,\partial_n^*(1\o \partial_n^*1\o 1))\in M$ then  $(1\o E_D)\partial_i$ extends to a bounded operator on $L^2(M,\tau)$.
\item  If $(X_1,...,X_n)\in (A,\tau)$ have conjugate variables $(\partial_1^*1\o 1,...,\partial_n^*1\o 1)\in M$ and second order conjugate variables  $(\partial_1^*(1\o \partial_1^*1\o 1),...,\partial_n^*(1\o \partial_n^*1\o 1))\in M$ then the  unbounded densely defined operator $\partial_{i_1,...,i_k}^k:M\to M^{\oeh{D} (k+1)}$ is weak-* closable with closure $\overline{\partial_{i_1,...,i_k}^k}^{eh}$,
and $\partial_{i_1,...,i_k}^k:L^2(M,\tau)\to L^2(M,\tau)^{\o_D (k+1)}$ is closable with closure $\overline{\partial_{i_1,...,i_k}^k}.$

Moreover, for $k\leq 3$ {(resp. $k\leq 2$)} the conclusions about the $eh$ extension and for $k\leq 2$ { (resp. $k\leq 1$)} for  the $L^2$ extension hold assuming only $(\partial_1^*1\o 1,...,\partial_n^*1\o 1)\in M$ { (resp. $L^2(M)$). 

Finally, if $F\in C^{k,0}_{tr}(A,A_{R,conj(1_{k\geq 1}/2+1_{k\geq 3}/2+1_{k\geq 4})}^n:B,E_D)$ and $\|X_i\|\leq R$ then $F(X)\in D(\overline{\partial_{i_1,...,i_k}^k}^{eh})$ and $\overline{\partial_{i_1,...,i_k}^k}^{eh}(F(X))=[\partial_{i_1,...,i_k}^k(F)](X)$}
\item{ If  $(X_1,...,X_n)\in (A,\tau)$ have conjugate variables $(\partial_1^*1\o 1,...,\partial_n^*1\o 1)\in M$ then  $\partial_i^*$  is a weak-* continuous bounded operator $D(\overline{\partial_i\o_D1}^{eh}\oplus\overline{1\o_D\partial_i}^{eh})\to M$ and if moreover they have second order conjugate variables it extends to a bounded operator $M\oeh{D} {M}\to L^2(M)$.}
\end{enumerate}
\end{Lemma}

\begin{proof}
(1) Using \cite[Prop 14, Th15]{dabsetup}, we have a canonical weak-* continuous completely contractive map $M\oeh{D} M\subset L^2(M)\o_DL^2(M).$ Thus closability follows from closability as a map valued in the Hilbert space $L^2(M)\o_DL^2(M)$. The densely defined adjoint is then given by Voiculescu's formula  $B\langle X_1,...,X_n\rangle \o_D B\langle X_1,...,X_n\rangle,$:
\begin{equation}\label{Dan}\partial_i^*(a\o_D b)=a\partial_i^*(1\o 1)b-(1\o E_D)(\partial_i(a))b- a(E_D\o 1)(\partial_i(b)).\end{equation}
 This shows  the first result. 
The reasoning for $\partial_i\o_D 1, 1\o_D \partial_i$ is similar.
{To check the derivation property it suffices to take bounded nets $U_n\to U, V_\nu\to V$ and to use the weak-* continuity of $.\#_K.$ obtained in Proposition \ref{OnePermutation} from Theorem \ref{finite2}.(2)  in order to take the limit successively in $n,\nu$ of $\overline{\partial_i\o_D 1}^{eh}(U_n\#V_\nu)=\overline{\partial_i\o_D 1}^{eh}(U_n)\#_2V_\nu+U_n\#\overline{\partial_i\o_D 1}^{eh}(V_\nu)$ } 

(2)  The second result is the relative variant of  \cite[Remark 11, Lemma 12]{Dab08}. 

(3) The third result then follows similarly from the first using also the second result. It always suffices to show weak-* closability from $M$ (or $L^2(M)$) with value an $L^2$ tensor product, for which one needs densely defined adjoints with value $L^1(M)$ or $L^2(M)$ respectively.

We detail only the case $k=2,3$. From Voiculescu's formula, for $a,b,c,d\in B\langle X_1,...,X_n\rangle$, one deduces:
\begin{align*}(\partial_{i_1,i_2}^2)^*&(a\o_D b \o_D c)=\partial_{i_1}^*(a\o[b\partial_{i_2}^*(1\o 1)c-bE_D\o1\partial_{i_2}(c)-1\o E_D\partial_{i_2}(b)c])
\\&=[a\partial_{i_1}^*1\o 1-1\o E_D\partial_{i_1}(a)][b\partial_{i_2}^*(1\o 1)c-b(E_D\o1)\partial_{i_2}(c)-1\o E_D\partial_{i_2}(b)c])
\\&-a(E_D\o 1)\partial_{i_1}[b\partial_{i_2}^*(1\o 1)c-b(E_D\o1)\partial_{i_2}(c)-1\o E_D)\partial_{i_2}(b)c]
\end{align*}
where the second line is in $M$ and the third in $L^2(M)$ by the second point as soon as the first order conjugate variables are in $M$ (resp.  both in $L^1(M)$ by the second point as soon as the first order conjugate variables are in $L^2(M)$). This gives the various statements in case $k=2$.

Likewise, we have :
\begin{align*}(\partial_{i_0,i_1,i_2}^3)^*&(a\o_D b \o_D c\o_D d)\\&=[a\partial_{i_1}^*1\o 1-1\o E_D\partial_{i_1}(a)](\partial_{i_1,i_2}^2)^*(b\o c \o d)-a(E_D\o 1\partial_{i_0})(\partial_{i_1,i_2}^2)^*(b\o c \o d)
\end{align*}
and the first term is in $L^2(M),$ the second in $L^1(M)$ by the second point and what we just proved, as soon as the first order conjugate variable are in $M$ (resp. both in $L^2(M)$ if we have first and second conjugate variables in $M$). 

The higher order terms are then similar to this last case when we have both 
first and second conjugate variables in $M$. All the higher adjoints are then valued in $L^2(M)$ on basic tensors from $B\langle X_1,...,X_n\rangle$.

For the compatibility with $C^k$ spaces, the non-commutative analytic functionals are clearly in the domain and the extension by density is straightforward (even with norm instead of weak-* convergence which is used at the analytic function level though).

(4) For  the fourth statement the $M$ valued extension only involves application of canonical maps associated to Haagerup tensor product to mimic  the formula above.  For the second part of the fourth statement, we extend each term of the formula above. First we know that  $a\o_D b\to a\xi_ib$ can be extended to $M\oeh{D} M$ since  $\xi_i\in D'\cap M$ (see e.g. \cite[Lemma 43.(2)]{dabsetup}). 
We  next write down explicit bounds for the last $L^2(M)$ valued extension. From the Cauchy-Schwarz inequality for Hilbert modules  one gets $(\sum_j a_j\xi_ib_j)^*\sum_j a_j\xi_ib_j\leq \|\sum_j a_ja_j^*\|(\sum_jb_j^*\xi_i^*\xi_ib_j)$
so that $$\|\sum_j a_j\xi_ib_j\|_2^2\leq \|\sum_j a_ja_j^*\|\|\xi_i\|_2^2\|\sum_j b_jb_j^*\|,$$ and moreover $$\|\sum_j a_j\xi_ib_j\|^2\leq \|\sum_j a_ja_j^*\|\|\xi_i\|^2\|\sum_j b_j^*b_j\|,$$
Likewise  we get, \begin{eqnarray*}
\|\sum_ja_j(E_D\o 1)\partial_i(b_j)\|_2^2&\leq &\|\sum_j a_ja_j^*\|\sum_j\|(E_D\o 1)(\partial_i(b_j)\|_2^2\\
&\leq &\|\sum_j a_ja_j^*\|\|(E_D\o 1)\partial_i\|^2\sum_j\|b_j\|_2^2\end{eqnarray*}
and replacing $b_j$ by $a_j^*$, $a_j$, by $b_j^*$ :$$\|\sum_j(1\o E_D)(\partial_i(a_j))b_j\|_2^2\leq  \|\sum_j b_j^*b_j\|\|(E_D\o 1)\partial_i\|^2\sum_j\|a_j\|_2^2$$
giving the last claimed extension  (using the canonical expression for elements in the extended Haagerup product in \cite{M05}).
\end{proof}
We finally recall Voiculescu's extension result for free products:

\begin{Lemma}\label{ExtFDQ}
Assume that the conjugate variables to $X_1,\dots,X_n$ exist.  Consider the unique extension $\hat{\partial}_i$ on $B\langle X_1,...,X_n,S_t,t>0\rangle$ of the free difference quotient derivations $\partial_i$ satisfying the Leibniz rule and $\hat{\partial}_i(S_t)=0$. Then  $\hat{\partial}_i^*(1\o 1)=\partial_i^*(1\o 1)$.
\end{Lemma}
Let $U\subset A_{R,conj}^n$, $\mathscr{A}=A*_D(D\otimes W^*(S_t^{(i)},i=1,...,n,t\geq 0))$, and recall that we defined in subsection \ref{freebrownian}: $U_A=\{X\in \mathscr{A}_{R}^n , X\in U \}\subset A_{R,conj}^n$.  
Given any inclusion $i:\mathscr{A}\to A$ set $U_A'=\{X\in \mathscr{A}_{R}^n , i(X)\in U \}$.  If $U$ is invariant under 
trace preserving isomorphisms (as will be the case for us), the space $U_A'$ does not depend on the choice of the inclusion $i$.

{For all spaces with cyclic variants here, $\mathscr{A}^{\oeh{D}n}$ is replaced by $\mathscr{M}^{\oehc{D}n}$, with $\mathscr{M}=W^*(B,X_1,...,X_n, S_{t}, t>0)$ so that Proposition \ref{TensorCyclicEvaluation} can be applied to all the variables $X_1,...,X_n, S_{t}$.}

\subsection{Conditional expectations and $C^{k,l}$ functions} 
Recall the spaces $C^{k,l}_{tr,V}(A,U:\mathscr{B},D:\mathscr{S}),$
$C^{k,l}_{tr}(A,U:\mathscr{B},D:\mathscr{S}_{\geq u}),$ etc. from subsection \ref{freebrownian}. They are convenient spaces to define semigroups thanks to the following result. The composition maps are variants of the previous subsection and the new conditional expectations are based of the behaviour for extended Haagerup products of free difference quotients of our previous Lemma \ref{conjvar}.

\begin{Proposition}\label{ExpectationCkl}
\begin{enumerate}
\item Let $k,l$ ($k\geq l$ when required in the definition of the space) and $U\subset A_{R,conj0}^n$ (resp. { $U\subset A_{R,conj(1/2)}^n$ if $k\geq 1$, resp. $U\subset A_{R,conj1}^n$, if $k\geq 3$} resp. $U\subset A_{R,conj2}^n$, if $k\geq 4$) . Then 
 $E_B:\mathscr{B}=B*_D(D\otimes W^*(S_t^{(i)},i=1,...,n,t\geq 0))\to B$ gives rise to contractions $$E_0:(C^{k,l}_{tr}(A,U:\mathscr{B},E_D:\mathscr{S}), \|.\|_{k,l,U}) \to (C^{k,l}_{tr}(A,U:B,E_D),\|.\|_{k,l,U}),$$ 
 $$E_0:C^{k,l;\epsilon_1,\epsilon_2}_{tr,V}(A,U:\mathscr{B},E_D:\mathscr{S})\to C^{k,l;\epsilon_1,\epsilon_2}_{tr,V}(A,U:B,E_D),$$ $$E_0:C^{k,l}_{tr,V}(A,U:\mathscr{B},E_D:\mathscr{S})\to C^{k,l}_{tr,V}(A,U:B,E_D), \quad k\geq 2$$ { and likewise for cyclic variants : $C^{k,l}_{tr,c}(A,U:\mathscr{B},E_D:\mathscr{S})\to C^{k,l}_{tr,c}(A,U:B,E_D)$, $C^{k,l;\epsilon_1,\epsilon_2}_{tr,V,c}(A,U:\mathscr{B},E_D:\mathscr{S})\to C^{k,l;\epsilon_1,\epsilon_2}_{tr,V,c}(A,U:B,E_D)$, $C^{k,l}_{tr,V,c}(A,U:\mathscr{B},E_D:\mathscr{S})\to C^{k,l}_{tr,V,c}(A,U:B,E_D)$}. { They are also contractions for the seminorms $\|.\|_{k,l,U\geq 1}$ and $\|.\|_{C^{k,l}_{tr,V}(A,U:\mathscr{B},E_D:\mathscr{S}),\geq 1}.$}

{We also have similarly for $u>0$ $$E_u:C^{k,l}_{tr}(A,U:\mathscr{B},E_D:\mathscr{S})\to C^{k,l}_{tr}(A,U:\mathscr{B},E_D:\mathscr{S}_u),$$  $$E_u:C^{k,l;\epsilon_1,\epsilon_2}_{tr,V}(A,U:\mathscr{B},E_D:\mathscr{S})\to C^{k,l;\epsilon_1,\epsilon_2}_{tr,V}(A,U:\mathscr{B},E_D:\mathscr{S}_u),$$ $$E_u:C^{k,l}_{tr,V}(A,U:\mathscr{B},E_D:\mathscr{S})\to C^{k,l}_{tr,V}(A,U:\mathscr{B},E_D:\mathscr{S}_u),$$ such that  $E_0\circ E_u=E_0$ }{ and $E_0= E_u\circ \theta_u'= E_0\circ \theta_u'.$
}

\item Moreover, the extension result of Corollary \ref{compositionCkl} is also valid for{ any {$U\subset A_{R,conj0}^n,V\subset A_{S,conj0}^n$} giving composition maps $\circ$}:
\begin{align*}\circ:C^{k+l}_c(A,V:B,D)&\times {Comp(U_A,V_A',}C^{k,l}_{tr,W}(A,U:\mathscr{B},E_D:\mathscr{S}))\\ &\to C^{k,l}_{tr,W}(A,U:\mathscr{B},E_D:\mathscr{S}),\end{align*}
\begin{align*}\circ:C^{k,l}_{tr}(A,V:B,D)&\times {Comp(U_A,V_A',}C^{k,l}_{tr}(A,U:\mathscr{B},E_D:\mathscr{S}))\\ &\to C^{k,l}_{tr}(A,U:\mathscr{B},E_D:\mathscr{S})\end{align*}({here and in the next also for $(k,l)\in \N^2$}),
\begin{align*}\circ:C^{k,l;0,\epsilon_2}_{tr}(A,V:B,D)&\times {Comp(U_A,V_A',}C^{k,l;0,1}_{tr}(A,U:\mathscr{B},E_D:\mathscr{S}))\\ & \to C^{k,l;0,\epsilon_2}_{tr}(A,U:\mathscr{B},E_D:\mathscr{S}),\end{align*}  
\begin{align*}\circ:C^{k}_c(A,V:B,D)&\times {Comp(U_A,V_A',}C^{k}_{c}(A,U:\mathscr{B},D:\mathscr{S}))\\ &\to C^{k}_{c}(A,U:\mathscr{B},D:\mathscr{S}),
\end{align*} and as in Lemma \ref{Ckc} a map $\iota':C^{k+l}_{c}(A,U:\mathscr{B},D:\mathscr{S}))\to C^{k,l}_{tr,W,c}(A,U:\mathscr{B},E_D:\mathscr{S}),$
{$\iota':C^{k+l+1}_{c}(A,U:\mathscr{B},D:\mathscr{S}))\to C^{k,l,\epsilon_1,\epsilon_2}_{tr,W,c}(A,U:\mathscr{B},E_D:\mathscr{S}).$}
 We also have $(.)\circ\theta_u'(.)=\theta_u'((.)\circ (.))$ on the above spaces.

\item Finally, we also have a similar composition map $\circ_u$ for any $u>0$ { for $(k,l)\in \N^2$~:} 
$$C^{k,l}_{tr}(A,V:\mathscr{B},E_D:\mathscr{S}_{\geq u}))\times {Comp(U_A,V_A',}C^{k,l}_{tr}(A,U:\mathscr{B},E_D:\mathscr{S}_u))\qquad\qquad$$
$$\qquad\qquad \to C^{k,l}_{tr}(A,U:\mathscr{B},E_D:\mathscr{S}),$$

{ $$C^{k,l;0,\epsilon_2}_{tr}(A,V:\mathscr{B},E_D:\mathscr{S}_{\geq u}))\times {Comp(U,V,}(C^{k,l;0,1}_{tr}(A,U:\mathscr{B},E_D:\mathscr{S}_u))$$}
$$\qquad\qquad \to C^{k,l;0,\epsilon_2}_{tr}(A,U:\mathscr{B},E_D:\mathscr{S})),$$ $\epsilon_2\in\{-1,0,1\}$ and we have : $(.)\circ[(.)\circ_u(.)]=[(.)\circ(.)]\circ_u(.)$ and $E_B(.)\circ (.)=E_u(\theta_u'(.)\circ_u(.)):C^{k,l}_{tr}(A,U:\mathscr{B},E_D:\mathscr{S}))\times {Comp(U_A,V_A',}C^{k,l}_{tr}(A,U:\mathscr{B},E_D:\mathscr{S}_u))\to C^{k,l}_{tr}(A,U:\mathscr{B},E_D:\mathscr{S}_u).$ 
\end{enumerate}
\end{Proposition}
\begin{proof}
By density, it suffices to prove contractivity restricting  to the polynomial variant of the space  $C^0_{b,tr}(U,B\langle X_1,...,X_n:D,R\rangle)$. But if $P$ is in the partial evaluation $\eta_S(B_c\{X_1,...,X_n,S_{t_1},...,S_{t_m}-S_{t_{m-1}}:\mathscr{B},E_D,\max[R,\max_{i=2,n} 2(t_i-t_{i-1})]\C\})\}$, it is easy to see by definition of free semicircular variables with amalgamation that 

\noindent
$ E_A(P(E_{D,X}))=Q(E_{D,X})$ for some $Q\in B\{X_1,...,X_n:{B},E_D,R\}.$ $Q$ is the same as $P$ where brownian variables are replaced by sums over formal conditional expectations.

 More precisely, let $P=\epsilon_{m,\sigma}(P')$, for $$m\in M_{2k}(X_1,...,X_n,Z_1=S_{t_1},...,Z_m=S_{t_m}-S_{t_{m-1}},Y), \sigma\in NC_2(2k)$$  with $P'\in B^{\oehc{D}|m|+1},$ a typical monomial in the direct sum for analytic functions with expectation in the component indexed by $(m,\sigma)$. Recall that $Y$ variables and the pairing $\sigma$ indicate the position of conditional expectations. Let $\pi_m: NC_2(2k+|m|_Z)\to NC(2k)$ the restriction to the indices of $Y$ variables in the monomial $m(X_1=1,...,X_n=1,Z_1,...,Z_m,Y)$ and $\pi_{m,i,k}: NC_2(2k+|m|_Z)\to NC(|m|_{(Z_i)^{(k)}})$, $i=1,...,m,k=1,...,n$ the restriction to indices of the variables in position $(Z_i)^{(k)}.$ Note that this is valued in pair partitions when $(Z_i)^{(k)}$ variables are only paired within themselves.

Then, the conditional expectation is obtained by replacing with pairings and conditional expectations the brownian variables in an appropriate way so that  we define with for convenience $t_0=0$ :
$$E_B(P)=Q:=\sum_{\textrm{\tiny$\begin{array}{c}\pi\in NC_2(2k+|m|_Z)\\ \pi_m(\pi)=\sigma\\\pi_{m,i,k}(\pi)\in NC_2(|m|_{(Z_i)^{(k)}})\end{array}$}}\epsilon_{m(X_1,...,X_n,Z_1=Y,...Z_m=Y,Y),\pi}(P')\prod_{i=1}^m(t_i-t_{i-1})^{|m|_{Z_i}/2},$$
so that the relation above $E_A(P(E_{D,X})(X))=(E_B(P))(E_{D,X})(X), X\in A_R^n$ is easy to check by definition of free Brownian motions.
Note that \begin{equation}\label{DeltaExpectation}(\Delta+\delta_\Delta)(E_0(P))=E_0((\Delta+\delta_\Delta)(P))\end{equation} (where of course $\Delta$ only applies on $X_i$ variables) since, using the definition in the proof of Proposition \ref{DeltaAnalytic}, both expressions correspond to having a supplementary sum over pairs of $X_i$ variables giving a partition not  crossing the previous ones and replaced by a formal $E.$

Using relation \eqref{CyclicDifferential} with $e,H_i$ in the smaller algebra $A$, one sees that for $e\in A$, \begin{equation}\label{CyclicGradExpectation}E_A[\mathscr{D}_{i,e}(P)(E_{D,X}(X)]=\mathscr{D}_{i,e}(Q)(E_{D,X}(X))\end{equation} and we can extend this directly to the cyclic gradient of Proposition \ref{CalcDiffAnalytic}. For $e\in B\{X_1,...,X_n:{B},E_D,R\}$ we have \begin{equation}\label{CyclicGradExpectation2}E_0[\mathscr{D}_{i,e}(P)]=\mathscr{D}_{i,e}(E_0(P)).\end{equation} Indeed, for $e,P$ monomials, since $e$ has no dependence in $S_t$'s, there is a bijection between pairs of $S_t$'s appearing in each monomial after and before applying $\mathscr{D}_{i,e}$. Since cyclic permutations keep non-crossing partitions  the result is thus an easy combinatorial rewriting.

 It thus remains to check contractivity estimates to extend $E_0$ to spaces of $C^k$ functions. 

For $X\in U$, $P$ as before $\partial^l_i(Q)(E_{D,X})(X)=\overline{\partial^l_i}^{eh}[Q(E_{D,X})(X)]$ by Lemma \ref{conjvar} {(we only use it when $k\geq 1$, the various conditions on $U$ also when $k\geq 4$ comes from this application)}, and by duality from Lemma \ref{ExtFDQ}, one gets it equals to $\overline{\partial^l_i}^{eh}[E_A(P(E_{D,X})(X))]=(E_A^{\o_{eh} l+1})(\overline{\partial^l_i}^{eh}[(P(E_{D,X})(X))])$ and thus one gets by functoriality of Haagerup tensor product:
$$\|\partial^l_i(E_0(P))(X)\|_{A^{\o_{ehD (l+1)}}}\leq \|(\partial^l_i(P))(X)\|_{A^{\o_{ehD (l+1)}}}.$$

{ Here it is crucial to note that for all cyclic variants that by Proposition \ref{CyclicPermutations}.(3) if $\overline{\partial^l_i}^{eh}[(P(E_{D,X})(X))]$ is in a cyclic extended Haagerup tensor product, it remains there after application of $(E_A^{\o_{eh} l+1})$.}

Likewise, the full differential commute with conditional expectation (which is a linear bounded map, we thus get the bound for all parts of the seminorm involving free difference quotients and full differentials. { We thus proved  contractivity on $C^{k,l}_{tr}$-spaces.

Since $(\Delta_V+\delta_V)(P)=(\Delta_0+\delta_\Delta)(P)+d_XP.(\mathscr{D}_1V,...,\mathscr{D}_nV)$ the previous results give $(\Delta_V+\delta_V)(E_B(P))=E_B((\Delta_V+\delta_V)(P))$ so that since in this case $k\geq 2$, the choice of the seminorm chosen with this term is compatible with contractivity.
The contractivity of the term with cyclic gradients is also easy with the previous established commutation relation, so that one gets the stated contractivity on $C^{k,l}_{tr,V}$-spaces. Obtaining multiplication maps is as easy as before in this context and by arguments of stability of subspaces  for $C^{k}_{c}$-spaces.

The variant $E_u$ and its relations are obvious.}
\end{proof}

 \subsection{Regular Change of variables for Conjugate variables}
 The computation of conjugate variables along change of variables we used to identify conjugate variables of our transport maps are explained in the next Lemma \ref{adjointAppendix} with the differentiation along a path of such change of variables.

Let $M=W^*(X_1,...,X_n,B)$ for $(X_1,...,X_n)\in (A,\tau).$ We will soon assume those variables have enough conjugate variables relative to $D$ in presence of $B$.

\begin{Lemma}\label{adjointAppendix}
Assume $ W^*(B,X_1,...,X_n)=M$  is such that $X\mapsto \langle e_D, X\#e_D\rangle$ is a trace on $D'\cap M\oeh{D} M$.

Let $(X_1,...,X_n)\in U'\subset A_{S,conj}^n,S>0$ and thus have conjugates variables $(\partial_1^*1\o 1,...,\partial_n^*1\o 1)\in M^n$ 
 relative to $B,E_D$. Take $F=F^*\in (C_{tr,c}^{k,l}(A,U'))^n$, with $k\geq 2$. 
   
   Then $(Y_1,...,Y_n)=F(X_1,...,X_n)$ have conjugate variables in $M$ 
  as soon as $\Vert{}1-\mathscr{J} F\Vert{}_{M_n(M\oehc{D} M)}<1,$ with $(\mathscr{J} F)_{ij}=\partial_{j,X}Y_i.$
   Moreover, we have, setting $C(F)=\frac{1}{1-\Vert{}1-\mathscr{J} F\Vert{}_{M_n(M\oehc{D} M)}}$  :
\begin{align*}&\| \partial_{j,Y}^*1\o 1 \|\leq  C(F)\|\partial_j^*1\o 1\|+C(F)^2\left(\sum_{k\neq j}\|\sigma[(\mathscr{J} F)_{kj}]\|_{M\otimes_{ehD} M}\right)\sum_{k\neq j}\|\partial_k^*1\o 1\|\\&+C(F)^2\sum_{\textrm{\tiny$\begin{array}{c}k,l,m\in[1,n]\\(\epsilon,\eta)\in\{(1,0),(0,1)\}\end{array}$}}
\|\overline{1^{\o_D\epsilon}\o_D\partial_k\o_D 1^{\o_D\eta}}^{eh}(\sigma[(\mathscr{J} F)_{lm}])\|_{M^{\oeh{D}3}}.\end{align*}
   
\end{Lemma}
\begin{proof}

This proof is a variant relative to $D$ of Lemma 3.1 in \cite{alice-shlyakhtenko-transport}.

Take $P\in B\langle X_1,...,X_n,D,R,\C\rangle$, $R\geq\max(S,\sup_{X\in U'}\|F_i(X)\|)$, then $P(Y)$ satisfies the natural extension of formula \eqref{CompoCkltrFormula} from the proof of  Lemma \ref{compositionCklFixed} and so we get the equation in $D'\cap M\oeh{D}M$: $$\partial_{i,X}P(Y)=\sum_{j=1}^n(\partial_{j}(P))(Y)\#\partial_{i,X}Y_j.$$
 Note that from the assumption on $(\mathscr{J} F)_{ij}=\partial_{j,X}Y_i$, one deduces that  $\mathscr{J} F$ is 
invertible in $M_n(M\oehc{D} M)$ so that one gets:$$\partial_{i}(P)(Y)=\sum_{j=1}^n\partial_{j,X}P(Y)\#[(\mathscr{J} F)^{-1}]_{ji}.$$

Thus applying the weak-* continuity of Theorem \ref{finite2}.(2) to introduce $E_{D'}$ (and then remove it in the next-to-last line), the assumed traciality and applying 
\eqref{ScalarProductTrace} to $X=[(\mathscr{J} F)^{-1}]_{ji}^*,Y=E_{D'}(\partial_{j,X}P(Y))$, we get: 
\begin{align*}\langle e_D,(\partial_{i}(P))(Y)\#e_D\rangle&=\sum_{j=1}^n\langle e_D,E_{D'}(\partial_{j,X}P(Y))\#[(\mathscr{J} F)^{-1}]_{ji}\#e_D\rangle
\\&=\sum_{j=1}^n\tau(([(\mathscr{J} F)^{-1}]_{ji}^*)^*E_{D'}[\partial_{j,X}P(Y)])
\\&=\sum_{j=1}^n\langle[(\mathscr{J} F)^{-1}]_{ji}^*)\#e_D,E_{D'}[\partial_{j,X}P(Y)]\#e_D\rangle
\\&=\sum_{j=1}^n\langle([(\mathscr{J} F)^{-1}]_{ji}^*)\#e_D,\overline{\partial_{j,X}}^{L^2}(P(Y))\rangle.
\end{align*}
Thus if we check that $([(\mathscr{J} F)^{-1}]_{ji}^*)\#e_D\in D(\overline{\partial_{j,X}}^*)$
 we  will deduce the existence of the conjugate variable and the equality $$\partial_{i,Y}^*(1\o_D1)=\sum_{j=1}^n{\partial_{j,X}}^*[([(\mathscr{J} F)^{-1}]_{ji}^*)\#e_D].$$

Note that in any representation with $X\in U$ as above  $$\overline{\partial_{i,X}}^{eh}F_j(X)=[(\partial_{i,X}(\iota(F))](X)=[(\partial_{i,X}(\iota(F)^*)](X)=(\partial_{i,X}(\iota(F))](X)^\star=[\overline{\partial_{i,X}}^{eh}F_j(X)]^\star$$
where the last $\star$ is the one of $M\oehc{D}M,$ and one uses natural properties of evaluation extended using the one on polynomials since $X\in A^n_{S,UltraApp}.$

Now, since ($\mathscr{J} F)_{ij}\in M\oehc{D}M$ one can note that $(\sigma(\mathscr{J} F)_{ij})$ is well defined in $D'\cap M\oeh{D}M$ and $(\mathscr{J} F)_{ij}^*=\sigma[(\mathscr{J} F)_{ij}^\star]=[\sigma((\mathscr{J} F)_{ij})](X)$ and thus from Lemma \ref{conjvar} (3), the assumption
 $\sigma((\mathscr{J} F)_{ij})\in D(\overline{\partial_k\o 1}^{eh}\oplus \overline{1\o \partial_k}^{eh})$, Neumann series and from the derivation property in (1) of  the same Lemma,  so does $(\sigma(\mathscr{J} F))^{-1}_{ij}$ and for instance, one gets as expected $$\overline{\partial_k\o 1}^{eh}[(\sigma(\mathscr{J} F))^{-1}_{ij}]=-\sum_{l,m}(\sigma(\mathscr{J} F)^{-1}_{il}))\#[\overline{\partial_k\o 1}^{eh}((\sigma(\mathscr{J} F))_{lm})]\#_2(\sigma(\mathscr{J} F))^{-1}_{mj}).$$

Thus from part (4) of the same Lemma, one gets that $\partial_{i,Y}^*(1\o_D1)$ exists and is in $M$ and the expected bound easily follows from the proof of this statement giving the appropriate extension of Voiculescu's formula. Only note that for $j\neq i$ \begin{align*}&[([\sigma(\mathscr{J} F)^{-1}]_{ji})]\#{\partial_{j,X}}^*(1\o 1)\\&=\sum_{N=1}^\infty \sum_{n=0}^{N-1}\sum_{k\neq i} [(\sigma(\mathscr{J} F-1)^{N-n-1})_{jk}]\#[\sigma(\mathscr{J} F)]_{ki})]\#[\sigma(\mathscr{J} F-1)_{ii}]^n\#{\partial_{j,X}}^*(1\o 1).\end{align*}

\end{proof}

\subsection{Various continuity properties}

\medskip

We start by checking the continuity in $\alpha$ of our various maps. Recall that $A^n_{R/3}\subset A^n_{R,\alpha}$ independently of $\alpha\in [0,1].$

\begin{Lemma}
If we assume the Assumption of Lemma \ref{EstimXt} with $V,W\in C^{k+2}_{c}(A,2R:\mathscr{B},D), U\subset A^n_{R/3}$, then  $X:\alpha\mapsto X_t(\alpha)$ is continuous on  $[0,1]$ with value  $C^0([0,T],C^{k}_{c}(A,U:\mathscr{B},D:\mathscr{S}))$.
\end{Lemma}

\begin{proof}
For the continuity in $\alpha$ of $X$, we have :
\begin{align*} 
&X_t(\alpha)-X_t(\alpha') =-\frac{1}{2}\int_0^t du [\mathscr{D} V_\alpha-\mathscr{D} V_{\alpha'}](X_u(\alpha'))\\& -\frac{1}{2}\int_0^t du\left[\int_0^1d\beta\partial\mathscr{D} V_\alpha (\beta X_u(\alpha)+(1-\beta)X_u(\alpha'))\right]\# [X_u(\alpha)-X_u(\alpha')]
\end{align*}
Using the argument in Lemma \ref{lemhyp} with $\partial\mathscr{D} V_\alpha(X_u)$ replaced by  $$\left[\int_0^1d\beta\partial\mathscr{D} V_\alpha (\beta X_u(\alpha)+(1-\beta)X_u(\alpha'))\right]\geq c Id,$$ with the positivity coming since our notion of positivity is a closed convex cone, one gets:
\begin{align*} 
&\|X_t(\alpha)-X_t(\alpha')\| \leq e^{-ct/2}\int_0^t due^{cu/2} (\sum_i\|[\mathscr{D}_i V_\alpha-\mathscr{D}_i V_{\alpha'}](X_u(\alpha'))\|^2)^{1/2}
\end{align*}
This converges uniformly on $[0,T]$ to $0$ when $\alpha\to \alpha'$ using the corresponding continuity of $V_\alpha.$

Similarly, one gets bounds inductively using \eqref{HigherProcess} in decomposing the higher order term \begin{align*}\partial_j\mathscr{D}_i V_\alpha (X_u(\alpha))& \# (\partial^{k}_{(j_1,...,j_k)}X_u^{(j)}(\alpha))-\partial_j\mathscr{D}_i V_{\alpha'} (X_u(\alpha')) \# (\partial^{k}_{(j_1,...,j_k)}X_u^{(j)}(\alpha'))\\&=(\partial_j\mathscr{D}_i V_\alpha-\partial_j\mathscr{D}_i V_{\alpha'}) (X_u(\alpha)) \# (\partial^{k}_{(j_1,...,j_k)}X_u^{(j)}(\alpha))\\&+[\partial_j\mathscr{D}_i V_{\alpha'} (X_u(\alpha))-\partial_j\mathscr{D}_i V_{\alpha'} (X_u(\alpha'))] \# (\partial^{k}_{(j_1,...,j_k)}X_u^{(j)}(\alpha))\\&+(\partial_j\mathscr{D}_i V_{\alpha'} (X_u(\alpha')) \# (\partial^{k}_{(j_1,...,j_k)}X_u^{(j)}(\alpha))-(\partial^{k}_{(j_1,...,j_k)}X_u^{(j)}(\alpha')))\end{align*}

The last line is treated by  Lemma \ref{lemhyp}, the first line and lower order terms tend to zero uniformly on compact by continuity of $\alpha \mapsto V_\alpha$ or inductively, in the second line (and corresponding terms for lower order terms) , $V_\alpha$ is approximated (uniformly in $\alpha$) by analytic functions to get a Lipschitz function, and use the previous bound on $\|X_t(\alpha)-X_t(\alpha')\|$. Note that the Lipschitz property  could have been treated by explicit bounds on derivatives except for the lowest order term having highest derivative in $V$, namely $k+2$, for which it is crucial that our definition of $C^{k+2}_c$ imply a uniform continuity of the highest derivative via uniform approximation by analytic functions as explained. This concludes the uniform convergence statement in $\alpha$.
\end{proof} 
We also obtain the corresponding result for semigroups. 
 \begin{Lemma}\label{ContSemigAlpha}
If we assume the Assumption of Proposition \ref{semig} with $V,W\in C_c^{k+l+2}(A,2R:B,D)$ ($k\in\{2,3\}, l\geq 1$) then for every $T>0$ each $P\in C_c^{k+l}(A,A_{R,conj }^n:B,E_D)$, $\varphi^{.\prime}(P):\alpha\mapsto \varphi^{\alpha\prime}(P)$ is continuous on  $[0,1]$ with value  $$C^0([0,T],C_{tr,V_0}^{k,l}(A,A_{R/3,conj }^n:B,E_D)).$$
\end{Lemma}

\begin{proof}Recall that $C_{tr,V_0}^{k,l}(A,A_{R/3,conj}^n:B,E_D)=C_{tr,V_\alpha}^{k,l}(A,A_{R/3,conj}^n:B,E_D)$ with equivalent norms for $k\geq 2, l\geq 1$. The result follows by composing the composition map and expectations of Proposition \ref{ExpectationCkl} with our previous Lemma since for $X\in A_{R/3,conj}^n$, $X_t(X)\in A_{R,conj }^n$ for all $t$ so that the composition condition is satisfied.
 \end{proof} 
 
\subsection{Conjugate variables along free SDE's} 
 
The following result is an adaptation in free probability of (a special case of) Lemma 4.2 in \cite{RoTh02}, except that we have to use Ito Formula for the proof instead of Girsanov Theorem, not (yet) available in free probability. This is also an extension to our new classes of $C^2$ functions of a result first explained by the first author in \cite{Dab09}.

\begin{Proposition}\label{boundConj}
Assume the Assumption of Proposition \ref{ConvSDE}(a) with $V\in C^{4}_c(A,R:B,D).$  Assume moreover that, for $M=W^*(B,X_0)*_D(D\otimes W^*(S_t,t>0)$,  $\tau=\langle e_D, .\#e_D\rangle$ is a trace on $D'\cap M\oeh{D} M$ as in the conclusion of Theorem \ref{Finite3}.(3) and Proposition \ref{CyclicPermutations}.(2).

Consider on [0,T] the unique solution obtained there: 
$$X_t(X_0)=X_0+ S_t -\frac{1}{2} \int_0^t \mathscr{D} V (X_u(X_0)) du$$

Then $X_t^1,...,X_t^n$ have bounded conjugate variables in presence of $B$ relative $E_D$, and the corresponding $i$-th conjugate variable is given by $$\xi_s^i=\frac{1}{s}E_{W^*(B,X_s^1,...,X_s^n)}\left(X_s^i-X_0^i-\int_0^sdt \ \frac{t}{2}F_{\mathscr{D}_i V }(X_t^1,...,X_t^n)\right)+\frac{1}{2}\mathscr{D}_i V(X_s^1,...,X_s^n),$$

where for $W\in C^{2}_c(A,R:B,D)$ we defined:$F_W(X)=\frac{1}{2}\Delta_V(W)(X).$
\end{Proposition}

\begin{proof}
\begin{step} Obtaining a differential equation from Ito formula.\end{step}

We have to prove that $\tau(\langle 1\o_D 1,\partial P(X_t^1,...,X_t^n)\rangle)=\tau(\xi_t^i P(X_t^1,...,X_t^n))$ for an ordinary $B$-non-commutative polynomial $P$ (in the algebra generated by $B,X_1,...,X_n$).  Let us write $\delta_s$ the following (Malliavin) Derivation operator defined on $B$-non-commutative polynomials in $X_u^i$'s (as usual one can assume them algebraically free without loss of generality): 
$$\delta_s(P(X_{s_1}^{i_1},...,X_{s_n}^{i_n}))=\sum_{j}(\partial_{(j)}(P))(X_{s_1}^{i_1},...,X_{s_n}^{i_n}) (s\wedge s_j),$$
where $\partial_{(j)}$ is the $B-E_D$-free difference quotient in the $j$-th variable for $P$ (sending $X_{s_j}^{i_j}$ to $(1\o_D 1)_{i_j}$ having only an $i_j$-th non-zero component). Obviously, $\delta_t P(X_t^1,...,X_t^n)=t\partial P(X_t^1,...,X_t^n)$ so that it suffices to show:
$$\tau(\langle (1\o_D1)_i,\delta_s P(X_s^1,...,X_s^n)\rangle)-\tau(\Xi_s^iP_s)=0,$$
for $\Xi_s^i=X_s^i-X_0^i-\int_0^sdt\ \frac{t}{2} F_{\mathscr{D}_i V}(X_t^1,...,X_t^n)+\frac{s}{2}\mathscr{D}_i V(X_s^1,...,X_s^n),$ and any non-commutative polynomial $P_s=P(X_s^1,...,X_s^n)$. 
We will first prove using Ito formula a differential equation for the above differences.

Applying Ito formula, one gets ($\partial_{j}$ the ordinary difference quotient):
\begin{align*}P_t=P(X_t^1,...,X_t^n)&=P(X_0^1,...,X_0^n)+\int_0^tds\frac{1}{2}\Delta_V(P)(X_s^1,...,X_s^n)\\
&\qquad+\int_0^t \partial(P)(X_s^1,...,X_s^n)\#dS_s.\end{align*}
Let us write for short $\beta_s=\frac{1}{2}\Delta_V(P)(X_s^1,...,X_s^n)$.

Thus, let us compute likewise~: $$\tau(P_t(X_t^i-X_0^i))=\int_0^tds \tau(P_s (-\frac{1}{2}\mathscr{D}_i V(X_s))+\beta_s(X_s^i-X_0^i)+\langle 1\o_D1, \partial_i(P)(X_s^1,...,X_s^n)\rangle_{B\langle X\rangle}).$$

\begin{align*}\tau(P_t t\mathscr{D}_i V(X_t))&=\int_0^tds\  \tau(P_s \mathscr{D}_i V(X_s)+P_s s F_{\mathscr{D}_i V}(X_s^1,...,X_s^n)+\beta_s s\mathscr{D}_i V(X_s))\\ & +\int_0^tds\tau(\langle \partial(P^*)(X_s^1,...,X_s^n),\partial(s\mathscr{D}_i V(X_s) )\rangle_).\end{align*}

Thus \begin{eqnarray*}\tau(P_t\Xi_t^i)&=&\int_0^tds\left( \tau(\beta_s\Xi_s^i)+\tau(\langle 1\o_D1,\partial_i(P)(X_s^1,...,X_s^n)\rangle)\right)\\
&&
-\int_{0}^{t }ds \tau(\langle \partial(P^*)(X_s^1,...,X_s^n),\partial(\frac{s}{2}V_i(s,X_s) )\rangle).\end{eqnarray*}

Using similarly Ito's formula on tensor products:

\begin{align*}\tau(\langle (1\o_D1)_i,  \delta_t P\rangle)&=\int_0^tds \tau\big(\langle (1\o_D1), \partial_i (P)(X_s)\rangle_{L^2(B\langle X_s \rangle,E_D)})\\
& + \frac{s}{2}\tau(\langle (1\o_D1)_i,(\Delta_V\o 1+1\o \Delta_V)\partial P(X_s)\rangle\big)
\\& =\int_0^tds \tau(\langle (1\o_D1), \partial_i (P)(X_s)\rangle_{L^2(B\langle X_s \rangle,E_D)})\\
& + \tau(\langle (1\o_D1)_i,\delta_s\beta_s-\sum_j\partial_j(P(X_s))\#\frac{1}{2}\delta_s\mathscr{D}_j V(X_s)\rangle)
\end{align*}
where we used the elementary relation  applied to a polynomial $P$: $$(\Delta_V\o 1+1\o \Delta_V)\partial(.)
=\partial\Delta_V(.)-\sum_j\partial_j(.)\#\partial\mathscr{D}_j V.$$

But of course we can use the fundamental property for cyclic gradients $\partial_i\mathscr{D}_j V(X_s)=\rho(\partial_j\mathscr{D}_i V(X_s))=(\partial_j\mathscr{D}_i V(X_s))^*$ with the rotation $\rho(a\o b)= b\o a$ extended to cyclic Haagerup tensor products and using $V=V^*$. Thus, one gets:
\begin{align*}\tau(\langle (1\o_D1)_i,\sum_j\partial_j(P(X_s))\#\delta_s\mathscr{D}_j V(X_s)\rangle)=\sum_js\tau(\langle (1\o_D1),\partial_j(P(X_s))\#(\partial_j\mathscr{D}_i V(X_s))^*\rangle)
\end{align*}
Rewritten with the notation of Theorem \ref{Finite1}--\ref{Finite3} so that one can use our traciality assumption, this is 
\begin{align*}&\sum_js\langle e_D,\partial_j(P(X_s))\#(\partial_j\mathscr{D}_i V(X_s))^*\#e_D\rangle\\&=
\sum_js\langle e_D,E_{D'}(\partial_j(P(X_s)))\#(\partial_j\mathscr{D}_i V(X_s))^*\#e_D\rangle\\&=
\sum_js\langle e_D,(\partial_j\mathscr{D}_i V(X_s))^*\#E_{D'}(\partial_j(P(X_s)))\#e_D\rangle\\&=
\sum_js\langle (\partial_j\mathscr{D}_i V(X_s))\#e_D,\partial_j(P(X_s))\#e_D\rangle
\end{align*}
Note that we introduced in the second line the projection on the commutant using the weak-* continuity obtained in  Theorem \ref{finite2}.(2). In the next-to-last line, after using traciality, we used \eqref{ScalarProductTrace}. In the last line we removed the conditional expectation using the fact that $(\partial_j\mathscr{D}_i V(X_s))\#e_D$ commutes with $D$. Finally, we have $(\partial_j(P(X_s))\#e_D)^*=\partial_j(P^*(X_s))\#e_D$ and $[(\partial_j\mathscr{D}_i V(X_s))\#e_D]^*=(\partial_j\mathscr{D}_i V(X_s))\#e_D$ since $V=V^*$ and thus $$\langle (\partial_j\mathscr{D}_i V(X_s))\#e_D,\partial_j(P(X_s))\#e_D\rangle=\langle\partial_j(P^*(X_s))\#e_D, (\partial_j\mathscr{D}_i V(X_s))\#e_D\rangle.$$

We have thus obtained: \begin{align*}\tau&(\langle (1\o_D1)_i,\delta_t P(X_t^1,...,X_t^n)\rangle)=\int_0^t ds\tau(\langle 1\o_D1,\partial_i P(X_s^1,...,X_s^n)\rangle)\\
&+\int_0^tds\tau(\langle (1\o_D1)_i,\delta_s \beta_s\rangle)-\int_0^tds \tau(\langle \partial(P^*)(X_s^1,...,X_s^n),\delta_s V_i(s,X_s)\rangle)\end{align*}
Summing up, we have obtained our ``differential equation":\begin{equation}\label{diffConjVarSDE}\tau(P_t\Xi_t^i)-\tau(\langle S_i,\delta_t P(X_t^1,...,X_t^n)\rangle_{B\langle X_s \rangle})=\int_0^tds\ \tau(\beta_s\Xi_s^i)-\tau(\langle S_i,\delta_s \beta_s\rangle_{B\langle X_s \rangle}).\end{equation}

\begin{step} Case with $V\in B\langle X_1,..., X_n:D,R,\C\rangle$ of finite degree $p+1$ (i.e. ``usual" polynomial with all terms in the $\ell^1$ direct sum of order higher than $p+2$ vanishing).\end{step}

Let us write $$M_n:=n\max_i\|\mathscr{D}_i V\|_{B\langle X_1,..., X_n:D,1,\C\rangle}=En.$$ Let $p$ be the maximum degree of $\mathscr{D}_i V$. Let $R\geq \sup_{s\in [0,T],i}\|X_t^i\|$.

 Let $\tilde{M}_n:=M_n+2n(\frac{R^p}{R^p-1})^2=Dn$. Finally, let $\theta$ a time such that for all monomials $P$, all $t\leq \theta$ we have already established what we want (for instance at the beginning $\theta=0$): $$\tau(P_t\Xi_t^i)-\tau(\langle S_i,\delta_t P(X_t^1,...,X_t^n)\rangle_{B\langle X_s \rangle})=0.$$
 
 Let us show quickly using \eqref{diffConjVarSDE} that for $P$ monomial of degree less than $n=kp$ (with coefficient in extended Haagerup norm less than 1 i.e. of norm less that $1$ in $B\langle X_1,..., X_n:D,1,\C\rangle$), we have for $t\geq\theta$ (since by definition the left hand side is 0 before)~: \begin{align*}\tau&(P_t\Xi_t^i)-\tau(\langle S_i,\delta_t P(X_t^1,...,X_t^n)\rangle_{B\langle X_s \rangle})\\ &\leq \frac{(t-\theta)^l(C+pT)F^{l}(k+2l)^{2l+1}}{2^l(l!)^2}R^{(k+l)p}
 =:A_l(t,k),\end{align*}
  where $C=sup_{[0,T]}\|\Xi_t^i\|<\infty$ and $F=\max \left(\frac{p^2}{R-1},Ep\right)$.
  
   We prove this by induction on $l$. Initialization at $l=0$ is obvious by boundedness of $X_t$ by $R\geq 1$.
 
 To prove induction step, note that $\beta_s=\frac{1}{2}\Delta_V(P)(X_s^1,...,X_s^n)$ contain two types of terms. The term coming from the first order part is a finite sum monomials of degree less than $(k+1)p$. Each of these terms will be bounded by the induction Assumption at level $l$ by $A_l(s,k+1)$ times the norm of the coefficient in the extended Haagerup tensor product, which all sums up to $\max_i\|\mathscr{D}_i V\|_{B\langle X_1,..., X_n:D,1,\C\rangle}\leq M_n/n=E.$ Finally the number of sums due to derivation can always be crudely bounded by $n=kp$, the degree of $P$. We thus obtain a bound $FkA_l(s,k+1)$ for this first order term.

  The other terms come from the second order derivative, we have of course at most $n(n-1)/2$ pairs of terms selected by the derivative, but we have to pay attention to their degrees. For sure we have at most $n$ terms with a given space $l\leq n$ between the two $1\o 1$ inserted by the derivative, in that case the degree is at most $kp-l$ after taking the conditional expectation $E_D$, and we have a bound by $R^l$ to bound the coefficient induced by this conditional expectation (corresponding to the variables $X_s$ inside, below $E_D$). Let us gather terms by taking only into account the integer part $i$ of $l/p$. We have thus at most $np$ terms with such an integer part, all of degree at most $(k-i)p$, with $R^{ip}$ plus a factor $1, R, ...,R^{p-1}$ depending of the exact degree in the group. 
At the end one thus gets~:
  \begin{align*}\tau&(P_t\Xi_t^i)-\tau(\langle S_i,\delta_t P(X_t^1,...,X_t^n)\rangle)\\ &\leq\int_\theta^tds A_l(s,k+1)Fk+\sum_{{i=0}}^k A_l(s,i)np R^{p(k-i)}\frac{R^p-1}{R-1}
\\ &\leq\int_\theta^tds A_l(s,k+1)Fk+\sum_{i=0}^k A_l(s,i)k R^{p(k-i)}FR^p 
  .\end{align*}

We have just  used our induction Assumption and we reorder a bit our expression to factorize powers of R and replace $A_l$ by its value to get:  
    \begin{align*}\tau&(P_t\Xi_t^i)-\tau(\langle S_i,\delta_t P(X_t^1,...,X_t^n)\rangle_{B\langle X_s \rangle})
 \\ &\leq FkR^{p(k+1)}\int_\theta^tds \frac{A_l(s,k+1)}{R^{p(k+1)}}+\sum_{i=0}^k \frac{A_l(s,i)}{R^{pi}}
\\ &\leq FkR^{p(k+1+l)}\int_\theta^tds \sum_{i=0}^{k+1}(i+2l)^{2l+1} \frac{(s-\theta)^l(C+pT)F^{l}}{2^l(l!)^2}
  .\end{align*}
We can now use  an easy comparison to integral of a Riemann sum. $$\sum_{i=0}^{k+1}(i+2l)^{2l+1}\leq \sum_{i=1}^{k+1+2l}i^{2l+1}\leq \frac{(k+2+2l)^{2l+2}}{2l+2}.$$

  Computing the integral, we therefore proved:
  \begin{align*}\tau&(P_t\Xi_t^i)-\tau(\langle S_i,\delta_t P(X_t^1,...,X_t^n)\rangle_{B\langle X_s \rangle})
\\ &\leq FkR^{p(k+1+l)}\frac{(k+2+2l)^{2l+2}}{2l+2} \frac{(t-\theta)^{l+1}(C+pT)F^{l}}{2^l(l!)((l+1)!)}\leq A_{l+1}(t,k)
  .\end{align*}

  Let us finally estimate $$A_{l}(t,k)=2l(C+pT)R^{kp}\frac{((k/2l)+1)^{2l+1}(l)^{2l}}{(l!)^2}(4R^{p}F(t-\theta)/2)^l.$$
 Note that   $$((k/2l)+1)^{2l+1}\leq \exp((2l+1)k/2l)\leq \exp{2k}$$ and by Stirling's formula $$(l)^{2l}/(l!)^2\sim e^{2l}/(2\pi l)$$ we conclude that as soon as $4R^{p}F(t-\theta)e^2/2<1$, i.e. when $t-\theta<2/e^24R^{p}F$  (independent of $k$), $A_{l}(t,k)\to_{l\to\infty} 0$, so that one easily deduces by induction one can take $\theta=T$.

\begin{step} Case of general $V$.\end{step}
Take a sequence $V_n$ as in step 2 converging to $V$ in $C^{4}_c(A,R:B,D).$ Note that we can assume the $V_n$ to be $(c',R)$ h-convex for some $c'<c$. Let us write $X_t(V_n), X_t(V)$ the solutions given by Proposition \ref{ConvSDE}, and call $\Xi_t(V_n), \Xi_t(V)$ the formulas from step (1) and let us show that $$\sup_{t\in [0,T]}\max(\|X_t(V_n)- X_t(V)\|,\|\Xi_t(V_n)- \Xi_t(V)\|)\to_{n\to \infty}0.$$

This is roughly the same argument as in the previous subsection for continuity in $\alpha$. Note that
\begin{align*} 
&X_t(V_n)-X_t(V) =-\frac{1}{2}\int_0^t du [\mathscr{D} V_n-\mathscr{D} V](X_u(V))\\& -\frac{1}{2}\int_0^t du\left[\int_0^1d\beta\partial\mathscr{D} V_n (\beta X_u(V_n)+(1-\beta)X_u(V))\right]\# [X_u(V_n)-X_u(V)].
\end{align*}
Using the argument in Lemma \ref{lemhyp} with $\partial\mathscr{D} V_\alpha(X_u)$ replaced by  $$\left[\int_0^1d\beta\partial\mathscr{D} V_n (\beta X_u(V_n)+(1-\beta)X_u(V))\right]\geq c' Id,$$ with the positivity coming since our notion of positivity is a closed convex cone, one gets:
\begin{align*} 
&\|X_t(V_n)-X_t(V)\| \leq e^{-c't/2}\int_0^t due^{c'u/2} (\sum_i\|[\mathscr{D}_i V_n-\mathscr{D}_i V](X_u(V))\|^2)^{1/2}
\end{align*}
This converges uniformly on $[0,T]$ to $0$ when $n\to\infty$ using the corresponding limit $V_n\to V$ and the a priori bounds on the norm of the process $X_u(V)$ on $[0,T].$ (Doing this for small $T$ first, this in particular ensures a bound for $X_t(V_n)$ for $t$ huge enough without assuming the assumption of Proposition \ref{ConvSDE}(b) for $V_n$.)
The convergence of $\Xi(V_n)$ is then straightforward by the explicit formula. We can then take the limit in the conjugate variable equation to conclude.\end{proof}

\subsection{Examples of h-convex potentials} 
 \label{ConvexPotentialSection}

We first produce an elementary example in $1$ variable.

\begin{Lemma}\label{basicquartic}
If $v(X_1)=\mu\frac{X_1^2}{2}+\lambda \frac{X_1^3}{3}+\nu\frac{X_1^4}{4}\in \C\langle X_1,...,X_n\rangle\subset B_c\langle X_1,...,X_n;D,R,\C\rangle$ for $\nu>0,\lambda^2\leq 8\mu\nu/3$, then for any $B,D$, $v=v^*\in  C^{2}_c(A,R:B,D)$  is $(0,R)$-convex for any $R$. 
\end{Lemma} 
 
\begin{proof}From the computation on algebraic tensor products inside cyclic tensor products, in the proof of Proposition \ref{CyclicPermutations}, it is clear that $v(X_1)\in B_c\langle X_1,...,X_n;D,R,\C\rangle$.
Note that \begin{eqnarray*}H(X_1)&=&{{{   \partial_1}}}\mathscr D_1v\\
&=&\nu(X_1^2\otimes 1+X_1\otimes X_1+1\otimes X_1^2)+\lambda(X_1\o 1+ 1\o X_1)+\mu 1\o 1
\\&=&\left(X_1+\frac{\lambda}{2\nu}\right)^2\otimes \frac{\nu}{2}+ \frac{\nu}{2}\otimes\left(X_1+\frac{\lambda}{2\nu}\right)^2 + \frac{\nu}{2}\left(X_1\otimes 1+1\otimes X_1+\frac{\lambda}{2\nu}1\otimes 1\right)^2\\
&&+ \left(\mu- \frac{3\lambda^2}{8\nu}\right)1\otimes 1.\end{eqnarray*}
Thus let $B\subset (M,\tau)$, fix $X_1=X_1^*\in D'\cap M$  and let us observe that
$$e^{-t H(X_1)}=\sum_{k=0}^\infty\frac{(-t\nu)^k}{k!}e^{-t(\nu X_1^2+\lambda X_1+\mu/2)}(X_1^k\otimes X_1^k) e^{-t(\nu X_1^2+\lambda X_1+\mu/2)}$$
belongs to $ M\oh{D} M\subset M\oeh{D} M.$

Of course, the sum even converges in a projective tensor product, and we want to estimate its norm. Recall that 
$$M\oeh{D} M\subset CB_{M',M'}(D'\cap B(L^2(M)),B(L^2(M))\subset CB(B(L^2(M)),B(L^2(M))$$
 completely isometrically.

We now get an alternative integral formula. For convenience, we let 
$$Y_1=\left(X_1+\frac{\lambda}{2\nu}\right)\frac{\sqrt{\nu}}{\sqrt{2}}\,.$$ Using Cauchy product formula of absolutely converging series, one gets:
\begin{eqnarray*}e^{-t H(X_1)}&=&e^{-(\mu- \frac{3\lambda^2}{8\nu})t}
e^{-t Y_1^2}\sum_{k=0}^\infty \frac{[-t\nu (X_1\o 1+1\o X_1+\frac{\lambda}{2\nu})^2/2]^k}{k!}e^{-tY_1^2}\\&=&e^{-(\mu- \frac{3\lambda^2}{8\nu})t}
\int_{\R}d\sigma e^{-\sigma^2/2}e^{-tY_1^2}\sum_{k=0}^\infty \frac{[i\sqrt{t\nu}\sigma(X_1\o 1+1\o X_1+\frac{\lambda}{2\nu})]^k}{k!}e^{-tY_1^2}
\\
&=&\frac{e^{-(\mu- \frac{3\lambda^2}{8\nu})t}}{\sqrt{2\pi}}\int_{\R}d\sigma e^{-\sigma^2/2}e^{-tY_1^2}e^{i\sqrt{t\nu}\sigma(X_1\o 1+1\o X_1+\frac{\lambda}{2\nu})}e^{-tY_1^2}\\
&=&\frac{e^{-(\mu- \frac{3\lambda^2}{8\nu})t}}{\sqrt{2\pi}}\int_{\R}d\sigma e^{-\sigma^2/2}e^{-tY_1^2+i\sqrt{t\nu}\sigma (X_1+\frac{\lambda}{4\nu})}\otimes e^{-tY_1^2+i\sqrt{t\nu}\sigma (X_1+\frac{\lambda}{4\nu})},\end{eqnarray*}
where the second line is obtained using moments of standard gaussian  variables (and Fubini Theorem).
Using Hermite polynomials $H_n(x)=\frac{(-1)^n}{\sqrt{n!}}e^{x^2/2}\frac{d}{dx}e^{-x^2/2}$ as orthonormal basis, and $\xi_i\in L^2(M)$, one obtains by using the orthogonal decomposition in $L^2(d\gamma), d\gamma(\sigma) =e^{-\sigma^2/2}d\sigma/\sqrt{2\pi}$, 
$$\langle \xi_1,e^{-tY_1+i\sqrt{t\nu}\sigma (X_1+\frac{\lambda}{4\nu})}\xi_2\rangle =\sum_{n=0}^\infty H_n(\sigma)\langle \xi_1,c_n(X_1)\xi_2\rangle$$
which yields
\begin{align*}e^{-t H(X_1)}&=e^{-(\mu- \frac{3\lambda^2}{8\nu})t}\sum_{n=0}^\infty
c_n(X_1)\otimes 
c_n(X_1),\\
 c_n(X_1)&=\frac{1}{\sqrt{2\pi}}\int_{\R}d\sigma e^{-\sigma^2/2}H_n(\sigma)e^{-tY_1^2+i\sqrt{t\nu}\sigma (X_1+\frac{\lambda}{4\nu})}.\end{align*}
Indeed, to make this  identification  in $M\o_{eh}M=CB_{M',M'}(K(L^2(M),B(L^2(M)))$ (see \cite{BlecherPaulsen} for the equality), we first identify the two sides after evaluation on a finite rank operator, say in using the orthogonal decomposition recalled earlier  \begin{align*}
&\int_{\R}d\gamma(\sigma) \langle \xi_1,e^{-tY_1^2+i\sqrt{t\nu}\sigma (X_1+\frac{\lambda}{4\nu})}\xi_2\rangle \langle \xi_3, e^{-tY_1^2+i\sqrt{t\nu}\sigma (X_1+\frac{\lambda}{4\nu})}\xi_4\rangle \\
&=\sum_{n=0}^\infty
\langle \xi_1,c_n(X_1)\xi_2\rangle 
\langle \xi_3,c_n(X_1)\xi_4\rangle.\end{align*}
Then, if both sides extend to compact operators, one obtains the claimed equality. We already said the left hand side does (for instance by our previous bound
 on $e^{-t H(X_1)}$ obtained from the series expansion) and the right hand side will by our next bound {giving the contractivity property}.

Thus, for instance from \cite{M05}, when $\mu\ge \frac{3\lambda^2}{8\nu}$:
\begin{align*}||e^{-t H(X_1)}||_{M\oeh{D} M}&\leq  ||\sum_{n=0}^\infty c_n(X_1)c_n(X_1)^*||
\end{align*}

But note that for $\xi\in L^2(M)$, with $(e_j)_{j\in \N}$ an orthonormal basis of this space, we first get using Parseval equality and Tonelli Fubini Theorem to switch the sum over $j$:

\begin{align*}&\int_{\R}d\gamma(\sigma) \langle e^{-tY_1^2+i\sqrt{t\nu}\sigma (X_1+\frac{\lambda}{4\nu})}\xi,e^{-tY_1^2+i\sqrt{t\nu}\sigma (X_1+\frac{\lambda}{4\nu})}\xi\rangle
\\&=\sum_j\int_{\R}d\gamma(\sigma) \langle e^{-tY_1^2+i\sqrt{t\nu}\sigma (X_1+\frac{\lambda}{4\nu})}\xi, e_j\rangle
\langle e_j,e^{-tY_1^2+i\sqrt{t\nu}\sigma (X_1+\frac{\lambda}{4\nu})}\xi\rangle
\\&=\sum_j\sum_n
|\langle e_j,c_n(X_1)\xi\rangle|^2
\\&=\sum_n\sum_j
|\langle e_j,c_n(X_1)\xi\rangle|^2=\sum_n
\|c_n(X_1)\xi\|^2
=\langle \xi,\sum_{n=0}^\infty c_n(X_1)^*c_n(X_1)\xi\rangle
\end{align*}
where the third line is obtained by using Parseval equality again this time on $L^2(d\gamma)$, and again Fubini-Tonelli and Parseval.
Thus, we got, since $\nu>0$:
\begin{eqnarray*}
\sum_{n=0}^\infty c_n(X_1)^*c_n(X_1)&=&\frac{1}{\sqrt{2\pi}}\int_{\R}d\sigma e^{-\sigma^2/2}(e^{-tY_1^2+i\sqrt{t\nu}\sigma (X_1+\frac{\lambda}{4\nu})})^*e^{-tY_1^2+i\sqrt{t\nu}\sigma (X_1+\frac{\lambda}{4\nu})}\\
&=&\frac{1}{\sqrt{2\pi}}\int_{\R}d\sigma e^{-\sigma^2/2}e^{-2tY_1^2}\end{eqnarray*}
 is a contraction and so is $\sum_{n=0}^\infty c_n(X_1)c_n(X_1)^*$. Finally, from \eqref{finitesum}, it is easy to see in truncating the series that $\sigma( e^{-t H(X_1)})=e^{-t H(X_1)}$ and this concludes to :
$$||e^{-t H(X_1)}||_{M\oehc{D} M}\leq 1.$$
\end{proof} 
 
In order to deduce a more general example, we need to describe more explicitly the norm structure we put on $M_n(M\oehc{D} M)$ to obtain various contractive maps.

\begin{Lemma}\label{diagmult}
There is a completely contractive map $$\ell^\infty([\![1,n]\!], M\oehc{D} M)\to CB(\ell^2([\![1,n]\!],M^{\oehc{D}m}),\ell^2([\![1,n]\!],M^{\oehc{D}m}))$$ corresponding to action by diagonal matrices. Especially, there is a contractive diagonal embedding $(\ell^\infty([\![1,n]\!], M\oehc{D} M))\to M_n(M\oehc{D} M).$
\end{Lemma}
\begin{proof}
First recall that in \cite{PisierBook}, the operator space structure of $\ell^2([\![1,n]\!],M^{\oehc{D}m})$ is described as the interpolation of $\ell^\infty([\![1,n]\!])\o_{min}M^{\oehc{D}m}=\ell^\infty([\![1,n]\!])\o_{h}M^{\oehc{D}m}$ and $\ell^1([\![1,n]\!])\hat{\o}M^{\oehc{D}m}=\ell^1([\![1,n]\!])\o_{h}M^{\oehc{D}m}$ (the first equality comes from the fact both operator space product are injective and \cite[Lemma 9.2.4, Prop 9.3.1]{EffrosRuan} that imply the same result with $\ell^\infty([\![1,n]\!])$ replaced by $M_n(\C)$, the second equality reduces to the first one after taking duals, the computation of dual of Haagerup tensor product is known in this case from \cite[Cor 9.4.8]{EffrosRuan}  and for the projective tensor product see \cite[Prop 8.1.2, 8.1.8]{EffrosRuan}). From the interpolation result of Haagerup tensor products \cite[Th 5.22]{PisierBook}, one deduces the complete isometry $\ell^2([\![1,n]\!],M^{\oehc{D}m})=\ell^2_{oh}([\![1,n]\!])\o_{h}M^{\oehc{D}m}.$

We will start from this description to get our map. From the universal property of the projective tensor product (and agreement of Haagerup and extended Haagerup tensor products in the finite dimensional case), it suffices to get a canonical completely contractive map $$(\ell^\infty([\![1,n]\!])\o_{eh}M^{\oehc{D}2})\widehat{\o}(\ell^2_{oh}([\![1,n]\!])\o_{eh}M^{\oehc{D}m})\to \ell^2_{oh}([\![1,n]\!])\o_{eh}M^{\oehc{D}m}.$$
To reach this goal, we compose several known complete contractions. First we start with the shuffle map from \cite[Lemma 8]{dabsetup}:
\begin{align*}(\ell^\infty([\![1,n]\!])&\o_{h}M^{\oehc{D}2})\widehat{\o}(\ell^2_{oh}([\![1,n]\!])\o_{h}M^{\oehc{D}m})\\&\to \ell^2_{oh}([\![1,n]\!])\o_{h}\left((\ell^\infty([\![1,n]\!])\o_{h}M^{\oehc{D}2})\widehat{\o}M^{\oehc{D}m}\right)
\\&\to \ell^2_{oh}([\![1,n]\!])\o_{h}\ell^\infty([\![1,n]\!])\o_{h}\left(M^{\oehc{D}2}\widehat{\o}M^{\oehc{D}m}\right).\end{align*}

We compose this map with a canonical multiplication map $ \ell^2_{oh}([\![1,n]\!])\o_{h}\ell^\infty([\![1,n]\!])\to \ell^2_{oh}([\![1,n]\!])$. It is obtained by interpolation from the map $\ell^\infty([\![1,n]\!])\o_{h}\ell^2_{c}([\![1,n]\!])\to \ell^2_{c}([\![1,n]\!])$ from \cite[3.1.3, Prop 3.1.7]{BLM} and the symmetric map $\ell^2_{r}([\![1,n]\!])\o_{h}\ell^\infty([\![1,n]\!])\to \ell^2_{r}([\![1,n]\!])$ which we interpolate after noticing that $\ell^2_{c}([\![1,n]\!])\o_{h}\ell^\infty([\![1,n]\!])=\ell^2_{c}([\![1,n]\!])\o_{min}\ell^\infty([\![1,n]\!])=\ell^\infty([\![1,n]\!])\o_{min}\ell^2_{c}([\![1,n]\!])=\ell^\infty([\![1,n]\!])\o_{h}\ell^2_{c}([\![1,n]\!])$. This multiplication of course gives the expected diagonal matrix action.

The multiplication map we finally  want $M^{\oehc{D}2}\widehat{\o}M^{\oehc{D}m}\to M^{\oehc{D}m}$ is of course the one we built in Proposition \ref{CyclicPermutations} (1). By density of the algebraic tensor product, it suffices to get a contractivity on basic tensors. Since the target norm  is induced form $M^{\o_{eshc D}m}$, it suffices to get the contractivity with this target space. This decomposes in various contractivity for each flip (using the fonctoriality of nuclear tensor product). We thus have to see that 
$\#:M^{\oeh{D}2}\widehat{\o}(D'\cap M^{\oeh{D}m})\to M^{\oeh{D}m}$ and $\#_i:M^{\oeh{D}m}\widehat{\o}(D'\cap M^{\oeh{D}2})\to M^{\oeh{D}m}$ are complete contractions. This is obvious from complete contractivity of composition of $CB$ maps.
\end{proof}

\begin{Lemma}\label{fullquartic}
Let $A=(A_{i,j})\in M_n(\R) $ a positive matrix with  $A\geq cI_n$ and $(\lambda_{i,j})\in M_{n,k}(\R) , \mu\in [0,\infty[^k$, $\upsilon_j(x)=\nu_{j,2}\frac{x^2}{2}+\nu_{j,3} \frac{x^3}{3}+\nu_{j,4}\frac{x^4}{4}$ for $\nu_{j,4}>0,\nu_{j,3}^2\leq 8\nu_{j,2}\nu_{j,4}/3.$
Let 
$$V(X)=\sum_{j=1}^k\mu_j\upsilon_j\left(\sum_{i=1}^n\lambda_{i,j}X_i\right)
+\sum_{i,j=1}^nA_{i,j}X_iX_j\,.$$ 
Then,for any $B,D$,  $V(X)\in \C\langle X_1,...,X_n\rangle\subset B_c\langle X_1,...,X_n;D,R,\C\rangle,$
$V=V^*\in  C^{6}_c(A,R:B,D)$  is $(c,R)$-h-convex for any $R$. 

Moreover, let $P=P^*\in \C\langle u_1,...,u_n\rangle$ a $*$-polynomial in unitary variables, and define for $\varepsilon>0$
$$\mathcal{V}(X)=V(X)+\varepsilon P(\frac{\sqrt{-1}+X_1}{\sqrt{-1}-X_1}, \cdots,\frac{\sqrt{-1}+X_n}{\sqrt{-1}-X_n})\,.$$
Then, for any $R>0$ and any $c'\in [0,c)$, there exists $\varepsilon_R>0$ so that for $\varepsilon\in [-\varepsilon_R,\varepsilon_R]$,  
$W\in C^{6}_c(A,R:B,D)$  is $(c',R)$ h-convex.
\end{Lemma} 
 \begin{proof}
 From the additivity of  positivity, the positivity elements form a cone, so that it suffices to consider $k=1$ and even to show that $W(X)=\upsilon_1\left(\sum_{i=1}^n\lambda_{i,1}X_i\right)$ is (0,R) convex. But with the notation of the previous proof $$\partial_i\mathscr{D}_jW
 =\lambda_{j,1}\lambda_{i,1}H(\sum_{i=1}^n\lambda_{i,1}X_i).$$
 
 Let us call $P=\sum_i\lambda_{i,1}^2>0$. From the previous proof of Lemma \ref{basicquartic}, one deduces  $||e^{-t H(\sum_{i=1}^n\lambda_{i,1}X_i)}||_{M\oehc{D} M}\leq 1$. 
  Let us fix an orthogonal matrix $O$ with $O_{j,1}=\lambda_{j,1}/\sqrt{P}$ and write the matrix $A=(\partial_i\mathscr{D}_jW)_{j,i}$ as $A=O(O^* AO)O^*$. Note that $(O^*AO)_{j,i}=0$ except for $(O^*AO)_{1,1}=PH(\sum_{i=1}^n\lambda_{i,1}X_i).$ Thus $e^{-t A}=Oe^{-t (O^* AO)} O^*$. To conclude, note that $O,O^*$ are contractions in $M_n(M\oehc{D} M)$ since their action coincides with $O^\epsilon\o1$ on $\ell^2([\![1,n]\!],M^{\oehc{D}m})=\ell^2_{oh}([\![1,n]\!])\o_{h}M^{\oehc{D}m}.$ Finally, $e^{-t (O^* AO)}=Diag(e^{-tP H(\sum_{i=1}^n\lambda_{i,1}X_i)},1,...,1)$ and each term in the diagonal matrix is a contraction, so that one can apply Lemma \ref{diagmult} to conclude to $||e^{-t (O^* AO)}||_{M_n(M\oehc{D} M)}\leq 1.$
  
  We finally consider the case where the polynomial is pertubed.
 { In order to check that $\mathcal V\in C^{6}_c(A,R:B,D)$, since this space is obviously an algebra, it suffices to check $P_t(X)=\frac{1}{t\sqrt{-1}-X_1}\in C^{6}_c(A,R:B,D)$ for $t>0$. For $t$ large enough, a geometric series converging in $C^{6}_c(A,R:B,D)$ shows this. The set of such $t$ is thus non-empty, it is easy to check that $C^{6}_c(A,R:B,D)$ has an equivalent Banach algebra norm, then, a Neumann series gives the set of $t$ is open. It remains to see it is closed in $ ]0,\infty[$ to get the result by connectivity. An easy computation shows that $||P_t||_{6,0,A_R^n}\leq \sum_{k=0}^61/t^{k+1}$ as soon as we showed $P_t$ is in the space above, since $\partial^k_{(1,...,1)}P_t(X_1)=(k!) P_t(X_1)^{\otimes k+1}.$ When $t_n\to t>0$, and using 
 $$P_{t}(X_1)-P_s(X_1)=-P_{t}(X_1)(t-s)\sqrt{-1}P_s(X_1)$$
  one easily gets the convergence $||P_{t_n}-P_t||_{6,0,A_R^n}\to 0$ (in getting a Cauchy sequence and identifying the limit with $P_t$). It only remains to check the stated $h$-convexity. It suffices to take the coefficients of $P$ small enough so that  $b=(\partial_i\mathscr{D}_j(\mathcal{V}-V))_{j,i}$ has a norm  $||b||:=||b||_{M_n(M\oehc{D} M)}<c$ and in this case $c'=c-||b||$ is appropriate. Indeed, let $a= (\partial_i\mathscr{D}_j(V))_{j,i}$, we can use the Dyson series: $$e^{-t(a+b)}=e^{-ta}+\sum_{k=1}^\infty\int_0^tds_1\int_0^{s_1}ds_2\cdots\int_0^{s_{k-1}}ds_ke^{-(t-s_1)a}be^{-(s_1-s_2)a}\cdots be^{-(s_{k-1}-s_k)a}be^{-s_ka} ,$$
 and one obtains:
 $$||e^{-t(a+b)}||_{M_n(M\oehc{D} M)}\leq e^{-tc}+\sum_{k=1}^\infty\int_0^tds_1\int_0^{s_1}ds_2\cdots\int_0^{s_{k-1}}ds_ke^{-tc}||b||^k=e^{-t(c-||b||)}.
$$ 
 }
 \end{proof}

It remains to check the other assumptions on $\mathcal{V}$. We need variants of results from \cite[Th 3.4]{guionnet-edouard:combRM} and \cite{alice-shlyakhtenko-freeDiffusions}.

\begin{Proposition}\label{ConcentrationNorm}Let $V$ be of the form of $\mathcal{V}$ in Lemma \ref{fullquartic}, and $(c,R)$ h-convex for all $R>0$. Consider the probability on $(M_N(\C)_{sa})^{n}$ given (for some normalization constant $Z_{V,N}$) by :
$$\mu_{V,N}(dx)=\frac{1}{Z_{V,N}}e^{-N Tr(V(X_1,..,X_n))}dLeb_{(M_N(\C)_{sa})^{n}}(dX)$$
 Let $A_{1}^N,..., A_{n}^N$ of law $\mu_{V,N}$ (on a same probability space), we have a constant $C>0$ such that a.s.: $$\limsup _{N\to \infty}\max_i||A_{i}^N||_\infty\leq C,$$
 and for $K\in\N^*$\begin{equation}\label{IntegralUnifBound}
 \limsup _{N\to \infty}E_{\mu_{V,N}}(1_{\{||A_{i,l}^N||_\infty\geq C\}}\frac{1}{N}Tr((A_{i,l}^N)^{2K}))=0.
 \end{equation}
Moreover, 
for any non-commutative polynomial $P\in \C\langle X\rangle\otimes_{alg}\C\langle X\rangle$
 $$\lim_{N\to \infty}\left|E_{\mu_{V,N}}(\frac{1}{N^2}(Tr\otimes Tr)(P(A_1,...,A_k))-\frac{1}{N^2}\left[(E_{\mu_{V,N}}\circ Tr)\otimes (E_{\mu_{V,N}}\circ Tr)\right](P)\right|=0 .$$
  \end{Proposition}
\begin{proof} The proof is identical to \cite[Th 3.4]{guionnet-edouard:combRM} since $X_1,\ldots,X_n\mapsto \Tr{V}(X_1,\ldots,X_n)$ is convex, with Hessian bounded below by $c$, on the space of Hermitian matrices. In fact, one can check that any $h$-convex function ${V}$ satisfies this property.
\end{proof}

\begin{Theorem}\label{SDV}
Let $V$ be of the form of $\mathcal{V}$ in Lemma \ref{fullquartic}, and $(0,R)$ h-convex for all $R$.
Consider, the law absolutely continuous with respect to the law $P_{G^N}$ of GUE $G^N$:$$d\mu_{V,N}(X)=\frac{1}{Z_{V,N}}e^{-N Tr(V(X_1,..,X_n))}dP_{G^N}(X).$$
 Then $E_{\mu_{V,N}}\circ\frac{1}{N}Tr$ converges  in law to a tracial state $\tau_V$ which is the law of self-adjoint  variables $X(V)$ (of norm bounded by some $R$) and the unique solution with this property to the equation $(SD_V)$, for $G(X)=\tau_{X}(V)$:
$$ \forall P\in \C\langle X_1,...,X_n\rangle , \ (\tau_V\otimes \tau_V)(\partial_{X_i}(P))=\tau_V(X_iP)+\tau_V(\mathcal{D}_i V P).$$
Moreover, there is a solution on $\R_+$ given by Proposition \ref{ConvSDE} with potential $V_0+V$ 
and $\tau_V$ is the unique stationary $R^\omega$-embeddable 
trace for this free SDE. 
\end{Theorem}
Note that the $R^\omega$-embeddability assumption in the uniqueness is not really necessary but we stick to that case in order to be consistent and use our previous setting.
\begin{proof}  \setcounter{Step}{0}
\begin{step}Defining limit variables in a von Neumann algebra ultraproduct.\end{step}
Consider a non-principal ultrafilter $\omega$ on $\N$ and the tracial von Neumann algebra ultraproducts $\mathcal{L}^\omega=L^2(M_N(L^\infty(\mu_{V,N}))^\omega$, $\mathcal{M}^\omega=M_N(L^\infty(\mu_{{{V}},N}))^\omega$. Considering $A_1^N,...,A_n^N$ the canonical hermitian variables in $M_N(L^\infty(\mu_{V,N}))$, we know from \eqref{IntegralUnifBound} that $||A_i^N1_{\{||A_i^N||\leq C\}}- A_i^N||_2\to 0$ so that $X_i^\omega=(A_i^N)^\omega=(A_i^N1_{\{||A_i^N||\leq C\}})^\omega\in \mathcal{M}^\omega$. We thus also fix $B_i^N=A_i^N1_{\{||A_i^N||\leq C\}}.$ 

This gives a tracial state $\tau_{X^\omega}.$  Let us check that any such state satisfies $(SD_V)$. 

\begin{step}Showing $(SD_V)$.\end{step}

As in \cite{guionnet-edouard:combRM}, we use an integration by parts formula on $\mu_{V,N}$ which gives $\forall P\in \C\langle X_1,...,X_m\rangle$:
\begin{align*}\ E_{\mu_{V,N}}&\left(\frac{1}{N}Tr(A_i^NP(A_1^N,...,A_m^N))+ \frac{1}{N}Tr(N\nabla_{A_i^N}G(A_1^N,...,A_m^N)P(A_1^N,...,A_m^N))\right)\\&=E_{\mu_{V,N}}\left((\frac{1}{N}Tr\otimes\frac{1}{N}Tr)(\partial_{X_i}P)(A_1^N,...,A_m^N)\right)\end{align*}
and the second concentration result in Proposition \ref{ConcentrationNorm} implies that the right hand side converges when $N\to \omega$ to $(\tau_{X^\omega}\otimes \tau_{X^\omega})(\partial_i(P))$.
One thus obtains the relation in taking of limit to $\omega$ of the integration by parts relation. 
Moreover, note that this implies $\tau_{X^\omega}$ has finite Fisher information.

\begin{step}Properties and use of the SDE.\end{step}
Let $X_0=X^\omega$ or a $R^\omega$-embeddable solution of $(SD_V)$, which ensures $X_0\in A_{R/3,App}^n$ in the scalar case $B=\C$.
The application of our Proposition \ref{ConvSDE} thus gives a unique solution $X_t(X_0)$ on $[0,\infty[$ solving
$$X_t(X_0)=X_0-\frac{1}{2}\int_0^t \mathcal D V(X_s(X_0)) ds-\frac{1}{2}\int_0^t X_s(X_0) ds +S_t.$$
Considering another solution starting at $Y_0$, one obtains: $$||X_t(X_0)-X_t(Y_0)||_2^2\leq e^{-ct}||X_0-Y_0||_2^2.$$

Then exponential decay implies that the laws  $\tau_{X_t(X^\omega)}$ and
$\tau_{X_t(X^{\omega'})}$ are arbitrarily close for $t\to\infty$ and since they are equal to $\tau_{X^\omega}$ and
$\tau_{X^{\omega'}}$ by stationarity, one deduces that $X^\omega$ have the same law for any ultrafilter. Similarly, $(SD_V)$ has a unique $R^\omega$-embeddable solution and the exponential decay implies a stationary state for the SDE is unique too.

\begin{step}Conclusion on the limit of $E_{\mu_{V,N}}\circ\tau_.$.\end{step}
The law $E_{\mu_{V,N}}\circ\frac{1}{N}Tr$ is close to $E_{\mu_{V,N}}\circ\tau_{B^N}$ for $N$ large enough and this second law lies in the compact set $\mathcal{S}_C^n$ (tracial state space of the universal free product $C([-C,C])^{*n}$ with the weak-* topology) and from the result on ultrafilter limits the sequence has a unique limit point there (any such limit point being a $\tau_{X^\omega}$). We thus deduce by compactness the claimed convergence.
\end{proof}

\begin{Corollary}\label{corfullquartic}
Let $V,V+W$ be of the form of $\mathcal{V}$ in Lemma \ref{fullquartic}, and thus  $(c,R)$ h-convex for all $R$ and some $c>0$. Then they satisfy Assumption \ref{thehyphyphyp}.
\end{Corollary}
\begin{proof}
The application of the previous Theorem gives existence of solution of $(SD_{V_\alpha}), \alpha\in [0,1]$ which is $R^\omega$-embeddable or equivalently $L(F_\infty)^\omega$-embeddable which is a reformulation of $A_{R,UltraApp}^n$ in the case $B=\C$. Everything else comes from Lemma \ref{fullquartic} and stability of $(c,R)$ h-convexity under taking convex combinations.
\end{proof}

\end{document}